\pdfoutput=1

\documentclass[10pt]{article}

\usepackage{xspace}
\usepackage{url}
\usepackage{mathtools}
\usepackage{amssymb}
\usepackage{amsthm}
\usepackage{empheq}
\usepackage{latexsym}
\usepackage{enumitem}
\usepackage{eurosym}
\usepackage{dsfont}
\usepackage{appendix}
\usepackage{color} 
\usepackage{frcursive}
\usepackage[utf8]{inputenc}
\usepackage[T1]{fontenc}
\usepackage{geometry}
\usepackage{multirow}
\usepackage{todonotes}
\usepackage{lmodern}
\usepackage{anyfontsize}
\usepackage{stmaryrd}
\usepackage[english]{babel}
\usepackage[english=british]{csquotes}
\usepackage{mathtools}
\usepackage{accents}
\usepackage{color}
\usepackage{comment}
\usepackage{bm}

\allowdisplaybreaks

\usepackage{natbib}
\setcitestyle{numbers,open={[},close={]}}

\makeatletter \@addtoreset{equation}{section}

\newtheorem{theorem}{Theorem}[section]
\newtheorem{assumption}[theorem]{Assumption}

\newtheorem{lemma}[theorem]{Lemma}
\newtheorem{proposition}[theorem]{Proposition}

\newtheorem{definition}[theorem]{Definition}
\newtheorem{remark}[theorem]{Remark}

\usepackage[unicode]{hyperref}

\usepackage[capitalize,nameinlink,noabbrev]{cleveref}

\crefname{assumption}   {assumption}   {assumptions}
\Crefname{assumption}   {Assumption}   {Assumptions}

\crefname{assumptio}    {assumption}   {assumptions}   
\Crefname{assumptio}    {Assumption}   {Assumptions}

\crefname{assumptionalt}{assumption}   {assumptions}
\Crefname{assumptionalt}{Assumption}   {Assumptions}

\crefname{lemma}        {lemma}        {lemmas}
\Crefname{lemma}        {Lemma}        {Lemmas}

\crefname{proposition}  {proposition}  {propositions}
\Crefname{proposition}  {Proposition}  {Propositions}

\crefname{definition}   {definition}   {definitions}
\Crefname{definition}   {Definition}   {Definitions}

\AddToHook{env/assumption/begin}{\crefalias{theorem}{assumption}}
\AddToHook{env/lemma/begin}{\crefalias{theorem}{lemma}}
\AddToHook{env/proposition/begin}{\crefalias{theorem}{proposition}}
\AddToHook{env/definition/begin}{\crefalias{theorem}{definition}}
\definecolor{red}{rgb}{0.7,0.15,0.15}
\definecolor{green}{rgb}{0,0.5,0}
\definecolor{blue}{rgb}{0,0,0.7}
\hypersetup{colorlinks, linkcolor={red}, citecolor={green}, urlcolor={blue}}

\DeclareUnicodeCharacter{014D}{\=o}
\newcommand{\smallertext}[1]{\text{\fontsize{5}{5}\selectfont$#1$}}
\newcommand{\smalltext}[1]{\text{\fontsize{4}{4}\selectfont$#1$}}
\newcommand{\tinytext}[1]{\text{\fontsize{3}{3}\selectfont$#1$}}

\def\balpha{\bm{\alpha}}

\newcommand\cA{\mathcal A}
\newcommand\cB{\mathcal B}
\newcommand\cC{\mathcal C}

\newcommand\cE{\mathcal E}
\newcommand\cF{\mathcal F}
\newcommand\cG{\mathcal G}

\newcommand\cI{\mathcal I}

\newcommand\cL{\mathcal L}
\newcommand\cM{\mathcal M}

\newcommand\cO{\mathcal O}
\newcommand\cP{\mathcal P}

\newcommand\cT{\mathcal T}

\newcommand\cW{\mathcal W}

\newcommand\PP{\mathbb P}

\def \E{\mathbb{E}}
\def \F{\mathbb{F}}
\def \G{\mathbb{G}}
\def \H{\mathbb{H}}

\def \L{\mathbb{L}}
\def \M{\mathbb{M}}
\def \N{\mathbb{N}}
\def \P{\mathbb{P}}
\def \Q{\mathbb{Q}}
\def \R{\mathbb{R}}

\def \X{\mathbb{X}}
\def \Y{\mathbb{Y}}
\def \Z{\mathbb{Z}}

\newcommand{\x}{\mathbf{x}}

\def\eps{\varepsilon}
\def\d{\mathrm{d}}
\DeclareMathOperator*{\argmax}{arg\,max}

\DeclareMathOperator*{\esssup}{ess\,sup}

\usepackage{graphicx}
\usepackage{subfig}

\begin{document}

\title{Variance strikes back: sub-game--perfect Nash equilibria in time‑inconsistent $N$‑player games, and their mean‑field sequel}

\author{Dylan {\sc Possama\"{i}} \footnote{ETH Z\"urich, Department of Mathematics, Switzerland, dylan.possamai@math.ethz.ch. This author gratefully acknowledges support from the SNF project MINT 205121-21981.}\and Chiara {\sc Rossato} \footnote{ETH Z\"urich, Department of Mathematics, Switzerland, chiara.rossato@math.ethz.ch.} }

\date{\today}

\maketitle

\begin{abstract}
We investigate a time-inconsistent, non-Markovian finite-player game in continuous time, where each player's objective functional depends non-linearly on the expected value of the state process. As a result, the classical Bellman optimality principle no longer applies. To address this, we adopt a two-layer game-theoretic framework and seek sub-game--perfect Nash equilibria both at the intra-personal level, which accounts for time inconsistency, and at the inter-personal level, which captures strategic interactions among players. We first characterise sub-game--perfect Nash equilibria and the corresponding value processes of all players through a system of coupled backward stochastic differential equations. We then analyse the mean-field counterpart and its sub-game--perfect mean-field equilibria, described by a system of McKean-Vlasov backward stochastic differential equations. Building on this representation, we finally prove the convergence of sub-game--perfect Nash equilibria and their corresponding value processes in the $N$-player game to their mean-field counterparts.

\bigskip
\noindent{\bf Key words:} time inconsistency, mean-variance, non-Markovian stochastic games, mean-field games, sub-game--perfect equilibria. \bigskip
\end{abstract}

\allowdisplaybreaks
\section{Introduction}\label{introTI}

This paper studies time-inconsistent stochastic differential games in which each player's objective is characterised by a non-linear function of the expected value of their outcome. Such non-linearity captures risk-sensitive behaviour toward uncertain outcomes, implying that each player exhibits mean--variance-type preferences. Since the seminal work of \citeauthor*{markowitz1959portfolio} \cite{markowitz1959portfolio}, mean--variance preferences have played a central role in the economics and finance literature and have seen renewed attention over the past two decades.

\medskip
A central challenge in handling preferences characterised by a non-linear function of the expectation of future outcomes, is that the classical dynamic programming approach cannot be applied directly, as the iterated-expectation property (or the so-called Bellman optimality principle) no longer holds when one insists on optimising the objective functional. Consequently, this leads to dynamic inconsistency, since the optimal action may depend on the point in time at which the decision is made, and the agent may therefore have an incentive to deviate from their initial plan. \citeauthor*{strotz1955myopia} \cite{strotz1955myopia} was the first to formulate the conceptual framework for analysing time inconsistency, emphasising in his seminal work on time-consistent planning that an agent should select `the best plan among those that (they) will actually follow.' When the agent recognises that their present self and future selves may have conflicting preferences, then \cite{strotz1955myopia} describes two different approaches that can be followed: the strategy of pre-commitment and the strategy of consistent planning. In the former, the agent makes a decision that is optimal today and commits to it, simply disregarding the fact that, at a later point in time, such a choice could no longer be optimal. In the latter, the agent compromises by choosing the current action that is optimal in light of the inter-temporal conflict, adopting a game-theoretic perspective on time inconsistency. We can therefore interpret the problem as a non-cooperative game in which the agent's selves at different points in time are considered as players, each seeking what is referred to in the literature as an intra-personal equilibrium, or equivalently, a sub-game--perfect equilibrium, a term first introduced by \citeauthor*{selten1965spieltheoretische} \cite{selten1965spieltheoretische}.

\medskip
If \citeauthor*{strotz1955myopia} pioneered the analysis of time-inconsistent behaviour, this strand of research was subsequently developed by \citeauthor*{pollak1968consistent} \cite{pollak1968consistent} and \citeauthor*{peleg1973existence} \cite{peleg1973existence}, who formalised the idea by modelling time-inconsistent problems as non-cooperative games among an agent's successive selves. In this framework, each temporal self determines the control at their corresponding time, optimising their own objective while anticipating future re-optimisation. This perspective has been adopted in many recent works on time-inconsistency arising from non-exponential discounting. Notably, \citeauthor*{ekeland2006being} \cite{ekeland2006being,ekeland2008equilibrium,ekeland2010golden} introduced the first rigorous definition of sub-game--perfect equilibrium for deterministic control problems. \citeauthor*{ekeland2008investment} \cite{ekeland2008investment} extended these ideas to continuous-time stochastic models analysing how non-exponential discounting affects investment--consumption policies in a Merton-like problem. Regarding the time-inconsistency associated with mean--variance preferences, \citeauthor*{basak2010dynamic} \cite{basak2010dynamic} were among the first to apply and extend the consistent planning approach to mean--variance portfolio optimisation, a research direction further advanced by \citeauthor*{wang2011continuous} \cite{wang2011continuous}, \citeauthor*{hu2012time} \cite{hu2012time,hu2017time}, \citeauthor*{wei2013markowitz} \cite{wei2013markowitz}, \citeauthor*{he2013optimal} \cite{he2013optimal}, \citeauthor*{czichowsky2013time} \cite{czichowsky2013time}, \citeauthor*{bensoussan2014time} \cite{bensoussan2014time}, \citeauthor*{bjork2014mean} \cite{bjork2014mean}, \citeauthor*{kronborg2015inconsistent} \cite{kronborg2015inconsistent} and \citeauthor*{djehiche2016characterization} \cite{djehiche2016characterization}. Mean--variance criteria have also been studied in insurance-related problems, as in \citeauthor*{li2015time} \cite{li2015time} and \citeauthor*{zeng2016robust} \cite{zeng2016robust}. It is also worth noting the works of \citeauthor*{bjork2017time} \cite{bjork2016time2,bjork2017time,bjork2021time} (see also \citeauthor*{lindensjo2019regular} \cite{lindensjo2019regular} and the survey by \citeauthor*{he2022who} \cite{he2022who}), who developed a comprehensive framework for addressing a broad class of time-inconsistent stochastic control problems in continuous time within the Markovian setting. In the non-Markovian setting, \citeauthor*{hernandez2023me} \cite{hernandez2023me} provided a rigorous proof of an extended dynamic programming principle and fully characterised the time-inconsistent problem through a system of backward stochastic differential equations (BSDEs).

\medskip
In a consistent-planning perspective, an agent optimises their decisions while accounting for intra-personal conflict, and therefore by correctly anticipating the actions of their selves in the future. A strategy that resolves this internal dynamic is called an intra-personal equilibrium and it has been extensively studied in the literature. In discrete time, the notion of intra-personal equilibrium is widely agreed upon and provides a mathematical formulation of \citeauthor*{strotz1955myopia}'s ideas. However, in continuous time, several alternative definitions have been proposed to capture the subtleties of temporal consistency. The most widely adopted formulation is the first-order approximation approach, known as weak intra-personal equilibrium and pioneered by \cite{ekeland2006being}. However, this definition does not guarantee that the equilibrium corresponds to an optimum of the payoff function, as it may merely identify a stationary point, and consequently, the agent may still have an incentive to deviate from a given weak equilibrium strategy. To overcome this limitation, \citeauthor*{huang2021strong} \cite{huang2021strong} introduced the notion of strong intra-personal equilibrium in the context of an infinite-time stochastic control problem, where an agent controls the generator of a time-homogeneous, continuous-time, finite-state Markov chain. \citeauthor*{he2021equilibrium} \cite{he2021equilibrium} showed that strong equilibrium strategies do not always exist. Motivated by this non-existence result, they suggested the concept of regular intra-personal equilibrium, which they showed to be stronger than the weak intra-personal equilibrium, and provided a sufficient condition under which these two notions coincide. The notion of intra-personal equilibrium is extended to the non-Markovian setting in \cite{hernandez2023me}, where it is defined as a strategy from which the agent has no incentive to deviate over a short period of time unless such a deviation yields an incremental reward positively proportional to the duration of that period, resembling the definition of weak equilibrium in the Markovian setting.

\medskip
While the works discussed so far focus on time-inconsistent control problems involving a single agent, we are particularly interested in extending the analysis to multiple interacting players, and ultimately in studying the continuum limit with mean-field interactions. When several players exhibit time-inconsistent preferences, the resulting analysis involves two interdependent levels of strategic interaction. At the inter-personal level, each agent's control affects the objectives of the others, leading to the classical notion of Nash equilibrium among players. At the intra-personal level, each agent faces a dynamic game against their future selves, induced by their time-inconsistent objectives. Each sophisticated agent therefore seeks a sub-game--perfect Nash equilibrium, that is, a strategy that constitutes an intra-personal equilibrium internally and a Nash equilibrium externally. Equivalently, a sub-game--perfect Nash equilibrium is a Nash equilibrium across both levels simultaneously: no agent has an incentive to deviate given the strategies of the others (inter-personal equilibrium), and no temporal self of any agent wishes to deviate given the continuation of their own strategy (intra-personal equilibrium). Despite its relevance, existing literature on time-inconsistent problems has primarily focused on the single-agent case. Only a few works consider the multi-agent setting, where two intertwined levels of strategic interaction arise. In the context of time-inconsistent contract theory, \citeauthor*{cetemen2021renegotiation} \cite{cetemen2021renegotiation} studied a contracting problem in which the principal exhibits non-exponential discounting, while \citeauthor*{hernandez2024time} \cite{hernandez2024time} analysed the case of a time-inconsistent sophisticated agent whose reward is determined by the solution of a backward stochastic Volterra integral equation. Focusing on non-cooperative interactions, \citeauthor*{huang2022time} \cite{huang2022time} investigated a non--zero-sum Dynkin game in discrete time under non-exponential discounting, while \citeauthor*{lazrak2023present} \cite{lazrak2023present} analysed a linear--quadratic zero-sum game in which the two players discount performance at a non-constant rate when lobbying for investment in a wind turbine farm. \citeauthor*{huang2023partial} \cite{huang2023partial} studied a mean--variance portfolio optimisation game in which a finite number of investors determine their strategies under both full and partial information. To the best of our knowledge, only two works have explored this direction for large-population systems. \citeauthor*{wang2023time} \cite{wang2023time} considered a time-inconsistent linear--quadratic mean-field game, while \citeauthor*{bayraktar2025time} \cite{bayraktar2025time} analysed the convergence of equilibria in $N$-player games toward a mean-field game equilibrium in a discrete-time Markov decision game with non-exponential discounting.

\medskip
In this paper, we develop a general framework for non-cooperative stochastic games with finitely many players, formulated under the weak formulation, in which the drift of each player's state process depends on the states and controls of all agents. Each player faces a non-Markovian stochastic control problem that is time-inconsistent due to the presence of a non-linear function of the expected value of future outcomes in their objective functional. We adopt the perspective of sophisticated agents, who are aware of the time-inconsistent nature of their preferences and anticipate future re-optimisation. Within this setting, we introduce a notion of sub-game--perfect Nash equilibrium, adapting the definition in \cite{hernandez2023me} to our problem, which involves two intertwined layers of strategic interaction: the intra-personal equilibrium, ensuring consistency among an agent's temporal selves, and the inter-personal equilibrium, capturing the mutual influence among different agents. Extending the results of \cite[Section 7]{hernandez2023me} and \cite[Section 2.4]{hernandez2021general}, our equilibrium notion allows us to prove an extended dynamic programming principle. Consequently, each equilibrium strategy constitutes a Nash equilibrium across all players and is time-consistent, in the sense that neither players nor players' future selves have an incentive to deviate from the strategy given its continuation. Leveraging this extended dynamic programming principle, we provide a rigorous BSDE characterisation of the sub-game--perfect Nash equilibria and the associated value processes in the $N$-player game. Specifically, the equilibria correspond to the fixed-points of a vector-valued Hamiltonian, and the resulting system is a $3N$-dimensional system of BSDEs with quadratic growth, whose well-posedness is both necessary and sufficient to characterise the time-inconsistent multi-agent problem. 

\medskip
In the case of a symmetric game, as the number of players increases, the dimension of the BSDE system associated with the $N$-player game grows accordingly, making it increasingly challenging to find a solution. For this reason, our second objective is to study the mean-field game and the corresponding sub-game--perfect mean-field equilibria, and to analyse the convergence problem, in the spirit of \citeauthor*{lauriere2022backward} \cite{lauriere2022backward} and \citeauthor*{possamai2025non} \cite{possamai2025non}. We adapt the equilibrium notion from the multi-agent problem to the mean-field setting and, by leveraging once again the extended dynamic programming principle, we characterise the sub-game--perfect mean-field equilibria and the associated value processes through a three-dimensional system of coupled McKean--Vlasov BSDEs with quadratic growth. Building on this representation, and under the assumption of uniqueness of the sub-game--perfect mean-field equilibrium, we establish that the BSDE system describing the $N$-player game converges to the McKean-Vlasov BSDEs associated with the mean-field game, by relying on propagation of chaos results for forward--backward stochastic differential equations (FBSDEs). To the best of our knowledge, this is the first result in the literature establishing such a convergence in a time-inconsistent setting. Crucially, the extended version of Bellman's optimality principle is what makes this possible. It enables us to carry out the entire convergence analysis directly at the BSDE level, providing a natural route that remains fully compatible with the weak formulation and, moreover, one that yields explicit non-asymptotic convergence rates. In contrast, the two other standard approaches to convergence in time-consistent mean-field theory face obstacles in this time-inconsistent setting. Analytic methods based on the master equation, such as the work of \citeauthor*{cardaliaguet2019master} \cite{cardaliaguet2019master} in the time-consistent case, would require a full PDE analysis of a master equation, which has not yet been derived for time-inconsistent problems tackled under the consistent planning approach. On the other hand, the probabilistic, compactness-based method pioneered by \citeauthor*{lacker2020convergence} \cite{lacker2020convergence} relies on a relaxed notion of equilibrium, which also has not been fully elucidated so far, and, in addition, would not allow one to derive explicit rates.

\medskip
The rest of the paper is organised as follows. \Cref{section:ProbSetting} introduces the probabilistic framework and establishes the notation used throughout the paper for both the multi-agent and mean-field games. \Cref{section:gameFinitePlayer} formulates the multi-agent game, defines the notion of sub-game--perfect Nash equilibrium, and provides the corresponding BSDE characterisation in \Cref{necessity_toBSDE} and \Cref{sufficiency_toBSDE}, which relies on the extended dynamic programming principle. \Cref{section:only2Players_sameResulT} presents a particular example with only two players in which the solution coincides with that of the corresponding McKean--Vlasov differential game, while \Cref{subsection:zeroSumGame} examines the special case of a two-player zero-sum game, highlighting the resulting simplifications of the associated BSDE system in this setting. \Cref{meanFieldGame} provides a complete description of the mean-field problem, establishing the BSDE characterisation by presenting the necessary condition in \Cref{necessity_toBSDE_meanFieldGame} and the sufficient condition in \Cref{sufficiency_toBSDE_meanFieldGame}. Finally, \Cref{section:convRes_mainResult} presents the convergence of the sub-game--perfect Nash equilibria and the associated value processes to their mean-field counterparts, as stated in \Cref{theorem:convergenceTheorem}, and includes a representative example to illustrate the proof in a simple setting.

\medskip
{\small\textbf{Notation.} 
Let $\N$ be the set of non-negative integers, $\N^\star$ the set of positive integers, $\R_\smallertext{+}$ the non-negative real line, and $\R^\star_\smallertext{+}$ the positive real line. For $(a,b) \in [-\infty, +\infty]^2$, we write $a \vee b \coloneqq \max\{a,b\}$ and $a \wedge b \coloneqq \min\{a,b\}$. Fix $p \in \N^\star$; for $(a,b) \in \R^p \times \R^p$, let $a \cdot b$ denote the inner product with the associated Euclidean norm $\|\cdot\|$. When $p=1$, we use $|\cdot|$ to denote the modulus. Given a Polish space $(E,d_E)$, for every vector $e \in E^p$, we define $e^{\smallertext{-}i} \in E^{p\smallertext{-}1}$ as the vector obtained by removing the $i$-th coordinate of $e$, and $\tilde{e} \otimes_i e^{\smallertext{-}i} \in E^p$ as the vector whose $i$-th coordinate is equal to $\tilde{e}$, for any $(i,\tilde{e}) \in \{1,\ldots,p\} \times E$. These notations extend to matrices as well. When considering elements with an upper index $N \in \N$, we write $e^{i,N}$ (resp. $e^{N,\smallertext{-}i}$) instead of $(e^N)^i$ (resp. $(e^N)^{\smallertext{-}i}$). For $(m,n) \in \N^\star \times \N^\star$, let $E^{m \times n}$ be the space of $m \times n$ matrices with $E$-valued entries. For $M \in E^{m \times n}$, we denote its transpose by $M^\top$, and if $M \in E^{m \times m}$, we denote its trace by $\mathrm{Tr}[M]$. The spectral norm of $M$ is denoted by $\|M\|$.

\medskip
We denote by $\cP(E)$ the set of all probability measures on the measurable
space $(E,\cB(E))$, where $\cB(E)$ is the Borel $\sigma$-algebra of $E$. The set $\cP(E)$ is endowed with the topology induced by the weak convergence of measures. For any $k \in \N^\star$, we denote by $\cP_k(E)$ the subset of $\cP(E)$ consisting of probability measures with finite $k$-th moment. The set $\cP_k(E)$ is equipped with the $k$-Wasserstein distance, denoted by $\cW_k$. Given a vector $e \in E^p$, we define the empirical measure associated to $e$ as $L^p(e) \coloneqq (1/p) \sum_{\ell=1}^p \delta_{e^\smalltext{\ell}}$, where $\delta_{e^\smalltext{\ell}}$ denotes the Dirac measure at the coordinate $e^\ell$.

\medskip
For $(p,q) \in \N^\star \times \N$, we set $\cC_p^q(E)$ as the space of functions from $E$ to $\R^p$ which are at least $q$ times continuously differentiable. If $q=1$, we simplify the notation to $\cC_p(E) \coloneqq \cC_p^1(E)$ and $\cC_{p,b}(E) \coloneqq \cC_{p,b}^1(E)$. For a given time horizon $T \in \R^\star_\smallertext{+}$, we suppress the dependence on $E$ when $E = [0, T]$. For any $f \in \cC_{p}$, we define $\|f\|_\infty \coloneqq \sup_{t \in [0,T]} \|f(t)\|$. Given another Polish space $(A,d_A)$, for the product space $\tilde{E} = \cC_p \times A$, we define the metric $d_{\tilde{E}}$ as $d_{\tilde{E}}((f,a),(\tilde{f},\tilde{a})) \coloneqq (\|f-\tilde{f}\|_\infty^2+d^2_A(a,\tilde{a}))^{1/2}$.

\medskip
Consider a filtered probability space $(\Omega, \cF, \F \coloneqq (\cF_t)_{t \in [0,T]}, \P)$. Given a random variable $\xi$, we define $\E^\P[\xi] \coloneqq \E^\P[\xi \vee 0] - \E^\P[(-\xi) \vee 0]$, with the convention that $+\infty - \infty = -\infty$. We denote by $\F^{\P_{\smalltext{+}}}$ the $\P$-augmentation of $\F$, and by $\rm{Prog}(\F)$ the progressive $\sigma$-algebra on $\Omega\times[0,T]$. For $s \in [0,T]$ and $t\in [s,T]$, we write $\cT_{s,t}(\F)$ for the set of $[s,t]$-valued $\F$--stopping times. Given two processes $\alpha$ and $\beta$, valued in the same space, we denote their concatenation $\alpha \otimes_t \beta \coloneqq \alpha \mathbf{1}_{[0,t)} + \beta \mathbf{1}_{[t,T]}$, where $t \in [0,T]$. Let $M$ be an $(\F,\P)$--local-martingale in the sense of \citeauthor*{jacod2003limit} \cite[Definition I.1.45]{jacod2003limit} with continuous trajectories $\P$--a.s., we denote its stochastic exponential by $\cE(M) \coloneqq \exp(M- [M]/2)$. For $p \in \N^\star$, we introduce the spaces $\H^2(\F,\P,\R^p)$ and $\H^2_{\mathrm{loc}}(\F,\P,\R^p)$ of, respectively, $\R^p$-valued, $\F$-predictable processes such that
\begin{align*}
\E^\P \bigg[ \int_0^T \|Z_t\|^2 \d t \bigg]  < +\infty,\; \text{respectively,}\; \int_0^T \|Z_t\|^2 \d t<+\infty,\; \P\text{\rm--a.s.}
\end{align*}

}
\section{Probabilistic setting}\label{section:ProbSetting}

Before introducing the stochastic basis on which we define the stochastic differential game, we fix $(N,m,d) \in \N^\star \times \N^\star \times \N^\star$, where $N$ represents the number of players, $m$ is the dimension of the state process of each player, and $d$ is the dimension of the Brownian motions driving these state processes. We work on a fixed probability space $(\Omega,\cF,\P)$, where $\Omega$ is a Polish space and $\cF \coloneqq \cB(\Omega)$ is its Borel $\sigma$-algebra. We consider a sequence of $\P$-independent, $\R^d$-valued Brownian motions $(W^i)_{i \in \N^{\smalltext{\star}}} = ((W^i_t)_{t \in [0,T]})_{i \in \N^{\smalltext{\star}}}$. For each $i \in \N^\star$, let $\G^i = (\cG^i_t)_{t \in [0,T]}$ denote the natural filtration generated by $W^i$, where $\cG^i_t \coloneqq \sigma(W^i_s: s \in [0,t])$. In addition, we introduce a family of $\R^m$-valued random variables $(X^i_0)_{i \in \N^{\smalltext{\star}}}$, $\P$-independent of the family of Brownian motions $(W^i)_{i \in \N^{\smalltext{\star}}}$. For each $i \in \N^\star$, we define the enlarged filtration $\F^i = (\cF^i_t)_{t \in [0,T]}$ by $\cF^i_t \coloneqq \sigma(\cG^i_{t} \cup \sigma(X^i_0))$. Finally, we define the joint filtration $\F_N = (\cF_{N,t})_{t \in [0,T]}$ by $\cF_{N,t} \coloneqq \bigvee_{i = 1}^N \cF^i_t$.

\medskip
A natural question is whether each $W^i$ remains a Brownian motion with respect to the enlarged filtrations $\F^i$ and $\F_N$. The answer is affirmative. This follows from L\'evy's characterisation of Brownian motion (see, for instance, \citeauthor*{weizsaecker1990stochastic} \cite[Theorem 9.1.1]{weizsaecker1990stochastic}) since the process $W^i$ remains a martingale with respect to the enlarged filtrations $\F^i$ and $\F_N$, as established by \citeauthor*{grigorian2023enlargement} \cite[Theorem 2]{grigorian2023enlargement}, and its quadratic variation process does not depend on the filtration. A further question is whether the martingale representation property is preserved under such initial enlargement. The answer is again affirmative. 

\begin{lemma}\label{lemma:martingaleRepresentation}
Let $M$ be an $(\F^i,\P)$-martingale $($respectively, an $(\F_N,\P)$-martingale$)$. There exists a unique process $Z \in \H^2_{\mathrm{loc}}(\F^i,\P,\R^d)$ $($respectively, $(Z^\ell)_{\ell \in \{1,\ldots,N\}}$ with each $Z^\ell \in \H^2_{\mathrm{loc}}(\F_N,\P,\R^d))$ such that
\begin{equation*}
M_{t} = M_0 + \int_0^t Z_s \cdot \d W^i_s \; \bigg(\text{\rm respectively,} \; M_{t} = M_0 + \int_0^t \sum_{\ell =1}^N Z^\ell_s \cdot \d W^\ell_s\bigg),\; t \in [0,T],\; \P \text{\rm --a.s.}
\end{equation*}
Moreover, this property is preserved under an equivalent change of measure, see for instance {\rm\cite[Theorem III.5.24]{jacod2003limit}}.
\end{lemma}
To the best of our knowledge, this property has not been shown without assuming the usual conditions on the filtration. For completeness, we therefore provide a proof within our framework, which we defer to the \Cref{appendix:martingaleRep} for readability. In particular, \Cref{lemma:martingaleRepresentation} implies that every $(\F^i,\P)$-martingale (respectively, every $(\F_N,\P)$-martingale) admits a $\P$-modification that is right-continuous and $\P\text{--a.s.}$ continuous.

\medskip
For each $i \in \N^\star$, we introduce a Borel-measurable function $\sigma^i:[0,T] \times \cC_m \longrightarrow \R^{m \times d}$. We assume that the stochastic differential equation ({\rm SDE})
\begin{equation}\label{eq:XstrongSolSDE}
X^i_t = X^i_0 + \int_0^t \sigma^i_s(X^i_{\cdot \land s}) \d W^i_s, \; t \in [0,T], \; \P\text{\rm--a.s.},
\end{equation}
admits a unique strong solution on $(\Omega,\cF,\F^i,\P)$. It is well known that existence and strong uniqueness hold, for example, when the function $\sigma^i$ is locally Lipschitz-continuous in its spatial (path) variable, uniformly in time, see \cite[Theorem 11.1.1]{weizsaecker1990stochastic}. Furthermore, the family $(X^i)_{i \in \N^{\smalltext{\star}}}$ consists of mutually $\P$-independent processes provided that the initial conditions $(X^i_0)_{i \in \N^{\smalltext{\star}}}$ are $\P$-independent, and each $X^i$ is right-continuous with $\P$--a.s. continuous paths, see \cite[Lemma 4.3.5]{weizsaecker1990stochastic}.

\medskip
In the context of the $N$-player game, we write $\X^N \coloneqq (X^1,\ldots,X^N)$. Let $(A,d_A)$ be a non-empty compact Polish space. We define the set of admissible strategies $\cA^N_N$ as the collection of $\F_N$-predictable, $A^N$-valued processes $\balpha \coloneqq (\alpha^{1}, \ldots, \alpha^{N})$. Similarly, let $\cA_N$ and $\cA_N^{N\smallertext{-}1}$ denote the sets of $\F_N$-predictable, $A$-valued and $A^{N-1}$-valued processes, respectively. For each $i \in \{1,\ldots,N\}$, we introduce the drift function
\begin{align*}
b^i:[0,T] \times \cC_m \times \cP_2(\cC_{m} \times A) \times A \longrightarrow \R^d,
\end{align*}
which is assumed to be bounded and Borel-measurable in all its arguments. For each $\balpha \in \cA^N_N$, we then define the probability measure $\P^{\balpha,N}$ on $(\Omega,\cF)$ by
\begin{equation}\label{eq:changeMeasureNplayerGame}
\frac{\d \P^{\balpha,N}}{\d \P} \coloneqq \cE\Bigg( \int_0^\cdot \sum_{\ell =1}^N b^\ell_t\big(X^\ell_{\cdot \land t},L^N\big(\X^N_{\cdot \land t},\balpha_t\big),\alpha^{\ell}_t\big) \cdot \d W^\ell_t \Bigg)_T.
\end{equation}
We recall the notation $L^N(\X^N,\balpha)$ for the empirical measure associated with the $(\R^{d} \times A)^N$-valued process $(X^N,\balpha)$, as introduced at end of the introduction. Particularly, for each $i\in\{1,\ldots,N\}$, we have
\begin{equation}\label{eq:dinXchangedMeasureNplayerGame}
X^i_t=X^i_0+\int_0^t \sigma^i_s(X^i_{\cdot \land s}) b^i_s\big(X^i_{\cdot \land s},L^N\big(\X^N_{\cdot \land s},\balpha_s\big),\alpha^i_s\big) \d s + \int_0^t \sigma^i_s(X^i_{\cdot \land s}) \d \big(W_s^{\balpha,N}\big)^{i}, \; t\in[0,T], \; \P\text{\rm--a.s.},
\end{equation}
where, by Girsanov's theorem, the process
\begin{align*}
\big(W^{\balpha,N}_t\big)^{i} \coloneqq W^i_t-\int_0^t b^i_s\big(X^i_{\cdot \land s},L^N\big(\X^N_{\cdot \land s},\balpha_s\big),\alpha^i_s\big) \d s, \; t \in [0,T],
\end{align*}
is an $(\F_N,\P^{\balpha,N})$--Brownian motion.

\medskip
When describing the mean-field game, our main objective is to analyse the convergence of the $N$-player game to its mean-field counterpart. To this end, we consider a symmetric $N$-player game. Specifically, we assume that the functions characterising the $N$-player game are independent of the player index $i \in \{1,\ldots, N\}$, meaning that they are identical for all players. Finally, we give special attention to player $1$, whom we refer to as the representative agent. We introduce the set $\mathfrak{P}$ of Borel-measurable functions $[0,T] \ni t \longmapsto \xi_t \in \cP_2(\cC_m \times A)$ and the set $\cA_1$ of admissible strategies, \emph{i.e.}, the collection of $\F^1$-predictable, $A$-valued processes, to specify the probability measure $\P^{\alpha,\xi}$ on $(\Omega,\cF)$ by
\begin{equation}\label{eq:changeMeasureMeanField}
\frac{\d \P^{\alpha,\xi}}{\d \P} \coloneqq \cE\bigg( \int_0^\cdot b_t\big(X^1_{\cdot \land t},\xi_t,\alpha_t\big) \cdot \d W^1_t \bigg)_T, \; (\alpha,\xi) \in \cA_1 \times \mathfrak{P}.
\end{equation}
We have that
\begin{align*}
X^1_t &= X^1_0+\int_0^t \sigma_s(X^1_{\cdot \land s}) b_s\big(X^1_{\cdot \land s},\xi_s,\alpha_s\big) \d s + \int_0^t \sigma_s(X^1_{\cdot \land s}) \d \big(W_s^{\alpha,\xi}\big)^1 \\
&\coloneqq X^1_0+\int_0^t \sigma_s(X^1_{\cdot \land s}) b_s\big(X^1_{\cdot \land s},\xi_s,\alpha_s\big) \d s + \int_0^t \sigma_s(X^1_{\cdot \land s}) \d \bigg( W^1_s - \int_0^s b_r\big(X^1_{\cdot \land r},\xi_r,\alpha_r\big) \d r \bigg), \; t\in[0,T], \; \P\text{\rm--a.s.}
\end{align*}

\subsection{Regular conditional probability distributions}\label{section:rcpd}

Let $\cM(\Omega)$ denote the set of all probability measures on the measurable space $(\Omega,\cF)$, which is also a Polish space under the weak convergence topology. We consider the probability space $(\Omega, \cF,\M)$ for some $\M \in \cM(\Omega)$, together with a filtration $\F \coloneqq (\cF_t)_{t \in [0,T]}$. In the following, $\M$ will denote an arbitrary probability measure introduced above, either specified by $\balpha \in \cA_N^N$ or by $(\alpha,\xi) \in \cA_1 \times \mathfrak{P}$, and we will work with the filtrations $\F_N$ and $\F^1$. Since $\cF$ is the Borel $\sigma$-algebra of $\Omega$, there exists a family of regular conditional probability distributions (r.c.p.d.s) $(\M_\omega^\tau)_{\omega \in \Omega}$ of $\M$ with respect to $\cF_\tau$, for any stopping time $\tau \in \cT_{0,T}(\F)$ (see \citeauthor*{stroock1997multidimensional} \cite[Theorem 1.1.6 and Theorem 1.1.8]{stroock1997multidimensional}). That is
\begin{enumerate}[label={$(\roman*)$}]
\item\label{firstrcpd} for each $A \in \cF_\tau$ and $B \in \cF$, the function $\omega \longmapsto \M_\omega^\tau(B)$ is $\cF_\tau$-measurable, and $\M[A \cap B] = \int_A \M_\omega^\tau(B) \M(\d \omega)$;
\item\label{secondrcpd} $\M_\omega^\tau\big[[\omega]_\tau\big] = 1$, $\text{\rm for}\;\M\text{\rm--a.e.}\;\omega\in\Omega$, where $[\omega]_\tau \coloneqq \bigcap \{ A \in \cF: A \in \cF_\tau \; \text{and} \; \omega \in A\}$.
\end{enumerate}
Moreover, by \cite[Corollary 1.1.7]{stroock1997multidimensional}, for any $\M$-integrable random variable $X$ on $(\Omega, \cF,\M)$ it holds that
\begin{align*}
\E^{\M_\smalltext{\omega}^\smalltext{\tau}}[X] = \E^\M[X|\cF_\tau](\omega), \; \text{\rm for}\;\M\text{\rm--a.e.}\;\omega\in\Omega.
\end{align*}
Note that for each stopping time $\tau\in \cT_{0,T}(\F)$, the family $(\M_\omega^\tau)_{\omega\in\Omega}$ is uniquely determined by \ref{firstrcpd}-\ref{firstrcpd} only up to a $\P$-null set.

\medskip
Given a stopping time $\tau \in \cT_{0,T}(\F_N)$ and the probability measure $\P^{\balpha,N}$ on $(\Omega,\cF)$ introduced in \Cref{eq:changeMeasureNplayerGame}, we denote by $(\P^{\balpha,N,\tau}_\omega)_{\omega \in \Omega}$ the family of r.c.p.d.s of $\P^{\balpha,N}$ given $\cF_{N,\tau}$. By \cite[Theorem 1.2.10]{stroock1997multidimensional}, it follows that for any $i \in \{1,\ldots,N\}$, and for $\P^{\balpha,N}\text{\rm--a.e.} \; \omega\in\Omega$
\begin{align*}
\begin{split}
X^i_t &= X^i_{\tau(\omega)}(\omega) + \int_{\tau(\omega)}^t \sigma^i_s(X^i_{\cdot \land s}) b^i_s\big(X^i_{\cdot \land s},L^N(\X^N_{\cdot \land s},\balpha_s),\alpha^i_s\big) \d s + \int_{\tau(\omega)}^t \sigma^i_s(X^i_{\cdot \land s}) \d \big(W_s^{\balpha,N,\tau,\omega}\big)^i, \; t\in[\tau(\omega),T], \; \P^{\balpha,N,\tau}_\omega\text{\rm--a.s.} ,
\end{split}
\end{align*}
where 
\begin{align*}
W^{\balpha,N,\tau,\omega}_t \coloneqq W^{\balpha,N}_t - W^{\balpha,N}_{t \land \tau(\omega)}, \; t \in [\tau(\omega),T],
\end{align*}
is an $(\F_N,\P^{\balpha,N,\tau}_\omega)$--Brownian motion. Similarly, for a stopping time $\tau \in \cT_{0,T}(\F^1)$ we define $(\P^{\alpha,1,\tau}_\omega)_{\omega \in \Omega}$ as the family of r.c.p.d.s of $\P^{\alpha,\xi}$, introduced in \Cref{eq:changeMeasureMeanField}, given $\cF^1_{\tau}$. Consequently, we have for $\P^{\alpha}\text{\rm--a.e.} \; \omega\in\Omega$
\begin{align*}
\begin{split}
X^1_t = X^1_{\tau(\omega)}(\omega) + \int_{\tau(\omega)}^t \sigma_s(X^1_{\cdot \land s}) b_s\big(X^1_{\cdot \land s},\xi_s,\alpha_s\big) \d s + \int_{\tau(\omega)}^t \sigma_s(X^1_{\cdot \land s}) \d W_s^{\alpha,1,\tau,\omega}, \; t\in[\tau(\omega),T], \; \P^{\alpha,1,\tau}_\omega\text{\rm--a.s.},
\end{split}
\end{align*}
where 
\begin{align*}
W^{\alpha,1,\tau,\omega}_t \coloneqq \big(W^{\alpha,\xi}_t\big)^1 - \big(W^{\alpha,\xi}_{t \land \tau(\omega)}\big)^1, \; t \in [\tau(\omega),T],
\end{align*}
is an $\big(\F^1,\P^{\alpha,1,\tau}_\omega\big)$--Brownian motion.

\section{The finitely many player game}\label{section:gameFinitePlayer}

After introducing the underlying probability space and the regular conditional probabilities, we now describe the game under consideration. We study an $N$-player game in which the payoff of player $i$, $i \in \{1,\ldots ,N\}$, given that the other players follow the strategy $\balpha^{N,\smallertext{-}i} \in \cA^{N-1}_N$, is defined as
\begin{align}\label{align:criteriumNplayerGame}
\notag J^i\big(t,\omega,\alpha;\balpha^{N,\smallertext{-}i}\big) &\coloneqq \E^{\P^{\smalltext{\alpha} \smalltext{\otimes}_\tinytext{i} \smalltext{\balpha}^{\tinytext{N}\tinytext{,}\tinytext{-}\tinytext{i}}\smalltext{,}\smalltext{N}\smalltext{,}\smalltext{t}}_\smalltext{\omega}} \bigg[ \int_t^T f^i_s\big(X^i_{\cdot \land s}, L^N\big(\X^N_{\cdot \land s},(\alpha \otimes_i \balpha^{N,\smallertext{-}i})_s\big), \alpha_s\big) \d s + g^i\big(X^i_{\cdot \land T}, L^N\big(\X^N_{\cdot \land T}\big)\big) \bigg] \\
&\; \quad + G^i\Big( \E^{\P^{\smalltext{\alpha} \smalltext{\otimes}_\tinytext{i} \smalltext{\balpha}^{\tinytext{N}\tinytext{,}\tinytext{-}\tinytext{i}}\smalltext{,}\smalltext{N}\smalltext{,}\smalltext{t}}_\smalltext{\omega}} \big[ \varphi^i_1(X^i_{\cdot \land T}) \big] ,  \E^{\P^{\smalltext{\alpha} \smalltext{\otimes}_\tinytext{i} \smalltext{\balpha}^{\tinytext{N}\tinytext{,}\tinytext{-}\tinytext{i}}\smalltext{,}\smalltext{N}\smalltext{,}\smalltext{t}}_\smalltext{\omega}} \big[ \varphi^i_2\big(L^N(\X^N_{\cdot \land T})\big) \big] \Big), \; (t,\omega,\alpha) \in [0,T] \times \Omega \times \cA_N,
\end{align}
where the functions $f^i: [0,T] \times \cC_m \times \cP_2(\cC_{m} \times A) \times A \longrightarrow \R$, $g^i: \cC_m \times \cP_2(\cC_{m}) \longrightarrow \R$, $G^i: \R \times \R \longrightarrow \R$, $\varphi^i_1: \cC_m \longrightarrow \R$, and $\varphi^i_2: \cP_2(\cC_{m}) \longrightarrow \R$ are all assumed to be Borel-measurable with respect to all their arguments. 

\medskip
Our objective is to characterise the Nash equilibria of the stochastic differential game just introduced, that is, to identify admissible strategies $\balpha^N \in \cA_N^N$ such that no player $i \in \{1,\ldots,N\}$ can improve their outcome by unilaterally deviating. Owing to the non-linear dependence on the mean in the payoff function $J^i$, as given in \eqref{align:criteriumNplayerGame}, each agent faces a time-inconsistent control problem. We assume that every agent is what is known in the literature as sophisticated---or ‘thrifty,’ as originally defined in \cite{strotz1955myopia}---and accordingly, we assume they adopt consistent planning strategies following the game-theoretic approach introduced by \cite{strotz1955myopia} and later formalised by \cite{ekeland2006being}, anticipating the behaviours of their future selves. As a consequence, each player faces an internal inter-temporal conflict and seeks a strategy that all of their future selves would consistently implement over time. Thus, each player competes not only against the other $N-1$ players but also against a continuum of their own future selves. This naturally leads to the notion of `games embedded in an $N$-player game', a complexity we address using the two-level game-theoretic framework presented by \cite{huang2022time}. At the intra-personal game level, each agent searches for time-consistent strategies, while at the inter-personal game level, each agent selects the best such strategy in response to the strategies of the other players. To formalise this, we introduce the concept of a \emph{sub-game--perfect Nash equilibrium}, extending the definition in \cite[Definition 2.6]{hernandez2023me} to the multi-agent setting.

\begin{definition}\label{def:NashEquilibrium}
Let $\hat\balpha^N \in \cA^N_N$, and $\varepsilon >0$. We define
\begin{align*}
&\ell_\varepsilon \coloneqq \inf \Big\{ \ell >0 : \exists (i,t,\alpha) \in \{1,\ldots,N\} \times [0,T] \times \cA_N, \;
\P \big[ J^i\big(t,\cdot,\hat\alpha^{i,N};\hat\balpha^{N,\smallertext{-}i}\big) < J^i\big(t,\cdot,\alpha\otimes_{t+\ell}\hat\alpha^{i,N};\hat\balpha^{N,\smallertext{-}i}\big) - \varepsilon \ell \big] >0 \Big\}.
\end{align*}
We say that $\hat\balpha^N$ is a {\rm sub-game--perfect Nash equilibrium} if for any $\varepsilon >0$, it holds that $\ell_\varepsilon >0$. We denote by $\mathcal{N}\!\mathcal{A}_{s,N}$ the set of all {\rm sub-game--perfect Nash equilibria}.
\end{definition}

\begin{remark}
\begin{enumerate}[label={$(\roman*)$}]
\item By construction, for any $\alpha \in \cA_N$ and $\ell > 0$, we have that $\alpha\otimes_{t+\ell}\hat\alpha^{i,N} \in \cA_N$. Moreover, for every $\cF_{N,T}$-measurable random variable $\eta$, we observe that
\begin{align*}
\E^{\Q^{\smalltext{t}}_\smalltext{\omega}} [ \eta ] = \E^{\P^{\smalltext{t}}_\smalltext{\omega}} \bigg[ \cE\bigg( \int_t^\cdot h_s \cdot \d W_s^i \bigg)_T \eta \bigg],
\end{align*}
where $\Q$ is the probability measure defined via the stochastic exponential $\d \Q / \d \P \coloneqq \cE\big( \int_0^\cdot h_s \cdot \d W^i_s \big)_T$. As a consequence, by the definition of the payoff function, we observe that $J^i(t,\omega,\alpha;\hat\balpha^{N,\smallertext{-}i})$ depends on the strategy $\alpha \in \cA_N$ only through its values on the interval $[t,T]$. In particular, if we define $\alpha^{i,t,\ell} \coloneqq \alpha \mathbf{1}_{[t,t+\ell)} + \hat\alpha^{i,N} \mathbf{1}_{[t+\ell, T]}$, it holds that $J^i(t,\omega,\alpha\otimes_{t+\ell}\hat\alpha^{i,N};\hat\balpha^{N,\smallertext{-}i}) = J^i(t,\omega,\alpha^{i,t,\ell};\hat\balpha^{N,\smallertext{-}i})$ for $\P\text{\rm --a.e.}\;\omega \in \Omega$;
\item As already discussed in the introduction, there are several definitions of \emph{intra-personal equilibrium} in the time-inconsistent literature. Firstly, we note we cannot directly apply the notion of \emph{strong intra-personal equilibrium} from {\rm\cite[Definition 2.2]{huang2021strong}} to our multi-agent setting, as {\rm\cite[Proposition 4.9]{he2021equilibrium}} provides a mean--variance problem for which no \emph{strong intra-personal equilibrium} exists. Instead, we adopt an extension of {\rm\cite[Definition 2.6]{hernandez2023me}} because it is well suited to establishing an extended dynamic programming principle. As the authors themselves explain in {\rm\cite[Section 3.1]{hernandez2023me}}, following the consistent planning approach, each sophisticated player must select a strategy that coordinates with their future selves, thereby yielding a time-consistent game under equilibrium, from which a dynamic programming principle naturally follows. It is important to highlight that we require each agent to choose a strategy that reconciles with all their future selves. Consequently, if a {\rm sub-game--perfect Nash equilibrium} $\hat\balpha^N \in \cA^N_N$ exists, then for every $\varepsilon >0$ there exists $\ell \in  (0,\ell_\varepsilon)$ such that, for all $(i,t,\alpha) \in \{1,\ldots,N\} \times [0,T] \times \cA_N$, it holds that 
\begin{align}\label{align:condEllEquilibrium}
J^i\big(t,\omega,\hat\alpha^{i,N};\hat\balpha^{N,\smallertext{-}i}\big) - J^i\big(t,\omega,\alpha\otimes_{t+\ell}\hat\alpha^{i,N};\hat\balpha^{N,\smallertext{-}i}\big) \geq - \varepsilon \ell, \; \P \text{\rm--a.e.} \; \omega \in \Omega.
\end{align}
It is important that the above condition holds for all $\ell < \ell_\varepsilon$, rather than merely along a sequence as is the case for a {\rm weak intra-personal equilibrium}. As shown in {\rm \Cref{appendix:martingaleRep}}, this local property is fundamental for proving the extended dynamic programming principle; see {\rm\cite[Section 3.1]{hernandez2023me}} for further discussion.
\item It is worth noting that, in contrast to {\rm\cite[Definition 2.6]{hernandez2023me}}, the quantifier $`$there exists $t \in [0,T]$' in our definition appears outside the probability. This plays a crucial role in the characterisation of both the value functions and the equilibria through the \emph{BSDE} system \eqref{BSDE_NplayerGame}, or equivalently \eqref{align:BSDEalphaMaxHam}. In particular, as shown in the proof of {\rm\Cref{sufficiency_toBSDE}}, this is essential when demonstrating that the well-posedness of the {\rm BSDE} system is sufficient to ensure the existence of a {\rm sub-game--perfect Nash equilibrium} $\hat\balpha^N$ as a maximiser of the Hamiltonian. Given any $\varepsilon >0$, we can construct $\ell_\varepsilon >0$ such that for any $(\ell,i,t,\alpha) \in (0,\ell_\varepsilon) \times \{1,\ldots,N\} \times [0,T] \times \cA_N$ the condition stated in {\rm\Cref{align:condEllEquilibrium}} is satisfied. However, this does not imply that for every $(\ell,i,\alpha) \in (0,\ell_\varepsilon) \times \{1,\ldots,N\} \times \cA_N$, the following holds
\begin{align*}
J^i\big(t,\omega,\hat\alpha^{i,N};\hat\balpha^{N,\smallertext{-}i}\big) - J^i\big(t,\omega,\alpha\otimes_{t+\ell}\hat\alpha^{i,N};\hat\balpha^{N,\smallertext{-}i}\big) \geq - \varepsilon \ell, \; \text{for any} \; t \in [0,T], \; \P \text{\rm--a.e.} \; \omega \in \Omega.
\end{align*}
This is due to the fact that the payoff function $J^i$, for each $i \in \{1,\ldots,N\}$, is defined via {\rm r.c.p.d.s,} as in \eqref{align:criteriumNplayerGame}, and each {\rm r.c.p.d.} is uniquely defined up to $\P$--null sets. As a result, there is no a priori guarantee that it possesses any regularity with respect to the time $t \in [0,T]$.
\end{enumerate}

\end{remark}

\subsection{BSDE characterisation of sub-game--perfect Nash equilibria}\label{section:BSDEchar_Nplayer}

Since the problem \eqref{align:criteriumNplayerGame} is time-inconsistent, Bellman’s optimality principle does not apply in this setting. Nevertheless, following \cite{hernandez2023me}, one can formulate what we refer to as the `extended dynamic programming principle' to overcome this difficulty. Throughout this section, we assume that there exists a sub-game--perfect Nash equilibrium $\hat\balpha^N \in \mathcal{N}\!\mathcal{A}_{s,N}$, although it may not be unique. We then consider the processes
\begin{align}\label{align:MiNi_starDef}
M_t^{i,\star,N} \coloneqq \E^{\P^{\smalltext{\hat\balpha}^\tinytext{N}\smalltext{,}\smalltext{N}\smalltext{,}\smalltext{t}}_\smalltext{\cdot}} \big[ \varphi^i_1(X^i_{\cdot \land T}) \big], \; \text{and} \; N_t^{i,\star,N} \coloneqq \E^{\P^{\smalltext{\hat\balpha}^\tinytext{N}\smalltext{,}\smalltext{N}\smalltext{,}\smalltext{t}}_\smalltext{\cdot}} \big[ \varphi^i_2\big(L^N\big(\X^N_{\cdot \land T}\big)\big) \big], \; t \in [0,T].
\end{align}

\begin{remark}
\begin{enumerate}[label={$(\roman*)$}]
\item As previously mentioned, we do not assume the existence of a unique sub-game--perfect Nash equilibrium $\hat\balpha^N$. Consequently, the processes just defined should be understood as depending on each specific choice of $\hat\balpha^N$, although we omit this dependence to simplify the notation.
\item When the $N$-player game is symmetric, meaning that the data of the game are identical across all players, the second process in \eqref{align:MiNi_starDef} becomes independent of the index $i \in \{1,\ldots,N\}$. In this case, we denote it by
\begin{align*}
N_t^{\star,N} \coloneqq \E^{\P^{\smalltext{\hat\balpha}^\tinytext{N}\smalltext{,}\smalltext{N}\smalltext{,}\smalltext{t}}_\smalltext{\cdot}} \big[ \varphi_2\big(L^N(\X^N_{\cdot \land T})\big) \big], \; t \in [0,T].
\end{align*}
\end{enumerate}
\end{remark}

Before stating the extended dynamic programming principle, we introduce the following assumptions.
\begin{assumption}\label{assumpDPP_NplayerGame}
Let $i \in \{1,\ldots,N\}$. 
\begin{enumerate}[label={$(\roman*)$}]
\item\label{boundedPhisMart_NplayerGame} The functions $\cC_m \ni x \longmapsto \varphi^i_1(x)$ and $\cP_2(\cC_{m \times N}) \ni \xi \longmapsto \varphi^i_2(\xi)$ are bounded$;$
\item\label{Glip_NplayerGame} the function $\R \times \R \ni (m^{\star},n^{\star}) \longmapsto G^i(m^{\star},n^{\star})$ is twice continuously differentiable with Lipschitz-continuous derivatives $\partial_{m}G^i$, $\partial_{n}G^i$, $\partial^2_{m,m}G^i$, $\partial^2_{m,n}G^i$, and $\partial^2_{n,n}G^i;$
\item\label{modulusCondExpMart_NplayerGame} there exists a constant $c>0$ and a modulus of continuity $\rho$ such that, for any $(\balpha,t,\tilde{t},t^\prime) \in \cA^N_N \times [0,T] \times [t,T] \times [\tilde{t},T]$, we have
\begin{align*}
\E^{\P^{\smalltext{\balpha}\smalltext{,}\smalltext{N}\smalltext{,}\smalltext{t}}_{\smalltext{\cdot}}} \Big[ \big| \E^{\P^{\smalltext{\balpha}\smalltext{,}\smalltext{N}\smalltext{,}\smalltext{\tilde{t}}}_\smalltext{\cdot}} \big[ M^{i,\star,N}_{t^\prime} \big] - M^{i,\star,N}_{\tilde{t}} \big|^2 +  \big| \E^{\P^{\smalltext{\balpha}\smalltext{,}\smalltext{N}\smalltext{,}\smalltext{\tilde{t}}}_\smalltext{\cdot}} \big[ N^{i,\star,N}_{{t}^\prime} \big] - N^{i,\star,N}_{\tilde{t}} \big|^2 \Big] \leq c |t^\prime-\tilde{t}| \rho \big(|t^\prime-\tilde{t}|\big), \; \P\text{\rm--a.s.}
\end{align*}
\end{enumerate}
\end{assumption}

\begin{remark}
Let $i \in \{1,\ldots,N\}$. As stated in {\rm\cite[Lemma 7.2]{hernandez2023me}}, and equivalently in {\rm\cite[Lemma 2.4.2]{hernandez2021general}}, the inequality concerning the first term involving $M^{i,\star,N}$ in {\rm\Cref{assumpDPP_NplayerGame}.\ref{modulusCondExpMart_NplayerGame}} holds provided that $\varphi^i_1$ admits bounded first-order $\nabla_{x} \varphi^i_1$ and bounded second-order $\nabla^2_{x} \varphi^i_1$ vertical derivatives, in the sense of {\rm\citeauthor*{cont2010change} \cite[Definition 8]{cont2010change}}. Additionally, we must require that the process $\mathfrak{A}^{i,\balpha}$ defined by 
\begin{align*}
\mathfrak{A}^{i,\balpha}_t \coloneqq \nabla_{x} \varphi^i_1(X^i_{\cdot \land t}) \Big(\sigma^i_t(X^i_{\cdot \land t}) b^i_t\big(X^i_{\cdot \land t},L^N\big(\X^N_{\cdot \land t},\balpha_t\big),\alpha^{i}_t\big)\Big)^\top + \frac{1}{2} \mathrm{Tr}\big[\sigma^i_t(X^i_{\cdot \land t}) (\sigma^i_t(X^i_{\cdot \land t}))^\top \big] \nabla^2_{x} \varphi^i_1(X^i_{\cdot \land t}), \; t \in [0,T],
\end{align*}
is $\P$--square-integrable, for any $\balpha \in \cA^N_N$. A similar conclusion holds for the second term $N^{i,\star,N}$, assuming analogous conditions on the composed function $\cC_{m \times N} \ni \mathbf{x} \longmapsto \widetilde{\varphi}^i_2 \coloneqq \varphi^i_2(L^N(\mathbf{x}))$. For brevity, we omit the explicit formulation of these conditions, as the notation would become excessively heavy.
\end{remark}

\begin{theorem}\label{thm:DPP_NplayerGame}
Let {\rm\Cref{assumpDPP_NplayerGame}} hold, and let $\hat\balpha^N \in \mathcal{N}\!\mathcal{A}_{s,N}$ be a sub-game--perfect Nash equilibrium. Then, for any $i \in \{1,\ldots,N\}$ and any $(t,\tilde{t}) \in [0,T] \times [t,T]$, the value process $V^{i,N} \coloneqq J^i(\cdot,\cdot,\hat\alpha^{i,N};\hat\balpha^{N,\smallertext{-}i})$ satisfies
\begin{align*}
V^{i,N}_t = \esssup_{\alpha \in {\cA}_\smalltext{N}} \E^{\P^{\smalltext{\alpha} \smalltext{\otimes}_\tinytext{i} \smalltext{\hat\balpha}^{\tinytext{N}\tinytext{,}\tinytext{-}\tinytext{i}}\smalltext{,}\smalltext{N}\smalltext{,}\smalltext{t}}_\smalltext{\cdot}} \bigg[ &V^{i,N}_{\tilde{t}} + \int_t^{\tilde{t}} f^i_s\big(X^i_{\cdot \land s}, L^N\big(\X^N_{\cdot \land s},(\alpha \otimes_i \hat\balpha^{N,\smallertext{-}i})_s\big), \alpha_s\big) \d s\\
& - \frac{1}{2} \int_t^{\tilde{t}} \partial^2_{m,m} G^i\big(M^{i,\star,N}_s, N^{i,\star,N}_s\big) \d \big[M^{i,\star,N}\big]_s - \frac{1}{2} \int_t^{\tilde{t}} \partial^2_{n,n} G^i\big(M^{i,\star,N}_s, N^{i,\star,N}_s\big) \d \big[N^{i,\star,N}\big]_s \\
& - \int_t^{\tilde{t}} \partial^2_{m,n} G^i\big(M^{i,\star,N}_s, N^{i,\star,N}_s\big) \d \big[M^{i,\star,N},N^{i,\star,N}\big]_s \bigg], \; \P\text{\rm--a.s.}
\end{align*}
\end{theorem}

Although the result is a reformulation of \cite[Lemma 7.2]{hernandez2023me}, we provide a complete proof in our setting. This is necessary not only because our definition of equilibrium differs slightly, but also because the proof in \cite{hernandez2023me} is formulated on the canonical function space, whereas we work on a general Polish space. The proof is deferred to \Cref{section:DPP} for readability and is carried out for player 1; the argument for the remaining players $i \in \{1,\ldots,N\}\setminus\{1\}$ follow analogously. In the proof of the extended dynamic programming principle for the value process $V^1$, we introduce the auxiliary probability measure $\Q$, defined by
\begin{align*}
\frac{\d \Q}{\d \P} \coloneqq \cE\Bigg( \int_0^\cdot \sum_{\ell \in \{1,\ldots,N\}\setminus\{1\}} b^\ell_t\big(X^\ell_{\cdot \land t},L^N\big(\X^N_{\cdot \land t},\hat\balpha^N_t\big),\hat\alpha^{\ell,N}_t\big) \cdot \d W^\ell_t \Bigg)_T.
\end{align*}

\medskip
The extended dynamic programming principle allows us to relate each sub-game--perfect Nash equilibria $\hat\balpha^N \in \mathcal{N}\!\mathcal{A}_{s,N}$ and the corresponding value processes $(V^{i,N})_{i \in \{1,\ldots,N\}}$ in the $N$-player game to a fully coupled system of FBSDEs, as detailed in \Cref{necessity_toBSDE}. In particular, sub-game--perfect Nash equilibria correspond to fixed points of the associated vector-valued Hamiltonian, which we introduce below. Before doing so, we establish some preliminary notation. For each $i \in \{1,\ldots,N\}$, we define the function $H^{i,N}: [0,T] \times \cC_{m \times N} \times \R^{d \times N} \times \R \times \R \times \R^{d \times N} \times \R^{d \times N} \times A \times A^N \longrightarrow \R$ as follows
\begin{align*}
H^{i,N}_t\big(\x,\mathbf{z},m^\star,n^\star,\mathsf{z}^{m,\star},\mathsf{z}^{n,\star},a,\mathbf{e}^N\big) &\coloneqq h^i_t\big(\x,z^{i},a,\mathbf{e}^N\big) + \sum_{\ell\in\{1,\dots,N\}\setminus\{i\}} z^{\ell} \cdot b^\ell_t\big(x^\ell, L^N(\x,a \otimes_i \mathbf{e}^{N,\smallertext{-}i}), e^{N,\ell}\big) \\
&\;\quad-\frac{1}{2} \partial^2_{m,m} G^i(m^\star,n^\star) \sum_{\ell =1}^N \|z^{m,\ell,\star}\|^2 - \partial^2_{m,n} G^i(m^\star,n^\star) \sum_{\ell =1}^N z^{m,\ell,\star} \cdot z^{n,\ell,\star} \\
&\;\quad  -\frac{1}{2} \partial^2_{n,n} G^i(m^\star,n^\star) \sum_{\ell =1}^N \|z^{n,\ell,\star}\|^2,
\end{align*}
where for any $(t,\x,z,\mathbf{e}^N,a) \in [0,T] \times \cC_{m \times N} \times \R^d \times A^N \times A$
\begin{align*}
h^i_t\big(\x,z,a,\mathbf{e}^N\big) \coloneqq f^i_t\big(x^i, L^N(\x,a \otimes_i \mathbf{e}^{N,\smallertext{-}i}), a\big) +  z \cdot b^i_t\big(x^i, L^N(\x,a \otimes_i \mathbf{e}^{N,\smallertext{-}i}), a\big).
\end{align*}
The function just introduced is typically referred to as the Hamiltonian associated with the problem faced by the player $i$. However, in the context of the $N$-player game, the appropriate notion must account for the simultaneous optimisation performed by all players. For this reason, we introduce the vector-valued function $H^N: [0,T] \times \cC_{m \times N} \times (\R^{d \times N})^N \times \R^N \times \R^N \times (\R^{d \times N})^N \times (\R^{d \times N})^N \times A^N \times A^N \longrightarrow \R^N$ defined by
\begin{align*}
    H^N_t\big(\mathbf{x}, \mathbf{z}, \mathsf{m}^\star, \mathsf{n}^\star, \mathsf{z}^{m,\star}, \mathsf{z}^{n,\star}, \mathbf{a}^N, \mathbf{e}^N\big) \coloneqq
    \begin{pmatrix}
    H^{1,N}_t\big(\x,\mathbf{z}^1,{m}^{1,\star},{n}^{1,\star},\mathsf{z}^{1,m,\star},\mathsf{z}^{1,n,\star},{a}^{1,N},\mathbf{e}^N\big) \\
    \vdots \\
    H^{N,N}_t\big(\x,\mathbf{z}^N,{m}^{N,\star},{n}^{N,\star},\mathsf{z}^{N,m,\star},\mathsf{z}^{N,n,\star},a^{N,N},\mathbf{e}^N\big)
    \end{pmatrix}.
\end{align*}

\begin{definition}\label{def:fixedPointHamiltonian}
Let $(t, \mathbf{x}, \mathbf{z}, \mathsf{m}^\star, \mathsf{n}^\star, \mathsf{z}^{m,\star}, \mathsf{z}^{n,\star}) \in [0,T] \times \cC_{m \times N} \times (\R^{d \times N})^N \times \R^N \times \R^N \times (\R^{d \times N})^N \times (\R^{d \times N})^N$. A vector $\mathbf{a}^N \coloneqq (a^{1,N},\ldots,a^{N,N})\in A^N$ is said to be a fixed-point of the Hamiltonian $H^N$ if, for any $i \in \{1,\ldots,N\}$, it holds that
\begin{align*}
a^{i,N} \in \argmax_{a \in A} \big\{H^{i,N}_t\big(\mathbf{x},\mathbf{z}^i,\mathsf{m}^{i,\star},\mathsf{n}^{i,\star},\mathsf{z}^{i,m,\star},\mathsf{z}^{i,n,\star},a,\mathbf{a}^N\big)\big\}.
\end{align*}
We denote the set of such fixed-points by $\cO_{N}(t, \mathbf{x}, \mathbf{z}, \mathsf{m}^\star, \mathsf{n}^\star, \mathsf{z}^{m,\star}, \mathsf{z}^{n,\star})$.
\end{definition}

\begin{remark}
Every fixed-point $\mathbf{a}^N$ of the Hamiltonian $H^N$ is a function of the form $\mathbf{a}^N(t, \mathbf{x}, \mathbf{z}, \mathsf{m}^\star, \mathsf{n}^\star, \mathsf{z}^{m,\star}, \mathsf{z}^{n,\star})$, as it is clear from the definition. At this stage, we do not impose any regularity assumptions on such functions. However, we will later require additional regularity when we aim to fully characterise the set of sub-game--perfect Nash equilibria in terms of fixed-points of $H^N$.
\end{remark}

It is natural at this point to introduce the system of BSDEs associated with a sub-game--perfect Nash equilibrium $\hat\balpha^N \in \mathcal{N}\!\mathcal{A}_{s,N}$. We write the system under the probability measure $\P^{\hat\balpha^{\smalltext{N}}}$; this formulation is particularly convenient for the subsequent analysis aimed at proving the convergence of the $N$-player game to its mean-field limit.
\begin{align}\label{BSDE_NplayerGame}
\notag Y^{i,N}_t &= g^i\big(X^i_{\cdot \land T}, L^N\big(\X^N_{\cdot \land T}\big)\big) + G^i\big(\varphi^i_1(X^i_{\cdot \land T}),\varphi^i_2\big(L^N(\X^N_{\cdot \land T})\big)\big) \\
\notag &\quad+ \int_t^T f^i_s\big(X^i_{\cdot \land s},L^N\big(\X^N_{\cdot \land s},\hat\balpha^N_s\big),\hat\alpha^{i,N}_s\big) \d s - \int_t^T \partial^2_{m,n}G^i\big(M^{i,\star,N}_s,N^{i,\star,N}_s\big) \sum_{\ell =1}^N Z^{i,m,\ell,\star,N}_s \cdot Z^{i,n,\ell,\star,N}_s \d s \\
\notag &\quad- \frac{1}{2} \int_t^T \Bigg(\partial^2_{m,m}G^i\big(M^{i,\star,N}_s,N^{i,\star,N}_s\big) \sum_{\ell =1}^N \big\|Z^{i,m,\ell,\star,N}_s\big\|^2 + \partial^2_{n,n}G^i\big(M^{i,\star,N}_s,N^{i,\star,N}_s\big) \sum_{\ell =1}^N \big\|Z^{i,n,\ell,\star,N}_s\big\|^2\Bigg) \d s \\
\notag &\quad- \int_t^T \sum_{\ell =1}^N Z^{i,\ell,N}_s \cdot \d \big(W_s^{\hat\balpha^\smalltext{N},N}\big)^\ell, \; t \in [0,T], \; \P \text{\rm--a.s.}, \\
\notag M^{i,\star,N}_t &= \varphi^i_1(X^i_{\cdot \land T}) - \int_t^T \sum_{\ell =1}^N Z^{i,m,\ell,\star,N}_s \cdot \d \big(W_s^{\hat\balpha^\smalltext{N},N}\big)^\ell, \; t \in [0,T], \; \P \text{\rm--a.s.}, \\
N^{i,\star,N}_t &= \varphi^i_2\big(L^N\big(\X^N_{\cdot \land T}\big)\big) - \int_t^T \sum_{\ell =1}^N Z^{i,n,\ell,\star,N}_s \cdot \d \big(W_s^{\hat\balpha^\smalltext{N},N}\big)^\ell, \; t \in [0,T], \; \P \text{\rm--a.s.}
\end{align}

We can now state the characterisation result, which consists of two separate parts, whose proof is postponed to \Cref{section:characterisation}. The first addresses the necessity of the system: given a sub-game--perfect Nash equilibrium and the corresponding value processes $(V^{i,N})_{i \in \{1,\ldots,N\}}$ of the game, one can construct a solution to the BSDE system \eqref{BSDE_NplayerGame}. The second result is a verification result: it shows that any sufficiently integrable solution to \eqref{BSDE_NplayerGame}, where $\hat\balpha^N$ is given as a suitable fixed point of the Hamiltonian defined in \Cref{def:fixedPointHamiltonian}, allows one to construct a sub-game--perfect Nash equilibrium.

\begin{assumption}\label{assumption:regularityForNecessity}
Let $i \in \{1,\ldots,N\}$. There exists some $p \geq 1$ and a vector $\mathbf{a}_0^N \in A^N$ such that 
\begin{align*}
\sup_{\balpha \in \cA^\smalltext{N}_\smalltext{N}} \E^{\P^{\smalltext{\balpha}}} \bigg[ \big|g^i\big(X^i_{\cdot \land T}, L^N\big(\X^N_{\cdot \land T}\big)\big)\big|^p + \int_0^T \big|H^{i,N}_t\big(\X^N_{\cdot \land t},\mathbf{0},0,0,\mathbf{0},\mathbf{0},a^{i,N}_0,\mathbf{a}^N_0\big)\big|^p \d t \bigg] <+\infty.
\end{align*}
\end{assumption}

\begin{proposition}\label{necessity_toBSDE}
Let {\rm\Cref{assumpDPP_NplayerGame}} and {\rm\Cref{assumption:regularityForNecessity}} hold, and let $\hat\balpha^N \in \mathcal{N}\!\mathcal{A}_{s,N}$ be a sub-game--perfect Nash equilibrium. Then, one can construct a tuple 
\[
(\Y^N,\Z^N,\M^{\star,N},\N^{\star,N},\Z^{m,\star,N},\Z^{n,\star,N}) \coloneqq (Y^{i,N}, \Z^{i,N}, M^{i,\star,N}, N^{i,\star,N}, \Z^{i,m,\star,N},\Z^{i,n,\star,N})_{i\in\{1,\ldots,N\}},
\]
that satisfies the system \eqref{BSDE_NplayerGame} and, for some $p \geq 1$, the integrability condition
\begin{align*}
&\sup_{\balpha \in \cA^\smalltext{N}_\smalltext{N}} \E^{\P^{\smalltext{\balpha}}} \Bigg[ \sup_{t \in [0,T]} \big|Y^{i,N}_t\big|^p + \bigg(\int_0^T \sum_{\ell =1}^N \big\|{Z}_t^{i,\ell,N}\big\|^2 \d t\bigg)^\frac{p}{2} \Bigg] \\
&+\esssup_{\balpha \in \cA^\smalltext{N}_\smalltext{N}} \Bigg\{\sup_{t \in [0,T]} \Big\{ \big| M^{i,\star,N}_t\big| + \big| N^{i,\star,N}_t \big| \Big\} + \sup_{\tau \in \cT_{0,T}} \E^{\P^{\smalltext{\balpha}\smalltext{,}\smalltext{N}\smalltext{,}\smalltext{\tau}}_{\smalltext{\cdot}}} \Bigg[ \int_\tau^T \sum_{\ell =1}^N \Big( \big\|Z^{i,m,\ell,\star,N}_t\big\|^2 + \big\|Z^{i,n,\ell,\star,N}_t\big\|^2 \Big) \d t \Bigg] \Bigg\} < +\infty.
\end{align*}
It holds that $V^{i,N}_t = Y^{i,N}_t$, $\P\text{\rm--a.s.}$, for every $t \in [0,T]$. Moreover,
\begin{align}\label{condOpFixedPoints}
\hat\balpha^N_t \in \cO_{N}\big(t, \X^N_{\cdot \land t}, \Z^N_t, \M^{\star,N}_t, \N^{\star,N}_t, \Z^{m,\star,N}_t, \Z^{n,\star,N}_t\big), \; \text{\rm for}\; \d t \otimes \d \P\text{\rm--a.e.} \; (t,\omega) \in [0,T] \times \Omega.
\end{align}
\end{proposition}

Before stating the sufficiency of the BSDE system, we introduce the following condition imposed on the function that identifies the maximisers of the Hamiltonian.

\begin{assumption}\label{assumption:borelMeasurability_Lip}
For every fixed point $\mathbf{a}^N$ of the Hamiltonian $H^N$, there is a Borel-measurable function $\mathfrak{a}^N: [0,T] \times \cC_{m \times N} \times (\R^{d \times N})^N \times \R^N \times \R^N \times (\R^{d \times N})^N \times (\R^{d \times N})^N \longrightarrow A^N$ such that $\mathfrak{a}^N(t, \mathbf{x}, \mathbf{z}, \mathsf{m}^\star, \mathsf{n}^\star, \mathsf{z}^{m,\star}, \mathsf{z}^{n,\star}) = \mathbf{a}^N$.
\end{assumption}

\begin{proposition}\label{sufficiency_toBSDE}
Let {\rm\Cref{assumpDPP_NplayerGame}.\ref{Glip_NplayerGame}} and {\rm\Cref{assumption:borelMeasurability_Lip}} hold. Let $(\Y^N,\Z^N,\M^{\star,N},\N^{\star,N},\Z^{m,\star,N},\Z^{n,\star,N})$ be a solution to the following system of {\rm BSDEs}
\begin{align}\label{align:BSDEalphaMaxHam}
\notag Y^{i,N}_t &= g^i\big(X^i_{\cdot \land T}, L^N(\X^N_{\cdot \land T})\big) + G^i\big(\varphi^i_1(X^i_{\cdot \land T}),\varphi^i_2\big(L^N(\X^N_{\cdot \land T})\big)\big) \\
\notag &\quad+ \int_t^T f^i_s\big(X^i_{\cdot \land s},L^N(\X^N_{\cdot \land s},\hat\balpha^N_s),\hat\alpha^{i,N}_s\big) \d s - \int_t^T \partial^2_{m,n}G^i\big(M^{i,\star,N}_s,N^{i,\star,N}_s\big) \sum_{\ell =1}^N Z^{i,m,\ell,\star,N}_s \cdot Z^{i,n,\ell,\star,N}_s \d s \\
\notag &\quad- \frac{1}{2} \int_t^T \Bigg(\partial^2_{m,m}G^i\big(M^{i,\star,N}_s,N^{i,\star,N}_s\big) \sum_{\ell =1}^N \big\|Z^{i,m,\ell,\star,N}_s\big\|^2 +\partial^2_{n,n}G^i\big(M^{i,\star,N}_s,N^{i,\star,N}_s\big) \sum_{\ell=1}^N \big\|Z^{i,n,\ell,\star,N}_s\big\|^2 \Bigg)\d s \\
\notag &\quad- \int_t^T \sum_{\ell = 1}^N Z^{i,\ell,N}_s \cdot \d \big(W_s^{\hat\balpha^\smalltext{N},N}\big)^\ell, \; t \in [0,T], \; \P \text{\rm--a.s.}, \\
\notag M^{i,\star,N}_t &= \varphi^i_1(X^i_{\cdot \land T}) - \int_t^T \sum_{\ell =1}^N Z^{i,m,\ell,\star,N}_s \cdot \d \big(W_s^{\hat\balpha^\smalltext{N},N}\big)^\ell, \; t \in [0,T], \; \P \text{\rm--a.s.}, \\
\notag N^{i,\star,N}_t &= \varphi^i_2\big(L^N\big(\X^N_{\cdot \land T}\big)\big) - \int_t^T \sum_{\ell=1}^N Z^{i,m,\ell,\star,N}_s \cdot \d \big(W_s^{\hat\balpha^\smalltext{N},N}\big)^\ell, \; t \in [0,T], \; \P \text{\rm--a.s.} \\
\notag \hat\balpha^N_t &\coloneqq \mathfrak{a}^N\big(t,\X^N_{\cdot \land t},\Z^N_t,\M^{\star,N}_t,\N^{\star,N}_t,\Z^{m,\star,N}_t,\Z^{n,\star,N}_t\big),  \\
\frac{\d \P^{\hat\balpha^\smalltext{N},N}}{\d \P} &\coloneqq \cE\Bigg( \int_0^\cdot \sum_{\ell =1}^N b^\ell_t\big(X^\ell_{\cdot \land t},L^N\big(\X^N_{\cdot \land t},\hat\balpha^N_t\big),\hat\alpha^{\ell,N}_t\big) \cdot \d W^\ell_t \Bigg)_T.
\end{align}
In addition, for each $i \in \{1,\ldots,N\}$, there exists $p \geq 1$ such that 
\begin{align}\label{align:estimatesBSDEsupSuff}
\notag &\sup_{\balpha \in \cA^\smalltext{N}_\smalltext{N}} \E^{\P^{\smalltext{\balpha}}} \Bigg[ \sup_{t \in [0,T]} \big|Y^{i,N}_t\big|^p + \bigg(\int_0^T \sum_{\ell =1}^N \big\|{Z}_t^{i,\ell,N}\big\|^2 \d t\bigg)^\frac{p}{2} \Bigg] \\
&+\esssup_{\balpha \in \cA^\smalltext{N}_\smalltext{N}} \Bigg\{\sup_{t \in [0,T]} \Big\{ \big| M^{i,\star,N}_t\big| + \big| N^{i,\star,N}_t \big| \Big\} + \sup_{\tau \in \cT_{0,T}} \E^{\P^{\smalltext{\balpha}\smalltext{,}\smalltext{N}\smalltext{,}\smalltext{\tau}}_{\smalltext{\cdot}}} \Bigg[ \int_\tau^T \sum_{\ell =1}^N \Big( \big\|Z^{i,m,\ell,\star,N}_t\big\|^2 + \big\|Z^{i,n,\ell,\star,N}_t\big\|^2 \Big) \d t \Bigg] \Bigg\} <+ \infty.
\end{align}
Consequently, it holds that $\hat\balpha^N \in \mathcal{N}\!\mathcal{A}_{s,N}$, and for any player index $i \in \{1,\ldots,N\}$, $J^i(t,\cdot,\hat\alpha_t^{i,N};\hat\balpha_t^{N,\smallertext{-}i}) = Y^{i,N}_t$, $\P\text{\rm--a.s.}$, for every $t\in [0,T]$.
\end{proposition}

\subsection{Illustrative examples: two players}\label{section:only2Players_sameResulT}

In this section, we present two illustrative examples for which the sub-game--perfect Nash equilibria, together with the associated value processes, can be computed explicitly using the BSDE formulation of the game, in the spirit of \Cref{necessity_toBSDE} and \Cref{sufficiency_toBSDE}. Throughout, we adopt the notation introduced in \Cref{section:gameFinitePlayer} and restrict to the one-dimensional setting $d=m=1$. We introduce a constant $\sigma > 0$ and assume that $A \coloneqq (-\bar{a},\bar{a})$, where $\bar{a}>0$ is chosen sufficiently large. Moreover, we consider coefficients of the form
\begin{align*}
\sigma_t(x^i) = \sigma, \; (t,x^i) \in [0,T] \times \cC, \; \text{and} \; b_t(x^i,a^i) = a^i, \; (t, x^i,a^i) \in [0,T] \times \cC \times A, \; \text{for} \; i \in \{1,\ldots,N\}.
\end{align*}

We focus on the case of two players, $N = 2$, and fix a risk-aversion parameter $\gamma >0$. Let $i \in \{1,2\}$ be a given player.

\subsubsection{First example}
When the other player $j \in \{1,2\}\setminus \{i\}$ follows the strategy $\alpha^{j} \in \cA_2$, the criterion of player $i$ is
\begin{equation}\label{equation:payoffExtwoplayers}
J^i(t,\alpha^i;\alpha^{j}) \coloneqq \E^{\P^{\smalltext{\alpha}^\tinytext{i} \smalltext{\otimes}_\tinytext{i} \smalltext{\alpha}^{\tinytext{j}}\smalltext{,}\smalltext{2}\smalltext{,}\smalltext{t}}_\smalltext{\omega}} \bigg[ \int_t^T \big((\alpha^j_s)^2-(\alpha^i_s)^2\big) \d s + X^i_T - \frac{\gamma}{2}(X^i_T)^2 \bigg] + \frac{\gamma}{2} \Big( \E^{\P^{\smalltext{\alpha}^\tinytext{i} \smalltext{\otimes}_\tinytext{i} \smalltext{\alpha}^{\tinytext{j}}\smalltext{,}\smalltext{2}\smalltext{,}\smalltext{t}}_\smalltext{\omega}} \big[ X^i_T \big| \cF_t \big]\Big)^2, \; (t, \alpha^i) \in [0,T] \times \cA_2.
\end{equation}

By the BSDE characterisation of the game, $\hat{\balpha}^{2} = (\hat{\alpha}^{1,2},\hat{\alpha}^{2,2}) \in \cA^2_2$ is a sub-game--perfect Nash equilibrium if and only if, for each $i \in \{1,2\}$,
\begin{align*}
\hat{\alpha}^{i,2}_t = \frac{Z^{i,i,2}_t}{2}, \; \d t \otimes \d \P\text{\rm--a.e.} \; (t,\omega) \in [0,T] \times \Omega,
\end{align*}
where $(Y^{i,2},(Z^{i,i,2},Z^{i,j,2}),M^{i,\star,2},(Z^{i,m,i,\star,2},Z^{i,m,j,\star,2}))$ solves the BSDE system
\begin{align*}
Y^{i,2}_t &= X^i_T + \frac{1}{4} \int_t^T \bigg( (Z^{i,i,2}_s)^2 + (Z^{j,j,2}_s)^2 + 2 Z^{i,j,2}_s Z^{j,j,2}_s - 2 \gamma \Big( \big(Z^{i,m,i,\star,2}_s\big)^2 + \big(Z^{i,m,j,\star,2}_s\big)^2 \Big)  \bigg) \d s\\
& \quad  - \int_t^T \frac{Z^{i,i,2}_s}{\sigma} \d X^i_s - \int_t^T \frac{Z^{i,j,2}_s}{\sigma} \d X^j_s, \; t \in [0,T], \; \P\text{\rm--a.s.} \\
M^{i,\star,2}_t &= X^i_T + \frac{1}{2} \int_t^T \big( Z^{i,m,i,\star,2}_s Z^{i,i,2}_s + Z^{i,m,j,\star,2}_s Z^{j,j,2}_s \big) \d s  - \int_t^T \frac{Z^{i,m,i,\star,2}_s}{\sigma} \d X^i_s - \int_t^T \frac{Z^{i,m,j,\star,2}_s}{\sigma} \d X^j_s, \; t \in [0,T], \; \P\text{\rm--a.s.}
\end{align*}
The quadratic structure of this system guarantees the existence of a unique solution (see \cite[Theorem 2.2]{cheridito2014bsdes}), namely
\begin{align*}
Y^{i,2}_t = X^i_t + \frac{\sigma^2}{2}(1-\gamma)(T-t), \; Z^{i,i,2}_t = \sigma, \; Z^{i,j,2}_t = 0, \; M^{i,\star,2}_t = X^i_t + \frac{\sigma^2}{2}(T-t), \; {Z}^{i,m,i,\star,2}_t = \sigma, \; {Z}^{i,m,j,\star,2}_t = 0.
\end{align*}
Consequently, there exists a unique sub-game--perfect Nash equilibrium, given by
\[
\hat{\balpha}^{2}_t = (\hat{\alpha}^{1,2}_t,\hat{\alpha}^{2,2}_t) = \Big(\frac{\sigma}{2},\frac{\sigma}{2}\Big), \; t\in [0,T].
\]

In contrast to the approach taken in the paper, and therefore also in the example above, which resolves the time inconsistency in the payoff \eqref{equation:payoffExtwoplayers} by adopting a game-theoretic interpretation that views the problem as a strategic interaction between the successive temporal versions of each player's preferences, we may instead rewrite the payoff as a functional of the law of the state processes. This provides a different perspective, under which the problem becomes a McKean--Vlasov control problem. In this formulation, the objective that player $i \in \{1,2\}$ aims to maximise at the initial time is
\begin{equation*}
J^i(0,\alpha^i;\alpha^{j}) \coloneqq \E^{\P^{\alpha^i \otimes_i \alpha^j}} \bigg[ \int_0^T \big((\alpha^i_s)^2 -(\alpha^j_s)^2\big) \d s + X^i_T - \frac{\gamma}{2}(X^i_T)^2 \bigg] + \frac{\gamma}{2} \bigg( \int_{\R} x^i \mathcal{L}^{\alpha^i \otimes_i \alpha^j}_{X^i_T}(\d x^i)\bigg)^2, \; \alpha^i \in \cA_2.
\end{equation*}
Our goal is to describe the Nash equilibria, that is, all $\tilde{\balpha}^2 = (\tilde{\alpha}^{1,2},\tilde{\alpha}^{2,2}) \in \mathcal{A}^2_2$ such that $J^i(0,\tilde{\alpha}^i;\tilde{\alpha}^{j}) \geq J^i(0,\alpha^i;\tilde{\alpha}^{j})$ for every $\alpha^i \in \mathcal{A}_2$. We then define the associated value function
\begin{equation*}
v^{i,2}(t,\mu^i) \coloneqq J^i(t,\tilde{\alpha}^{i,2};\tilde{\alpha}^{j,2}), \; (t,\mu^i) \in [0,T] \times \cP_2(\R),
\end{equation*}
where $\mu^i = \mathcal{L}^{\tilde{\balpha}^{\smalltext{2}}}_{X^i_t}$. Using the notion of differentiability with respect to probability measures (see, for instance, \citeauthor*{cardaliaguet2010notes} \cite[Section 6]{cardaliaguet2010notes}), this problem can be associated in a standard way with a system of second-order Hamilton--Jacobi--Bellman equations on the Wasserstein space (see \citeauthor*{bayraktar2018randomized} \cite[Remark 5.3]{bayraktar2018randomized})
\begin{align}\label{align:HJBequationViWassMc}
\notag &\partial_t v^i(t,\mu^i) + \int_\R \widetilde{H}^{i,2}_t\big(\partial_{\mu^i} v^i(t,\mu^i)(x^i),\partial_{x^i \mu^i} v^i(t,\mu^i)(x^i)\big) \mu^i(\d x^i) = 0, \; (t,\mu^i) \in [0,T) \times \mathcal{P}_2(\R), \\
&v^i(T,\mu^i) = \int_\R \bigg(x^i-\frac{\gamma}{2} (x^i)^2\bigg) \mu^i(\d x^i) + \frac{\gamma}{2} \bigg( \int_\R x^i \mu^i(\d x^i) \bigg)^2.
\end{align}
Here, the vector-valued Hamiltonian is given by
\begin{align*}
\widetilde{H}^2_t(p,q) \coloneqq 
\begin{pmatrix}
\widetilde{H}^{1,2}_t(p,q) \\
\widetilde{H}^{2,2}_t(p,q)
\end{pmatrix} \coloneqq
\begin{pmatrix}
-(\tilde{a}^{1,2})^2 + (\tilde{a}^{2,2})^2  + \sigma p \tilde{a}^{1,2} + \sigma^2 q/2 \\
-(\tilde{a}^{2,2})^2 + (\tilde{a}^{1,2})^2  + \sigma p \tilde{a}^{2,2} + \sigma^2 q/2
\end{pmatrix}, \; (t,p,q) \in [0,T] \times \R \times \R,
\end{align*}
where, analogously to \Cref{def:fixedPointHamiltonian}, $(\tilde{a}^{1,2},\tilde{a}^{2,2})$ denotes a fixed point of $\widetilde{H}^2_t(p,q)$ in the appropriate sense. By \citeauthor*{cheung2025viscosity} \cite[Theorem 5.2]{cheung2025viscosity}, the value function $v^{i,2}$ defined above is a viscosity solution of \eqref{align:HJBequationViWassMc}. Moreover, uniqueness of the viscosity solution is established in \cite[Theorem 5.3]{cheung2025viscosity}, and the unique solution admits the explicit form
\begin{align*}
v^{i,2}(t,\mu^i) = \int_\R \bigg(x^i-\frac{\gamma}{2} (x^i)^2\bigg) \mu^i(\d x^i) + \frac{\gamma}{2} \bigg( \int_\R x^i \mu^i(\d x^i) \bigg)^2 + \frac{\sigma^2}{2} (1-\gamma) (T-t).
\end{align*}
From this expression, we deduce that the unique Nash equilibrium satisfies 
\[
\tilde{\balpha}^2_t = (\hat{\alpha}^{1,2}_t,\hat{\alpha}^{2,2}_t) = \Big(\frac{\sigma}{2},\frac{\sigma}{2}\Big), \; t\in [0,T].
\]
Hence, $\tilde{\balpha}^2 = \hat{\balpha}^2$. In general, the game-theoretic formulation and the McKean--Vlasov formulation need not yield the same equilibrium. In this specific example, however, the $Z$-component $(Z^{i,i,2},Z^{i,j,2})$ appearing in the BSDE characterisation of the sub-game--perfect equilibrium is deterministic, and it coincides exactly with the measure derivative that determines the Nash equilibrium.

\subsubsection{Second example}
Continuing in the same setting, we now present another example of a payoff function for which the sub-game--perfect equilibrium coincides with the Nash equilibrium. We suppose that the other player $j \in \{1,2\}\setminus \{i\}$ follows a strategy $\alpha^{j} \in \cA_2$, the objective of player $i \in \{1,2\}$ is then given by
\begin{equation*}
J^i(t,\alpha^i;\alpha^{j}) \coloneqq \E^{\P^{\smalltext{\alpha}^\tinytext{i} \smalltext{\otimes}_\tinytext{i} \smalltext{\alpha}^{\tinytext{j}}\smalltext{,}\smalltext{2}\smalltext{,}\smalltext{t}}_\smalltext{\omega}} \bigg[\int_t^T \bigg( \frac{(\alpha^i_s)^2}{2} - \big(\alpha^i_s - \alpha^j_s\big)^2 \bigg) \d s + X^i_T - \frac{\gamma}{2}(X^i_T)^2 \bigg] +\frac{\gamma}{2} \Big( \E^{\P^{\smalltext{\alpha}^\tinytext{i} \smalltext{\otimes}_\tinytext{i} \smalltext{\alpha}^{\tinytext{j}}\smalltext{,}\smalltext{2}\smalltext{,}\smalltext{t}}_\smalltext{\omega}} \big[ X^i_T \big]\Big)^2, \; (t, \alpha^i) \in [0,T] \times \cA_2.
\end{equation*}
Analogously to the previous example, the BSDE characterisation shows that $\hat{\balpha}^{2} = (\hat{\alpha}^{1,2},\hat{\alpha}^{2,2}) \in \cA^2_2$ is a sub-game--perfect Nash equilibrium if and only if, for each $i \in \{1,2\}$,
\begin{align*}
\hat{\alpha}^{i,2}_t = -\frac{2 Z^{i,i,2}_t + Z^{j,j,2}_t}{3}, \; \d t \otimes \d \P\text{\rm--a.e.} \; (t,\omega) \in [0,T] \times \Omega,
\end{align*}
because the Hamiltonian maximisation condition now yields $\hat{\alpha}^{i,2}_t = 2 \hat{\alpha}^{j,2}_t + Z^{i,i,2}_t$, for $\d t \otimes \d \P\text{\rm--a.e.} \; (t,\omega) \in [0,T] \times \Omega$. The associated BSDE system admits the unique solution
\begin{align*}
Y^{i,2}_t = X^i_t-\frac{\sigma^2}{2}(1+\gamma)(T-t), \; Z^{i,i,2}_t = \sigma, \; Z^{i,j,2}_t = 0, \; M^{i,\star,2}_t = X^i_t - \sigma^2 (T-t), \; {Z}^{i,m,i,\star,2}_t = \sigma, \; {Z}^{i,m,j,\star,2}_t = 0, \; t\in [0,T].
\end{align*}
Hence, the sub-game--perfect equilibrium is uniquely determined by 
\[
\hat{\balpha}^{2}_t = (\hat{\alpha}^{1,2}_t,\hat{\alpha}^{2,2}_t) = (-\sigma,-\sigma), \; t\in [0,T].
\]
Similarly to the previous case, the value function associated with the corresponding McKean--Vlasov control problem is
\begin{align*}
v^i(t,\mu^i) = \int_\R \bigg(x^i-\frac{\gamma}{2} (x^i)^2\bigg) \mu^i(\d x^i) + \frac{\gamma}{2} \bigg( \int_\R x^i \mu^i(\d x^i) \bigg)^2 + \frac{\sigma^2}{2} (1+\gamma) (T-t), \; (t,\mu^i) \in [0,T] \times \cP_2(\R),
\end{align*}
and therefore $\tilde{\balpha}^2 = \hat{\balpha}^2$.

\subsection{The zero--sum game}\label{subsection:zeroSumGame}

When the game is restricted to two players, and hence $N=2$, and the sum of their payoffs is fixed at a constant, the setting corresponds to a zero-sum stochastic differential game. In this framework, our objective is to investigate how \Cref{necessity_toBSDE} and \Cref{sufficiency_toBSDE} can be reformulated, and in particular, to identify conditions under which the BSDE system characterising the game can be simplified by one dimension, being described by a single process $Y$ instead of the pair $(Y^{1,2},Y^{2,2})$. Before addressing this reduction, we present the formulation of the stochastic differential game.

\medskip
Unlike the $N$-player game described in \eqref{align:criteriumNplayerGame}, where each player has a distinct state process, here we assume that both players share the same state process $X^1$, which is the unique strong solution to \eqref{eq:XstrongSolSDE} on the filtered probability space $(\Omega,\cF,\mathbb{F}^1,\P)$. Moreover, for each $\balpha \coloneqq (\alpha^1,\alpha^2) \in \cA_1^2$, we define the probability measure $\P^{\balpha}$ on $(\Omega,\cF)$, whose density with respect to $\PP$ is given by
\begin{equation*}
\frac{\d \P^{\balpha}}{\d \P} \coloneqq \cE\bigg( \int_0^\cdot b_t\big(X^1_{\cdot \land t},\alpha^1_t,\alpha^2_t\big) \cdot \d W^1_t \bigg)_T,
\end{equation*}
where $b:[0,T] \times \cC_m \times A \times A \longrightarrow \R^d$ is assumed to be bounded and Borel-measurable. We note that this definition slightly abuses notation, since the function $b$ introduced here differs from the function $b:[0,T] \times \cC_m \times \cP_2(\cC_{m} \times A) \times A \longrightarrow \R^d$ considered in \Cref{section:ProbSetting}, which characterises the change of measure in both the symmetric $N$-player game and the mean-field game. Nevertheless, its role is essentially the same.

\medskip
We introduce the payoff
\begin{align}\label{align:criteriumZeroSum}
J(t,\omega,\balpha) \coloneqq \E^{\P^{\smalltext{\balpha}\smalltext{,}\smalltext{t}}_\smalltext{\omega}} \bigg[ \int_t^T f_s\big(X^1_{\cdot \land s}, \alpha^1_s,\alpha^2_s\big) \d s + g\big(X^1_{\cdot \land T}\big) \bigg] + G\Big( \E^{\P^{\smalltext{\balpha}\smalltext{,}\smalltext{t}}_\smalltext{\omega}} \big[ \varphi\big(X^1_{\cdot \land T}\big) \big] \Big), \; (t,\omega,\balpha) \in [0,T] \times \Omega \times \cA_1^2,
\end{align}
where the functions $f: [0,T] \times \cC_m \times A^2 \longrightarrow \R$, $g: \cC_m \longrightarrow \R$, $G: \R \longrightarrow \R$, and $\varphi: \cC_m \longrightarrow \R$ are assumed to be Borel-measurable with respect to all their arguments. Analogously to the $N$-player game, the goal is to find an admissible strategy $\hat\balpha \in \cA_1^2$ that constitutes a zero-sum sub-game--perfect Nash equilibrium, according to the following definition.

\begin{definition}\label{def:NashEquilibriumZS}
Let $\hat\balpha \in \cA_1^2$, and $\varepsilon >0$. We define
\begin{align*}
\ell_\varepsilon \coloneqq \inf \Big\{ \ell >0 : \exists (i,t,\alpha) \in \{1,2\} \times [0,T] \times \cA_1, \;
\P \big[ \omega \in \Omega: \big(J(t,\omega,\hat\balpha) - J\big(t,\omega,(\alpha\otimes_{t+\ell}\hat\alpha^{i}) \otimes_i \hat\alpha^{\smallertext{-}i}\big) \big) \cI^i < - \varepsilon \ell \big] >0 \Big\},
\end{align*}
where $\cI^i \coloneqq -\mathbf{1}_{\{i=1\}}+\mathbf{1}_{\{i=2\}}$. We say that $\hat\balpha$ is a {\rm zero-sum sub-game--perfect Nash equilibrium} if for any $\varepsilon >0$, it holds that $\ell_\varepsilon >0$. The set containing all {\rm zero-sum sub-game--perfect Nash equilibria} is denoted by $\mathcal{N}\!\mathcal{A}_{s,0}$.
\end{definition}

\begin{remark}
\begin{enumerate}[label={$(\roman*)$}]
\item Player 1 aims to minimise the previously introduced payoff, while player 2 seeks to maximise it. Accordingly, this zero-sum game reduces to studying the existence of local $\varepsilon$-saddle points, which naturally reflects the time-inconsistent nature of the problem;
\item the framework of the game is symmetric in the sense that each player chooses admissible controls, and thus we adopt what in the time-consistent literature is referred to as a control-against-control formulation. In particular, no player can choose a non-anticipative mapping from the other player’s set of controls to their own. Beyond being a direct sub-case of the $N$-player game introduced earlier, this choice is motivated by the practical fact that players rarely share their strategies with competitors.
\end{enumerate}
\end{remark}

This setting still falls within the framework of time-inconsistent problems due to the definition of the payoff \eqref{align:criteriumZeroSum}. Consequently, the extended dynamic programming approach presented in the previous section remains the fundamental tool for addressing this problem and characterising sub-game--perfect Nash equilibria. In what follows, we assume the existence of a zero-sum sub-game--perfect Nash equilibrium $\hat\balpha\in\cA_1^2$, without imposing uniqueness. Analogously to \Cref{align:MiNi_starDef}, we introduce the process
\begin{align*}
M_t^{\star} \coloneqq \E^{\P^{\smalltext{\hat\balpha}\smalltext{,}\smalltext{t}}_\smalltext{\cdot}} \big[ \varphi(X^1_{\cdot \land T}) \big], \; t \in [0,T].
\end{align*}

\begin{assumption}\label{assumption:regularityForNecessityZero}
\begin{enumerate}[label={$(\roman*)$}]
\item\label{boundedV} The function $\cC_m \ni x \longmapsto \varphi(x)$ is bounded$;$
\item\label{differentiableG} the function $\R \ni m^{\star} \longmapsto G(m^{\star})$ is twice continuously differentiable with Lipschitz-continuous derivatives $G^\prime$ and $G^{\prime\prime};$
\item there exists a constant $c>0$ and a modulus of continuity $\rho$ such that, for any $(\balpha,t,\tilde{t},t^\prime) \in \cA^2_1 \times [0,T] \times [t,T] \times [\tilde{t},T]$, we have
\begin{align*}
\E^{\P^{\smalltext{\balpha}\smalltext{,}\smalltext{t}}_{\smalltext{\cdot}}} \Big[ \big| \E^{\P^{\smalltext{\balpha}\smalltext{,}\smalltext{\tilde{t}}}_\smalltext{\cdot}} \big[ M^{\star}_{t^\prime} \big] - M^{\star}_{\tilde{t}} \big|^2 \Big] \leq c |t^\prime-\tilde{t}| \rho \big(|t^\prime-\tilde{t}|\big), \; \P\text{\rm--a.s.};
\end{align*}
\item\label{IsaacsA} there exists some $p \geq 1$ such that 
\begin{align*}
\sup_{\balpha \in \cA^\smalltext{2}_\smalltext{1}} \E^{\P^{\smalltext{\balpha}}} \bigg[ \big|g\big(X^1_{\cdot \land T}\big)\big|^p + \int_0^T \big|H^{0}_t(X^1_{\cdot \land t},0,0,0)\big|^p \d t \bigg] <+\infty,
\end{align*}
where 
\begin{align*}
H^{0}_t(x,z,\mathsf{m},z^{\star}) \coloneqq \inf_{a \in A} \sup_{\tilde{a} \in A} \big\{f_t\big(x, a, \tilde{a} \big) +  z \cdot b_t(x,a,\tilde{a}) \big\} - \frac{1}{2} G^{\prime\prime}(\mathsf{m}) \|z^{\star}\|^2, \; (t,x,z,\mathsf{m},z^{\star}) \in [0,T] \times \cC_m \times \R^d \times \R \times \R.
\end{align*}
\end{enumerate}
\end{assumption}

\begin{proposition}\label{prop: from0subTOBSDE}
Let {\rm\Cref{assumption:regularityForNecessityZero}} hold, and let $\hat\balpha \in \mathcal{NA}_{s,0}$ be a zero-sum sub-game--perfect Nash equilibrium. Then, one can construct a quadruple $(Y,Z,M^{\star},Z^{m,\star})$ that satisfies the following system $\P \text{\rm--a.s.}$
\begin{align}\label{align:BSDEzeroSum}
\notag Y_t &= g\big(X^1_{\cdot \land T}\big) + G\big(\varphi(X^1_{\cdot \land T})\big)\big) + \int_t^T \bigg( f_s\big(X^1_{\cdot \land s},\hat\alpha^1_s,\hat\alpha^2_s\big) + Z_s \cdot b_s\big(X^1_{\cdot \land s},\hat\alpha^1_s,\hat\alpha^2_s\big) - \frac{G^{\prime\prime}\big(M^{\star}_s\big)}{2} \|Z^{m,\star}_s\|^2 \bigg) \d s\\
&\quad - \int_t^T Z_s \cdot \d W_s^1, \; t \in [0,T], \\
M^{\star}_t &= \varphi\big(X^1_{\cdot \land T}\big) + \int_t^T Z^{m,\star}_s \cdot b_s\big(X^1_{\cdot \land s},\hat\alpha^1_s,\hat\alpha^2_s\big) \d s - \int_t^T Z^{m,\star}_s \cdot \d W_s^1, \; t \in [0,T].
\end{align}
It holds that $J(t,\cdot,\hat\balpha) = Y_t$, $\P\text{\rm--a.s.}$, for every $t \in [0,T]$. Moreover, for some $p \geq 1$, the following integrability condition holds true
\begin{align*}
&\sup_{\balpha \in \cA^\smalltext{2}_\smalltext{1}} \E^{\P^{\smalltext{\balpha}}} \Bigg[ \sup_{t \in [0,T]} |Y_t|^p + \bigg(\int_0^T \|{Z}_t\|^2 \d t\bigg)^\frac{p}{2} \Bigg] +\esssup_{\balpha \in \cA^\smalltext{2}_\smalltext{1}} \Bigg\{\sup_{t \in [0,T]} | M^{\star}_t| + \sup_{\tau \in \cT_{0,T}} \E^{\P^{\smalltext{\balpha}\smalltext{,}\smalltext{\tau}}_{\smalltext{\cdot}}} \bigg[ \int_\tau^T \|Z^{m,\star}_t\|^2 \d t \bigg] \Bigg\} < +\infty.
\end{align*}
Finally, we have, $\d t \otimes \d \P\text{\rm--a.e.} \; (t,\omega) \in [0,T] \times \Omega$,
\begin{align}\label{align:iCond}
\notag \sup_{\tilde{a} \in A} \inf_{a \in A} \big\{f_t(X^1_{\cdot \land t}, a, \tilde{a}) +  Z_t \cdot b_t(X^1_{\cdot \land t},a,\tilde{a}) \big\} &= \inf_{a \in A} \sup_{\tilde{a} \in A}\big\{f_t(X^1_{\cdot \land t}, a, \tilde{a}) + Z_t \cdot b_t(X^1_{\cdot \land t},a,\tilde{a}) \big\} \\
&=f_t\big(X^1_{\cdot \land t}, \hat\alpha^1_t, \hat\alpha^2_t \big) +  Z_t \cdot b_t\big(X^1_{\cdot \land t},\hat\alpha^1_t, \hat\alpha^2_t \big).
\end{align}
\end{proposition}

\begin{remark}
In contrast to the {\rm BSDE} system \eqref{BSDE_NplayerGame} that characterises the $N$-player game, {\rm\Cref{prop: from0subTOBSDE}} shows that to the zero-sum game we can associate a two-dimensional {\rm BSDE}, rather than a three-dimensional one. This seems intuitive since player 1 aims to minimise exactly what player 2 seeks to maximise. Hence, if a zero-sum sub-game--perfect Nash equilibrium exists, the value of the game is determined only by the process $(Y,M^\star)$.
\end{remark}

\begin{proof}
Let some $\hat{\balpha} \in \mathcal{N}\!\mathcal{A}_{s,0}$ be fixed. Given that \Cref{assumption:regularityForNecessityZero} is simply the counterpart of \Cref{assumpDPP_NplayerGame} and \Cref{assumption:regularityForNecessity} in the zero-sum framework, \Cref{necessity_toBSDE} ensures the existence of $(Y^{1,0},Y^{2,0},Z^{1,0},Z^{2,0},M^{\star},Z^{m,\star})$ such that
\begin{align}
\notag Y_t^{1,0} &= -g\big(X^1_{\cdot \land T}\big) - G\big(\varphi(X^1_{\cdot \land T})\big)\big) + \int_t^T \bigg( -f_s\big(X^1_{\cdot \land s},\hat\alpha^1_s,\hat\alpha^2_s\big) + Z^{1,0}_s \cdot b_s\big(X^1_{\cdot \land s},\hat\alpha^1_s,\hat\alpha^2_s\big) + \frac{G^{\prime\prime}\big(M^{\star}_s\big)}{2} \|Z^{m,\star}_s\|^2 \bigg) \d s \\
\notag &\quad - \int_t^T Z^{1,0}_s \cdot \d W_s^1, \; t \in [0,T], \; \P\text{\rm--a.s.} \\
\notag Y_t^{2,0} &= g\big(X^1_{\cdot \land T}\big) + G\big(\varphi(X^1_{\cdot \land T})\big)\big) + \int_t^T \bigg( f_s\big(X^1_{\cdot \land s},\hat\alpha^1_s,\hat\alpha^2_s\big) + Z^{2,0}_s \cdot b_s\big(X^1_{\cdot \land s},\hat\alpha^1_s,\hat\alpha^2_s\big) - \frac{G^{\prime\prime}\big(M^{\star}_s\big)}{2} \|Z^{m,\star}_s\|^2 \bigg) \d s \\
\notag &\quad - \int_t^T Z^{1,0}_s \cdot \d W_s^1, \; t \in [0,T], \; \P\text{\rm--a.s.}, \\
M^{\star}_t &= \varphi\big(X^1_{\cdot \land T}\big) + \int_t^T Z^{m,\star}_s \cdot b_s\big(X^1_{\cdot \land s},\hat\alpha^1_s,\hat\alpha^2_s\big) \d s - \int_t^T Z^{m,\star}_s \cdot \d W_s^1, \; t \in [0,T], \; \P\text{\rm--a.s.}
\end{align}
Furthermore, the condition expressed in \eqref{condOpFixedPoints} can be equivalently expressed, for $\d t \otimes \d \P\text{\rm--a.e.} \; (t,\omega) \in [0,T] \times \Omega$, as
\begin{align*}
\sup_{a \in A} \big\{-f_t\big(X^1_{\cdot \land t}, a, \hat\alpha^2_t \big) +  Z^{1,0}_t \cdot b_t\big(X^1_{\cdot \land t},a,\hat\alpha^2_t\big) \big\} = & -f_t\big(X^1_{\cdot \land t}, \hat\alpha^1_t, \hat\alpha^2_t \big) +  Z^{1,0}_t \cdot b_t\big(X^1_{\cdot \land t},\hat\alpha^1_t,\hat\alpha^2_t\big)  \\
= & -f_t\big(X^1_{\cdot \land t}, \hat\alpha^1_t, \hat\alpha^2_t \big) - Z^{2,0}_t \cdot b_t\big(X^1_{\cdot \land t},\hat\alpha^1_t,\hat\alpha^2_t\big)  \\
=&- \sup_{\tilde{a} \in A}\big\{f_t\big(X^1_{\cdot \land t}, \hat\alpha^1_t, \tilde{a} \big) + Z^{2,0}_t \cdot b_t\big(X^1_{\cdot \land t},\hat\alpha^1_t,\tilde{a}\big) \big\}.
\end{align*}

We also have that $Y^{1,0}_t = J(t,\cdot,\hat\balpha_t) = -Y^{2,0}_t$, $\P\text{\rm--a.s.}$, $t \in [0,T]$, which implies that 
\begin{equation*}
Z^{1,0}_t = \big[Y^{1,0}_t,W^1\big]_t = \big[-Y^{2,0},W^1\big]_t = -Z^{2,0}_t, \; \P\text{\rm--a.s.}
\end{equation*}
Then, for $\d t \otimes \d \P\text{\rm--a.e.} \; (t,\omega) \in [0,T] \times \Omega$, it holds that
\begin{align*}
\inf_{\tilde{a} \in A} \sup_{a \in A} \big\{-f_t\big(X^1_{\cdot \land t}, a, \tilde{a} \big) +  Z^{1,0}_t \cdot b_t\big(X^1_{\cdot \land t},a,\tilde{a}\big) \big\} &\leq \sup_{a \in A} \big\{ -f_t\big(X^1_{\cdot \land t}, a, \hat\alpha^2_t \big) +  Z^{1,0}_t \cdot b_t\big(X^1_{\cdot \land t},a,\hat\alpha^2_t\big) \big\} \\
& =- \sup_{\tilde{a} \in A}\big\{f_t\big(X^1_{\cdot \land t}, \hat\alpha^1_t, \tilde{a} \big) - Z^{1,0}_t \cdot b_t\big(X^1_{\cdot \land t},\hat\alpha^1_t,\tilde{a}\big) \big\} \\
&\leq - \inf_{a \in A} \sup_{\tilde{a} \in A} \big\{f_t\big(X^1_{\cdot \land t}, a, \tilde{a} \big) -  Z^{1,0}_t \cdot b_t\big(X^1_{\cdot \land t},a,\tilde{a}\big) \big\} \\
&= \sup_{a \in A} \inf_{\tilde{a} \in A}\big\{-f_t\big(X^1_{\cdot \land t}, a, \tilde{a} \big) + Z^{1,0}_t \cdot b_t\big(X^1_{\cdot \land t},a,\tilde{a}\big) \big\},
\end{align*}
which in turn yields \Cref{align:iCond}.
\end{proof}

Before presenting the result that the two-dimensional BSDE system in \eqref{align:BSDEzeroSum} also provides a sufficient condition for characterising equilibria in the zero-sum setting, we first introduce the following assumption.

\begin{assumption}\label{assumption:genIsaacsZero}
There exist two Borel-measurable functions $\mathfrak{a}^{i,0}: [0,T] \times \cC_{m} \times \R^d \longrightarrow A$, $i \in \{1,2\}$, such that, for all $(t,x,z,a,\tilde{a}) \in [0,T] \times \cC_m \times \R^d \times A \times A$, it holds that
\begin{gather*}
f_t(x, \mathfrak{a}^{1,0}(t,x,z), \mathfrak{a}^{2,0}(t,x,z) ) +  z \cdot b_t(x,\mathfrak{a}^{1,0}(t,x,z),\mathfrak{a}^{2,0}(t,x,z)) \leq f_t(x, a, \mathfrak{a}^{2,0}(t,x,z) ) +  z \cdot b_t(x,a,\mathfrak{a}^{2,0}(t,x,z)),\\
f_t(x, \mathfrak{a}^{1,0}(t,x,z),\tilde{a}) +  z \cdot b_t(x,\mathfrak{a}^{1,0}(t,x,z),\tilde{a}) \leq f_t(x, \mathfrak{a}^{1,0}(t,x,z), \mathfrak{a}^{2,0}(t,x,z) ) +  z \cdot b_t(x,\mathfrak{a}^{1,0}(t,x,z),\mathfrak{a}^{2,0}(t,x,z)).
\end{gather*}
\end{assumption}

We refer to the previous assumption as the generalised Isaacs condition, as it extends the classical Isaacs condition:
\begin{equation*}
\inf_{a \in A} \sup_{\tilde{a} \in A} \big\{ f_t(x, a, \tilde{a} ) +  z \cdot b_t(x,a,\tilde{a}) \big\} = \sup_{\tilde{a} \in A} \inf_{a \in A} \big\{ f_t(x, a, \tilde{a} ) +  z \cdot b_t(x,a,\tilde{a}) \big\}, \; (t,x,z) \in [0,T] \times \cC_m \times \R^d.
\end{equation*}
Consequently, under \Cref{assumption:genIsaacsZero}, the Hamiltonian function $H^0:[0,T] \times \cC_m \times \R^d \times \R \times \R \longrightarrow \R$, introduced in \Cref{assumption:regularityForNecessityZero}.\ref{IsaacsA}, can be rewritten as
\begin{equation*}
H^{0}_t(x,z,\mathsf{m},z^{\star}) = \sup_{\tilde{a} \in A} \inf_{a \in A} \big\{ f_t(x, a, \tilde{a} ) +  z \cdot b_t(x,a,\tilde{a}) \big\} - \frac{1}{2} G^{\prime\prime}(\mathsf{m}) \|z^{\star}\|^2.
\end{equation*}

\begin{proposition}\label{prop: BSDETOzero}
We suppose that {\rm\Cref{assumption:regularityForNecessityZero}.\ref{differentiableG}} and {\rm\Cref{assumption:genIsaacsZero}} are satisfied. Let $(Y,Z,M^{\star},Z^{m,\star})$ be a solution to the following system of {\rm BSDEs}
\begin{align*}
\notag Y_t &= g\big(X^1_{\cdot \land T}\big) + G\big(\varphi(X^1_{\cdot \land T})\big)\big) + \int_t^T H^0_s\big(X^1_{\cdot \land s},Z_t,M^{\star}_s,Z^{m,\star}_s\big) \d s - \int_t^T Z_s \cdot \d W_s^1, \; t \in [0,T], \; \P\text{\rm--a.s.}, \\
M^{\star}_t &= \varphi\big(X^1_{\cdot \land T}\big) + \int_t^T Z^{m,\star}_s \cdot b_s\big(X^1_{\cdot \land s},\hat\alpha^1_s,\hat\alpha^2_s\big) \d s - \int_t^T Z^{m,\star}_s \cdot \d W_s^1, \; t \in [0,T], \; \P\text{\rm--a.s.},
\end{align*}
and satisfy, for some $p \geq 1$, the following integrability condition
\begin{align*}
&\sup_{\balpha \in \cA^\smalltext{2}_\smalltext{1}} \E^{\P^{\smalltext{\balpha}}} \Bigg[ \sup_{t \in [0,T]} |Y_t|^p + \bigg(\int_0^T \|{Z}_t\|^2 \d t\bigg)^\frac{p}{2} \Bigg] +\esssup_{\balpha \in \cA^\smalltext{2}_\smalltext{1}} \Bigg\{\sup_{t \in [0,T]} | M^{\star}_t| + \sup_{\tau \in \cT_{0,T}} \E^{\P^{\smalltext{\balpha}\smalltext{,}\smalltext{\tau}}_{\smalltext{\cdot}}} \bigg[ \int_\tau^T \|Z^{m,\star}_t\|^2 \d t \bigg] \Bigg\} < +\infty.
\end{align*}
If we define $\hat\balpha_t \coloneqq \big(\mathfrak{a}^{1,0}(t,X^1_{\cdot \land t},Z_t),\mathfrak{a}^{2,0}(t,X^1_{\cdot \land t},Z_t)\big)$, $t \in [0,T]$, it holds that $\hat\balpha \in \mathcal{N}\!\mathcal{A}_{s,0}$, and $J(t,\cdot,\hat\balpha) = Y_t$, $\P\text{\rm--a.s.}$, $t\in [0,T]$.
\end{proposition}

We omit the proof, as it follows exactly the same steps as the proof of \Cref{sufficiency_toBSDE}.

\begin{remark}
Thus, {\rm\Cref{prop: BSDETOzero}}, together with {\rm\Cref{prop: from0subTOBSDE}}, shows that a two-dimensional {\rm BSDE} fully describes the zero-sum game under the assumption that both agents are sophisticated. If the two agents are of the pre-committed type, meaning they do not revise their initially chosen strategies even if this leads to time-inconsistency, a similar result is given by {\rm\cite[Proposition 5.2]{djehiche2020optimal}}.
\end{remark}

\section{The mean-field game}\label{meanFieldGame}

In this section, we describe the mean-field game that is formally associated with the stochastic differential game in which the reward of each player is given by \Cref{align:criteriumNplayerGame}. This construction is carried out on the same probability space introduced in \Cref{section:ProbSetting}, where the $N$-player game is also defined. To begin, let $\xi \in \mathfrak{P}$ be a Borel-measurable function $[0,T] \ni t \longmapsto \xi_t \in \cP_2(\cC_m \times A)$, as in the notation introduced in \Cref{section:rcpd}. The criterion for the representative agent is then defined, for any $(t,\omega,\alpha) \in [0,T] \times \Omega \times \cA_1$, by
\begin{align}\label{align:payoffMeanField}
J^1(t,\omega,\alpha;\xi) \coloneqq \E^{\P^{\smalltext{\alpha}\smalltext{,}\smalltext{1}\smalltext{,}\smalltext{t}}_\smalltext{\omega}} \bigg[ \int_t^T f_s\big(X^1_{\cdot \land s}, \xi_{s}, \alpha_s\big) \d s + g\big(X^1_{\cdot \land T}, \xi_{T}^x\big) \bigg] + G\Big(\E^{\P^{\smalltext{\alpha}\smalltext{,}\smalltext{1}\smalltext{,}\smalltext{t}}_\smalltext{\omega}} \big[ \varphi_1(X^1_{\cdot \land T}) \big],\varphi_2(\xi_{T}^x)\Big),
\end{align}
where $\xi_{T}^x \in \cP_2(\cC_m)$ denotes the first marginal of $\xi_{T}$. Intuitively, we can think of the mean-field game problem as consisting of two steps. First, for a given function $\xi \in \mathfrak{P}$, one solves a family of sub-games determined by the future selves of the representative agent. Then, one searches for a fixed point, namely a flow $\xi$ such that, when all the versions of the representative agents play optimally in response to it, their collective behaviour is consistent with $\xi$ itself. This idea is formalised by the notion of a sub-game--perfect mean-field equilibrium, which coincides with the standard equilibrium concept studied in the time-consistent mean-field literature.

\begin{definition}\label{def:meanFieldEquilibrium}
Let $\hat\alpha \in \cA_1$, and $\varepsilon >0$. We define
\begin{align*}
\ell_\varepsilon \coloneqq \inf \Big\{ \ell >0 : \exists (t,\alpha) \in [0,T] \times \cA_1, \; \P \big[\big\{\omega \in \Omega: J^1\big(t,\omega,\hat\alpha;\xi\big) < J^1\big(t,\omega,\alpha\otimes_{t+\ell}\hat\alpha;\xi\big) - \varepsilon \ell \big\}\big] >0 \Big\}.
\end{align*}
We say that $\hat\alpha$ is a {\rm sub-game--perfect mean-field equilibrium} for this mean-field game if it holds that $\ell_\varepsilon >0$ for any $\varepsilon >0$, and $\P^{\hat\alpha} \circ (X^1_{\cdot \land t}, \hat\alpha_t)^{-1} = \xi_t$, for $\d t\text{\rm--a.e.} \; t \in [0,T]$.
\end{definition}

\subsection{The characterising BSDE system}\label{section:BSDEsystemMeanField}

As in the case of the $N$-player game, our goal is to reduce the mean-field game to a system of BSDEs, with the ultimate aim of showing that the BSDEs derived in \Cref{BSDE_NplayerGame} converge to those characterising the mean-field game. This approach is inspired by \cite{possamai2025non}, which builds on the backward propagation of chaos techniques developed by \citeauthor*{lauriere2022backward} \cite{lauriere2022backward}. Before proceeding, we introduce the Hamiltonian $H: [0,T] \times \cC_{m} \times \cP_2(\cC_m \times A) \times \R^d \times \R \times \R \times \R^d \times \R^d \longrightarrow \R$, which characterises the control problem faced by the representative player. It is given by
\begin{align*}
H_t(x,\xi,z,m^\star,n^\star,z^{m,\star},z^{n,\star}) &\coloneqq \sup_{a \in A} \big\{h_t(x,\xi,z,a)\big\} - \partial^2_{m,n} G(m^\star,n^\star) z^{m,\star} \cdot z^{n,\star}\\
&\quad -\frac{1}{2} \big( \partial^2_{m,m} G(m^\star,n^\star) \|z^{m,\star}\|^2 + \partial^2_{n,n} G(m^\star,n^\star) \|z^{n,\star}\|^2 \big),
\end{align*}
where 
\begin{align*}
h_t(x,\xi,z,a) \coloneqq f_t(x, \xi, a) +  z \cdot b_t(x, \xi, a), \; (t,x,\xi,z,a) \in [0,T] \times \cC_m \times \cP_2(\cC_m \times A) \times \R^d \times A.
\end{align*}
Assuming the existence of a sub-game--perfect mean-field equilibrium $\hat\alpha \in \cA_1$, we define the processes 
\begin{align*}
M_t^{1,\star} \coloneqq \E^{\P^{\smalltext{\hat\alpha}\smalltext{,}\smalltext{1}\smalltext{,}\smalltext{t}}_\smalltext{\cdot}} \big[ \varphi_1(X^1_{\cdot \land T}) \big] \; \text{and} \; N_t^{\star} \coloneqq \varphi_2(\cL_{\hat\alpha}(X^1_{\cdot \land T})), \; t \in [0,T].
\end{align*}
We also define the value process associated with $\hat\alpha$ by
\begin{align*}
V^1_t \coloneqq J^1\big(t,\cdot,\hat\alpha;(\cL_{\hat\alpha}(X^1_{\cdot \land s}, \hat\alpha_s))_{s\in[0,T]}\big), \; t \in [0,T].
\end{align*}
For notational convenience, we write $\cL_{\hat\alpha}(\eta)$ to denote the law of any random variable $\eta$ under the probability measure $\P^{{\hat\alpha}}$.

\begin{remark}
The processes $(M^{1,\star},N^\star)$ defined above are the analogues of those introduced in \eqref{align:MiNi_starDef} for the $N$-player game. Since the empirical measure $L^N(\X^N_{\cdot \land T})$ converges to the law of $X^1_{\cdot \land T}$, which is deterministic, the corresponding counterpart of $N^{\star,N}$ is the constant process $N^{\star}$. Consequently, it can be represented as the solution of a {\rm BSDE} with zero generator, and with a constant $Y$-component and identically zero $Z$-component. In light of this simplification, the Hamiltonian introduced previously reduces to $H: [0,T] \times \cC_{m} \times \cP_2(\cC_m \times A) \times \R^d \times \R \times \R \times \R^d \longrightarrow \R$ defined by
\begin{align}\label{eq:hamiltonian_meanField}
H_t\big(x,\xi,z,m^\star,n^\star,z^{m,\star}\big) &\coloneqq \sup_{a \in A} \big\{h_t(x, \xi, z, a)\big\} - \frac{1}{2} \partial^2_{m,m} G(m^\star,n^\star) \|z^{m,\star}\|^2.
\end{align}
We adopt this reduced form in what follows and continue to refer to it as the Hamiltonian associated with the mean-field game, using the same notation $H$ for simplicity.
\end{remark}

\begin{assumption}\label{assumpNecessity_meanField}
\begin{enumerate}[label={$(\roman*)$}]
\item\label{boundedPhisMart_meanField} The functions $\cC_m \ni x \longmapsto \varphi_1(x)$ and $\cP_2(\cC_{m \times N}) \ni \xi \longmapsto \varphi_2(\xi)$ are bounded$;$
\item\label{Glip_meanField} the function $\R \times \R \ni (m^{\star},n^{\star}) \longmapsto G(m^{\star},n^{\star})$ is twice continuously differentiable with Lipschitz-continuous derivatives $\partial_{m}G(m^{\star},n^{\star})$, $\partial_{n}G(m^{\star},n^{\star})$, $\partial^2_{m,m}G(m^{\star},n^{\star})$, $\partial^2_{m,n}G(m^{\star},n^{\star})$, $\partial^2_{n,n}G(m^{\star},n^{\star});$
\item\label{modulusCondExpMart_meanField} there exists a constant $c>0$ and a modulus of continuity $\rho$ such that, for any $(\alpha,t,\tilde{t},t^\prime) \in \cA_1 \times [0,T] \times [t,T] \times [\tilde{t},T]$, we have
\begin{align*}
\E^{\P^{\smalltext{\alpha}\smalltext{,}\smalltext{t}}_{\smalltext{\cdot}}} \Big[ \big| \E^{\P^{\smalltext{\balpha}\smalltext{,}\smalltext{N}\smalltext{,}\smalltext{\tilde{t}}}_\smalltext{\cdot}} \big[ M^{1,\star}_{t^\prime} \big] - M^{1,\star}_{\tilde{t}} \big|^2 +  \big| \E^{\P^{\smalltext{\alpha}\smalltext{,}\smalltext{\tilde{t}}}_\smalltext{\cdot}} \big[ N^{\star}_{{t}^\prime} \big] - N^{\star}_{\tilde{t}} \big|^2 \Big] \leq c |t^\prime-\tilde{t}| \rho \big(|t^\prime-\tilde{t}|\big), \; \P\text{\rm--a.s.};
\end{align*}
\item there exists some $p \geq 1$ and an element $a_0 \in A$ such that 
\begin{align*}
\sup_{\alpha \in \cA^\smalltext{1}} \E^{\P^{\smalltext{\alpha}}} \bigg[ \big|g\big(X^1_{\cdot \land T},\cL_{\alpha}(X^1_{\cdot \land T})\big)\big|^p + \int_0^T \big|h^{1}_s\big(X^1_{\cdot \land s},\cL_{\alpha}(X^1_{\cdot \land s},\alpha_s),\mathbf{0},a_0\big)\big|^p \d s \bigg] <+\infty.
\end{align*}
\end{enumerate}
\end{assumption}

\begin{proposition}\label{necessity_toBSDE_meanFieldGame}
Let {\rm\Cref{assumpNecessity_meanField}} hold, and let $\hat\alpha$ be a sub-game--perfect mean-field equilibrium. Then, there exists a quadruple $(Y^{1}, Z^{1,1}, M^{1,\star}, \Z^{1,m,1,\star,})$ satisfying the following {\rm BSDE} system
\begin{align}\label{BSDE_eanField_fromEquilibrium}
\notag Y^1_t &= g\big(X^1_{\cdot \land T}, \cL_{\hat\alpha}(X^1_{\cdot \land T})\big) + G\big(\varphi_1(X^1_{\cdot \land T}),\varphi_2(\cL_{\hat\alpha}(X^1_{\cdot \land T}))\big)+ \int_t^T f_s\big(X^1_{\cdot \land s},\cL_{\hat\alpha}(X^1_{\cdot \land s},\hat\alpha_s),\hat\alpha_s\big) \d s \\
\notag &\quad  - \frac{1}{2} \int_t^T \partial^2_{m,m}G\big(M^{1,\star}_s,N^{\star}_s\big) \big\|Z^{1,m,1,\star}_s\big\|^2 \d s - \int_t^T Z^{1,1}_s \cdot \d \big(W_s^{\hat\alpha}\big)^1, \; t \in [0,T], \; \P \text{\rm--a.s.}, \\
\notag M^{1,\star}_t &\;= \varphi_1(X^1_{\cdot \land T}) - \int_t^T Z^{1,m,1,\star}_s \cdot \d \big(W_s^{\hat\alpha}\big)^1, \; t \in [0,T], \; \P \text{\rm--a.s.},\\
N^{\star}_t &\coloneqq \varphi_2(\cL_{\hat\alpha}(X^1_{\cdot \land T})), \; t \in [0,T],\; \P\text{\rm--a.s.},
\end{align}
where, for notational convenience, we write $W^{\hat\alpha} \coloneqq W_s^{\hat\alpha,\xi}$ with $\xi_t \coloneqq \cL_{\hat\alpha}(X^1_{\cdot \land t}, \hat\alpha_t)$ for each $t \in [0,T]$. Furthermore, there exists some $p \geq 1$ such that
\begin{align*}
&\sup_{\alpha \in \cA^\smalltext{1}} \E^{\P^{\smalltext{\alpha}}} \Bigg[ \sup_{t \in [0,T]} \big|Y^{1}_t\big|^p + \bigg(\int_0^T \big\|{Z}_t^{1,1}\big\|^2 \d t\bigg)^\frac{p}{2} \Bigg] +\esssup_{\alpha \in \cA^\smalltext{1}} \Bigg\{\sup_{t \in [0,T]}  \big| M^{1,\star}_t\big| + \sup_{\tau \in \cT_{0,T}} \E^{\P^{\smalltext{\alpha}\smalltext{,}\smalltext{1}\smalltext{,}\smalltext{\tau}}_{\smalltext{\cdot}}} \bigg[ \int_\tau^T \big\|Z^{1,m,1,\star}_t\big\|^2 \d t \bigg] \Bigg\} < +\infty.
\end{align*}
The value process satisfies $V^1_t = Y^{1}_t$, $\P\text{\rm--a.s.}$, for any $t \in [0,T]$, and the sub-game--perfect mean-field equilibrium $\hat\alpha$ satisfies
\begin{align*}
\hat\alpha_t \in \argmax_{a \in A} \big\{h_t\big(X^1_{\cdot \land t}, \cL_{\hat\alpha}(X^1_{\cdot \land t},\hat\alpha_t), Z^{1,1}_t, a\big)\big\}, \; \text{\rm for}\; \d t \otimes \d \P\text{\rm--a.e.} \; (t,\omega) \in [0,T] \times \Omega.
\end{align*}
\end{proposition}

The proof is omitted, as it mirrors the argument given in \Cref{necessity_toBSDE}. Similarly, the sufficiency result follows along the lines of \Cref{sufficiency_toBSDE}, so its proof is also omitted. We first present the assumptions.

\begin{assumption}\label{assumption:borelMeasurability_Lip2}
Let $(t,x,\xi,z,m^\star,n^\star,z^{m,\star}) \in [0,T] \times \cC_{m} \times \cP_2(\cC_m \times A) \times \R^d \times \R \times \R \times \R^d$. For any maximiser $\hat{a}$ of the Hamiltonian $H_t(x, \xi, z, m^\star, n^\star, z^{m,\star})$ associated with the mean-field game and described in {\rm\Cref{eq:hamiltonian_meanField}}, there exists a Borel-measurable function $\mathfrak{a}: [0,T] \times \cC_{m} \times \cP_2(\cC_m \times A) \times \R^d \times \R \times \R \times \R^d \longrightarrow A$ satisfying $\mathfrak{a}(t,x, \xi, z, m^\star, n^\star, z^{m,\star}) = \hat{a}$.
\end{assumption}

\begin{proposition}\label{sufficiency_toBSDE_meanFieldGame}
Let {\rm\Cref{assumpNecessity_meanField}}.\ref{Glip_meanField} and {\rm\Cref{assumption:borelMeasurability_Lip2}} hold. We consider $(Y^{1}, Z^{1,1}, M^{1,\star}, \Z^{1,m,1,\star,})$ solving the following system of {\rm BSDEs}
\begin{align}\label{BSDE_eanField_toEquilibrium}
\notag Y^1_t &= g\big(X^1_{\cdot \land T}, \cL_{\hat\alpha}(X^1_{\cdot \land T})\big) + G\big(\varphi_1(X^1_{\cdot \land T}),\varphi_2(\cL_{\hat\alpha}(X^1_{\cdot \land T}))\big) + \int_t^T f_s\big(X^1_{\cdot \land s},\cL_{\hat\alpha}(X^1_{\cdot \land s},\hat\alpha_s),\hat\alpha_s\big) \d s \\
\notag &\quad - \frac{1}{2} \int_t^T \partial^2_{m,m}G\big(M^{1,\star}_s,N^{\star}_s\big) \big\|Z^{1,m,1,\star}_s\big\|^2 \d s - \int_t^T Z^{1,1}_s \cdot \d \big(W_s^{\hat\alpha}\big)^1, \; t \in [0,T], \; \P \text{\rm--a.s.}, \\
\notag M^{1,\star}_t &\;= \varphi_1(X^1_{\cdot \land T}) - \int_t^T Z^{1,m,1,\star}_s \cdot \d \big(W_s^{\hat\alpha}\big)^1, \; t \in [0,T], \; \P \text{\rm--a.s.}, \\
\notag N^{\star}_t &\coloneqq \varphi_2(\cL_{\hat\alpha}(X^1_{\cdot \land T})), \; t \in [0,T], \; \P\text{\rm--a.s.},\\
\notag \hat\alpha_t &\coloneqq \mathfrak{a}\big(t,X^1_{\cdot \land t},\cL_{\hat\alpha}(X^1_{\cdot \land t},\hat\alpha_t),Z^{1,1}_t,M^{1,\star}_t,Z^{1,m,1,\star}_t\big), \;  \mathrm{d}t\otimes\P\text{\rm--a.e.},\\
\frac{\d \P^{\hat\alpha}}{\d \P} &\coloneqq \cE\bigg( \int_0^\cdot b_s\big(X^1_{\cdot \land s},\cL_{\hat\alpha}(X^1_{\cdot \land s},\hat\alpha_s),\hat\alpha_s\big) \cdot \d W^1_s \bigg)_T,
\end{align}
where the $(\F^1,\P^{\hat\alpha})$--Brownian motion $(W^{\hat\alpha})^1$ is defined by
\begin{align*}
(W^{\hat\alpha}_t)^1 \coloneqq W^1_t-\int_0^t b_s\big(X^1_{\cdot \land s},\cL_{\hat\alpha}(X^1_{\cdot \land s},\hat\alpha_s),\hat\alpha_s\big) \d s, \; t \in [0,T].
\end{align*}
We additionally assume the existence of some $p \geq 1$ such that 
\begin{align*}
&\sup_{\alpha \in \cA^\smalltext{1}} \E^{\P^{\smalltext{\alpha}}} \Bigg[ \sup_{t \in [0,T]} \big|Y^{1}_t\big|^p + \bigg(\int_0^T \big\|{Z}_t^{1,1}\big\|^2 \d t\bigg)^\frac{p}{2} \Bigg] +\esssup_{\alpha \in \cA^\smalltext{1}} \Bigg(\sup_{t \in [0,T]} \big| M^{1,\star}_t\big| + \sup_{\tau \in \cT_{0,T}} \E^{\P^{\smalltext{\alpha}\smalltext{,}\smalltext{1}\smalltext{,}\smalltext{\tau}}_{\smalltext{\cdot}}} \bigg[ \int_\tau^T \big\|Z^{1,m,1,\star}_t\big\|^2 \d t \bigg] \Bigg) <+\infty.
\end{align*}
It then holds that $\hat\alpha$ is a sub-game--perfect mean-field equilibrium, and $J^1\big(t,\cdot,\hat\alpha;(\cL_{\hat\alpha}(X^1_{\cdot \land s}, \hat\alpha_s))_{s\in[0,T]}\big) = Y^{1}_t$, $\P\text{\rm--a.s.}$, for any $t\in [0,T]$.
\end{proposition}

\section{The convergence result}\label{section:convRes_mainResult}

In this section, we discuss the convergence to the mean-field game limit. Our analysis leverages the {\rm BSDE} characterisation of both the $N$-player game and its mean-field counterpart, building on the methodology developed in the time-consistent framework of \cite[Theorem 2.10]{possamai2025non}, where the authors prove the convergence of the $Y$-component of the {\rm BSDE} system associated with the $N$-player game to the corresponding component in the mean-field game at the initial time. In contrast, the problems we consider are inherently time-inconsistent. Consequently, the convergence analysis must be carried out over the entire time interval $[0,T]$, rather than being limited to the initial time. This is because, by the definitions of equilibrium in \Cref{def:NashEquilibrium} and \Cref{def:meanFieldEquilibrium}, one must account for the value of the problem from the perspective of each incarnation of each player's evolving preferences over time.

\begin{assumption}\label{assumpConvThm}
The following conditions are verufied
\begin{enumerate}[label={$(\roman*)$},leftmargin=1em]
\item\label{previousAssum}  {\rm\Cref{assumpDPP_NplayerGame}} and {\rm\Cref{assumption:regularityForNecessity}} hold$;$
\item\label{LambdaFunc} let $(t, \mathbf{x}, \mathbf{z}, \mathsf{m}^\star, \mathsf{n}^\star, \mathsf{z}^{m,\star}, \mathsf{z}^{n,\star}) \in [0,T] \times \cC_{m \times N} \times (\R^{d \times N})^N \times \R^N \times \R^N \times (\R^{d \times N})^N \times (\R^{d \times N})^N$. For every fixed point $\mathbf{a}^N \coloneqq (a^{1,N},\ldots,a^{N,N}) \in \cO_{N}(t, \mathbf{x}, \mathbf{z}, \mathsf{m}^\star, \mathsf{n}^\star, \mathsf{z}^{m,\star}, \mathsf{z}^{n,\star})$, there exits a Borel-measurable map $\Lambda: [0,T] \times \cC_m \times \cP_2(\cC_{m}) \times \R^d \times \R^d \times \R^d \times \R \longrightarrow A$ and a collection of functions $(\aleph^{i,N})_{i \in \{1,\ldots,N\}}$, where each $\aleph^{i,N}: [0,T] \times \cC_{m \times N} \times (\R^{d \times N})^N \times (\R^{d \times N})^N \times (\R^{d \times N})^N$, such that, for any $i\in\{1,\ldots,N\}$
\begin{align*}
a^{i,N}(t, \mathbf{x}, \mathbf{z}, \mathsf{m}^\star, \mathsf{n}^\star, \mathsf{z}^{m,\star}, \mathsf{z}^{n,\star}) = \Lambda_t\big(x^i,L^N(\mathbf{x}),z^{i,i},z^{i,m,i,\star},z^{i,n,i,\star},\aleph^{i,N}_t(\mathbf{x},\mathbf{z},\mathsf{z}^{m,\star},\mathsf{z}^{n,\star})\big).
\end{align*}
Moreover, for any $(t,x,\xi,m^\star,n^\star,z^{m,\star}) \in [0,T] \times \cC_{m} \times \cP_2(\cC_m \times A) \times \R^d \times \R \times \R \times \R^d$, we have
\begin{align*}
\Lambda_t\big(x,\xi^x,z,z^{m,\star},\mathbf{0},0\big) \in \argmax_{a \in A} \big\{h_t(x,\xi,z,a)\big\},
\end{align*}
where $\xi^x \in \cP_2(\cC_m)$ denotes the first marginal of $\xi;$
\item\label{lipLambda_growthAleph} the function $\Lambda: [0,T] \times \cC_m \times \cP_2(\cC_{m}) \times \R^d \times \R^d \times \R^d \times \R \longrightarrow A$, introduced in \ref{LambdaFunc}, is assumed to be Lipschitz-continuous with respect to all its arguments, with Lipschitz-constant $\ell_\Lambda >0$, and each function $\aleph^{i,N}: [0,T] \times \cC_{m \times N} \times (\R^{d \times N})^N \times (\R^{d \times N})^N \times (\R^{d \times N})^N$ satisfies that there exists a sequence $(R_N)_{N \in \N^\smalltext{\star}}$ with values in $\R_\smallertext{+}$ such that
\begin{align*}
\big|\aleph^{i,N}_t(\mathbf{x},\mathbf{z},\mathsf{z}^{m,\star},\mathsf{z}^{n,\star})\big| \leq R_N \Bigg( 1 + \|x^i\|_\infty + \sum_{\ell =1}^N \|z^{i,\ell}\| + \sum_{\ell =1}^N \|z^{i,m,\ell,\star}\| + \sum_{\ell=1}^N \|z^{n,\ell,\star}\| \Bigg),
\end{align*}
for any $(t,\mathbf{x},\mathbf{z},\mathsf{z}^{m,\star},\mathsf{z}^{n,\star}) \in [0,T] \times \cC_{m \times N} \times (\R^{d \times N})^N \times (\R^{d \times N})^N \times (\R^{d \times N})^N$. Moreover, $\lim_{N \rightarrow \infty} N R^2_N = 0$ and $\lim_{N \rightarrow \infty} N^2 R^2_N = h$, for some constant $h \in \R_\smallertext{+}^\star;$
\item\label{uniqueMeanField} the mean-field game admits a unique sub-game--perfect mean-field equilibrium $\hat\alpha \in\cA_1;$
\item\label{auxiliarySystem} for any $N \in \N^\star$, any filtered probability space $(\Omega^\prime,\cF^\prime,\mathbb{F}^\prime=(\cF^\prime_t)_{t \in [0,T]},\P^\prime)$, any family of $\R^m$-valued, $\cF^\prime_0$-measurable random variables $(X^i_0)_{i \in \N^{\smalltext{\star}}}$, and any family of $\P^\prime$-independent, $\R^d$-valued $(\mathbb{F}^\prime,\P^\prime)$--Brownian motions $(B^i)_{i \in \{1,\ldots,N\}}$ that are $\P^\prime$-independent of $(X^i_0)_{i \in \N^{\smalltext{\star}}}$, the {\rm FBSDE} admits exactly one solution $(\X^N,\Y^N,\Z^N,\M^{\star,N},\N^{\star,N},\Z^{m,\star,N},\Z^{n,\star,N})$
\begin{align*}
X^{i}_t &= X^i_0 + \int_0^t \sigma_s(X^{i}_{\cdot \land s}) b_s\big(X^{i}_{\cdot \land s},L^N\big(\X^N_{\cdot \land s},\balpha^N_s\big),\alpha^{i,N}_s\big) \d s + \int_0^t \sigma_s(X^i_{\cdot \land s}) \d B_s^i, \; t \in [0,T], \; \P^\prime \text{\rm--a.s.}, \\
\notag Y^i_t &= g\big(X^{i}_{\cdot \land T}, L^N\big(\X^N_{\cdot \land T}\big)\big) + G\big(\varphi_1(X^{i}_{\cdot \land T}),\varphi_2(L^N\big(\X^N_{\cdot \land T}\big)\big) + \int_t^T f_s\big(X^{i}_{\cdot \land s},L^N\big(\X^N_{\cdot \land s},\balpha^N_s\big),\alpha^{i,N}_s\big) \d s\\
\notag &\quad - \int_t^T \Bigg(\partial^2_{m,n}G\big(M^{i,\star,N}_s,N^{\star,N}_s\big) \sum_{\ell =1}^N Z^{i,m,\ell,\star,N}_s \cdot Z^{n,\ell,\star,N}_s +\partial^2_{m,m}G\big(M^{i,\star,N}_s,N^{\star,N}_s\big) \sum_{\ell =1}^N \big\|Z^{i,m,\ell,\star,N}_s\big\|^2 \Bigg)\d s\\
\notag&\quad - \frac{1}{2} \int_t^T \partial^2_{n,n}G\big(M^{i,\star,N}_s,N^{\star,N}_s\big) \sum_{\ell =1}^N \big\|Z^{n,\ell,\star,N}_s\big\|^2 \d s \\
\notag &\quad - \int_t^T \sum_{\ell =1}^N Z^{i,\ell,N}_s \cdot \d B_s^\ell, \; t \in [0,T], \; \P^\prime\text{\rm--a.s.}, \\
\notag M^{i,\star,N}_t &= \varphi_1\big(X^{i}_{\cdot \land T}\big) - \int_t^T \sum_{\ell=1}^N Z^{i,m,\ell,\star,N}_s \cdot \d B_s^\ell, \; t \in [0,T], \; \P^\prime\text{\rm--a.s.}, \\
\notag N^{\star,N}_t &= \varphi_2\big(L^N\big(\X^N_{\cdot \land T}\big)\big) - \int_t^T \sum_{\ell =1}^N Z^{n,\ell,\star,N}_s \cdot \d B_s^\ell, \; t \in [0,T], \; \P^\prime \text{\rm--a.s.}, \\
\alpha^{i,N}_t &= \Lambda_t\big(X^{i}_{\cdot \land t}, L^N\big(\X^N_{\cdot \land t}\big), Z^{i,i,N}_t, Z^{i,m,i,\star,N}_t,Z^{i,n,i,\star,N}_t,\mathbf{0}\big),\; \mathrm{d}t\otimes\P^\prime\text{\rm--a.e.},
\end{align*}
such that, for some $ p \geq 1$, 
\begin{align*}
&\E^{\P^\prime} \Bigg[ \sup_{t \in [0,T]} \Big\{ \big|Y^{i,N}_t\big|^p + \big| M^{i,\star,N}_t\big|^p + \big| N^{i,\star,N}_t \big|^p \Big\} +\Bigg(\int_0^T \sum_{\ell =1}^N \Big( \big\|{Z}_t^{i,\ell,N}\big\|^2 + \big\|Z^{i,m,\ell,\star,N}_t\big\|^2 + \big\|Z^{i,n,\ell,\star,N}_t\big\|^2 \d t\Bigg)^\frac{p}{2} \Bigg] <+\infty;
\end{align*}
\item\label{lip_gGf} the functions $(g+G)(\varphi_1,\varphi_2): \cC_m \times \cP_2(\cC_m) \longrightarrow \R$, defined by the composition $(x,\xi) \longmapsto (g+G)(\varphi_1,\varphi_2)(x,\xi)) \coloneqq g(x,\xi) + G(\varphi_1
(x), \varphi_2(\xi))$, $f: [0,T] \times \cC_m \times \cP_2(\cC_{m} \times A) \times A \longrightarrow \R$, $\varphi_1: \cC_m \longrightarrow \R$, and $\varphi_2: \cP_2(\cC_{m}) \longrightarrow \R$ are assumed to be Lipschitz-continuous, with Lipschitz-constants $\ell_{g+G,\varphi_\smalltext{1},\varphi_\smalltext{2}} >0$, $\ell_f >0$, $\ell_{\varphi_\smalltext{1}} >0$ and $\ell_{\varphi_\smalltext{2}} >0$, respectively$;$
\item\label{boundedSecondDerivativeG} the functions $\partial^2_{m,m}G$, $\partial^2_{m,n}G$, $\partial^2_{n,n}G: \R \times \R \longrightarrow \R$ are assumed to be Lipschitz-continuous, as already stated in \ref{previousAssum}, specifically in {\rm\Cref{assumpDPP_NplayerGame}}.\ref{Glip_NplayerGame}, with a common Lipschitz-constant $\ell_{\partial^\smalltext{2}G}>0;$
\item\label{boundedPhi_12} the functions $\varphi_1: \cC_m \longrightarrow \R$, and $\varphi_2: \cP_2(\cC_m) \longrightarrow \R$ are assumed to be bounded by constants $c_{\varphi_\smalltext{1}}$ and $c_{\varphi_\smalltext{2}}$, respectively, as specified in \ref{previousAssum}, particularly in {\rm\Cref{assumpDPP_NplayerGame}}.\ref{boundedPhisMart_NplayerGame}$;$
\item\label{lipSigma} the function $\sigma:[0,T] \times \cC_m \longrightarrow \R^{m \times d}$ satisfies the following growth condition
\begin{align*}
\|\sigma_t(x)\| \leq \ell_{\sigma} \big( 1 + \|x\|_\infty \big), \; (t,x) \in [0,T] \times \cC_m,
\end{align*}
and the following Lipschitz-continuity condition, with Lipschitz-constant $\ell_\sigma >0$
\begin{align*}
\sqrt{{\rm{Tr}}\big[(\sigma_t(x)-\sigma_t(\tilde{x}))(\sigma_t(x)-\sigma_t(\tilde{x}))^\top\big]} \leq \ell_\sigma \|x - \tilde{x}\|_\infty, \; (x,\tilde{x}) \in \cC_m \times \cC_m;
\end{align*}
\item\label{diss} for any $(t,x,\xi,a) \in [0,T] \times \cC_m \times \cP_2(\cC_m) \times A$, the function $\cP_2(\cC_m) \times A \ni (\xi,a) \longmapsto \sigma_t(x) b_t(x,\xi,a)$ is Lipschitz-continuous with Lipschitz-constant $\ell_{\sigma b}>0$, and the function $\cC_m \ni x \longmapsto \sigma_t(x) b_t(x,\xi,a)$ is dissipative, meaning that there exists a constant $K_{\sigma b} >0$ such that
\begin{align*}
(x_t - \tilde{x}_t) \cdot (\sigma_t(x) b_t(x,\xi,a) - \sigma_t(\tilde{x}) b_t(\tilde{x},\xi,a)) \leq - K_{\sigma b} \|x - \tilde{x}\|^2_\infty, \; (x,\tilde{x}) \in \cC_m \times \cC_m;
\end{align*}
\item\label{growth_f_gG} the initial conditions $(X^i_0)_{i \in \N^{\smalltext{\star}}}$ introduced in {\rm\Cref{section:ProbSetting}} are $\P$--{\rm i.i.d.}, and for some $\bar{p} \geq 1$, they verify
\begin{align*}
\E^\P\big[\|X^i_0\|^{2\bar{p}}\big] <+\infty, \; i \in \N^\star.
\end{align*}

Moreover, the function $f$ is such that there exists a constant $\ell_f>0$ and some $a_0 \in A$ with, for any $(t,x,\xi,a) \in [0,T] \times \cC_m \times \cP_2(\cC_m \times A) \times A$
\begin{align*}
|f_t(x,\xi,a)| \leq \ell_f \bigg( 1 + d_A^{\bar{p}}(a,a_0) + \|x\|^{\bar{p}}_\infty + \int_{\cC_\smalltext{m} \times A} \big(\|\tilde{x}\|_\infty^{\bar{p}} + d^{\bar{p}}_A(\tilde{a},a_0)\big) \xi(\d \tilde{x},\d \tilde{a}) \bigg).
\end{align*}
The function $g$ satisfies that there exists a constant $\ell_g>0$ such that
\begin{align*}
|g(x,\xi)| \leq \ell_g \bigg( 1 + \|x\|^{\bar{p}}_\infty + \int_{\cC_\smalltext{m}} \|\tilde{x}\|_\infty^{\bar{p}} \xi(\d \tilde{x}) \bigg), \; (x,\xi) \in \cC_m \times \cP_2(\cC_m).
\end{align*}
\end{enumerate}
\end{assumption}

\begin{remark}
\begin{enumerate}[label={$(\roman*)$}]\label{ass:comments}
\item\label{nocontrolzerofunctions} If there is no interaction through the strategies---namely, if the function $b$ in \eqref{eq:changeMeasureNplayerGame}, which defines the change of measure, and the function $f$ in the criterion definition \eqref{align:criteriumNplayerGame} do not depend on the strategies of the other players, and equivalently, if $b$ in \eqref{eq:changeMeasureMeanField} and $f$ in \eqref{align:payoffMeanField} depend only on the first marginal of $\xi$---then, the functions $(\aleph^{i,N})_{i \in\{1,\ldots,N\}}$ introduced in \emph{\Cref{assumpConvThm}.\ref{LambdaFunc}} vanish, and standard measurable selection arguments allow the construction of a Borel-measurable function $\Lambda$. In this case, each $R_N$ introduced in \Cref{assumpConvThm}.\ref{lipLambda_growthAleph} is equal to zero. On the other hand, in the presence of interaction through strategies, we refer to \emph{\cite[Section 2.4.1.1]{possamai2025non}} for a detailed discussion$;$

\item \emph{\Cref{assumpConvThm}.\ref{uniqueMeanField}} requires the uniqueness of the sub-game--perfect mean-field equilibrium, since this is equivalent to the uniqueness of the solution to the mean-field \emph{BSDE} system, as established in \emph{\Cref{necessity_toBSDE_meanFieldGame,sufficiency_toBSDE_meanFieldGame}}. In the proof, we rely on the uniqueness of the solution to the mean-field system; without it, we could only show that the $N$-player {\rm BSDE} system converges to some solution of the mean-field {\rm BSDE} system, which would not necessarily coincide with the value process of the mean-field game. It is in any case an expected assumption when one wants to prove full convergence of equilibria$;$

\item comparing \emph{\Cref{assumpConvThm}.\ref{auxiliarySystem}} with \emph{\cite[Assumption 2.9.(vi)]{possamai2025non}}, we note that our assumption is stronger. This is required for the Yamada--Watanabe result in \eqref{YW_lawEq}, since its proof relies on lifting the solutions of \eqref{align:intermediateSystemNplayerGame} and \eqref{align:intermediateSystemNplayerGame_rcpdMeanField} to a common probability space and applying pathwise uniqueness to construct a strong solution, which in turn ensures the law equality stated in \eqref{YW_lawEq}$;$

\item the $\P$--{\rm i.i.d.} assumption on the initial conditions $(X^i_0)_{i \in \N^{\smalltext{\star}}}$ in \emph{\Cref{assumpConvThm}.\ref{growth_f_gG}} ensures that the processes themselves $(X^i)_{i \in \N^{\smalltext{\star}}}$ are $\P$--{\rm i.i.d.}, which is necessary both to apply the strong law of large numbers and to construct independent copies of the mean-field game used in the estimates of \emph{\Cref{convTOmeanFieldGames}}.
\end{enumerate}
\end{remark}

\begin{theorem}\label{theorem:convergenceTheorem}
Let {\rm\Cref{assumpConvThm}} hold. In addition, assume that $K_{\sigma b} \geq \delta$, where $\delta >0$ is a constant depending on $\ell_b$, $\ell_\sigma$, $\ell_{\sigma b}$, $\ell_f$, $\ell_{\varphi_\smalltext{1}}$, $\ell_{\varphi_\smalltext{2}}$, $\ell_{g+G,\varphi_\smalltext{1},\varphi_\smalltext{2}}$, $\ell_{\partial^\smalltext{2}G}$, $\ell_\Lambda$ and $T$. Let $(\hat{\alpha}^{1,N})_{N \in \mathbb{N}^\star}$ be a sequence of sub-game--perfect Nash equilibria for the multi-player game, and let $(V^{1,N})_{N \in \mathbb{N}^\star}$ denote the associated value processes. Then, $(V^{1,N})_{N \in \mathbb{N}^\star}$ converges to the value process $V^1$ of the mean-field game. More precisely, there exist a constant $C>0$ and a function $\eta: \Omega \times [0,T] \times \N^\star \longrightarrow \R^\star_{\smalltext{+}}$ such that, for any $N \in \mathbb{N}^\star$, for any $u \in [0,T]$,
\begin{align}\label{align:convValueProcesses}
\big|V^{1,N}_u(\omega)-V^1_u(\omega)\big|^2 \leq C \big( \eta(\omega,u,N) + \gamma(\omega,u,N) \big), \; \P\text{\rm--a.e.} \; \omega\in\Omega,
\end{align}
where 
\begin{gather*}
\eta(\omega,u,N) \coloneqq \eta\big(R_N,(\|X^{i}_u(\omega)\|)_{i \in \{1,\ldots,N\}}\big), \; (\omega,u,N) \in \Omega \times [0,T] \times \N^\star, \; \text{with} \; \lim_{N \rightarrow \infty} \eta(\omega,u,N) = 0, \; \P\text{\rm--a.e.} \; \omega\in\Omega, \\
\gamma(\omega,u,N) \coloneqq \sup_{ t \in [u,T]} \E^{\P^{\smalltext{\hat\alpha}\smalltext{,}\smalltext{N}\smalltext{,}\smalltext{u}}_\smalltext{\omega}} \Big[ \cW_2^2\big(L^N\big(\X^N_{\cdot \land t}\big),\cL_{\hat\alpha}(X_{\cdot \land t})\big) + \cW_2^2\big(L^N(\hat\alpha_t),\cL_{\hat\alpha}(\hat\alpha_t)\big) \Big], \; (\omega,u,N) \in \Omega \times [0,T] \times \N^\star.
\end{gather*}
Here, $L^N(\hat\alpha)$ denotes the empirical measure of the $A^N$-valued process $(\hat\alpha^1,\ldots,\hat\alpha^N)$, where each $\hat\alpha^i$ is the unique sub-game--perfect mean-field equilibrium for the mean-field game driven by the state process $X^i$, $i \in \{1,\ldots,N\}$. Moreover, the sequence of sub-game--perfect Nash equilibria $(\hat{\alpha}^{1,N})_{N \in \mathbb{N}^\star}$ converges to the sub-game--perfect mean-field equilibrium $\hat{\alpha}$ in the sense that there exists a constant $C>0$ such that, for any $N \in \mathbb{N}^\star$, for any $u \in [0,T]$,
\begin{align}\label{align:convEquilibria}
\int_u^T \cW^2_2\big(\P^{{\hat\balpha}^\smalltext{N}{,}{N}{,}{u}}_{\omega} \circ ( \hat\alpha^{1,N}_t )^{-1}, \P^{{\hat\alpha}{,}{N}{,}{u}}_{\omega} \circ ( \hat\alpha_t )^{-1} \big) \d t \leq C \big( \eta(\omega,u,N) + \gamma(\omega,u,N) \big), \; \P\text{\rm--a.e.} \; \omega\in\Omega.
\end{align}
\end{theorem}

\begin{remark}
\begin{enumerate}[label={$(\roman*)$}]
\item To the best of our knowledge, {\rm \Cref{theorem:convergenceTheorem}} is the first convergence result for time-inconsistent games in the literature. The only other work addressing a convergence problem is {\rm\cite{bayraktar2025time}}, which studies a time-inconsistent mean-field Markov decision game in discrete time and shows that the mean-field equilibrium provides an approximate optimal strategy when applied to the corresponding $N$-player game, but only in a precommitment sense. This result does not contradict ours because it considers time-inconsistency arising from non-exponential discounting, whereas we focus on mean--variance type preferences. We show that, under the assumption of uniqueness of the sub-game--perfect mean-field equilibrium, the {\rm BSDE} system describing the $N$-player game converges to the McKean--Vlasov \emph{BSDE} associated with the mean-field game. In this context, the existence and uniqueness of the sub-game--perfect mean-field equilibrium is equivalent to the well-posedness of the McKean--Vlasov {\rm BSDE} described in \eqref{BSDE_eanField_fromEquilibrium}, or equivalently \eqref{BSDE_eanField_toEquilibrium}. Given the nature of this {\rm BSDE}, in which both the driving Brownian motion and the underlying probability measure are part of the solution, together with quadratic growth, proving existence and uniqueness is challenging in general. However, in the mean--variance setting, the system is finite-dimensional, which makes it more tractable. In contrast, non-exponential discounting leads to an infinite-dimensional \emph{BSDE} system, as shown in \emph{\cite[Theorem 3.10 and Theorem 3.12]{hernandez2023me}}, where well-posedness is expected to be extremely difficult to obtain. We believe this may be one of the fundamental reasons for the potential convergence failure highlighted by {\rm \cite{bayraktar2025time}}$;$
\item from the bounds in \eqref{align:convValueProcesses} and \eqref{align:convEquilibria} for the value processes and for the sub-game--perfect equilibria, respectively, we can additionally derive quantitative convergence rates. The key observation is that the constants $C$ appearing in both estimates depend only on the parameters of the game and are independent of $N$. Consequently, the convergence rates are entirely determined by the behaviour of the functions $\eta$ and $\gamma$. More precisely, the function $\eta $ originates from the estimates in \eqref{align:firstPart_finalEst} and takes the form
\begin{align*}
\eta(\omega,u,N) &\coloneqq R_N^2\Bigg(1+\| X^{1}_{u}(\omega)\|^{2}+\frac1N\sum_{\ell =1}^N\| X^{\ell}_{u}(\omega)\|^{2}\Bigg)+C NR_N^2\Bigg(1+\| X^{1}_{u}(\omega)\|^{2\bar{p}}+\frac1N\sum_{\ell =1}^N\| X^{\ell}_{u}(\omega)\|^{2\bar{p}}\Bigg)\\
	&\;\quad+ NR_N^4\Bigg(1+\frac1N\sum_{\ell =1}^N\| X^{\ell}_{u}(\omega)\|^{2}\Bigg) (1+N), \; (\omega,u,N) \in \Omega \times [0,T] \times \N^\star,
\end{align*}
where $\bar{p}$ is introduced in \emph{\Cref{assumpConvThm}.\ref{growth_f_gG}}. If there is no interaction through the strategies, then, as already discussed in {\rm\Cref{ass:comments}.\ref{nocontrolzerofunctions}}, the function $R_N$ vanishes, and consequently so does $\eta$. In the presence of interaction, as already noted in \eqref{align:LLN} $($equivalently, in \eqref{align:LLN_withLIM}$)$, \emph{\Cref{assumpConvThm}.\ref{lipLambda_growthAleph}} and the strong law of large numbers yield
\begin{align*}
\lim_{N \rightarrow \infty} \eta(\omega,u,N) = 0, \; \P\text{\rm--a.e.} \; \omega\in\Omega, \; \text{for any} \; u \in [0,T].
\end{align*}
Moreover, the rate of this convergence is determined jointly by the convergence rate of the sequence $(R_N)_{N \in \N^\smalltext{\star}}$ introduced in \emph{\Cref{assumpConvThm}.\ref{lipLambda_growthAleph}} and by the convergence rate provided by the strong law of large numbers, for which the literature provides explicit rates, see for instance the seminal work \emph{\citeauthor*{marcinkiewicz1938quelques} \cite[Theorem 5]{marcinkiewicz1938quelques}}. In particular, if there exists $q \in [1,2)$ such that $\E^\P\big[\|X^1_0\|^{2\bar{p}q}\big] <+\infty$, than there exists a constant $C>0$ such that, for any $u \in [0,T]$,
\begin{align*}
\frac1N\sum_{\ell =1}^N \Big( \|X^{\ell}_{u}(\omega)\|^{2} - \E^\P \big[ \| X^{1}_{u} \|^{2} \big] \Big) + \frac1N\sum_{\ell =1}^N \Big( \|X^{\ell}_{u}(\omega)\|^{2\bar{p}} - \E^\P \big[ \| X^{1}_{u} \|^{2\bar{p}} \big] \Big) \leq \frac{C}{N^{1-\frac{1}{q}}}, \; \P\text{\rm--a.e.} \; \omega\in\Omega.
\end{align*}
Regarding the function $\gamma$, explicit convergence rates are provided, for instance, by \emph{\citeauthor*{van2023weak} \cite[Theorem~2.7.1]{van2023weak}.} In particular, if $m=1$, then for any $\alpha\in(0,1/2)$, there exists a constant $C>0$ such that, for any $(\omega,u) \in \Omega \times [0,T]$,
\begin{align*}
\gamma(\omega,u,N) \leq \frac{C}{(\log N)^{2 \alpha}}.
\end{align*}
Sharper bounds are obtained by \emph{\citeauthor*{fournier2015rate} \cite[Theorem 1]{fournier2015rate}} under the assumption that the functions $f$, $g$, $\varphi_1$, $\varphi_2$ and $b$ depends only on the law of $X^i_t$ rather than on the entire path $X^i_{\cdot \land t}$, $t \in [0,T]$, for each $i \in \N^\star$, and the set $A$ satisfies $A \subset \R^k$ for some $k \in \N^\star$. More precisely, if $\bar{p}>1$, then for any $q \in (2,2\bar{p}]$ there exists some constant $C>0$ such that, for any $(\omega,u) \in \Omega \times [0,T]$,
\begin{align*}
\gamma(\omega,u,N) \leq C \big(r_{N,m,q} + r_{N,k,q}\big),
\end{align*}
where, for $n \in \N^\star$
\begin{align*}
r_{N,n,q} \coloneqq \begin{cases}
N^{-1/2} + N^{-(q-2)/q}, 
\; \text{if } n < 4 \text{ and } q \neq 4, \\
N^{-1/2}\log(1+N) + N^{-(q-2)/q}, 
\; \text{if } n = 4 \text{ and } q \neq 4, \\
N^{-2/n} + N^{-(q-2)/q}, 
\; \text{if } n > 4 \text{ and } q \neq n/{(n-2)}.
\end{cases}
\end{align*}
\end{enumerate}
\end{remark}

\subsection{A representative example: the convergence}

Before turning to the proof of \Cref{theorem:convergenceTheorem}, we first consider a simple illustrative model to show the convergence of a symmetric $N$-player time-inconsistent game to its mean-field counterpart. Building on the set-up and notation introduced in the previous sections, we consider a Markovian $N$-player game in which the state process of each player $i \in \{1,\ldots,N\}$ evolves according to
\begin{equation*}
X^i_t = X^i_0 + \sigma W^i_t, \; t \in [0,T], \; \P\text{\rm--a.s.},
\end{equation*}
where $\sigma >0$ is constant. To simplify notation, we also assume that $d=m=1$. Let $f: \R \longrightarrow \R$ and $g: \R \longrightarrow \R$ be Lipschitz-continuous functions. Motivated by the criterion studied in \cite[Section 4]{possamai2025non}, we introduce a mean--variance modification to account for time-inconsistency. Hence, if the other players follow the strategy $\balpha^{N,\smallertext{-}i} \in \cA^{N-1}_N$, the payoff faced by player $i\in \{1,\ldots,N\}$ is defined by
\begin{align}\label{align:criterionEx}
\notag J^i\big(t,\omega,\alpha;\balpha^{N,\smallertext{-}i}\big) &\coloneqq \E^{\P^{\smalltext{\alpha} \smalltext{\otimes}_\tinytext{i} \smalltext{\balpha}^{\tinytext{N}\tinytext{,}\tinytext{-}\tinytext{i}}\smalltext{,}\smalltext{N}\smalltext{,}\smalltext{t}}_\smalltext{\omega}} \Bigg[ \int_t^T \bigg( -\frac{(\alpha_s)^2}{2} + \frac{\kappa_1}{N} \sum_{\ell =1}^N f(X^\ell_s) + \frac{\kappa_2}{N} \bigg( \alpha_s + \sum_{\ell \in \{1,\ldots,N\}\setminus\{i\}} \alpha_s^{\ell,N} \bigg) \bigg) \d s  \Bigg] \\
&\; \quad +\E^{\P^{\smalltext{\alpha} \smalltext{\otimes}_\tinytext{i} \smalltext{\balpha}^{\tinytext{N}\tinytext{,}\tinytext{-}\tinytext{i}}\smalltext{,}\smalltext{N}\smalltext{,}\smalltext{t}}_\smalltext{\omega}}\bigg[g(X^i_T) - \frac{\gamma}{2} (X^i_T)^2\bigg]+ \frac{\gamma}{2} \Big( \E^{\P^{\smalltext{\alpha} \smalltext{\otimes}_\tinytext{i} \smalltext{\balpha}^{\tinytext{N}\tinytext{,}\tinytext{-}\tinytext{i}}\smalltext{,}\smalltext{N}\smalltext{,}\smalltext{t}}_\smalltext{\omega}} \big[ X^i_T \big] \Big)^2, \; (t,\omega,\alpha) \in [0,T] \times \Omega \times \cA_N,
\end{align}
where 
\begin{equation*}
\frac{\d \P^{\alpha \otimes_{\smalltext{i}} \balpha^{\smalltext{N}\smalltext{,}\smalltext{-}\smalltext{i}},N}}{\d \P} \coloneqq \cE\Bigg( \int_0^\cdot \bigg( (\alpha_t - k X^i_t) \d W^i_t + \sum_{\ell \in \{1,\ldots,N\}\setminus\{i\}} (\alpha^{\ell,N}_t-kX^\ell_s) \d W^\ell_t \bigg) \Bigg)_T.
\end{equation*}
Moreover, for simplicity, we assume that each strategy takes values in the interval $A \coloneqq (-\bar{a},\bar{a})$ for some $\bar{a}>0$.

\medskip
Rather than immediately introducing the BSDE system that describes the problem, we begin by analysing the corresponding PDE system, which allows for a simplified derivation of explicit solutions. The extended dynamic programming principle in \Cref{thm:DPP_NplayerGame} directly leads to the fundamental PDE system for the value functions. For each $i \in \{1,\ldots,N\}$, we have
\begin{align*}
&\partial_t v^{i,N}(t,x) +  \sigma \sum_{\ell=1}^N (\hat{a}^{\ell,N} - k x^\ell) \partial_{x^\smalltext{\ell}} v^{i,N}(t,x)  + \frac{\sigma^2}{2} \sum_{\ell=1}^N \partial^2_{x^\smalltext{\ell},x^\smalltext{\ell}} v^{i,N}(t,x) \\
&\quad - \frac{(\hat{a}^{i,N})^2}{2} + \frac{\kappa_1}{N} \sum_{\ell=1}^N f(x^\ell) + \frac{\kappa_2}{N} \sum_{\ell=1}^N \hat{a}^{\ell,N} - \frac{\gamma \sigma^2}{2} \sum_{\ell=1}^N \big(\partial_{x^\smalltext{\ell}} v^{m,i,N}(t,x)\big)^2 = 0, \; (t,x) \in [0,T) \times \R^N, \\
&\partial_t v^{m,i,N}(t,x) + \sigma \sum_{\ell=1}^N (\hat{a}^{\ell,N} - k x^\ell) \partial_{x^\smalltext{\ell}} v^{m,i,N}(t,x) + \frac{\sigma^2}{2} \sum_{\ell=1}^N \partial^2_{x^\smalltext{\ell},x^\smalltext{\ell}} v^{m,i,N}(t,x)=0 , \; (t,x) \in [0,T) \times \R^N, \\
&\big(v^{i,N}(T,x),v^{m,i,N}(T,x)\big) = (g(x^i),x^i).
\end{align*}
Here, each $\hat{\alpha}^{i,N}$ maximises the $i$-th component of the $N$-player Hamiltonian, that is,
\begin{align*}
\hat{\alpha}^{i,N} \in \argmax_{a \in (-\bar{a},\bar{a})} \bigg\{ \sigma (a - k x^i) \partial_{x^\smalltext{i}} v^{i,N}(t,x) - \frac{a^2}{2} + \frac{\kappa_2}{N} a \bigg\},
\end{align*} 
Introducing the projection operator onto the set $A$, denoted by $\text{P}_A$, we may write
\begin{align*}
\hat{\alpha}^{i,N}_t = \hat{\alpha}^{i,N}\big(t,x,\partial_{x^\smalltext{i}} v^{i,N}(t,x)\big) = \text{P}_A \bigg(\sigma \partial_{x^\smalltext{i}} v^{i,N}(t,x) + \frac{\kappa_2}{N}\bigg).
\end{align*}

To obtain a simple explicit solution, we further assume that the functions $f$ and $g$ are linear, specifically $f(x) = g(x) = x$, for $x \in \R$. Under this assumption, one can construct functions $(v^{i,N},v^{m,i,N})$ of the form
\begin{align*}
v^{i,N}(t,x) = \mathrm{e}^{\sigma k (t-T)} x^i + \frac{\kappa_1}{\sigma k N} \big(1-\mathrm{e}^{\sigma k (t-T)}\big) \sum_{\ell=1}^N x^\ell + \eta^N(t),\; v^{m,i,N}(t,x) = \mathrm{e}^{\sigma k (t-T)} x^i + \eta^{m,N}(t), \; (t,x) \in [0,T] \times \R^N,
\end{align*}
where ${\eta}^N$ and ${\eta}^{m,N}$ uniquely solve the following ODEs on the interval $[0,T)$
\begin{align*}
 \dot{\eta}^N(t) &= - \bigg( \sigma \mathrm{e}^{\sigma k (t-T)} + \frac{\kappa_1}{kN} \big( 1-\mathrm{e}^{\sigma k (t-T)}\big) + \frac{\kappa_2}{N} \bigg) \bigg( \frac{\sigma}{2} \mathrm{e}^{\sigma k (t-T)} + \bigg( \frac{\kappa_1}{k} \big( 1-\mathrm{e}^{\sigma k (t-T)}\big) + \kappa_2 \bigg) \bigg( 1 - \frac{1}{2N}\bigg) \bigg)\\
&\quad + \frac{\gamma\sigma^2}{2} \mathrm{e}^{2\sigma k (t-T)}, \\
\dot{\eta}^{m,N}(t) &= - \sigma \bigg( \sigma \mathrm{e}^{\sigma k (t-T)} + \frac{\kappa_1}{kN} \big( 1-\mathrm{e}^{\sigma k (t-T)}\big) + \frac{\kappa_2}{N} \bigg) \mathrm{e}^{\sigma k (t-T)},
\end{align*}
with terminal conditions $ {\eta}^N(T) = 0$ and $ {\eta}^{m,N}(T) = 0$. The $N$-player game is equivalently characterised by the BSDE system
\begin{align}\label{align:bsdeExample_nplayergame}
\notag Y^{i,N}_t &= X^i_T + \int_t^T \bigg( -\frac{1}{2} \bigg( \text{P}_A \bigg(Z^{i,i,N}_s + \frac{\kappa_2}{N} \bigg)\bigg)^2 + \frac{\kappa_1}{N} \sum_{\ell =1}^N X^\ell_s + \frac{\kappa_2}{N}  \sum_{\ell =1}^N \text{P}_A \bigg( Z^{\ell,\ell,N}_s + \frac{\kappa_2}{N} \bigg) \\
\notag&\quad + \sum_{\ell =1}^N Z^{i,\ell,N}_s \bigg( \text{P}_A \bigg(Z^{\ell,\ell,N}_s + \frac{\kappa_2}{N}\bigg) - k X^\ell_s \bigg) - \frac{\gamma}{2} \sum_{\ell =1}^N \big(Z^{i,m,\ell,\star,N}_s\big)^2 \bigg) \d s \\
\notag &\quad- \int_t^T \sum_{\ell =1}^N \frac{Z^{i,\ell,N}_s}{\sigma} \d X_s^\ell, \; t \in [0,T], \; \P \text{\rm--a.s.}, \\
M^{i,\star,N}_t &= X^i_T + \int_t^T \sum_{\ell =1}^N Z^{i,m,\ell,\star,N}_s \bigg( \text{P}_A \bigg(Z^{\ell,\ell,N}_s + \frac{\kappa_2}{N}\bigg) - k X^\ell_s \bigg) \d s - \int_t^T \sum_{\ell =1}^N \frac{Z^{i,m,\ell,\star,N}_s}{\sigma} \d X_s^\ell, \; t \in [0,T], \; \P \text{\rm--a.s.}
\end{align}
Given its quadratic growth structure, the system admits a unique solution (see \cite[Theorem 2.2]{cheridito2014bsdes}). Using the expression for the value functions $(v^{i,N},v^{m,i,N})$ derived from the PDE formulation, the associated $Z$-components are explicitly given by 
\begin{equation*}
({Z}^{i,\ell,N}_t,{Z}^{i,m,\ell,\star,N}_t) = \bigg(\sigma \mathrm{e}^{\sigma k (t-T)} \delta^i_\ell + \frac{\kappa_1}{k N} \big(1-\mathrm{e}^{\sigma k (t-T)}\big),\sigma \mathrm{e}^{\sigma k (t-T)} \delta^i_\ell\bigg),\;t\in [0,T].
\end{equation*}
Consequently, there exists a unique sub-game--perfect Nash equilibrium given by
\begin{equation*}
\hat\alpha^{i,N}_t = \text{P}_A \bigg( \sigma \mathrm{e}^{\sigma k (t-T)} + \frac{\kappa_1}{k N} \big(1-\mathrm{e}^{\sigma k (t-T)}\big) + \frac{\kappa_2}{N}\bigg),\;t\in [0,T].
\end{equation*}

\medskip
Analogously to the approach used in the proof of \Cref{theorem:convergenceTheorem}, for each $i \in\{1,\ldots,N\}$, we introduce a BSDE system associated with the state process $X^i$:
\begin{align}\label{align:BSDEexampleMeanField}
\notag Y^{i}_t &\;= X^i_T + \int_t^T \bigg( -\frac{(\hat\alpha^i_s)^2}{2} +  \E^{\P^{\smalltext{\hat\alpha^i}}}[\kappa_1 X^i_s + \kappa_2 Z^i_s ] + Z^i_s ( \hat\alpha^i_s - k X^i_s) - \frac{\gamma}{2} \big({Z}^{i,m,\star}_s\big)^2 \bigg) \d s - \int_t^T \frac{Z^{i}_s}{\sigma} \d X^i_s, \; t \in [0,T], \; \P \text{\rm--a.s.}, \\
\notag M^{i,\star}_t &\;= X^i_T + \int_t^T {Z}^{i,m,\star}_s (\hat\alpha^i_s - k X^i_s) \d s  - \int_t^T \frac{{Z}^{i,m,\star}_s}{\sigma} \d X^i_s, \; t \in [0,T], \; \P \text{\rm--a.s.}, \\
\notag \hat\alpha^i_t &\coloneqq \text{P}_A(Z^i_t), \; t \in [0,T], \\
\frac{\d \P^{\hat\alpha^i}}{\d \P} &\coloneqq \cE\bigg( \int_0^\cdot (\hat\alpha^i_s- k X^i_s) \d W^i_s \bigg)_T.
\end{align}
We assume that each system admits a unique solution, which in turn implies the uniqueness of the sub-game--perfect mean-field equilibrium $\hat\alpha^i$. The unique solution is explicitly given by
\begin{equation*}
(Y^i_t,M^{i,\star}_t,Z^i_t,Z^{i,m,\star}_t) = \bigg( \mathrm{e}^{\sigma k (t-T)} X^i_t + \eta(t), \mathrm{e}^{\sigma k (t-T)} X^i_t + \frac{\sigma}{2k} \big( 1-\mathrm{e}^{2\sigma k (t-T)}\big), \sigma\mathrm{e}^{\sigma k (t-T)}, \sigma\mathrm{e}^{\sigma k (t-T)} \bigg),
\end{equation*}
where
\begin{align*}
\eta(t) = \int_t^T \bigg( \sigma \mathrm{e}^{\sigma k (s-T)} \bigg( \frac{\sigma\mathrm{e}^{\sigma k (s-T)}}{2} + \kappa_2 \bigg) + \kappa_1  \bigg( \mathrm{e}^{-\sigma k (s-t)} X^i_t + \frac{\sigma}{2k} \big( \mathrm{e}^{\sigma k (s-T)} - \mathrm{e}^{-\sigma k (s+T-2t)} \big) \bigg)  - \frac{\gamma\sigma^2}{2} \mathrm{e}^{2\sigma k (t-T)} \bigg) \d s.
\end{align*}

\medskip
Having introduced the two BSDE systems \eqref{align:bsdeExample_nplayergame} and \eqref{align:BSDEexampleMeanField}, which describe the $N$-player game and its mean-field counterpart, we now turn to the convergence analysis. Throughout these computations we work under the assumption that $\bar{a}$, and equivalently the set $A$, is sufficiently large so that the projection operator $\text{P}_A$ can be omitted. Since we are studying a simple and fully explicit example, and that the functions $(\aleph^{i,N})_{i \in \{1,\ldots,N\}}$ introduced in \Cref{assumpConvThm}.\ref{LambdaFunc} reduce here to the constant value $\kappa_2/N$, it is not necessary to introduce an intermediate system. We can therefore proceed directly to proving the convergence result. To this end, we introduce the probability measure
\begin{align*}
\frac{\d \P^{\hat\alpha,N}}{\d \P} &\coloneqq \cE\Bigg( \int_0^\cdot \sum_{\ell =1}^N (\hat\alpha^\ell_t- k X^\ell_t) \d W^\ell_t \Bigg)_T.
\end{align*}
By construction, we have 
\begin{align*}
\E^{\P^{\smalltext{\hat\alpha}\smalltext{,}\smalltext{N}}}[X^i_t] = \E^{\P^{\smalltext{\hat\alpha^i}}}[X^i_t], \; t \in [0,T], \; \text{for all} \; i \in \{1,\ldots,N\}.
\end{align*}
Moreover, we note that all components of the sub-game--perfect Nash equilibrium $\hat\balpha^N= (\hat{\alpha}^{1,N}, \ldots, \hat{\alpha}^{N,N})$ are identical, and similarly all the sub-game--perfect mean-field equilibria $\hat{\alpha}^i$ coincide with each other, for all $i \in \{1,\ldots,N\}$. In addition, we have
\begin{align*}
Z^{i,m,\ell,\star,N}_t = Z^{i,m,\star}_t \delta^i_\ell, \; t \in [0,T], \; \text{for all} \; (i,\ell) \in \{1,\ldots,N\}^2.
\end{align*}
We now fix $\beta > 0$, whose value will be chosen later, and apply It\^o's formula to the processes $\mathrm{e}^{\beta t} \big|\delta Y^{i,N}_t\big|^2 \coloneqq \mathrm{e}^{\beta t} \big|Y^{i,N}_t-Y^i_t\big|^2$ and $\mathrm{e}^{\beta t} \big|\delta M^{i,\star,N}_t\big|:= \mathrm{e}^{\beta t} \big|M^{i,\star,N}_t - M^{i,\star}_t\big|^2$, for $t \in [0,T]$. This yields
\begin{align*}
& \mathrm{e}^{\beta t} \big|\delta Y^{i,N}_t\big|^2 + \int_t^T \sum_{\ell =1}^N \big(Z^{i,\ell,N}_s - Z^i_s \delta^i_\ell \big)^2 \d s \\
&= - \beta \int_t^T \mathrm{e}^{\beta s} \big|\delta Y^{i,N}_s\big|^2 \d s \\
& \quad + 2 \int_t^T \mathrm{e}^{\beta s} \delta Y^{i,N}_s \bigg( \big( \hat{\alpha}^{i,N}_s - \hat\alpha^i_s \big) \bigg( - \frac{\big( \hat{\alpha}^{i,N}_s - \hat\alpha^i_s \big)}{2} + \kappa_2 + \sum_{\ell =1}^N Z^{i,\ell,N}_s \bigg) + \kappa_1 \bigg( \frac{1}{N} \sum_{\ell =1}^N X^\ell_s - \E^{\P^{\smalltext{\hat\alpha}\smalltext{,}\smalltext{N}}}[X^i_s] \bigg) \bigg) \d s\\
\notag &\quad- \int_t^T \sum_{\ell =1}^N \big(Z^{i,\ell,N}_s - Z^i_s \delta^i_\ell \big) \d \big(W_s^{\hat\alpha}\big)^\ell, \; \P \text{\rm--a.s.}, \\
&\mathrm{e}^{\beta t} \big|\delta M^{i,\star,N}_t\big|^2 + \int_t^T \sum_{\ell =1}^N \big(Z^{i,m,\ell,\star,N}_s - Z^{i,m,\star}_s \delta^i_\ell \big) \d s \\
&= - \beta \int_t^T \mathrm{e}^{\beta s} \big|\delta M^{i,\star,N}_s\big|^2 \d s \\
& \quad + 2 \int_t^T \mathrm{e}^{\beta s} \delta M^{i,\star,N}_s \sum_{\ell =1}^N Z^{i,m,\ell,\star,N}_s \big( \hat{\alpha}^{i,N}_s - \hat\alpha^i_s \big) \d s, \; \P \text{\rm--a.s.}
\end{align*}
Applying Young's inequality, taking expectations under $\P^{\hat\alpha}$, and choosing $\beta >1$, we obtain the existence of a constant $c >0$, independent of $N$, such that 
\begin{align*}
& \E^{\P^{\smalltext{\hat\alpha}\smalltext{,}\smalltext{N}}} \Big[ \mathrm{e}^{\beta t} \big|\delta Y^{i,N}_t\big|^2 \Big] + \E^{\P^{\smalltext{\hat\alpha}\smalltext{,}\smalltext{N}}} \Big[ \mathrm{e}^{\beta t} \big|\delta M^{i,\star,N}_t\big|^2 \Big]  \\
&\leq \E^{\P^{\smalltext{\hat\alpha}\smalltext{,}\smalltext{N}}} \Bigg[ \int_t^T \mathrm{e}^{\beta s} \bigg( \big( \hat{\alpha}^{i,N}_s - \hat\alpha^i_s \big) \bigg( - \frac{\big( \hat{\alpha}^{i,N}_s - \hat\alpha^i_s \big)}{2} + \kappa_2 + \sum_{\ell =1}^N Z^{i,\ell,N}_s \bigg) + \kappa_1 \bigg( \frac{1}{N} \sum_{\ell =1}^N X^\ell_s - \E^{\P^{\smalltext{\hat\alpha}\smalltext{,}\smalltext{N}}}[X^i_s] \bigg)  \bigg)^2 \d s \Bigg] \\
&\leq c \bigg( \frac{1}{N} + \frac{1}{N^2} \bigg) \xrightarrow[N \to \infty]{} 0.
\end{align*}
Thus, if the mean-field BSDE system \eqref{align:BSDEexampleMeanField} admits a unique solution, we conclude that the value functions of the $N$-player stochastic differential game with the mean--variance criterion described in \eqref{align:criterionEx} converge to their mean-field game counterparts at rate $1/N$. Moreover, the sub-game--perfect Nash equilibrium also converges to the sub-game mean-field equilibrium with the same rate, as follows directly from their explicit expressions. This verifies the result of \Cref{theorem:convergenceTheorem} in the context of this fully explicit example.

\subsection{Proof of Theorem \ref{theorem:convergenceTheorem}}

This section presents the proof of the convergence result stated in \Cref{theorem:convergenceTheorem}. Our approach is inspired by the steps of \cite[Theorem 2.10]{possamai2025non} and consists in introducing two auxiliary {\rm BSDE} systems that serve to bridge the convergence from the $N$-player game to its mean-field counterpart. Despite the similar structure, the {\rm BSDE} systems arising in our setting are of a different nature, owing to their quadratic growth, which necessitates additional estimates. In addition, the inherent time-inconsistency of our problems requires us to establish convergence at every time $t \in [0,T]$, rather than only at the initial time. For this purpose, we work with some r.c.d.p.s given $\cF_{N,t}$ for each $t \in [0,T]$, and we rewrite the systems \eqref{BSDE_NplayerGame} and \eqref{BSDE_eanField_toEquilibrium}, describing the $N$-player game and the mean-field counterpart respectively, under these conditional measures. The proof is structured into the following steps:

\begin{enumerate}[label={$(\roman*)$}]
\item in \Cref{{section:keyBSDEsystems}}, for a fixed sub-game--perfect Nash equilibrium $\hat\balpha^N$, we write the corresponding BSDE system \eqref{align:systemMeanFieldGame} under the families of r.c.p.d.s $(\P^{\hat\balpha^{\smalltext{N}},N,u}_\omega)_{(\omega,u) \in \Omega \times [0,T]}$. By \Cref{assumpConvThm}.\ref{LambdaFunc}, $\hat\balpha^N$ is described by a Borel-measurable Hamiltonian maximiser $\Lambda$, which also characterises the mean-field equilibrium $\hat\alpha^i$ for each mean-field game associated with the state process $X^i$, $i \in \{1,\ldots,N\}$, following \Cref{sufficiency_toBSDE_meanFieldGame}. The BSDE system for the mean-field game, given in \eqref{align:systemMeanFieldGame_rcpd}, is defined with respect to the families of r.c.p.d.s $(\P^{\hat\alpha,N,u}_\omega)_{(\omega,u) \in \Omega \times [0,T]}$, where the measure $\P^{\hat\alpha,N}$ defined in \eqref{align:measureALLmeanFieldGames} describes all $N$ copies of the mean-field game. Finally, we introduce an auxiliary FBSDE system in \eqref{align:intermediateSystemNplayerGame}, which serves as a bridge: the proof then shows that the $N$-player system converges to this auxiliary system, which in turn converges to the mean-field system;
\item\label{secondSTEPproof} \Cref{subsubsection:fromNplayerGameToAuxiliary} is devoted to the proof of convergence from the $N$-player BSDE system to the intermediate system. The proof is organised into several sub-steps. In \textbf{Step 1}, we derive estimates for the martingale terms $(M^{i,\star,N},N^{\star,N})$ and $(\widetilde M^{i,\star,N},\widetilde N^{\star,N})$; in \textbf{Step 2} we establish bounds for the difference between of the value process $Y^{i,N}$ and $\widetilde Y^{i,N}$; in \textbf{Step 3}, we obtain estimates for the forward components $X^i$ and $\widetilde X^{i}$; and finally, in \textbf{Step 4}, we combine all previous estimates to conclude the desired convergence;
\item \Cref{convTOmeanFieldGames} concludes the proof by showing that the intermediate system converges to the $N$ copies of the mean-field system. In particular, we first introduce a second intermediate system in \eqref{align:intermediateSystemNplayerGame_rcpdMeanField}, which coincides with the auxiliary FBSDE system in \eqref{align:intermediateSystemNplayerGame} but is defined with respect to the fixed Brownian motions $((W^{\hat\alpha,N})^i)_{i \{1,\ldots,N\}}$ instead of $((W^{\hat\balpha^{\smalltext{N}},N})^i)_{i \{1,\ldots,N\}}$. By the Yamada--Watanabe theorem, we have $\widetilde{Y}^{i,N}_u = \overline{Y}^{i,N}_u$ $\P\text{\rm--a.s.}$ for any time $u\in[0,T]$. Hence, the proof is complete once we show that this second auxiliary system converges to the mean-field system. The proof follows the same structure as in \ref{secondSTEPproof}. Namely, \textbf{Step 1} derives estimates for the backward components $(\overline{Y}^{i,N},\overline M^{i,\star,N},\overline N^{\star,N})$ and $(Y^{i},M^{i,\star},N^{\star})$; \textbf{Step 2} gives the necessary estimates for the forward components $\overline X^i$ and $X^i$; and all these estimates are combined in \textbf{Step 4};
\item the final part of the proof is given in \Cref{subsubsection:conEquilibria}, where we show the convergence of a sub-game--perfect Nash equilibrium to the sub-game--perfect mean-field equilibrium. 
\end{enumerate}

\subsubsection{Setting up the key systems}\label{section:keyBSDEsystems}

We fix a sub-game--perfect Nash equilibrium $\hat\balpha^N = (\hat\alpha^{1,N}, \ldots, \hat\alpha^{N,N}) \in \mathcal{N}\!\mathcal{A}_{s,N}$ and denote by $(V^{i,N})_{i \in \{1,\ldots,N\}}$ the associated value processes. Then, for each player $i \in \{1,\ldots,N\}$, it follows from \Cref{necessity_toBSDE} and \Cref{assumpConvThm}.\ref{LambdaFunc} that $V^{i,N}_t = Y^{i,N}_t$, $\P\text{\rm--a.s.}$, for all $t \in [0,T]$, where the processes
\[
(\Y^N,\Z^N,\M^{\star,N},N^{\star,N},\Z^{m,\star,N},\Z^{n,\star,N}) \coloneqq (Y^{i,N}, \Z^{i,N}, M^{i,\star,N}, N^{\star,N}, \Z^{i,m,\star,N},\Z^{n,\star,N})_{i\in\{1,\ldots,N\}},
\]
solve the {\rm BSDE} system
\begin{align}\label{align:systemNplayerGame}
\notag X^{i}_t &= X^i_0 + \int_0^t \sigma_s(X^{i}_{\cdot \land s}) b_s\big(X^{i}_{\cdot \land s},L^N\big(\X^N_{\cdot \land s},\hat\balpha^N_s\big),\hat\alpha^{i,N}_s\big) \d s + \int_0^t \sigma_s(X^i_{\cdot \land s}) \d \big(W_s^{\hat\balpha^\smalltext{N},N}\big)^i, \; t \in [0,T], \; \P \text{\rm--a.s.}, \\
\notag Y^{i,N}_t &= g\big(X^{i}_{\cdot \land T}, L^N\big(\X^N_{\cdot \land T}\big)\big) + G\big(\varphi_1(X^{i}_{\cdot \land T}),\varphi_2(L^N\big(\X^N_{\cdot \land T}\big)\big)+ \int_t^T f_s\big(X^{i}_{\cdot \land s},L^N\big(\X^N_{\cdot \land s},\hat\balpha^N_s\big),\hat\alpha^{i,N}_s\big) \d s \\
\notag &\quad - \int_t^T \partial^2_{m,n}G\big(M^{i,\star,N}_s,N^{\star,N}_s\big) \sum_{\ell =1}^N Z^{i,m,\ell,\star,N}_s \cdot Z^{n,\ell,\star,N}_s \d s \\
\notag &\quad - \frac{1}{2} \int_t^T \partial^2_{m,m}G\big(M^{i,\star,N}_s,N^{\star,N}_s\big) \sum_{\ell =1}^N \big\|Z^{i,m,\ell,\star,N}_s\big\|^2 \d s - \frac{1}{2} \int_t^T \partial^2_{n,n}G\big(M^{i,\star,N}_s,N^{\star,N}_s\big) \sum_{\ell =1}^N \big\|Z^{n,\ell,\star,N}_s\big\|^2 \d s \\
\notag &\quad - \int_t^T \sum_{\ell =1}^N Z^{i,\ell,N}_s \cdot \d \big(W_s^{\hat\balpha^\smalltext{N},N}\big)^\ell, \; t \in [0,T], \; \P \text{\rm--a.s.}, \\
\notag M^{i,\star,N}_t &= \varphi_1\big(X^{i}_{\cdot \land T}\big) - \int_t^T \sum_{\ell =1}^N Z^{i,m,\ell,\star,N}_s \cdot \d \big(W_s^{\hat\balpha^\smalltext{N},N}\big)^\ell, \; t \in [0,T], \; \P \text{\rm--a.s.}, \\
\notag N^{\star,N}_t &= \varphi_2\big(L^N\big(\X^N_{\cdot \land T}\big)\big) - \int_t^T \sum_{\ell =1}^N Z^{n,\ell,\star,N}_s \cdot \d \big(W_s^{\hat\balpha^\smalltext{N},N}\big)^\ell, \; t \in [0,T], \; \P \text{\rm--a.s.}, \\
\hat\alpha^{i,N}_t &= \Lambda_t\big(X^{i}_{\cdot \land t}, L^N\big(\X^N_{\cdot \land t}\big), Z^{i,i,N}_t, Z^{i,m,i,\star,N}_t, Z^{n,i,\star,N}_t, \aleph^{i,N}_t\big), \; \d t \otimes \P \text{\rm--a.e.},
\end{align}
for some Borel-measurable function $\Lambda: [0,T] \times \cC_m \times \cP_2(\cC_{m}) \times \R^d \times \R^d \times \R^d \times \R \longrightarrow A$. We observe that, unlike in \Cref{necessity_toBSDE}, the process $N^{\star,N}$ does not depend on the player index $i$. This is because we are working under the assumption that the $N$-player game is symmetric, and in particular, the function $\varphi_2$ is identical for all players. Moreover, \Cref{assumpConvThm}.\ref{LambdaFunc} also implies that for any $(t,x,\xi,m^\star,n^\star,z^{m,\star}) \in [0,T] \times \cC_{m} \times \cP_2(\cC_m \times A) \times \R^d \times \R \times \R \times \R^d$, we have
\begin{align*}
\Lambda_t\big(x,\xi^x,z,z^{m,\star},\mathbf{0},0\big) \in \argmax_{a \in A}\big\{ h_t(x,\xi,z,a)\big\}.
\end{align*}
Since each mean-field game---each associated with the driving state process $X^i$, or equivalently with the Brownian motion $W^i$, $i \in \{1,\ldots,N\}$---admits a unique equilibrium $\hat\alpha^i$ by \Cref{assumpConvThm}.\ref{uniqueMeanField}, it follows from \Cref{sufficiency_toBSDE_meanFieldGame} that the following BSDE admits a unique solution
\begin{align}\label{align:systemMeanFieldGame}
\notag X^i_t &= X^i_0 + \int_0^t \sigma_s(X^i_{\cdot \land s}) b_s\big(X^i_{\cdot \land s},\cL_{\hat\alpha^{\smalltext{i}}}(X^i_{\cdot \land s},\hat\alpha^i_s),\hat\alpha^i_s\big) \d s + \int_0^t \sigma_s(X^i_{\cdot \land s}) \d \big(W_s^{\hat\alpha^\smalltext{i}}\big)^i, \; t \in [0,T], \; \P \text{\rm--a.s.}, \\
\notag Y^i_t &= g\big(X^i_{\cdot \land T}, \cL_{\hat\alpha^{\smalltext{i}}}(X^i_{\cdot \land T})) + G\big(\varphi_1(X^i_{\cdot \land T}),\varphi_2\big(\cL_{\hat\alpha^{\smalltext{i}}}(X^i_{\cdot \land T}))\big) + \int_t^T f_s\big(X^i_{\cdot \land s},\cL_{\hat\alpha^{\smalltext{i}}}\big(X^i_{\cdot \land s},\hat\alpha^i_s),\hat\alpha^i_s\big) \d s\\
\notag &\quad  - \frac{1}{2} \int_t^T \partial^2_{m,m}G\big(M^{i,\star}_s,N^{\star}_s\big) \big\|Z^{i,m,i,\star}_s\big\|^2 \d s - \int_t^T Z^{i,i}_s \cdot \d \big(W_s^{\hat\alpha^\smalltext{i}}\big)^i, \; t \in [0,T], \; \P \text{\rm--a.s.}, \\
\notag M^{i,\star}_t &= \varphi_1(X^i_{\cdot \land T}) - \int_t^T Z^{i,m,i,\star}_s \cdot \d \big(W_s^{\hat\alpha^\smalltext{i}}\big)^i, \; t \in [0,T], \; \P \text{\rm--a.s.}, \\
\notag N^{\star}_t &\coloneqq \varphi_2\big(\cL_{\hat\alpha^{\smalltext{i}}}(X^i_{\cdot \land T})\big), \; t \in [0,T], \\
\notag \hat\alpha^i_t &\coloneqq \Lambda_t\big(X^i_{\cdot \land t},\cL_{\hat\alpha^{\smalltext{i}}}(X^i_{\cdot \land t}),Z^{i,i}_t,Z^{i,m,i,\star}_t,\mathbf{0},0\big), \; \mathrm{d}t\otimes\P\text{\rm--a.e.}, \\
\frac{\d \P^{\hat\alpha^{\smalltext{i}}}}{\d \P} &\coloneqq \cE\bigg( \int_0^\cdot b_s\big(X^i_{\cdot \land s},\cL_{\hat\alpha^{\smalltext{i}}}(X^i_{\cdot \land s},\hat\alpha^i_s),\hat\alpha^i_s\big) \cdot \d W^i_s \bigg)_T.
\end{align}
It also holds that $V^i_t = Y^i_t$, $\P\text{\rm--a.s.}$, for any $t \in [0,T]$. Unlike the {\rm BSDE} system in \eqref{BSDE_eanField_toEquilibrium}, the system in \eqref{align:systemMeanFieldGame} depends on the index $i \in \{1,\ldots,N\}$, since we are considering $N$ identical copies of the mean-field game introduced in \eqref{align:payoffMeanField}, each driven by its own Brownian motion $W^i$. Nevertheless, it is straightforward to verify that $\cL_{\hat\alpha^{\smalltext{\ell}}}(X^\ell_{\cdot \land T}) = \cL_{\hat\alpha^{\smalltext{j}}}(X^j_{\cdot \land T})$ for all $(\ell,j) \in \{1,\ldots,N\}^2$, so that the process $N^\star$ is independent of the player index, analogously to the $N$-player game. Although each mean-field system is naturally described under its own probability measure $\P^{\hat\alpha^{\smalltext{i}}}$, it is convenient to introduce a single reference measure, equivalent to all $\P^{\hat\alpha^{\smalltext{i}}}$, $i \in \{1,\ldots,N\}$, and under which the system \eqref{align:systemMeanFieldGame} remains unchanged. To this end, we define the probability measure $\P^{\hat\alpha,N}$ by
\begin{align}\label{align:measureALLmeanFieldGames}
\frac{\d \P^{\hat\alpha,N}}{\d \P} &\coloneqq \cE\Bigg( \int_0^\cdot \sum_{\ell =1}^N b_t\big(X^\ell_{\cdot \land t},\cL_{\hat\alpha^{\smalltext{\ell}}}(X^\ell_{\cdot \land t},\hat\alpha^\ell_t),\hat\alpha^\ell_t\big) \cdot \d W^\ell_t \Bigg)_T.
\end{align}
It then follows that each of the families $\X^N \coloneqq (X^1,\ldots,X^N)$ and $\hat\alpha \coloneqq (\hat\alpha^1,\ldots,\hat\alpha^N)$ consists of $\P^{\hat\alpha,N}$--i.i.d.\ processes, and for any $i \in \{1,\ldots,N\}$, we have
\begin{align*}
\cL_{\hat\alpha}(X^i_{\cdot \land t}) = \cL_{\hat\alpha^{\smalltext{i}}}(X^i_{\cdot \land t}) \; \text{and} \;  \cL_{\hat\alpha}(X^i_{\cdot \land t},\hat\alpha^i_t) = \cL_{\hat\alpha^{\smalltext{i}}}(X^i_{\cdot \land t},\hat\alpha^i_t), \; t \in [0,T].
\end{align*}
Hence, in what follows, we adopt the notations $\cL_{\hat\alpha}(X_{\cdot \land t})$ and $\cL_{\hat\alpha}(X_{\cdot \land t},\hat\alpha_t)$ to denote these laws.

\medskip
Our objective is to prove convergence to the mean-field game limit over the entire time interval $[0,T]$. To this end, we fix an arbitrary time $u \in [0,T]$ and consider the families of r.c.p.d.s $(\P^{\hat\balpha^{\smalltext{N}},N,u}_\omega)_{\omega \in \Omega}$ and $(\P^{\hat\alpha,N,u}_\omega)_{\omega \in \Omega}$ of $\P^{\hat\balpha^{\smalltext{N}},N}$ and $\P^{\hat\alpha,N}$, respectively, given the $\sigma$-algebra $\cF_{N,u}$. Accordingly, the system in \eqref{align:systemNplayerGame} can be rewritten, yielding an equivalent representation of the $N$-player game that holds for $\P\text{\rm--a.e.} \; \omega\in\Omega$
\begin{align}\label{align:systemNplayerGame_rcpd}
\notag X^i_t &= X^i_u(\omega) + \int_u^t \sigma_s(X^i_{\cdot \land s}) b_s\big(X^i_{\cdot \land s},L^N\big(\X^N_{\cdot \land s},\hat\balpha^N_s\big),\hat\alpha^{i,N}_s\big) \d s + \int_u^t \sigma_s(X^i_{\cdot \land s}) \d \big(W_s^{\hat\balpha^\smalltext{N},N,u,\omega}\big)^i, \; t \in [u,T], \; \P^{\hat\balpha^\smalltext{N},N,u}_\omega \text{\rm--a.s.}, \\
\notag Y^{i,N}_t &= g\big(X^i_{\cdot \land T}, L^N\big(\X^N_{\cdot \land T}\big)\big) + G\big(\varphi_1(X^i_{\cdot \land T}),\varphi_2(L^N\big(\X^N_{\cdot \land T}\big)\big) + \int_t^T f_s\big(X^i_{\cdot \land s},L^N\big(\X^N_{\cdot \land s},\hat\balpha^N_s\big),\hat\alpha^{i,N}_s\big) \d s\\
\notag &\quad - \int_t^T \partial^2_{m,n}G\big(M^{i,\star,N}_s,N^{\star,N}_s\big) \sum_{\ell =1}^N Z^{i,m,\ell,\star,N}_s \cdot Z^{n,\ell,\star,N}_s \d s \\
\notag &\quad - \frac{1}{2} \int_t^T \partial^2_{m,m}G\big(M^{i,\star,N}_s,N^{\star,N}_s\big) \sum_{\ell =1}^N \big\|Z^{i,m,\ell,\star,N}_s\big\|^2 \d s - \frac{1}{2} \int_t^T \partial^2_{n,n}G\big(M^{i,\star,N}_s,N^{\star,N}_s\big) \sum_{\ell =1}^N \big\|Z^{n,\ell,\star,N}_s\big\|^2 \d s \\
\notag &\quad - \int_t^T \sum_{\ell =1}^N Z^{i,\ell,N}_s \cdot \d \big(W_s^{\hat\balpha^\smalltext{N},N,u,\omega}\big)^\ell, \; t \in [u,T], \; \P^{\hat\balpha^\smalltext{N},N,u}_\omega \text{\rm--a.s.},  \\
\notag M^{i,\star,N}_t &= \varphi_1(X^i_{\cdot \land T}) - \int_t^T \sum_{\ell =1}^N Z^{i,m,\ell,\star,N}_s \cdot \d \big(W_s^{\hat\balpha^\smalltext{N},N,u,\omega}\big)^\ell, \; t \in [u,T], \; \P^{\hat\balpha^\smalltext{N},N,u}_\omega \text{\rm--a.s.}, \\
\notag N^{\star,N}_t &= \varphi_2\big(L^N\big(\X^N_{\cdot \land T}\big)\big) - \int_t^T \sum_{\ell =1}^N Z^{n,\ell,\star,N}_s \cdot \d \big(W_s^{\hat\balpha^\smalltext{N},N,u,\omega}\big)^\ell, \; t \in [u,T], \; \P^{\hat\balpha^\smalltext{N},N,u}_\omega \text{\rm--a.s.}, \\
\hat\alpha^{i,N}_t &= \Lambda_t\big(X^i_{\cdot \land t}, L^N\big(\X^N_{\cdot \land t}\big), Z^{i,i,N}_t, Z^{i,m,i,\star,N}_t, Z^{n,i,\star,N}_t, \aleph^{i,N}_t\big), \; \d t \otimes \P^{\hat\balpha^\smalltext{N},N,u}_\omega \text{\rm--a.e.}
\end{align}
Similarly, the system associated with the mean-field game, which holds $\P\text{\rm--a.e.} \; \omega\in\Omega$, is given by
\begin{align}\label{align:systemMeanFieldGame_rcpd}
\notag X^i_t &= X^i_u(\omega) + \int_u^t \sigma_s(X^i_{\cdot \land s}) b_s\big(X^i_{\cdot \land s},\cL_{\hat\alpha}(X_{\cdot \land s},\hat\alpha_s),\hat\alpha^i_s\big) \d s + \int_u^t \sigma_s(X^i_{\cdot \land s}) \d \big(W_s^{\hat\alpha,N,u,\omega}\big)^i, \; t \in [u,T], \; \P^{\hat\alpha,N,u}_\omega \text{\rm--a.s.}, \\
\notag Y^i_t &= g\big(X^i_{\cdot \land T}, \cL_{\hat\alpha}(X_{\cdot \land T})) + G\big(\varphi_1(X^i_{\cdot \land T}),\varphi_2\big(\cL_{\hat\alpha}(X_{\cdot \land T}))\big) + \int_t^T f_s\big(X^i_{\cdot \land s},\cL_{\hat\alpha}\big(X_{\cdot \land s},\hat\alpha_s),\hat\alpha^i_s\big) \d s\\
\notag &\quad  - \frac{1}{2} \int_t^T \partial^2_{m,m}G\big(M^{i,\star}_s,N^{\star}_s\big) \big\|Z^{i,m,i,\star}_s\big\|^2 \d s  - \int_t^T Z^{i,i}_s \cdot \d \big(W_s^{\hat\alpha,N,u,\omega}\big)^i, \; t \in [u,T], \; \P^{\hat\alpha,N,u}_\omega \text{\rm--a.s.}, \\
\notag M^{i,\star}_t &= \varphi_1(X^i_{\cdot \land T}) - \int_t^T Z^{i,m,i,\star}_s \cdot \d \big(W_s^{\hat\alpha,N,u,\omega}\big)^i, \; t \in [u,T], \; \P^{\hat\alpha,N,u}_\omega \text{\rm--a.s.}, \\
\notag N^{\star}_t &\coloneqq \varphi_2\big(\cL_{\hat\alpha}(X_{\cdot \land T})\big), \; t \in [u,T], \; \P^{\hat\alpha,N,u}_\omega \text{\rm--a.s.},\\
\hat\alpha^i_t &\coloneqq \Lambda_t\big(X^i_{\cdot \land t},\cL_{\hat\alpha}(X_{\cdot \land t}),Z^{i,i}_t,Z^{i,m,i,\star}_t,\mathbf{0},0\big), \; \mathrm{d}t\otimes\P^{\hat\alpha,N,u}_\omega \text{\rm--a.e.}
\end{align}

\medskip
Since the payoff \eqref{align:criteriumNplayerGame} depends on the strategies of the other players, the sub-game--perfect Nash equilibrium $\hat\balpha^N$ depends on $(\aleph^{i,N})_{i \in \{1,\ldots,N\}}$, which encodes the interactions among players. To establish convergence, it is therefore convenient to proceed in two steps: first, by showing that the $N$-player BSDE system \eqref{align:systemNplayerGame_rcpd} converges to an intermediate system, and second, by proving that this intermediate system converges to the mean-field system \eqref{align:systemMeanFieldGame_rcpd}. For $\P\text{\rm--a.e.} \; \omega\in\Omega$, the intermediate system takes the form
\begin{align}\label{align:intermediateSystemNplayerGame}
\notag \widetilde X^{i}_t &= X^i_u(\omega) + \int_u^t \sigma_s(\widetilde X^i_{\cdot \land s}) b_s\big(\widetilde X^i_{\cdot \land s},L^N\big(\widetilde \X^N_{\cdot \land s},\widetilde\balpha^N_s\big),\widetilde\alpha^{i,N}_s\big) \d s + \int_u^t \sigma_s(\widetilde X^i_{\cdot \land s}) \d \big(W_s^{\hat\balpha^\smalltext{N},N,u,\omega}\big)^i, \; t \in [u,T], \; \P^{\hat\balpha^\smalltext{N},N,u}_\omega \text{\rm--a.s.}, \\
\notag\widetilde Y^{i,N}_t &= g\big(\widetilde X^i_{\cdot \land T}, L^N\big(\widetilde \X^N_{\cdot \land T})\big) + G\big(\varphi_1(\widetilde Xi_{\cdot \land T}),\varphi_2\big(L^N\big(\widetilde \X^N_{\cdot \land T}\big)\big)\big) + \int_t^T f_s\big(\widetilde X^i_{\cdot \land s},L^N\big(\widetilde \X^N_{\cdot \land s},\widetilde\balpha^N_s\big),\widetilde\alpha^{i,N}_s\big) \d s\\
\notag &\quad  - \int_t^T \partial^2_{m,n}G\big(\widetilde M^{i,\star,N}_s,\widetilde N^{\star,N}_s\big) \sum_{\ell =1}^N \widetilde Z^{i,m,\ell,\star,N}_s \cdot \widetilde Z^{n,\ell,\star,N}_s \d s \\
\notag &\quad - \frac{1}{2} \int_t^T \partial^2_{m,m}G\big(\widetilde M^{i,\star,N}_s,\widetilde N^{\star,N}_s\big) \sum_{\ell =1}^N \big\|\widetilde Z^{i,m,\ell,\star,N}_s\big\|^2 \d s - \frac{1}{2} \int_t^T \partial^2_{n,n}G\big(\widetilde M^{i,\star,N}_s,\widetilde N^{\star,N}_s\big) \sum_{\ell =1}^N \big\|\widetilde Z^{n,\ell,\star,N}_s\big\|^2 \d s \\
\notag &\quad - \int_t^T \sum_{\ell =1}^N \widetilde Z^{i,\ell,N}_s \cdot \d \big(W_s^{\hat\balpha^\smalltext{N},N,u,\omega}\big)^\ell, \; t \in [u,T], \; \P^{\hat\balpha^\smalltext{N},N,u}_\omega \text{\rm--a.s.}, \\
\notag \widetilde M^{i,\star,N}_t & = \varphi_1(\widetilde X^i_{\cdot \land T}) - \int_t^T \sum_{\ell =1}^N \widetilde Z^{i,m,\ell,\star,N}_s \cdot \d \big(W_s^{\hat\balpha^\smalltext{N},N,u,\omega}\big)^\ell, \; t \in [u,T], \; \P^{\hat\balpha^\smalltext{N},N,u}_\omega \text{\rm--a.s.}, \\
\notag\widetilde N^{\star,N}_t & = \varphi_2\big(L^N(\widetilde \X^N_{\cdot \land T})\big) - \int_t^T \sum_{\ell =1}^N \widetilde Z^{n,\ell,\star,N}_s \cdot \d \big(W_s^{\hat\balpha^\smalltext{N},N,u,\omega}\big)^\ell, \; t \in [u,T], \; \P^{\hat\balpha^\smalltext{N},N,u}_\omega \text{\rm--a.s.}, \\
\widetilde \alpha^{i,N}_t &\coloneqq \Lambda_t\big(\widetilde X^{i,N}_{\cdot \land t}, L^N(\widetilde \X^N_{\cdot \land t}), \widetilde Z^{i,i,N}_t, \widetilde Z^{i,m,i,\star,N}_t, \widetilde Z^{n,i,\star,N}_t, 0\big), \; \mathrm{d}t\otimes\P^{\hat\alpha,N,u}_\omega \text{\rm--a.e.}
\end{align}

\subsubsection{The auxiliary system as a bridge from the finitely many player game}\label{subsubsection:fromNplayerGameToAuxiliary}

In this section, we derive estimates for the {\rm FBSDE} system corresponding to the difference between the system in \Cref{align:systemNplayerGame_rcpd} and the intermediate system in \Cref{align:intermediateSystemNplayerGame}. These estimates are obtained through repeated applications of It\^o's formula. In what follows, for any process $\eta^i \in \big\{X^i,Y^{i,N},M^{i,\star,N},N^{\star,N}, \Z^{i,N},\Z^{i,m,N},\Z^{n,N}\big\}$, $i \in \{1,\ldots,N\}$, we denote
\begin{align*}
\delta \eta^i_t = \eta^i_t - \widetilde{\eta}^i_t, \; t \in [u,T].
\end{align*}
Additionally, we introduce a constant $\beta >0$, whose value will be specified at the end of the section.

\medskip
\textbf{Step 1: estimates for the martingale terms}

\medskip
We fix an index $i \in \{1,\ldots,N\}$ and apply It\^o's formula to the process $\mathrm{e}^{\beta t} \big|\delta M^{i,\star,N}_t\big|^2$, for $t \in [u,T]$, to obtain that, for $\P\text{\rm--a.e.} \; \omega\in\Omega$,
\begin{align}\label{align:MIto_firstDifference}
\notag &\mathrm{e}^{\beta t} \big|\delta M^{i,\star,N}_t\big|^2 + \int_t^T \mathrm{e}^{\beta s} \sum_{\ell\in \{1,\ldots,N\}} \big\|\delta Z^{i,m,\ell,\star,N}_s\big\|^2 \d s \\
\notag &= \mathrm{e}^{\beta T} \big| \varphi_1(X^{i}_{\cdot \land T}) - \varphi_1(\widetilde X^{i}_{\cdot \land T}) \big|^2 - \beta \int_t^T \mathrm{e}^{\beta s} \big|\delta M^{i,\star,N}_s\big|^2 \d s \\
\notag &\quad- 2 \int_t^T \mathrm{e}^{\beta s} \delta M^{i,\star,N}_s \sum_{\ell=1}^N \delta Z^{i,m,\ell,\star,N}_s \cdot \d \big(W_s^{\hat\balpha^\smalltext{N},N,u,\omega}\big)^\ell\\ 
&\leq \ell^2_{\varphi_\smalltext{1}} \mathrm{e}^{\beta T} \big\|\delta X^{i}_{\cdot \land T}\big\|^2_\infty - 2 \int_t^T \mathrm{e}^{\beta s} \delta M^{i,\star,N}_s \sum_{\ell=1}^N \delta Z^{i,m,\ell,\star,N}_s \cdot \d \big(W_s^{\hat\balpha^\smalltext{N},N,u,\omega}\big)^\ell, \; t \in [u,T], \; \P^{\hat\balpha^\smalltext{N},N,u}_\omega \text{\rm--a.s.},
\end{align}
where the inequality is a consequence of \Cref{assumpConvThm}.\ref{lip_gGf}, which establishes the Lipschitz-continuity of $\varphi_1$. For any $\eta>0$, the Burkholder--Davis--Gundy's inequality with constant $c_{1,{\smallertext{\rm BDG}}}$ independent of both $N \in \N^\star$ and $\omega \in \Omega$, as given in \citeauthor*{osekowski2010sharp} \cite[Theorem 1.2]{osekowski2010sharp}, together with Young's inequality, yields
\begin{align*}
\E^{\P^{\smalltext{\hat\balpha}^\tinytext{N}\smalltext{,}\smalltext{N}\smalltext{,}\smalltext{u}}_\smalltext{\omega}} \bigg[ \sup_{t \in [u,T]} \mathrm{e}^{\beta t} \big|\delta M^{i,\star,N}_t\big|^2 \bigg]&\leq \ell^2_{\varphi_\smalltext{1}} \E^{\P^{\smalltext{\hat\balpha}^\tinytext{N}\smalltext{,}\smalltext{N}\smalltext{,}\smalltext{u}}_\smalltext{\omega}} \Big[ \mathrm{e}^{\beta T} \big\|\delta X^{i}_{\cdot \land T}\big\|^2_\infty \Big]+\eta 4c^2_{1,\smallertext{\rm BDG}}\E^{\P^{\smalltext{\hat\balpha}^\tinytext{N}\smalltext{,}\smalltext{N}\smalltext{,}\smalltext{u}}_\smalltext{\omega}} \bigg[ \sup_{t \in [u,T]} \mathrm{e}^{\beta t} \big|\delta M^{i,\star,N}_t\big|^2\bigg] \\
&\quad+ \frac1\eta\E^{\P^{\smalltext{\hat\balpha}^\tinytext{N}\smalltext{,}\smalltext{N}\smalltext{,}\smalltext{u}}_\smalltext{\omega}} \Bigg[  \int_u^T \mathrm{e}^{\beta t} \sum_{\ell =1}^N \big\|\delta Z^{i,m,\ell,\star,N}_t\big\|^2 \d t \Bigg], \; \P\text{\rm--a.e.} \; \omega\in\Omega.
\end{align*}
For any $\eta\in (0,1/(4c^2_{1,\smallertext{\rm BDG}}))$, it holds that
\begin{align}\label{align:MboundSup}
&\notag \E^{\P^{\smalltext{\hat\balpha}^\tinytext{N}\smalltext{,}\smalltext{N}\smalltext{,}\smalltext{u}}_\smalltext{\omega}} \bigg[ \sup_{t \in [u,T]} \mathrm{e}^{\beta t} \big|\delta M^{i,\star,N}_t\big|^2 \bigg]\\
&\leq \frac{1}{(1-\eta4c^2_{1,\smallertext{\rm BDG}})} \Bigg( \ell^2_{\varphi_\smalltext{1}} \E^{\P^{\smalltext{\hat\balpha}^\tinytext{N}\smalltext{,}\smalltext{N}\smalltext{,}\smalltext{u}}_\smalltext{\omega}} \Big[ \mathrm{e}^{\beta T} \big\|\delta X^{i}_{\cdot \land T}\big\|^2_\infty \Big] + \frac1\eta\E^{\P^{\smalltext{\hat\balpha}^\tinytext{N}\smalltext{,}\smalltext{N}\smalltext{,}\smalltext{u}}_\smalltext{\omega}} \Bigg[  \int_u^T \mathrm{e}^{\beta t} \sum_{\ell =1}^N \big\|\delta Z^{i,m,\ell,\star,N}_t\big\|^2 \d t \Bigg] \Bigg), \; \P\text{\rm--a.e.} \; \omega\in\Omega.
\end{align}
The right-hand side is finite, being the sum of two finite terms. The first term is finite due to the boundedness of the drift function $b$, which, together with \Cref{assumpConvThm}.\ref{lipSigma}, guarantees that $\|\delta X^{i}_{\cdot \land T}\|_\infty$ has finite moments of any order under any probability measure $\P^{{\balpha}^\smalltext{N}{,}{N}{,}{u}}_{\omega}$ for $\balpha \in A^N_N$. The second term is finite by the estimates in \Cref{necessity_toBSDE} and \Cref{assumpConvThm}.\ref{auxiliarySystem}. Consequently, the stochastic integral in \eqref{align:MIto_firstDifference} is an $(\F_N,\P^{{\hat\balpha}^\smalltext{N}{,}{N}{,}{u}}_{\omega})$-martingale since
\begin{align*}
&\E^{\P^{\smalltext{\hat\balpha}^\tinytext{N}\smalltext{,}\smalltext{N}\smalltext{,}\smalltext{u}}_\smalltext{\omega}} \Bigg[ \sup_{t \in [u,T]} \Bigg| \int_u^t \mathrm{e}^{\beta t} \delta M^{i,\star,N}_t \sum_{\ell=1}^N \delta Z^{i,m,\ell,\star,N}_t \cdot \d \big(W_t^{\hat\balpha^\smalltext{N},N,u,\omega}\big)^\ell \Bigg| \Bigg] \\
&\leq \frac{c_{1,\smallertext{\rm BDG}}}{2} \E^{\P^{\smalltext{\hat\balpha}^\tinytext{N}\smalltext{,}\smalltext{N}\smalltext{,}\smalltext{u}}_\smalltext{\omega}} \Bigg[ \sup_{t \in [u,T]} \mathrm{e}^{\beta t} \big|\delta M^{i,\star,N}_t\big|^2 + \int_u^T \mathrm{e}^{\beta t} \sum_{\ell =1}^N \big\|\delta Z^{i,m,\ell,\star,N}_t\big\|^2 \d t \Bigg],\; \P\text{\rm--a.e.} \; \omega\in\Omega.
\end{align*}
As a result, from \Cref{align:MIto_firstDifference}, we obtain
\begin{equation*}
\E^{\P^{\smalltext{\hat\balpha}^\tinytext{N}\smalltext{,}\smalltext{N}\smalltext{,}\smalltext{u}}_\smalltext{\omega}} \Bigg[  \int_u^T \mathrm{e}^{\beta t} \sum_{\ell =1}^N \big\|\delta Z^{i,m,\ell,\star,N}_t\big\|^2 \d t \Bigg] \leq \ell^2_{\varphi_\smalltext{1}} \E^{\P^{\smalltext{\hat\balpha}^\tinytext{N}\smalltext{,}\smalltext{N}\smalltext{,}\smalltext{u}}_\smalltext{\omega}} \Big[ \mathrm{e}^{\beta T} \big\|\delta X^{i}_{\cdot \land T}\big\|^2_\infty \Big], \; \P\text{\rm--a.e.} \; \omega\in\Omega,
\end{equation*}
which, combined with the estimate in \eqref{align:MboundSup}, also implies
\begin{align*}
\E^{\P^{\smalltext{\hat\balpha}^\tinytext{N}\smalltext{,}\smalltext{N}\smalltext{,}\smalltext{u}}_\smalltext{\omega}} \bigg[ \sup_{t \in [u,T]} \mathrm{e}^{\beta t} \big|\delta M^{i,\star,N}_t\big|^2 \bigg]\leq \frac{(1+\eta)\ell^2_{\varphi_{\smalltext{1}}}}{\eta(1-\eta4c^2_{1,\smallertext{\rm BDG}})}\E^{\P^{\smalltext{\hat\balpha}^\tinytext{N}\smalltext{,}\smalltext{N}\smalltext{,}\smalltext{u}}_\smalltext{\omega}} \Big[ \mathrm{e}^{\beta T} \big\|\delta X^{i}_{\cdot \land T}\big\|^2_\infty \Big], \; \P\text{\rm--a.e.} \; \omega\in\Omega,
\end{align*}
for any $\eta\in (0,1/(4c^2_{1,\smallertext{\rm BDG}}))$. We may therefore conclude that
\begin{align}\label{align:deltaM_finalEst}
\notag&\E^{\P^{\smalltext{\hat\balpha}^\tinytext{N}\smalltext{,}\smalltext{N}\smalltext{,}\smalltext{u}}_\smalltext{\omega}} \Bigg[ \sup_{t \in [u,T]} \mathrm{e}^{\beta t} \big|\delta M^{i,\star,N}_t\big|^2 + \int_u^T \mathrm{e}^{\beta t} \sum_{\ell =1}^N \big\|\delta Z^{i,m,\ell,\star,N}_t\big\|^2 \d t \Bigg] \\
&\leq \ell^2_{\varphi_{\smalltext{1}}}\underbrace{\bigg(1+\frac{1+\eta}{\eta(1-4c^2_{1,\smallertext{\rm BDG}}\eta)}\bigg)}_{\eqqcolon c^\smalltext{\star}(\eta) }\E^{\P^{\smalltext{\hat\balpha}^\tinytext{N}\smalltext{,}\smalltext{N}\smalltext{,}\smalltext{u}}_\smalltext{\omega}} \Big[ \mathrm{e}^{\beta T} \big\|\delta X^{i}_{\cdot \land T}\big\|^2_\infty \Big]\leq \ell^2_{\varphi_{\smalltext{1}}}c^\star \E^{\P^{\smalltext{\hat\balpha}^\tinytext{N}\smalltext{,}\smalltext{N}\smalltext{,}\smalltext{u}}_\smalltext{\omega}} \Big[ \mathrm{e}^{\beta T} \big\|\delta X^{i}_{\cdot \land T}\big\|^2_\infty \Big], \; \P\text{\rm--a.e.} \; \omega\in\Omega,
\end{align}
where
\begin{equation}\label{eq:cStarConst}
c^\star\coloneqq \min_{\eta \in(0,1/(4c^\smalltext{2}_{\smalltext{1}\smalltext{,}\smalltext{\rm BDG}}{)}{)}}c^\star(\eta)=1+\Big(\sqrt{1+4c^2_{1,\smallertext{\rm BDG}}}+2c_{1,\smallertext{\rm BDG}}\Big)^2.
\end{equation}

\medskip
By analogous reasoning, the same argument applies to the process $\mathrm{e}^{\beta t} \big|\delta N^{\star,N}_t\big|^2$, for $t \in [u,T]$, which satisfies, for $\P\text{\rm--a.e.} \; \omega\in\Omega$
\begin{align*}
&\mathrm{e}^{\beta t} \big|\delta N^{\star,N}_t\big|^2 + \int_t^T \mathrm{e}^{\beta s} \sum_{\ell =1}^N \big\|\delta Z^{n,\ell,\star,N}_s\big\|^2 \d s \\
		&= \mathrm{e}^{\beta T} \big| \varphi_2\big(L^N\big(\X^N_{\cdot \land T}\big)\big) - \varphi_2\big(L^N\big(\widetilde \X^N_{\cdot \land T}\big)\big) \big|^2 - \beta \int_t^T \mathrm{e}^{\beta s} \big|\delta N^{\star,N}_s\big|^2 \d s - 2 \int_t^T \mathrm{e}^{\beta s} \delta N^{\star,N}_s \sum_{\ell =1}^N \delta Z^{n,\ell,\star,N}_s \cdot \d \big(W_s^{\hat\balpha^N,N,u,\omega}\big)^\ell \\ 
		&\leq \frac{\ell^2_{\varphi_\smalltext{2}}}{N} \mathrm{e}^{\beta T} \sum_{\ell =1}^N \big\|\delta X^{\ell}_{\cdot \land T}\big\|^2_\infty  - 2 \int_t^T \mathrm{e}^{\beta s} \delta N^{\star,N}_s \sum_{\ell =1}^N \delta Z^{n,\ell,\star,N}_s \cdot \d \big(W_s^{\hat\balpha^{\smalltext{N}},N,u,\omega}\big)^\ell, \; t \in [u,T], \; \P^{\hat\balpha^\smalltext{N},N,u}_\omega \text{\rm--a.s.}
\end{align*}
We then deduce the estimate
\begin{align}\label{align:deltaN_finalEst}
&\E^{\P^{\smalltext{\hat\balpha}^\tinytext{N}\smalltext{,}\smalltext{N}\smalltext{,}\smalltext{u}}_\smalltext{\omega}} \Bigg[ \sup_{t \in [u,T]} \mathrm{e}^{\beta t} \big|\delta N^{\star,N}_t\big|^2 + \int_u^T \mathrm{e}^{\beta t} \sum_{\ell =1}^N \big\|\delta Z^{n,\ell,\star,N}_t\big\|^2 \d t \Bigg]\leq \frac{\ell^2_{\varphi_{\smalltext{2}}}c^\star}{N} \E^{\P^{\smalltext{\hat\balpha}^\tinytext{N}\smalltext{,}\smalltext{N}\smalltext{,}\smalltext{u}}_\smalltext{\omega}} \Bigg[ \mathrm{e}^{\beta T} \sum_{\ell=1}^N \big\|\delta X^{\ell}_{\cdot \land T}\big\|^2_\infty \Bigg], \; \P\text{\rm--a.e.} \; \omega\in\Omega.
\end{align}

\medskip
\textbf{Step 2: estimates for the value processes}

\medskip
We fix a player index  $i \in \{1,\ldots,N\}$. Applying It\^o's formula to $\mathrm{e}^{\beta t} \big|\delta Y^{i,N}_t\big|^2$, for $t \in [u,T]$, yields that
\begin{align*}
&\mathrm{e}^{\beta t} \big|\delta Y^{i,N}_t\big|^2 + \int_t^T \mathrm{e}^{\beta s} \sum_{\ell =1}^N \big\|\delta Z^{i,\ell,N}_s\big\|^2 \d s \\
	&= \mathrm{e}^{\beta T} \Big| g\big(X^i_{\cdot \land T}, L^N\big(\X^N_{\cdot \land T}\big)\big) + G\big(\varphi_1(X^i_{\cdot \land T}),\varphi_2\big(L^N\big(\X^N_{\cdot \land T}\big)\big)\big) - g\big(\widetilde X^i_{\cdot \land T}, L^N\big(\widetilde \X^N_{\cdot \land T}\big)\big) - G\big(\varphi_1(\widetilde X^i_{\cdot \land T}),\varphi_2\big(L^N\big(\widetilde \X^N_{\cdot \land T}\big)\big)\big) \Big|^2 \\
&\quad- \beta \int_t^T \mathrm{e}^{\beta s} \big|\delta Y^{i,N}_s\big|^2 \d s + 2 \int_t^T \mathrm{e}^{\beta s} \delta Y^{i,N}_s \big( f_s\big(X^i_{\cdot \land s},L^N\big(\X^N_{\cdot \land s},\hat\balpha^N_s\big),\hat\alpha^{i,N}_s\big) - f_s\big(\widetilde X^i_{\cdot \land s},L^N\big(\widetilde \X^N_{\cdot \land s},\widetilde\balpha^N_s\big),\widetilde\alpha^{i,N}_s\big) \big) \d s \\
&\quad- 2 \int_t^T \mathrm{e}^{\beta s} \delta Y^{i,N}_s \bigg( \partial^2_{m,n}G\big(M^{i,\star,N}_s,N^{\star,N}_s\big) \sum_{\ell =1}^N Z^{i,m,\ell,\star,N}_s \cdot Z^{n,\ell,\star,N}_s -\partial^2_{m,n}G\big(\widetilde M^{i,\star,N}_s,\widetilde N^{\star,N}_s\big) \sum_{\ell =1}^N \widetilde Z^{i,m,\ell,\star,N}_s \cdot \widetilde Z^{n,\ell,\star,N}_s \bigg) \d s \\
&\quad- \int_t^T \mathrm{e}^{\beta s} \delta Y^{i,N}_s \bigg( \partial^2_{m,m}G\big(M^{i,\star,N}_s,N^{\star,N}_s\big) \sum_{\ell =1}^N \big\|Z^{i,m,\ell,\star,N}_s\big\|^2  - \partial^2_{m,m}G\big(\widetilde M^{i,\star,N}_s,\widetilde N^{\star,N}_s\big) \sum_{\ell =1}^N \big\|\widetilde Z^{i,m,\ell,\star,N}_s\big\|^2 \bigg) \d s \\
&\quad- \int_t^T \mathrm{e}^{\beta s} \delta Y^{i,N}_s \bigg( \partial^2_{n,n}G\big(M^{i,\star,N}_s,N^{\star,N}_s\big) \sum_{\ell =1}^N \big\|Z^{n,\ell,\star,N}_s\big\|^2  - \partial^2_{n,n}G\big(\widetilde M^{i,\star,N}_s,\widetilde N^{\star,N}_s\big) \sum_{\ell =1}^N \big\|\widetilde Z^{n,\ell,\star,N}_s\big\|^2 \bigg) \d s \\
&\quad- 2 \int_t^T \mathrm{e}^{\beta s} \delta Y^{i,N}_s \sum_{\ell =1}^N \delta Z^{i,\ell,N}_s \cdot \d \big(W_s^{\hat\balpha^{\smalltext{N}},N,u,\omega}\big)^\ell, \; t \in [u,T], \; \P^{\hat\balpha^\smalltext{N},N,u}_\omega \text{\rm--a.s.}, \; \text{for} \; \P\text{\rm--a.e.} \; \omega\in\Omega.
\end{align*}

We fix some $\eps_1>0$, and assume that $\beta \geq 3 \ell^2_f/\eps_1$. Moreover, by \Cref{assumpConvThm}.\ref{lip_gGf} and Young's inequality, it follows that 
\begin{align*}
&\mathrm{e}^{\beta t} \big|\delta Y^{i,N}_t\big|^2 + \int_t^T \mathrm{e}^{\beta s} \sum_{\ell =1}^N \big\|\delta Z^{i,\ell,N}_s\big\|^2 \d s \\
	&\leq 2 \ell_{g+G,\varphi_\smalltext{1},\varphi_\smalltext{2}}^2 \mathrm{e}^{\beta T} \Big( \big\|\delta X^i_{\cdot \land T}\big\|^2_\infty + \cW_2^2\big(L^N\big(\X^N_{\cdot \land T}\big),L^N\big(\widetilde \X^N_{\cdot \land T}\big)\big) \Big) \\
&\quad+ \varepsilon_1 \int_t^T \mathrm{e}^{\beta s} \Big( \big\|\delta X^i_{\cdot \land s}\big\|^2_\infty + \cW^2_2\big(L^N\big(\X^N_{\cdot \land s},\hat\balpha^N_s\big),L^N\big(\widetilde \X^N_{\cdot \land s},\widetilde\balpha^N_s\big)\big) + d^2_A\big(\hat\alpha^{i,N}_s,\widetilde\alpha^{i,N}_s\big) \Big) \d s \\
&\quad- 2 \int_t^T \mathrm{e}^{\beta s} \delta Y^{i,N}_s \partial^2_{m,n}G\big(M^{i,\star,N}_s,N^{\star,N}_s\big) \sum_{\ell =1}^N Z^{i,m,\ell,\star,N}_s \cdot \delta Z^{n,\ell,\star,N}_s \d s \\
&\quad- 2 \int_t^T \mathrm{e}^{\beta s} \delta Y^{i,N}_s \partial^2_{m,n}G\big(M^{i,\star,N}_s,N^{\star,N}_s\big) \sum_{\ell =1}^N \widetilde Z^{n,\ell,\star,N}_s \cdot \delta Z^{i,m,\ell,\star,N}_s \d s \\
&\quad- 2 \int_t^T \mathrm{e}^{\beta s} \delta Y^{i,N}_s \Big( \partial^2_{m,n}G\big(M^{i,\star,N}_s,N^{\star,N}_s\big) - \partial^2_{m,n}G\big(\widetilde M^{i,\star,N}_s,\widetilde N^{\star,N}_s\big) \Big) \sum_{\ell =1}^N \widetilde Z^{i,m,\ell,\star,N}_s \cdot \widetilde Z^{n,\ell,\star,N}_s \d s \\
&\quad- \int_t^T \mathrm{e}^{\beta s} \delta Y^{i,N}_s \partial^2_{m,m}G\big(M^{i,\star,N}_s,N^{\star,N}_s\big) \sum_{\ell =1}^N Z^{i,m,\ell,\star,N}_s \cdot \delta Z^{i,m,\ell,\star,N}_s \d s \\
&\quad- \int_t^T \mathrm{e}^{\beta s} \delta Y^{i,N}_s \partial^2_{m,m}G\big(M^{i,\star,N}_s,N^{\star,N}_s\big) \sum_{\ell =1}^N \widetilde Z^{i,m,\ell,\star,N}_s \cdot \delta Z^{i,m,\ell,\star,N}_s \d s \\
&\quad- \int_t^T \mathrm{e}^{\beta s} \delta Y^{i,N}_s \Big(\partial^2_{m,m}G\big(M^{i,\star,N}_s,N^{\star,N}_s\big) - \partial^2_{m,m}G\big(\widetilde M^{i,\star,N}_s,\widetilde N^{\star,N}_s\big)\Big) \sum_{\ell =1}^N \big\|\widetilde Z^{i,m,\ell,\star,N}_s\big\|^2 \d s \\
&\quad- \int_t^T \mathrm{e}^{\beta s} \delta Y^{i,N}_s \partial^2_{n,n}G\big(M^{i,\star,N}_s,N^{\star,N}_s\big) \sum_{\ell=1}^N Z^{n,\ell,\star,N}_s \cdot \delta Z^{n,\ell,\star,N}_s \d s \\
&\quad- \int_t^T \mathrm{e}^{\beta s} \delta Y^{i,N}_s \partial^2_{n,n}G\big(M^{i,\star,N}_s,N^{\star,N}_s\big) \sum_{\ell =1}^N \widetilde Z^{n,\ell,\star,N}_s \cdot \delta Z^{n,\ell,\star,N}_s \d s \\
&\quad- \int_t^T \mathrm{e}^{\beta s} \delta Y^{i,N}_s \Big( \partial^2_{n,n}G\big(M^{i,\star,N}_s,N^{\star,N}_s\big) - \partial^2_{n,n}G\big(\widetilde M^{i,\star,N}_s,\widetilde N^{\star,N}_s\big) \Big) \sum_{\ell =1}^N \big\|\widetilde Z^{n,\ell,\star,N}_s\big\|^2 \d s \\
&\quad- 2 \int_t^T \mathrm{e}^{\beta s} \delta Y^{i,N}_s \sum_{\ell =1}^N \delta Z^{i,\ell,N}_s \cdot \d \big(W_s^{\hat\balpha^\smalltext{N},N,u,\omega}\big)^\ell \\
	&\leq 2 \ell_{g+G,\varphi_\smalltext{1},\varphi_\smalltext{2}}^2 \mathrm{e}^{\beta T} \Big( \big\|\delta X^i_{\cdot \land T}\big\|^2_\infty + \cW_2^2\big(L^N\big(\X^N_{\cdot \land T}\big),L^N\big(\widetilde \X^N_{\cdot \land T}\big)\big) \Big) \\
&\quad+ \varepsilon_1 \int_t^T \mathrm{e}^{\beta s} \Big( \big\|\delta X^i_{\cdot \land s}\big\|^2_\infty + \cW^2_2\big(L^N\big(\X^N_{\cdot \land s},\hat\balpha^N_s\big),L^N\big(\widetilde \X^N_{\cdot \land s},\widetilde\balpha^N_s\big)\big) + d^2_A\big(\hat\alpha^{i,N}_s,\widetilde\alpha^{i,N}_s\big) \Big) \d s \\
&\quad+ 2 c_{\partial^\smalltext{2}G} \int_t^T \mathrm{e}^{\beta s} \bigg| \d \bigg\langle \int_0^\cdot \delta Y^{i,N}_{r} \d M^{i,\star,N}_{r} , \delta N^{\star,N} \bigg\rangle_s \bigg| + 2 c_{\partial^\smalltext{2}G} \int_t^T \mathrm{e}^{\beta s} \bigg| \d \bigg\langle \int_0^\cdot \delta Y^{i,N}_{r} \d \widetilde N^{\star,N}_{r} , \delta M^{i,\star,N} \bigg\rangle_s \bigg| \\
&\quad+ 2 \int_t^T \mathrm{e}^{\beta s} \bigg| \d \bigg\langle \int_0^\cdot \delta Y^{i,N}_{r} \d \widetilde M^{i,\star,N}_{r} , \int_0^\cdot \Big( \partial^2_{m,n}G\big(M^{i,\star,N}_{r},N^{\star,N}_{r}\big) - \partial^2_{m,n}G\big(\widetilde M^{i,\star,N}_{r},\widetilde N^{\star,N}_{r}\big) \Big) \d \widetilde N^{\star,N}_{r} \bigg\rangle_s \bigg| \\
&\quad+ c_{\partial^\smalltext{2}G} \int_t^T \mathrm{e}^{\beta s} \bigg| \d \bigg\langle \int_0^\cdot \delta Y^{i,N}_{r} \d M^{i,\star,N}_{r} , \delta M^{i,\star,N} \bigg\rangle_s \bigg| + c_{\partial^\smalltext{2}G} \int_t^T \mathrm{e}^{\beta s} \bigg| \d \bigg\langle \int_0^\cdot \delta Y^{i,N}_{r} \d \widetilde M^{i,\star,N}_{r} , \delta M^{i,\star,N} \bigg\rangle_s \bigg| \\
&\quad+ \int_t^T \mathrm{e}^{\beta s} \bigg| \d \bigg\langle \int_0^\cdot \delta Y^{i,N}_{r}\d \widetilde M^{i,\star,N}_{r} , \int_0^\cdot \Big( \partial^2_{m,m}G\big(M^{i,\star,N}_{r},N^{\star,N}_{r}\big) - \partial^2_{m,m}G\big(\widetilde M^{i,\star,N}_{r},\widetilde N^{\star,N}_{r}\big) \Big) \d \widetilde M^{i,\star,N}_{r} \bigg\rangle_s \bigg| \\
&\quad+ c_{\partial^\smalltext{2}G} \int_t^T \mathrm{e}^{\beta s} \bigg| \d \bigg\langle \int_0^\cdot \delta Y^{i,N}_{r} \d N^{\star,N}_{r} , \delta N^{\star,N} \bigg\rangle_s \bigg| + c_{\partial^\smalltext{2}G} \int_t^T \mathrm{e}^{\beta s} \bigg| \d \bigg\langle \int_0^\cdot \delta Y^{i,N}_{r} \d \widetilde N^{\star,N}_{r} , \delta N^{\star,N} \bigg\rangle_s \bigg| \\
&\quad+ \int_t^T \mathrm{e}^{\beta s} \bigg| \d \bigg\langle \int_0^\cdot \delta Y^{i,N}_{r} \d \widetilde N^{\star,N}_{r} , \int_0^\cdot \Big( \partial^2_{n,n}G\big(M^{i,\star,N}_{r},N^{\star,N}_{r}\big) - \partial^2_{n,n}G\big(\widetilde M^{i,\star,N}_{r},\widetilde N^{\star,N}_{r}\big) \Big) \d \widetilde N^{\star,N}_{r} \bigg\rangle_s \bigg| \\
&\quad- 2 \int_t^T \mathrm{e}^{\beta s} \delta Y^{i,N}_s \sum_{\ell =1}^N \delta Z^{i,\ell,N}_s \cdot \d \big(W_s^{\hat\balpha^\smalltext{N},N,u,\omega}\big)^\ell, \; t \in [u,T], \; \P^{\hat\balpha^\smalltext{N},N,u}_\omega \text{\rm--a.s.}, \; \text{for} \; \P\text{\rm--a.e.} \; \omega\in\Omega.
\end{align*}
The last inequality follows from the fact that $\partial_{m,m}^2 G(M^{i,\star,N}, N^{\star,N})$, $\partial_{m,n}^2 G(M^{i,\star,N}, N^{\star,N})$, and $\partial_{n,n}^2 G(M^{i,\star,N}, N^{\star,N})$ are uniformly bounded by a constant $c_{\partial^2 G} > 0$ that does not depend on $N \in \N^\star$ or $\omega \in \Omega$. This, in turn, follows from the continuity of the second-order derivatives of the function $G$ stated in \Cref{assumpConvThm}.\ref{boundedSecondDerivativeG} and the boundedness of the processes $M^{i,\star,N}$ and $N^{\star,N}$, which can be deduced from the estimates in \textbf{Step 1} and the boundedness of the functions $\varphi_1$ and $\varphi_2$ given in \Cref{assumpConvThm}.\ref{boundedPhi_12}. Furthermore, the Lipschitz-continuity assumption on $\Lambda$, stated in \Cref{assumpConvThm}.\ref{lipLambda_growthAleph}, implies that 
\begin{align*}
&d^2_A\big(\hat\alpha^{i,N}_t,\widetilde \alpha^{i,N}_t\big) \\
&= d^2_A\Big(\Lambda_t\big(X^i_{\cdot \land t}, L^N\big(\X^N_{\cdot \land t}\big), Z^{i,i,N}_t, Z^{i,m,i,\star,N}_t, Z^{n,i,\star,N}_t, \aleph^{i,N}_t\big), \Lambda_t\big(\widetilde X^i_{\cdot \land t}, L^N\big(\widetilde \X^N_{\cdot \land t}\big), \widetilde Z^{i,i,N}_t, \widetilde Z^{i,m,i,\star,N}_t, \widetilde Z^{n,i,\star,N}_t, 0\big)\Big) \\
&\leq 6 \ell_\Lambda^2 \Big( \big\|\delta X^i_{\cdot \land t}\big\|^2_\infty + \cW^2_2\big(L^N\big(\X^N_{\cdot \land t}\big),L^N\big(\widetilde \X^N_{\cdot \land t}\big)\big) + \big\|\delta Z^{i,i,N}_t\big\|^2 + \big\|\delta Z^{i,m,i,\star,N}_t\big\|^2 + \big\|\delta Z^{n,i,\star,N}_t\big\|^2 + \big|\aleph^{i,N}_t\big|^2 \Big) \\
&\leq 6 \ell_\Lambda^2 \Bigg( \big\|\delta X^i_{\cdot \land t}\big\|^2_\infty + \frac{1}{N} \sum_{\ell=1}^N \big\|\delta X^{\ell,N}_{\cdot \land t}\big\|^2_\infty + \big\|\delta Z^{i,i,N}_t\big\|^2 + \big\|\delta Z^{i,m,i,\star,N}_t\big\|^2 + \big\|\delta Z^{n,i,\star,N}_t\big\|^2 + \big|\aleph^{i,N}_t\big|^2 \Bigg).
\end{align*}

It follows that
\begin{align}\label{eq:YbeforeEverything}
\notag&\mathrm{e}^{\beta t} \big|\delta Y^{i,N}_t\big|^2 + \int_t^T \mathrm{e}^{\beta s} \sum_{\ell =1}^N \big\|\delta Z^{i,\ell,N}_s\big\|^2 \d s \\
	\notag&\leq 2 \ell_{g+G,\varphi_\smalltext{1},\varphi_\smalltext{2}}^2 \mathrm{e}^{\beta T} \Bigg( \big\|\delta X^i_{\cdot \land T}\big\|^2_\infty + \frac{1}{N} \sum_{\ell=1}^N \big\|\delta X^{\ell}_{\cdot \land T}\big\|^2_\infty \Bigg) + \varepsilon_1 \int_t^T \mathrm{e}^{\beta s} \big\|\delta X^{i}_{\cdot \land s}\big\|^2_\infty \d s + \frac{\varepsilon_1}{N} \int_t^T \mathrm{e}^{\beta s} \sum_{\ell =1}^N \|\delta X^{\ell}_{\cdot \land s} \|^2_{\infty} \d s \\
\notag&\quad+ \varepsilon_1 6 \ell^2_\Lambda \int_t^T \mathrm{e}^{\beta s} \Bigg( \big\|\delta X^{i}_{\cdot \land s}\big\|^2_\infty + \frac{1}{N} \sum_{\ell =1}^N \big\|\delta X^{\ell}_{\cdot \land s}\big\|^2_\infty + \big\|\delta Z^{i,i,N}_s\big\|^2 + \big\|\delta Z^{i,m,i,\star,N}_s\big\|^2 + \big\|\delta Z^{n,i,\star,N}_s\big\|^2 + \big|\aleph^{i,N}_s\big|^2 \Bigg) \d s \\
\notag&\quad+ \frac{\varepsilon_1 6 \ell^2_\Lambda}{N} \int_t^T \mathrm{e}^{\beta s} \sum_{\ell =1}^N \Big( 2 \big\|\delta X^{\ell}_{\cdot \land s}\big\|^2_\infty + \big\|\delta Z^{\ell,\ell,N}_s\big\|^2 + \big\|\delta Z^{\ell,m,\ell,\star,N}_s\big\|^2 + \big\|\delta Z^{n,\ell,\star,N}_s\big\|^2 + \big|\aleph^{\ell,N}_s\big|^2 \Big) \d s \\
\notag&\quad+ 2 c_{\partial^\smalltext{2}G} \int_t^T \mathrm{e}^{\beta s} \bigg| \d \bigg\langle \int_0^\cdot \delta Y^{i,N}_{r} \d M^{i,\star,N}_{r} , \delta N^{\star,N} \bigg\rangle_s \bigg| + 2 c_{\partial^\smalltext{2}G} \int_t^T \mathrm{e}^{\beta s} \bigg| \d \bigg\langle \int_0^\cdot \delta Y^{i,N}_{r} \d \widetilde N^{\star,N}_{r} , \delta M^{i,\star,N} \bigg\rangle_s \bigg| \\
\notag&\quad+ 2 \int_t^T \mathrm{e}^{\beta s} \bigg| \d \bigg\langle \int_0^\cdot \delta Y^{i,N}_{r} \d \widetilde M^{i,\star,N}_{r} , \int_0^\cdot \Big( \partial^2_{m,n}G\big(M^{i,\star,N}_{r},N^{\star,N}_{r}\big) - \partial^2_{m,n}G\big(\widetilde M^{i,\star,N}_{r},\widetilde N^{\star,N}_{r}\big) \Big) \d \widetilde N^{\star,N}_{r} \bigg\rangle_s \bigg| \\
\notag&\quad+ c_{\partial^\smalltext{2}G} \int_t^T \mathrm{e}^{\beta s} \bigg| \d \bigg\langle \int_0^\cdot \delta Y^{i,N}_{r} \d M^{i,\star,N}_{r} , \delta M^{i,\star,N} \bigg\rangle_s \bigg| + c_{\partial^\smalltext{2}G} \int_t^T \mathrm{e}^{\beta s} \bigg| \d \bigg\langle \int_0^\cdot \delta Y^{i,N}_{r} \d \widetilde M^{i,\star,N}_{r} , \delta M^{i,\star,N} \bigg\rangle_s \bigg| \\
\notag&\quad+ \int_t^T \mathrm{e}^{\beta s} \bigg| \d \bigg\langle \int_0^\cdot \delta Y^{i,N}_{r} \d \widetilde M^{i,\star,N}_{r} , \int_0^\cdot \Big( \partial^2_{m,m}G\big(M^{i,\star,N}_{r},N^{\star,N}_{r}\big) - \partial^2_{m,m}G\big(\widetilde M^{i,\star,N}_{r},\widetilde N^{\star,N}_{r}\big) \Big) \d \widetilde M^{i,\star,N}_{r} \bigg\rangle_s \bigg| \\
\notag&\quad+ c_{\partial^\smalltext{2}G} \int_t^T \mathrm{e}^{\beta s} \bigg| \d \bigg\langle \int_0^\cdot \delta Y^{i,N}_{r} \d N^{\star,N}_{r} , \delta N^{\star,N} \bigg\rangle_s \bigg| + c_{\partial^\smalltext{2}G} \int_t^T \mathrm{e}^{\beta s} \bigg| \d \bigg\langle \int_0^\cdot \delta Y^{i,N}_{r} \d \widetilde N^{\star,N}_{r} , \delta N^{\star,N} \bigg\rangle_s \bigg| \\
\notag&\quad+ \int_t^T \mathrm{e}^{\beta s} \bigg| \d \bigg\langle \int_0^\cdot \delta Y^{i,N}_{r}\d \widetilde N^{\star,N}_{r} , \int_0^\cdot \Big( \partial^2_{n,n}G\big(M^{i,\star,N}_{r},N^{\star,N}_{r}\big) - \partial^2_{n,n}G\big(\widetilde M^{i,\star,N}_{r},\widetilde N^{\star,N}_{r}\big) \Big) \d \widetilde N^{\star,N}_{r} \bigg\rangle_s \bigg| \\
&\quad- 2 \int_t^T \mathrm{e}^{\beta s} \delta Y^{i,N}_s\sum_{\ell =1}^N \delta Z^{i,\ell,N}_s \cdot \d \big(W_s^{\hat\balpha^\smalltext{N},N,u,\omega}\big)^\ell, \; t \in [u,T], \; \P^{\hat\balpha^\smalltext{N},N,u}_\omega \text{\rm--a.s.}, \; \text{for} \; \P\text{\rm--a.e.} \; \omega\in\Omega.
\end{align}

\medskip
Given that
\begin{align*}
&\E^{\P^{\smalltext{\hat\balpha}^\tinytext{N}\smalltext{,}\smalltext{N}\smalltext{,}\smalltext{u}}_\smalltext{\omega}} \Bigg[ \sup_{t \in [u,T]} \Bigg| \int_u^t \mathrm{e}^{\beta t} \delta Y^{i,N}_t \sum_{\ell=1}^N \delta Z^{i,\ell,N}_t \cdot \d \big(W_t^{\hat\balpha^\smalltext{N},N,u,\omega}\big)^\ell \Bigg| \Bigg] \\
&\leq \frac{c_{1,\smallertext{\rm BDG}}}{2} \E^{\P^{\smalltext{\hat\balpha}^\tinytext{N}\smalltext{,}\smalltext{N}\smalltext{,}\smalltext{u}}_\smalltext{\omega}} \Bigg[ \sup_{t \in [u,T]} \mathrm{e}^{\beta t} \big|\delta Y^{i,N}_t\big|^2 + \int_u^T \mathrm{e}^{\beta t} \sum_{\ell =1}^N \big\|\delta Z^{i,\ell,N}_t\big\|^2 \d t \Bigg] <+\infty,\; \P\text{\rm--a.e.} \; \omega\in\Omega,
\end{align*}
the stochastic integral in \eqref{eq:YbeforeEverything} is an $(\F_N,\P^{{\hat\balpha}^\smalltext{N}{,}{N}{,}{u}}_{\omega})$-martingale. Consequently,
\begin{align*}
&\E^{\P^{\smalltext{\hat\balpha}^\tinytext{N}\smalltext{,}\smalltext{N}\smalltext{,}\smalltext{u}}_\smalltext{\omega}} \Bigg[ \int_u^T \mathrm{e}^{\beta t} \sum_{\ell =1}^N \big\|\delta Z^{i,\ell,N}_t\big\|^2 \d t \Bigg] \\
		&\leq 2\ell_{g+G,\varphi_\smalltext{1},\varphi_\smalltext{2}}^2 \E^{\P^{\smalltext{\hat\balpha}^\tinytext{N}\smalltext{,}\smalltext{N}\smalltext{,}\smalltext{u}}_\smalltext{\omega}} \Bigg[ \mathrm{e}^{\beta T} \Bigg( \big\|\delta X^{i}_{\cdot \land T}\big\|^2_\infty + \frac{1}{N} \sum_{\ell =1}^N \big\|\delta X^{\ell}_{\cdot \land T}\big\|^2_\infty \Bigg) \Bigg] \\
&\quad+ \varepsilon_1 (1 + 6 \ell^2_\Lambda) \E^{\P^{\smalltext{\hat\balpha}^\tinytext{N}\smalltext{,}\smalltext{N}\smalltext{,}\smalltext{u}}_\smalltext{\omega}} \bigg[ \int_u^T \mathrm{e}^{\beta t} \big\|\delta X^{i}_{\cdot \land t}\big\|^2_\infty \d t \bigg]  + \frac{\varepsilon_1(1 + 18 \ell^2_\Lambda)}{N} \E^{\P^{\smalltext{\hat\balpha}^\tinytext{N}\smalltext{,}\smalltext{N}\smalltext{,}\smalltext{u}}_\smalltext{\omega}} \Bigg[ \int_u^T \mathrm{e}^{\beta t} \sum_{\ell =1}^N \big\|\delta X^{\ell}_{\cdot \land t} \big\|^2_{\infty} \d t \Bigg] \\
&\quad+ \varepsilon_1 6 \ell^2_\Lambda \E^{\P^{\smalltext{\hat\balpha}^\tinytext{N}\smalltext{,}\smalltext{N}\smalltext{,}\smalltext{u}}_\smalltext{\omega}} \bigg[ \int_u^T \mathrm{e}^{\beta t} \Big( \big\|\delta Z^{i,i,N}_t\big\|^2 + \big\|\delta Z^{i,m,i,\star,N}_t\big\|^2 + \big\|\delta Z^{n,i,\star,N}_t\big\|^2 + \big|\aleph^{i,N}_t\big|^2 \Big) \d t \bigg] \\
&\quad+ \frac{\varepsilon_1 6 \ell^2_\Lambda}{N} \E^{\P^{\smalltext{\hat\balpha}^\tinytext{N}\smalltext{,}\smalltext{N}\smalltext{,}\smalltext{u}}_\smalltext{\omega}} \Bigg[ \int_u^T \mathrm{e}^{\beta t} \sum_{\ell=1}^N \Big( \big\|\delta Z^{\ell,\ell,N}_t\big\|^2 + \big\|\delta Z^{\ell,m,\ell,\star,N}_t\big\|^2 + \big\|\delta Z^{n,\ell,\star,N}_t\big\|^2 + \big|\aleph^{\ell,N}_t\big|^2 \Big) \d t \Bigg] \\
&\quad+ 2 c_{\partial^\smalltext{2}G} \E^{\P^{\smalltext{\hat\balpha}^\tinytext{N}\smalltext{,}\smalltext{N}\smalltext{,}\smalltext{u}}_\smalltext{\omega}} \Bigg[ \int_u^T \mathrm{e}^{\beta t} \bigg| \d \bigg\langle \int_0^\cdot \delta Y^{i,N}_{r} \d M^{i,\star,N}_{r} , \delta N^{\star,N} \bigg\rangle_t \bigg| \Bigg] + 2 c_{\partial^\smalltext{2}G} \E^{\P^{\smalltext{\hat\balpha}^\tinytext{N}\smalltext{,}\smalltext{N}\smalltext{,}\smalltext{u}}_\smalltext{\omega}} \Bigg[ \int_u^T \mathrm{e}^{\beta t} \bigg| \d \bigg\langle \int_0^\cdot \delta Y^{i,N}_{r} \d \widetilde N^{\star,N}_{r} , \delta M^{i,\star,N} \bigg\rangle_t \bigg| \Bigg] \\
&\quad+ 2 \E^{\P^{\smalltext{\hat\balpha}^\tinytext{N}\smalltext{,}\smalltext{N}\smalltext{,}\smalltext{u}}_\smalltext{\omega}} \Bigg[ \int_u^T \mathrm{e}^{\beta t} \bigg| \d \bigg\langle \int_0^\cdot \delta Y^{i,N}_{r} \d \widetilde M^{i,\star,N}_{r} , \int_0^\cdot \Big( \partial^2_{m,n}G\big(M^{i,\star,N}_{r},N^{\star,N}_{r}\big) - \partial^2_{m,n}G\big(\widetilde M^{i,\star,N}_{r},\widetilde N^{\star,N}_{r}\big) \Big) \d \widetilde N^{\star,N}_{r} \bigg\rangle_t \bigg| \Bigg] \\
&\quad+ c_{\partial^\smalltext{2}G} \E^{\P^{\smalltext{\hat\balpha}^\tinytext{N}\smalltext{,}\smalltext{N}\smalltext{,}\smalltext{u}}_\smalltext{\omega}} \Bigg[ \int_u^T \mathrm{e}^{\beta t} \bigg| \d \bigg\langle \int_0^\cdot \delta Y^{i,N}_{r} \d M^{i,\star,N}_{r} , \delta M^{i,\star,N} \bigg\rangle_t \bigg| + c_{\partial^\smalltext{2}G}  \int_u^T \mathrm{e}^{\beta t} \bigg| \d \bigg\langle \int_0^\cdot \delta Y^{i,N}_{r} \d \widetilde M^{i,\star,N}_{r} , \delta M^{i,\star,N} \bigg\rangle_t \bigg| \Bigg] \\
&\quad+ \E^{\P^{\smalltext{\hat\balpha}^\tinytext{N}\smalltext{,}\smalltext{N}\smalltext{,}\smalltext{u}}_\smalltext{\omega}} \Bigg[ \int_u^T \mathrm{e}^{\beta t} \bigg| \d \bigg\langle \int_0^\cdot \delta Y^{i,N}_{r} \d \widetilde M^{i,\star,N}_{r} , \int_0^\cdot \Big( \partial^2_{m,m}G\big(M^{i,\star,N}_{r},N^{\star,N}_{r}\big) - \partial^2_{m,m}G\big(\widetilde M^{i,\star,N}_{r},\widetilde N^{\star,N}_{r}\big) \Big) \d \widetilde M^{i,\star,N}_{r} \bigg\rangle_t \bigg| \Bigg] \\
&\quad+ c_{\partial^\smalltext{2}G} \E^{\P^{\smalltext{\hat\balpha}^\tinytext{N}\smalltext{,}\smalltext{N}\smalltext{,}\smalltext{u}}_\smalltext{\omega}} \Bigg[ \int_u^T \mathrm{e}^{\beta t} \bigg| \d \bigg\langle \int_0^\cdot \delta Y^{i,N}_{r} \d N^{\star,N}_{r} , \delta N^{\star,N} \bigg\rangle_t \bigg| \Bigg] + c_{\partial^\smalltext{2}G} \E^{\P^{\smalltext{\hat\balpha}^\tinytext{N}\smalltext{,}\smalltext{N}\smalltext{,}\smalltext{u}}_\smalltext{\omega}} \Bigg[ \int_u^T \mathrm{e}^{\beta t} \bigg| \d \bigg\langle \int_0^\cdot \delta Y^{i,N}_{r} \d \widetilde N^{\star,N}_{r} , \delta N^{\star,N} \bigg\rangle_t \bigg| \Bigg] \\
&\quad+ \E^{\P^{\smalltext{\hat\balpha}^\tinytext{N}\smalltext{,}\smalltext{N}\smalltext{,}\smalltext{u}}_\smalltext{\omega}} \Bigg[ \int_u^T \mathrm{e}^{\beta t} \bigg| \d \bigg\langle \int_0^\cdot \delta Y^{i,N}_{r} \d \widetilde N^{\star,N}_{r} , \int_0^\cdot \Big( \partial^2_{n,n}G\big(M^{i,\star,N}_{r},N^{\star,N}_{r}\big) - \partial^2_{n,n}G\big(\widetilde M^{i,\star,N}_{r},\widetilde N^{\star,N}_{r}\big) \Big) \d \widetilde N^{\star,N}_{r} \bigg\rangle_t \bigg| \Bigg] \\
		&\leq 2 \ell_{g+G,\varphi_\smalltext{1},\varphi_\smalltext{2}}^2 \E^{\P^{\smalltext{\hat\balpha}^\tinytext{N}\smalltext{,}\smalltext{N}\smalltext{,}\smalltext{u}}_\smalltext{\omega}} \Bigg[ \mathrm{e}^{\beta T} \Bigg( \big\|\delta X^{i}_{\cdot \land T}\big\|^2_\infty + \frac{1}{N} \sum_{\ell =1}^N \big\|\delta X^{\ell}_{\cdot \land T}\big\|^2_\infty \Bigg) \Bigg] \\
&\quad+ \varepsilon_1 (1 + 6 \ell^2_\Lambda) \E^{\P^{\smalltext{\hat\balpha}^\tinytext{N}\smalltext{,}\smalltext{N}\smalltext{,}\smalltext{u}}_\smalltext{\omega}} \bigg[ \int_u^T \mathrm{e}^{\beta t} \big\|\delta X^{i}_{\cdot \land t}\big\|^2_\infty \d t \bigg]  + \frac{\varepsilon_1(1 + 18 \ell^2_\Lambda)}{N} \E^{\P^{\smalltext{\hat\balpha}^\tinytext{N}\smalltext{,}\smalltext{N}\smalltext{,}\smalltext{u}}_\smalltext{\omega}} \Bigg[ \int_u^T \mathrm{e}^{\beta t} \sum_{\ell =1}^N \big\|\delta X^{\ell}_{\cdot \land t} \big\|^2_{\infty} \d t \Bigg] \\
&\quad+ \varepsilon_1 6 \ell^2_\Lambda \E^{\P^{\smalltext{\hat\balpha}^\tinytext{N}\smalltext{,}\smalltext{N}\smalltext{,}\smalltext{u}}_\smalltext{\omega}} \bigg[ \int_u^T \mathrm{e}^{\beta t} \Big( \big\|\delta Z^{i,i,N}_t\big\|^2 + \big\|\delta Z^{i,m,i,\star,N}_t\big\|^2 + \big\|\delta Z^{n,i,\star,N}_t\big\|^2 + \big|\aleph^{i,N}_t\big|^2 \Big) \d t \bigg] \\
&\quad+ \frac{\varepsilon_1 6 \ell^2_\Lambda}{N} \E^{\P^{\smalltext{\hat\balpha}^\tinytext{N}\smalltext{,}\smalltext{N}\smalltext{,}\smalltext{u}}_\smalltext{\omega}} \Bigg[ \int_u^T \mathrm{e}^{\beta t} \sum_{\ell =1}^N \Big( \big\|\delta Z^{\ell,\ell,N}_t\big\|^2 + \big\|\delta Z^{\ell,m,\ell,\star,N}_t\big\|^2 + \big\|\delta Z^{n,\ell,\star,N}_t\big\|^2 + \big|\aleph^{\ell,N}_t\big|^2 \Big) \d t \Bigg] \\
&\quad+ 2 c_{\partial^\smalltext{2}G} \Bigg(\E^{\P^{\smalltext{\hat\balpha}^\tinytext{N}\smalltext{,}\smalltext{N}\smalltext{,}\smalltext{u}}_\smalltext{\omega}} \Bigg[ \int_u^T \mathrm{e}^{\beta t} \d \bigg\langle \int_0^\cdot \delta Y^{i,N}_{r} \d M^{i,\star,N}_{r} \bigg\rangle_t \Bigg]\Bigg)^{\frac{1}{2}} \Bigg(\E^{\P^{\smalltext{\hat\balpha}^\tinytext{N}\smalltext{,}\smalltext{N}\smalltext{,}\smalltext{u}}_\smalltext{\omega}} \Bigg[ \int_u^T \mathrm{e}^{\beta t} \d \big\langle \delta N^{\star,N} \big\rangle_t \Bigg]\Bigg)^{\frac{1}{2}} \\
&\quad+ 2 c_{\partial^\smalltext{2}G} \Bigg(\E^{\P^{\smalltext{\hat\balpha}^\tinytext{N}\smalltext{,}\smalltext{N}\smalltext{,}\smalltext{u}}_\smalltext{\omega}} \Bigg[ \int_u^T \mathrm{e}^{\beta t} \d \bigg\langle \int_0^\cdot \delta Y^{i,N}_{r} \d \widetilde N^{\star,N}_{r} \bigg\rangle_t \Bigg]\Bigg)^{\frac{1}{2}} \Bigg(\E^{\P^{\smalltext{\hat\balpha}^\tinytext{N}\smalltext{,}\smalltext{N}\smalltext{,}\smalltext{u}}_\smalltext{\omega}} \Bigg[ \int_u^T \mathrm{e}^{\beta t} \d \big\langle \delta M^{i,\star,N} \big\rangle_t \Bigg]\Bigg)^{\frac{1}{2}} \\
&\quad+ 2 \Bigg(\E^{\P^{\smalltext{\hat\balpha}^\tinytext{N}\smalltext{,}\smalltext{N}\smalltext{,}\smalltext{u}}_\smalltext{\omega}} \Bigg[\int_u^T \mathrm{e}^{\beta t} \d \bigg\langle \int_u^\cdot \delta Y^{i,N}_{r} \d \widetilde M^{i,\star,N}_{r} \bigg\rangle_t \Bigg]\Bigg)^{\frac{1}{2}} \\
&\quad \times \Bigg(\E^{\P^{\smalltext{\hat\balpha}^\tinytext{N}\smalltext{,}\smalltext{N}\smalltext{,}\smalltext{u}}_\smalltext{\omega}} \Bigg[ \int_u^T \mathrm{e}^{\beta t} \d \bigg\langle \int_0^\cdot \Big( \partial^2_{m,n}G\big(M^{i,\star,N}_{r},N^{\star,N}_{r}\big) - \partial^2_{m,n}G\big(\widetilde M^{i,\star,N}_{r},\widetilde N^{\star,N}_{r}\big) \Big) \d \widetilde N^{\star,N}_{r} \bigg\rangle_t \Bigg]\Bigg)^{\frac{1}{2}} \\
&\quad+ c_{\partial^\smalltext{2}G} \Bigg(\E^{\P^{\smalltext{\hat\balpha}^\tinytext{N}\smalltext{,}\smalltext{N}\smalltext{,}\smalltext{u}}_\smalltext{\omega}} \Bigg[ \int_u^T \mathrm{e}^{\beta t} \d \bigg\langle \int_0^\cdot \delta Y^{i,N}_{r} \d M^{i,\star,N}_{r} \bigg\rangle_t \Bigg]\Bigg)^{\frac{1}{2}} \Bigg(\E^{\P^{\smalltext{\hat\balpha}^\tinytext{N}\smalltext{,}\smalltext{N}\smalltext{,}\smalltext{u}}_\smalltext{\omega}} \Bigg[ \int_u^T \mathrm{e}^{\beta t} \d \big\langle \delta M^{i,\star,N} \big\rangle_t \Bigg]\Bigg)^{\frac{1}{2}} \\
&\quad+ c_{\partial^\smalltext{2}G} \Bigg(\E^{\P^{\smalltext{\hat\balpha}^\tinytext{N}\smalltext{,}\smalltext{N}\smalltext{,}\smalltext{u}}_\smalltext{\omega}} \Bigg[ \int_u^T \mathrm{e}^{\beta t} \d \bigg\langle \int_0^\cdot \delta Y^{i,N}_{r} \d \widetilde M^{i,\star,N}_{r} \bigg\rangle_t \Bigg]\Bigg)^{\frac{1}{2}} \Bigg(\E^{\P^{\smalltext{\hat\balpha}^\tinytext{N}\smalltext{,}\smalltext{N}\smalltext{,}\smalltext{u}}_\smalltext{\omega}} \Bigg[ \int_u^T \mathrm{e}^{\beta t} \d \big\langle \delta M^{i,\star,N} \big\rangle_t \Bigg]\Bigg)^{\frac{1}{2}} \\
&\quad+ \Bigg(\E^{\P^{\smalltext{\hat\balpha}^\tinytext{N}\smalltext{,}\smalltext{N}\smalltext{,}\smalltext{u}}_\smalltext{\omega}} \Bigg[ \int_u^T \mathrm{e}^{\beta t} \d \bigg\langle \int_0^\cdot \delta Y^{i,N}_{r} \d \widetilde M^{i,\star,N}_{r} \bigg\rangle_t \Bigg]\Bigg)^{\frac{1}{2}} \\
&\quad \times \Bigg(\E^{\P^{\smalltext{\hat\balpha}^\tinytext{N}\smalltext{,}\smalltext{N}\smalltext{,}\smalltext{u}}_\smalltext{\omega}} \Bigg[ \int_u^T \mathrm{e}^{\beta t} \d \bigg\langle \int_0^\cdot \Big( \partial^2_{m,m}G\big(M^{i,\star,N}_{r},N^{\star,N}_{r}\big) - \partial^2_{m,m}G\big(\widetilde M^{i,\star,N}_{r},\widetilde N^{\star,N}_{r}\big) \Big) \d \widetilde M^{i,\star,N}_{r} \bigg\rangle_t \Bigg]\Bigg)^{\frac{1}{2}} \\
&\quad+ c_{\partial^\smalltext{2}G} \Bigg(\E^{\P^{\smalltext{\hat\balpha}^\tinytext{N}\smalltext{,}\smalltext{N}\smalltext{,}\smalltext{u}}_\smalltext{\omega}} \Bigg[ \int_u^T \mathrm{e}^{\beta t} \d \bigg\langle \int_0^\cdot \delta Y^{i,N}_{r} \d N^{\star,N}_{r} \bigg\rangle_t \Bigg]\Bigg)^{\frac{1}{2}} \Bigg(\E^{\P^{\smalltext{\hat\balpha}^\tinytext{N}\smalltext{,}\smalltext{N}\smalltext{,}\smalltext{u}}_\smalltext{\omega}} \Bigg[ \int_u^T \mathrm{e}^{\beta t} \d \big\langle \delta N^{\star,N} \big\rangle_t \Bigg]\Bigg)^{\frac{1}{2}} \\
&\quad+ c_{\partial^\smalltext{2}G} \Bigg(\E^{\P^{\smalltext{\hat\balpha}^\tinytext{N}\smalltext{,}\smalltext{N}\smalltext{,}\smalltext{u}}_\smalltext{\omega}} \Bigg[ \int_u^T \mathrm{e}^{\beta t} \d \bigg\langle \int_0^\cdot \delta Y^{i,N}_{r} \d \widetilde N^{\star,N}_{r} \bigg\rangle_t \Bigg]\Bigg)^{\frac{1}{2}} \Bigg(\E^{\P^{\smalltext{\hat\balpha}^\tinytext{N}\smalltext{,}\smalltext{N}\smalltext{,}\smalltext{u}}_\smalltext{\omega}} \Bigg[ \int_u^T \mathrm{e}^{\beta t} \d \big\langle \delta N^{\star,N} \big\rangle_t \Bigg]\Bigg)^{\frac{1}{2}} \\
&\quad+ \Bigg(\E^{\P^{\smalltext{\hat\balpha}^\tinytext{N}\smalltext{,}\smalltext{N}\smalltext{,}\smalltext{u}}_\smalltext{\omega}} \Bigg[ \int_u^T \mathrm{e}^{\beta t} \d \bigg\langle \int_0^\cdot \delta Y^{i,N}_{r} \d \widetilde N^{\star,N}_{r} \bigg\rangle_t \Bigg]\Bigg)^{\frac{1}{2}} \\
&\quad \times \Bigg(\E^{\P^{\smalltext{\hat\balpha}^\tinytext{N}\smalltext{,}\smalltext{N}\smalltext{,}\smalltext{u}}_\smalltext{\omega}} \Bigg[ \int_u^T \mathrm{e}^{\beta t} \d \bigg\langle \int_0^\cdot \Big( \partial^2_{n,n}G(M^{i,\star,N}_{r},N^{\star,N}_{r}) - \partial^2_{n,n}G(\widetilde M^{i,\star,N}_{r},\widetilde N^{\star,N}_{r}) \Big) \d \widetilde N^{\star,N}_{r} \bigg\rangle_t \Bigg]\Bigg)^{\frac{1}{2}}, \; \P\text{\rm--a.e.} \; \omega\in\Omega,
\end{align*}
as a consequence of Kunita--Watanabe's inequality and Cauchy--Schwarz's inequality. By applying \citeauthor*{delbaen2010harmonic} \cite[Lemma 1.4]{delbaen2010harmonic}, using the boundedness of both functions $\varphi_1$ and $\varphi_2$, and applying Young's inequality for some $\varepsilon_2>0$, we deduce that for $\P\text{\rm--a.e.} \; \omega\in\Omega$
\begin{align*}
&\E^{\P^{\smalltext{\hat\balpha}^\tinytext{N}\smalltext{,}\smalltext{N}\smalltext{,}\smalltext{u}}_\smalltext{\omega}} \Bigg[ \int_u^T \mathrm{e}^{\beta t} \sum_{\ell =1}^N \big\|\delta Z^{i,\ell,N}_t\big\|^2 \d t \Bigg] \\
		&\leq 2\ell_{g+G,\varphi_\smalltext{1},\varphi_\smalltext{2}}^2 \E^{\P^{\smalltext{\hat\balpha}^\tinytext{N}\smalltext{,}\smalltext{N}\smalltext{,}\smalltext{u}}_\smalltext{\omega}} \Bigg[ \mathrm{e}^{\beta T} \Bigg( \big\|\delta X^i_{\cdot \land T}\big\|^2_\infty + \frac{1}{N} \sum_{\ell =1}^N \big\|\delta X^{\ell}_{\cdot \land T}\big\|^2_\infty \Bigg) \Bigg] \\
		&\quad+ \varepsilon_1 (1 + 6 \ell^2_\Lambda) \E^{\P^{\smalltext{\hat\balpha}^\tinytext{N}\smalltext{,}\smalltext{N}\smalltext{,}\smalltext{u}}_\smalltext{\omega}} \bigg[ \int_u^T \mathrm{e}^{\beta t} \big\|\delta X^i_{\cdot \land t}\big\|^2_\infty \d t \bigg]  + \frac{\varepsilon_1(1 + 18 \ell^2_\Lambda)}{N} \E^{\P^{\smalltext{\hat\balpha}^\tinytext{N}\smalltext{,}\smalltext{N}\smalltext{,}\smalltext{u}}_\smalltext{\omega}} \Bigg[ \int_u^T \mathrm{e}^{\beta t} \sum_{\ell =1}^N \big\|\delta X^\ell_{\cdot \land t} \big\|^2_{\infty} \d t \bigg] \\
&\quad+ \varepsilon_1 6 \ell^2_\Lambda \E^{\P^{\smalltext{\hat\balpha}^\tinytext{N}\smalltext{,}\smalltext{N}\smalltext{,}\smalltext{u}}_\smalltext{\omega}} \bigg[ \int_u^T \mathrm{e}^{\beta t} \Big( \big\|\delta Z^{i,i,N}_t\big\|^2 + \big\|\delta Z^{i,m,i,\star,N}_t\big\|^2 + \big\|\delta Z^{n,i,\star,N}_t\big\|^2 + \big|\aleph^{i,N}_t\big|^2 \Big) \d t \bigg] \\
&\quad+ \frac{\varepsilon_1 6 \ell^2_\Lambda}{N} \E^{\P^{\smalltext{\hat\balpha}^\tinytext{N}\smalltext{,}\smalltext{N}\smalltext{,}\smalltext{u}}_\smalltext{\omega}} \Bigg[ \int_u^T \mathrm{e}^{\beta t} \sum_{\ell =1}^N \Big( \big\|\delta Z^{\ell,\ell,N}_t\big\|^2 + \big\|\delta Z^{\ell,m,\ell,\star,N}_t\big\|^2 + \big\|\delta Z^{n,\ell,\star,N}_t\big\|^2 + \big|\aleph^{\ell,N}_t\big|^2 \Big) \d t \Bigg] \\
&\quad+ \varepsilon_2 2 c^2_{\partial^\smalltext{2}G} \big\|M^{i,\star,N}\big\|_{{\smallertext{\rm{BMO}}}_{\smalltext{[}\smalltext{u}\smalltext{,}\smalltext{T}\smalltext{]}}}^2  \E^{\P^{\smalltext{\hat\balpha}^\tinytext{N}\smalltext{,}\smalltext{N}\smalltext{,}\smalltext{u}}_\smalltext{\omega}} \bigg[ \sup_{t \in [u,T]} \mathrm{e}^{\beta t} \big|\delta Y^{i,N}_t\big|^2 \bigg] + \frac{1}{\varepsilon_2} \E^{\P^{\smalltext{\hat\balpha}^\tinytext{N}\smalltext{,}\smalltext{N}\smalltext{,}\smalltext{u}}_\smalltext{\omega}} \bigg[\int_u^T \mathrm{e}^{\beta t} \d \big\langle \delta N^{\star,N} \big\rangle_t \bigg] \\
&\quad+ \varepsilon_2 2 c^2_{\partial^\smalltext{2}G} \big\|\widetilde N^{\star,N}\big\|_{\smallertext{\rm{BMO}}_{\smalltext{[}\smalltext{u}\smalltext{,}\smalltext{T}\smalltext{]}}}^2  \E^{\P^{\smalltext{\hat\balpha}^\tinytext{N}\smalltext{,}\smalltext{N}\smalltext{,}\smalltext{u}}_\smalltext{\omega}} \bigg[ \sup_{t \in [u,T]} \mathrm{e}^{\beta t} \big|\delta Y^{i,N}_t\big|^2 \bigg] + \frac{1}{\varepsilon_2} \E^{\P^{\smalltext{\hat\balpha}^\tinytext{N}\smalltext{,}\smalltext{N}\smalltext{,}\smalltext{u}}_\smalltext{\omega}} \bigg[ \int_u^T \mathrm{e}^{\beta t} \d \big\langle \delta M^{i,\star,N} \big\rangle_t \bigg] \\	
&\quad+ \varepsilon_2 2 \big\|\widetilde M^{i,\star,N}\big\|_{\smallertext{\rm{BMO}}_{\smalltext{[}\smalltext{u}\smalltext{,}\smalltext{T}\smalltext{]}}}^2 \E^{\P^{\smalltext{\hat\balpha}^\tinytext{N}\smalltext{,}\smalltext{N}\smalltext{,}\smalltext{u}}_\smalltext{\omega}} \bigg[ \sup_{t \in [u,T]} \mathrm{e}^{\beta t} \big|\delta Y^{i,N}_t\big|^2 \bigg] \\
&\quad+ \frac{2\ell^2_{\partial^\smalltext{2}G}}{\varepsilon_2} \big\|\widetilde N^{\star,N}\big\|^2_{\smallertext{{\rm{BMO}}}_{\smalltext{[}\smalltext{u}\smalltext{,}\smalltext{T}\smalltext{]}}} \E^{\P^{\smalltext{\hat\balpha}^\tinytext{N}\smalltext{,}\smalltext{N}\smalltext{,}\smalltext{u}}_\smalltext{\omega}} \bigg[ \sup_{t \in [u,T]} \mathrm{e}^{\beta t} \Big(\big| \delta M^{i,\star,N}_t\big|^2 + \big|\delta N^{\star,N}_t\big|^2\Big) \bigg] \\
&\quad+ \varepsilon_2 c^2_{\partial^\smalltext{2}G} \big\|M^{i,\star,N}\big\|^2_{\smallertext{\rm{BMO}}_{\smalltext{[}\smalltext{u}\smalltext{,}\smalltext{T}\smalltext{]}}} \E^{\P^{\smalltext{\hat\balpha}^\tinytext{N}\smalltext{,}\smalltext{N}\smalltext{,}\smalltext{u}}_\smalltext{\omega}} \bigg[ \sup_{t \in [u,T]} \mathrm{e}^{\beta t} \big|\delta Y^{i,N}_t\big|^2\bigg] + \frac{1}{2\varepsilon_2} \E^{\P^{\smalltext{\hat\balpha}^\tinytext{N}\smalltext{,}\smalltext{N}\smalltext{,}\smalltext{u}}_\smalltext{\omega}} \bigg[ \int_u^T \mathrm{e}^{\beta t} \d \big\langle \delta M^{i,\star,N} \big\rangle_t \bigg] \\		
&\quad+ \varepsilon_2 c^2_{\partial^\smalltext{2}G} \big\|\widetilde M^{i,\star,N}\big\|^2_{\smallertext{\rm{BMO}}_{\smalltext{[}\smalltext{u}\smalltext{,}\smalltext{T}\smalltext{]}}}  \E^{\P^{\smalltext{\hat\balpha}^\tinytext{N}\smalltext{,}\smalltext{N}\smalltext{,}\smalltext{u}}_\smalltext{\omega}} \bigg[ \sup_{t \in [u,T]} \mathrm{e}^{\beta t} \big|\delta Y^{i,N}_t\big|^2 \bigg] + \frac{1}{2\varepsilon_2} \E^{\P^{\smalltext{\hat\balpha}^\tinytext{N}\smalltext{,}\smalltext{N}\smalltext{,}\smalltext{u}}_\smalltext{\omega}} \bigg[ \int_u^T \mathrm{e}^{\beta t} \d \big\langle \delta M^{i,\star,N} \big\rangle_t \bigg] \\	
&\quad+ \varepsilon_2 \big\|\widetilde M^{i,\star,N}\big\|^2_{\smallertext{\rm{BMO}}_{\smalltext{[}\smalltext{u}\smalltext{,}\smalltext{T}\smalltext{]}}} \E^{\P^{\smalltext{\hat\balpha}^\tinytext{N}\smalltext{,}\smalltext{N}\smalltext{,}\smalltext{u}}_\smalltext{\omega}} \bigg[ \sup_{t \in [u,T]} \mathrm{e}^{\beta t} \big|\delta Y^{i,N}_t\big|^2 \bigg] \\
&\quad+ \frac{\ell^2_{\partial^\smalltext{2}G}}{\varepsilon_2} \big\|\widetilde M^{i,\star,N}\big\|^2_{\smallertext{\rm{BMO}}_{\smalltext{[}\smalltext{u}\smalltext{,}\smalltext{T}\smalltext{]}}} \E^{\P^{\smalltext{\hat\balpha}^\tinytext{N}\smalltext{,}\smalltext{N}\smalltext{,}\smalltext{u}}_\smalltext{\omega}} \bigg[ \sup_{t \in [u,T]} \mathrm{e}^{\beta t}\Big( \big|\delta M^{i,\star,N}_t\big|^2 + \big|\delta N^{\star,N}_t\big|^2\Big) \bigg] \\
&\quad+ \varepsilon_2 c^2_{\partial^\smalltext{2}G} \big\|N^{\star,N}\big\|^2_{\smallertext{\rm{BMO}}_{\smalltext{[}\smalltext{u}\smalltext{,}\smalltext{T}\smalltext{]}}} \E^{\P^{\smalltext{\hat\balpha}^\tinytext{N}\smalltext{,}\smalltext{N}\smalltext{,}\smalltext{u}}_\smalltext{\omega}} \bigg[ \sup_{t \in [u,T]} \mathrm{e}^{\beta t} \big|\delta Y^{i,N}_t\big|^2 \bigg] + \frac{1}{2\varepsilon_2} \E^{\P^{\smalltext{\hat\balpha}^\tinytext{N}\smalltext{,}\smalltext{N}\smalltext{,}\smalltext{u}}_\smalltext{\omega}} \bigg[ \int_u^T \mathrm{e}^{\beta t} \d \big\langle \delta N^{\star,N} \big\rangle_t \bigg] \\	
&\quad+ \varepsilon_2 c^2_{\partial^\smalltext{2}G} \big\|\widetilde N^{\star,N}\big\|^2_{\smallertext{\rm{BMO}}_{\smalltext{[}\smalltext{u}\smalltext{,}\smalltext{T}\smalltext{]}}} \E^{\P^{\smalltext{\hat\balpha}^\tinytext{N}\smalltext{,}\smalltext{N}\smalltext{,}\smalltext{u}}_\smalltext{\omega}} \bigg[ \sup_{t \in [u,T]} \mathrm{e}^{\beta t} \big|\delta Y^{i,N}_t\big|^2 \bigg] + \frac{1}{2\varepsilon_2} \E^{\P^{\smalltext{\hat\balpha}^\tinytext{N}\smalltext{,}\smalltext{N}\smalltext{,}\smalltext{u}}_\smalltext{\omega}} \bigg[ \int_u^T \mathrm{e}^{\beta t} \d \big\langle \delta N^{\star,N} \big\rangle_t \bigg] \\	
&\quad+ \varepsilon_2 \big\|\widetilde N^{\star,N}\big\|^2_{\smallertext{\rm{BMO}}_{\smalltext{[}\smalltext{u}\smalltext{,}\smalltext{T}\smalltext{]}}} \E^{\P^{\smalltext{\hat\balpha}^\tinytext{N}\smalltext{,}\smalltext{N}\smalltext{,}\smalltext{u}}_\smalltext{\omega}} \bigg[ \sup_{t \in [u,T]} \mathrm{e}^{\beta t} \big|\delta Y^{i,N}_t\big|^2 \bigg] + \frac{\ell^2_{\partial^\smalltext{2}G}}{\varepsilon_2} \big\|\widetilde N^{\star,N}\big\|^2_{\smallertext{\rm{BMO}}_{\smalltext{[}\smalltext{u}\smalltext{,}\smalltext{T}\smalltext{]}}} \E^{\P^{\smalltext{\hat\balpha}^\tinytext{N}\smalltext{,}\smalltext{N}\smalltext{,}\smalltext{u}}_\smalltext{\omega}} \bigg[ \sup_{t \in [u,T]} \mathrm{e}^{\beta t} \Big(\big| \delta M^{i,\star,N}_t\big|^2 + \big|\delta N^{\star,N}_t\big|^2\Big) \bigg].
\end{align*}

We define
\begin{gather*}
c_{\smallertext{\rm{BMO}}_{\smalltext{[}\smalltext{u}\smalltext{,}\smalltext{T}\smalltext{]}}} \coloneqq 3 c^2_{\partial^\smalltext{2}G} \big\|M^{i,\star,N}\big\|^2_{\smallertext{\rm{BMO}}_{\smalltext{[}\smalltext{u}\smalltext{,}\smalltext{T}\smalltext{]}}} + (3+c^2_{\partial^\smalltext{2}G}) \big\|\widetilde M^{i,\star,N}\big\|^2_{\smallertext{\rm{BMO}}_{\smalltext{[}\smalltext{u}\smalltext{,}\smalltext{T}\smalltext{]}}} + c^2_{\partial^\smalltext{2}G} \big\|N^{\star,N}\big\|^2_{\smallertext{\rm{BMO}}_{\smalltext{[}\smalltext{u}\smalltext{,}\smalltext{T}\smalltext{]}}} + (1+3c^2_{\partial^\smalltext{2}G}) \big\|\widetilde N^{\star,N}\big\|^2_{\smallertext{\rm{BMO}}_{\smalltext{[}\smalltext{u}\smalltext{,}\smalltext{T}\smalltext{]}}},\\
\bar c_{\smallertext{\rm{BMO}}_{\smalltext{[}\smalltext{u}\smalltext{,}\smalltext{T}\smalltext{]}}} \coloneqq \big\|\widetilde M^{i,\star,N}\big\|^2_{\smallertext{\rm{BMO}}_{\smalltext{[}\smalltext{u}\smalltext{,}\smalltext{T}\smalltext{]}}} + 3 \big\|\widetilde N^{\star,N}\big\|^2_{\smallertext{\rm{BMO}}_{\smalltext{[}\smalltext{u}\smalltext{,}\smalltext{T}\smalltext{]}}} .
\end{gather*}
Although the notation is slightly abused, it is clear that all the above constants are uniformly bounded in $N \in \N^\star$, since $\varphi^1$ and $\varphi^2$ are assumed to be bounded, as stated in \Cref{assumpConvThm}.\ref{boundedPhi_12}, and the drift function $b$ is bounded as well, so much so that we can use \citeauthor*{herdegen2021equilibrium} \cite[Lemma A.1]{herdegen2021equilibrium} to ensure that the BMO-norms appearing are indeed uniformly bounded in $N \in \N^\star$ (and $\omega \in \Omega)$, both under $\P^{{\hat\balpha}^\smalltext{N}{,}{N}{,}{u}}_{\omega}$ and $\P^{{N}{,}{u}}_{\omega}$. It follows that for $\P\text{\rm--a.e.} \; \omega\in\Omega$
\begin{align}\label{eq:intDeltaZ}
\notag&\E^{\P^{\smalltext{\hat\balpha}^\tinytext{N}\smalltext{,}\smalltext{N}\smalltext{,}\smalltext{u}}_\smalltext{\omega}} \Bigg[ \int_u^T \mathrm{e}^{\beta t} \sum_{\ell =1}^N \big\|\delta Z^{i,\ell,N}_t\big\|^2 \d t \Bigg] \\
		\notag&\leq 2\ell_{g+G,\varphi_\smalltext{1},\varphi_\smalltext{2}}^2 \E^{\P^{\smalltext{\hat\balpha}^\tinytext{N}\smalltext{,}\smalltext{N}\smalltext{,}\smalltext{u}}_\smalltext{\omega}} \Bigg[ \mathrm{e}^{\beta T} \Bigg( \big\|\delta X^{i}_{\cdot \land T}\big\|^2_\infty + \frac{1}{N} \sum_{\ell =1}^N \big\|\delta X^{\ell}_{\cdot \land T}\big\|^2_\infty \Bigg) \Bigg] \\
\notag&\quad+ \varepsilon_1 (1 + 6 \ell^2_\Lambda) \E^{\P^{\smalltext{\hat\balpha}^\tinytext{N}\smalltext{,}\smalltext{N}\smalltext{,}\smalltext{u}}_\smalltext{\omega}} \bigg[\int_u^T \mathrm{e}^{\beta t} \big\|\delta X^{i}_{\cdot \land t}\big\|^2_\infty \d t \bigg]  + \frac{\varepsilon_1(1 + 18 \ell^2_\Lambda)}{N} \E^{\P^{\smalltext{\hat\balpha}^\tinytext{N}\smalltext{,}\smalltext{N}\smalltext{,}\smalltext{u}}_\smalltext{\omega}} \Bigg[ \int_u^T \mathrm{e}^{\beta t} \sum_{\ell=1}^N \big\|\delta X^{\ell}_{\cdot \land t} \big\|^2_{\infty} \d t \Bigg] \\
\notag&\quad+ \varepsilon_1 6 \ell^2_\Lambda \E^{\P^{\smalltext{\hat\balpha}^\tinytext{N}\smalltext{,}\smalltext{N}\smalltext{,}\smalltext{u}}_\smalltext{\omega}} \bigg[ \int_u^T \mathrm{e}^{\beta t} \Big( \big\|\delta Z^{i,i,N}_t\big\|^2 + \big\|\delta Z^{i,m,i,\star,N}_t\big\|^2 + \big\|\delta Z^{n,i,\star,N}_t\big\|^2 + \big|\aleph^{i,N}_t\big|^2 \Big) \d t \bigg] \\
\notag&\quad+ \frac{\varepsilon_1 6 \ell^2_\Lambda}{N} \E^{\P^{\smalltext{\hat\balpha}^\tinytext{N}\smalltext{,}\smalltext{N}\smalltext{,}\smalltext{u}}_\smalltext{\omega}} \Bigg[ \int_u^T \mathrm{e}^{\beta t} \sum_{\ell =1}^N \Big( \big\|\delta Z^{\ell,\ell,N}_t\big\|^2 + \big\|\delta Z^{\ell,m,\ell,\star,N}_t\big\|^2 + \big\|\delta Z^{n,\ell,\star,N}_t\big\|^2 + \big|\aleph^{\ell,N}_t\big|^2 \Big) \d t \Bigg] \\
\notag&\quad+ \frac{2}{\varepsilon_2} \E^{\P^{\smalltext{\hat\balpha}^\tinytext{N}\smalltext{,}\smalltext{N}\smalltext{,}\smalltext{u}}_\smalltext{\omega}} \Bigg[ \int_u^T \mathrm{e}^{\beta t} \sum_{\ell=1}^N \big\|\delta Z^{i,m,\ell,\star,N}_t\big\|^2 \d t \Bigg] + \frac{2}{\varepsilon_2} \E^{\P^{\smalltext{\hat\balpha}^\tinytext{N}\smalltext{,}\smalltext{N}\smalltext{,}\smalltext{u}}_\smalltext{\omega}} \Bigg[ \int_u^T \mathrm{e}^{\beta t} \sum_{\ell =1}^N \big\|\delta Z^{n,\ell,\star,N}_t\big\|^2 \d t \Bigg] \\	
&\quad+ \frac{\bar c_{\smallertext{\rm{BMO}}_{\smalltext{[}\smalltext{u}\smalltext{,}\smalltext{T}\smalltext{]}}} \ell^2_{\partial^\smalltext{2}G} }{\varepsilon_2} \E^{\P^{\smalltext{\hat\balpha}^\tinytext{N}\smalltext{,}\smalltext{N}\smalltext{,}\smalltext{u}}_\smalltext{\omega}} \bigg[ \sup_{t \in [u,T]} \mathrm{e}^{\beta t} \Big(\big| \delta M^{i,\star,N}_t\big|^2 + \big|\delta N^{\star,N}_t \big|^2\Big) \bigg] + \varepsilon_2  c_{\smallertext{\rm{BMO}}_{\smalltext{[}\smalltext{u}\smalltext{,}\smalltext{T}\smalltext{]}}} \E^{\P^{\smalltext{\hat\balpha}^\tinytext{N}\smalltext{,}\smalltext{N}\smalltext{,}\smalltext{u}}_\smalltext{\omega}} \bigg[ \sup_{t \in [u,T]} \mathrm{e}^{\beta t} \big|\delta Y^{i,N}_t\big|^2 \bigg].
\end{align}

\medskip
Applying the Burkholder--Davis--Gundy's inequality together with Young’s inequality for some $\varepsilon_3 > 0$, and then proceeding as in the previous steps by making use of \cite[Lemma 1.4]{delbaen2010harmonic}, followed by another application of Young’s inequality for some $\varepsilon_4>0$, we deduce from \Cref{eq:YbeforeEverything} that
\begin{align}\label{eq:supDeltaY}
\notag&\E^{\P^{\smalltext{\hat\balpha}^\tinytext{N}\smalltext{,}\smalltext{N}\smalltext{,}\smalltext{u}}_\smalltext{\omega}} \bigg[  \sup_{t \in [u,T]} \mathrm{e}^{\beta t} \big|\delta Y^{i,N}_t\big|^2 \bigg]\\
		\notag&\leq 2\ell_{g+G,\varphi_\smalltext{1},\varphi_\smalltext{2}}^2 \E^{\P^{\smalltext{\hat\balpha}^\tinytext{N}\smalltext{,}\smalltext{N}\smalltext{,}\smalltext{u}}_\smalltext{\omega}} \Bigg[ \mathrm{e}^{\beta T} \bigg( \big\|\delta X^{i}_{\cdot \land T}\big\|^2_\infty + \frac{1}{N} \sum_{\ell =1}^N \big\|\delta X^{\ell}_{\cdot \land T}\big\|^2_\infty \bigg) \Bigg] \\
\notag&\quad+ \varepsilon_1 (1 + 6 \ell^2_\Lambda) \E^{\P^{\smalltext{\hat\balpha}^\tinytext{N}\smalltext{,}\smalltext{N}\smalltext{,}\smalltext{u}}_\smalltext{\omega}} \bigg[\int_u^T \mathrm{e}^{\beta t} \big\|\delta X^{i}_{\cdot \land t}\big\|^2_\infty \d t \bigg]  + \frac{\varepsilon_1(1 + 18 \ell^2_\Lambda)}{N} \E^{\P^{\smalltext{\hat\balpha}^\tinytext{N}\smalltext{,}\smalltext{N}\smalltext{,}\smalltext{u}}_\smalltext{\omega}} \Bigg[ \int_u^T \mathrm{e}^{\beta t} \sum_{\ell =1}^N \big\|\delta X^{\ell}_{\cdot \land t} \big\|^2_{\infty} \d t \Bigg] \\
\notag&\quad+ \varepsilon_1 6 \ell^2_\Lambda \E^{\P^{\smalltext{\hat\balpha}^\tinytext{N}\smalltext{,}\smalltext{N}\smalltext{,}\smalltext{u}}_\smalltext{\omega}} \bigg[ \int_u^T \mathrm{e}^{\beta t} \Big( \big\|\delta Z^{i,i,N}_t\big\|^2 + \big\|\delta Z^{i,m,i,\star,N}_t\big\|^2 + \big\|\delta Z^{n,i,\star,N}_t\big\|^2 + \big|\aleph^{i,N}_t\big|^2 \Big) \d t \bigg] \\
\notag&\quad+ \frac{\varepsilon_1 6 \ell^2_\Lambda}{N} \E^{\P^{\smalltext{\hat\balpha}^\tinytext{N}\smalltext{,}\smalltext{N}\smalltext{,}\smalltext{u}}_\smalltext{\omega}} \Bigg[ \int_u^T \mathrm{e}^{\beta t} \sum_{\ell =1}^N \Big( \big\|\delta Z^{\ell,\ell,N}_t\big\|^2 + \big\|\delta Z^{\ell,m,\ell,\star,N}_t\big\|^2 + \big\|\delta Z^{n,\ell,\star,N}_t\big\|^2 + \big|\aleph^{\ell,N}_t\big|^2 \Big) \d t \Bigg] \\
\notag&\quad+ \frac{2}{\varepsilon_4} \E^{\P^{\smalltext{\hat\balpha}^\tinytext{N}\smalltext{,}\smalltext{N}\smalltext{,}\smalltext{u}}_\smalltext{\omega}} \Bigg[ \int_u^T \mathrm{e}^{\beta t} \sum_{\ell =1}^N \big\|\delta Z^{i,m,\ell,\star,N}_t\big\|^2 \d t \Bigg] + \frac{2}{\varepsilon_4} \E^{\P^{\smalltext{\hat\balpha}^\tinytext{N}\smalltext{,}\smalltext{N}\smalltext{,}\smalltext{u}}_\smalltext{\omega}} \Bigg[ \int_u^T \mathrm{e}^{\beta t} \sum_{\ell =1}^N \big\|\delta Z^{n,\ell,\star,N}_t\big\|^2 \d t \Bigg] \\	
\notag&\quad+ \frac{\bar c_{\smallertext{\rm{BMO}}_{\smalltext{[}\smalltext{u}\smalltext{,}\smalltext{T}\smalltext{]}}} \ell^2_{\partial^\smalltext{2}G} }{\varepsilon_4} \E^{\P^{\smalltext{\hat\balpha}^\tinytext{N}\smalltext{,}\smalltext{N}\smalltext{,}\smalltext{u}}_\smalltext{\omega}} \bigg[ \sup_{t \in [u,T]} \mathrm{e}^{\beta t} \Big(\big| \delta M^{i,\star,N}_t\big|^2 + \big|\delta N^{\star,N}_t \big|^2\Big) \bigg] + \varepsilon_4 c_{\smallertext{\rm{BMO}}_{\smalltext{[}\smalltext{u}\smalltext{,}\smalltext{T}\smalltext{]}}} \E^{\P^{\smalltext{\hat\balpha}^\tinytext{N}\smalltext{,}\smalltext{N}\smalltext{,}\smalltext{u}}_\smalltext{\omega}} \bigg[ \sup_{t \in [u,T]} \mathrm{e}^{\beta t} \big|\delta Y^{i,N}_t\big|^2 \bigg] \\
&\quad+ \varepsilon_3 4 c^2_{1,{\smallertext{\rm BDG}}} \E^{\P^{\smalltext{\hat\balpha}^\tinytext{N}\smalltext{,}\smalltext{N}\smalltext{,}\smalltext{u}}_\smalltext{\omega}} \bigg[ \sup_{t \in [u,T]} \mathrm{e}^{\beta t} \big|\delta Y^{i,N}_t\big|^2 \bigg] + \frac{1}{\varepsilon_3} \E^{\P^{\smalltext{\hat\balpha}^\tinytext{N}\smalltext{,}\smalltext{N}\smalltext{,}\smalltext{u}}_\smalltext{\omega}} \Bigg[ \int_u^T \mathrm{e}^{\beta t} \sum_{\ell =1}^N \big\|\delta Z^{i,\ell,N}_t\big\|^2 \d t \Bigg], \; \P\text{\rm--a.e.} \; \omega\in\Omega.
\end{align}

Combining \eqref{eq:intDeltaZ} and \eqref{eq:supDeltaY}, we get that
\begin{align*}
&\bigg( 1 - \varepsilon_3 4c^2_{1,\smallertext{\rm BDG}} - \bigg(\varepsilon_4 +\frac{\eps_2}{\eps_3}\bigg) c_{\smallertext{\rm{BMO}}_{\smalltext{[}\smalltext{u}\smalltext{,}\smalltext{T}\smalltext{]}}} \bigg) \E^{\P^{\smalltext{\hat\balpha}^\tinytext{N}\smalltext{,}\smalltext{N}\smalltext{,}\smalltext{u}}_\smalltext{\omega}} \bigg[ \sup_{t \in [u,T]} \mathrm{e}^{\beta t} \big|\delta Y^{i,N}_t\big|^2 \bigg] \\
&\leq \bigg(1+\frac1{\eps_3}\bigg) 2\ell_{g+G,\varphi_\smalltext{1},\varphi_\smalltext{2}}^2 \E^{\P^{\smalltext{\hat\balpha}^\tinytext{N}\smalltext{,}\smalltext{N}\smalltext{,}\smalltext{u}}_\smalltext{\omega}} \Bigg[ \mathrm{e}^{\beta T} \Bigg( \big\|\delta X^{i}_{\cdot \land T}\big\|^2_\infty + \frac{1}{N} \sum_{\ell =1}^N \big\|\delta X^{\ell}_{\cdot \land T}\big\|^2_\infty \Bigg) \Bigg] \\
&\quad+ \varepsilon_1 \bigg(1+\frac1{\eps_3}\bigg) (1 + 6 \ell^2_\Lambda) \E^{\P^{\smalltext{\hat\balpha}^\tinytext{N}\smalltext{,}\smalltext{N}\smalltext{,}\smalltext{u}}_\smalltext{\omega}} \bigg[\int_u^T \mathrm{e}^{\beta t} \big\|\delta X^{i}_{\cdot \land t}\big\|^2_\infty \d t \bigg]  + \varepsilon_1 \bigg(1+\frac1{\eps_3}\bigg) \frac{(1 + 18 \ell^2_\Lambda)}{N} \E^{\P^{\smalltext{\hat\balpha}^\tinytext{N}\smalltext{,}\smalltext{N}\smalltext{,}\smalltext{u}}_\smalltext{\omega}} \Bigg[ \int_u^T \mathrm{e}^{\beta t} \sum_{\ell =1}^N \big\|\delta X^{\ell}_{\cdot \land t} \big\|^2_{\infty} \d t \Bigg] \\
&\quad+ \varepsilon_1 \bigg(1+\frac1{\eps_3}\bigg) 6 \ell^2_\Lambda \E^{\P^{\smalltext{\hat\balpha}^\tinytext{N}\smalltext{,}\smalltext{N}\smalltext{,}\smalltext{u}}_\smalltext{\omega}} \bigg[ \int_u^T \mathrm{e}^{\beta t} \Big( \big\|\delta Z^{i,i,N}_t\big\|^2 + \big\|\delta Z^{i,m,i,\star,N}_t\big\|^2 + \big\|\delta Z^{n,i,\star,N}_t\big\|^2 + \big|\aleph^{i,N}_t\big|^2 \Big) \d t \bigg] \\
&\quad+ \varepsilon_1 \bigg(1+\frac1{\eps_3}\bigg) \frac{6 \ell^2_\Lambda}{N} \E^{\P^{\smalltext{\hat\balpha}^\tinytext{N}\smalltext{,}\smalltext{N}\smalltext{,}\smalltext{u}}_\smalltext{\omega}} \Bigg[ \int_u^T \mathrm{e}^{\beta t} \sum_{\ell =1}^N \Big( \big\|\delta Z^{\ell,\ell,N}_t\big\|^2 + \big\|\delta Z^{\ell,m,\ell,\star,N}_t\big\|^2 + \big\|\delta Z^{n,\ell,\star,N}_t\big\|^2 + \big|\aleph^{\ell,N}_t\big|^2 \Big) \d t \Bigg] \\
&\quad+ \bigg(\frac{1}{\varepsilon_4}+\frac{1}{\eps_3\eps_2}\bigg) 2 \E^{\P^{\smalltext{\hat\balpha}^\tinytext{N}\smalltext{,}\smalltext{N}\smalltext{,}\smalltext{u}}_\smalltext{\omega}} \Bigg[ \int_u^T \mathrm{e}^{\beta t} \sum_{\ell =1}^N \Big( \big\|\delta Z^{i,m,\ell,\star,N}_t\big\|^2 + \big\|\delta Z^{n,\ell,\star,N}_t\big\|^2 \Big) \d t \Bigg] \\	
&\quad+ \bigg(\frac{1}{\varepsilon_4}+\frac{1}{\eps_3\eps_2}\bigg) \bar c_{\smallertext{\rm{BMO}}_{\smalltext{[}\smalltext{u}\smalltext{,}\smalltext{T}\smalltext{]}}} \ell^2_{\partial^\smalltext{2}G} \E^{\P^{\smalltext{\hat\balpha}^\tinytext{N}\smalltext{,}\smalltext{N}\smalltext{,}\smalltext{u}}_\smalltext{\omega}} \bigg[ \sup_{t \in [u,T]} \mathrm{e}^{\beta t} \Big(\big| \delta M^{i,\star,N}_t\big|^2 + \big|\delta N^{\star,N}_t\big|^2\Big) \bigg], \; \P\text{\rm--a.e.} \; \omega\in\Omega,
\end{align*}
and hence, rearranging terms, we have
\begin{align}\label{eq:longestim}
\notag&\bigg( 1 - \varepsilon_3 4c^2_{1,\smallertext{\rm BDG}} - \bigg(\varepsilon_2 +\eps_4+\frac{\eps_2}{\eps_3}\bigg) c_{\smallertext{\rm{BMO}}_{\smalltext{[}\smalltext{u}\smalltext{,}\smalltext{T}\smalltext{]}}} \bigg) \E^{\P^{\smalltext{\hat\balpha}^\tinytext{N}\smalltext{,}\smalltext{N}\smalltext{,}\smalltext{u}}_\smalltext{\omega}} \bigg[ \sup_{t \in [u,T]} \mathrm{e}^{\beta t} \big|\delta Y^{i,N}_t\big|^2 \bigg] + \E^{\P^{\smalltext{\hat\balpha}^\tinytext{N}\smalltext{,}\smalltext{N}\smalltext{,}\smalltext{u}}_\smalltext{\omega}} \Bigg[ \int_u^T \mathrm{e}^{\beta t} \sum_{\ell =1}^N \big\|\delta Z^{i,\ell,N}_t\big\|^2 \d t \Bigg] \\
\notag&\leq \bigg(2+\frac1{\eps_3}\bigg) 2\ell_{g+G,\varphi_\smalltext{1},\varphi_\smalltext{2}}^2 \E^{\P^{\smalltext{\hat\balpha}^\tinytext{N}\smalltext{,}\smalltext{N}\smalltext{,}\smalltext{u}}_\smalltext{\omega}} \Bigg[ \mathrm{e}^{\beta T} \Bigg( \big\|\delta X^{i}_{\cdot \land T}\big\|^2_\infty + \frac{1}{N} \sum_{\ell =1}^N \big\|\delta X^{\ell}_{\cdot \land T}\big\|^2_\infty \Bigg) \Bigg] \\
\notag&\quad+ \varepsilon_1 \bigg(2+\frac1{\eps_3}\bigg) (1 + 6 \ell^2_\Lambda) \E^{\P^{\smalltext{\hat\balpha}^\tinytext{N}\smalltext{,}\smalltext{N}\smalltext{,}\smalltext{u}}_\smalltext{\omega}} \bigg[\int_u^T \mathrm{e}^{\beta t} \big\|\delta X^{i}_{\cdot \land t}\big\|^2_\infty \d t \bigg] + \varepsilon_1 \bigg(2+\frac1{\eps_3}\bigg) \frac{(1 + 18 \ell^2_\Lambda)}{N} \E^{\P^{\smalltext{\hat\balpha}^\tinytext{N}\smalltext{,}\smalltext{N}\smalltext{,}\smalltext{u}}_\smalltext{\omega}} \Bigg[ \int_u^T \mathrm{e}^{\beta t} \sum_{\ell =1}^N \|\delta X^{\ell}_{\cdot \land t} \|^2_{\infty} \d t \Bigg] \\
\notag&\quad+ \varepsilon_1 \bigg(2+\frac1{\eps_3}\bigg) 6 \ell^2_\Lambda \E^{\P^{\smalltext{\hat\balpha}^\tinytext{N}\smalltext{,}\smalltext{N}\smalltext{,}\smalltext{u}}_\smalltext{\omega}} \bigg[ \int_u^T \mathrm{e}^{\beta t} \Big( \big\|\delta Z^{i,i,N}_t\big\|^2 + \big\|\delta Z^{i,m,i,\star,N}_t\big\|^2 + \big\|\delta Z^{n,i,\star,N}_t\big\|^2 + \big|\aleph^{i,N}_t\big|^2 \Big) \d t \bigg] \\
\notag&\quad+ \varepsilon_1 \bigg(2+\frac1{\eps_3}\bigg) \frac{6 \ell^2_\Lambda}{N} \E^{\P^{\smalltext{\hat\balpha}^\tinytext{N}\smalltext{,}\smalltext{N}\smalltext{,}\smalltext{u}}_\smalltext{\omega}} \Bigg[ \int_u^T \mathrm{e}^{\beta t} \sum_{\ell =1}^N \Big( \big\|\delta Z^{\ell,\ell,N}_t\big\|^2 + \big\|\delta Z^{\ell,m,\ell,\star,N}_t\big\|^2 + \big\|\delta Z^{n,\ell,\star,N}_t\big\|^2 + \big|\aleph^{\ell,N}_t\big|^2 \Big) \d t \Bigg] \\
\notag&\quad+ \bigg(\frac{1}{\varepsilon_2}+\frac1{\eps_4}+\frac{1}{\eps_2\eps_3}\bigg) 2 \E^{\P^{\smalltext{\hat\balpha}^\tinytext{N}\smalltext{,}\smalltext{N}\smalltext{,}\smalltext{u}}_\smalltext{\omega}} \Bigg[ \int_u^T \mathrm{e}^{\beta t} \sum_{\ell=1}^N \Big( \big\|\delta Z^{n,\ell,\star,N}_t\big\|^2 + \big\|\delta Z^{i,m,\ell,\star,N}_t\big\|^2 \Big) \d t \Bigg] \\	
&\quad+ \bigg(\frac{1}{\varepsilon_2}+\frac1{\eps_4}+\frac{1}{\eps_2\eps_3}\bigg) \bar c_{\smallertext{\rm{BMO}}_{\smalltext{[}\smalltext{u}\smalltext{,}\smalltext{T}\smalltext{]}}} \ell^2_{\partial^\smalltext{2}G} \E^{\P^{\smalltext{\hat\balpha}^\tinytext{N}\smalltext{,}\smalltext{N}\smalltext{,}\smalltext{u}}_\smalltext{\omega}} \bigg[ \sup_{t \in [u,T]} \mathrm{e}^{\beta t} \Big(\big| \delta M^{i,\star,N}_t\big|^2 + \big|\delta N^{\star,N}_t\big|^2\Big) \bigg], \; \P\text{\rm--a.e.} \; \omega\in\Omega.
\end{align}

\medskip
By substituting the estimates from \eqref{align:deltaM_finalEst} and \eqref{align:deltaN_finalEst} into \Cref{eq:longestim}, $\P\text{\rm--a.e.} \; \omega\in\Omega$, we have
\begin{align}\label{eq:longestim2}
\notag&\bigg( 1 - \varepsilon_3 4c^2_{1,\smallertext{\rm BDG}} - \bigg(\varepsilon_2 +\eps_4+\frac{\eps_2}{\eps_3}\bigg) c_{\smallertext{\rm{BMO}}_{\smalltext{[}\smalltext{u}\smalltext{,}\smalltext{T}\smalltext{]}}} \bigg) \E^{\P^{\smalltext{\hat\balpha}^\tinytext{N}\smalltext{,}\smalltext{N}\smalltext{,}\smalltext{u}}_\smalltext{\omega}} \bigg[ \sup_{t \in [u,T]} \mathrm{e}^{\beta t} \big|\delta Y^{i,N}_t\big|^2 \bigg] +\E^{\P^{\smalltext{\hat\balpha}^\tinytext{N}\smalltext{,}\smalltext{N}\smalltext{,}\smalltext{u}}_\smalltext{\omega}} \Bigg[ \int_u^T \mathrm{e}^{\beta t} \sum_{\ell =1}^N \big\|\delta Z^{i,\ell,N}_t\big\|^2 \d t \Bigg] \\
\notag&\leq \bigg(2+\frac1{\eps_3}\bigg) 2\ell_{g+G,\varphi_\smalltext{1},\varphi_\smalltext{2}}^2 \E^{\P^{\smalltext{\hat\balpha}^\tinytext{N}\smalltext{,}\smalltext{N}\smalltext{,}\smalltext{u}}_\smalltext{\omega}} \Bigg[ \mathrm{e}^{\beta T} \Bigg( \big\|\delta X^{i}_{\cdot \land T}\big\|^2_\infty + \frac{1}{N} \sum_{\ell =1}^N \big\|\delta X^{\ell}_{\cdot \land T}\big\|^2_\infty \Bigg) \Bigg] \\
\notag&\quad+ \varepsilon_1 \bigg(2+\frac1{\eps_3}\bigg) (1 + 6 \ell^2_\Lambda) \E^{\P^{\smalltext{\hat\balpha}^\tinytext{N}\smalltext{,}\smalltext{N}\smalltext{,}\smalltext{u}}_\smalltext{\omega}} \bigg[\int_u^T \mathrm{e}^{\beta t} \big\|\delta X^{i}_{\cdot \land t}\big\|^2_\infty \d t \bigg]  + \varepsilon_1 \bigg(2+\frac1{\eps_3}\bigg) \frac{(1 + 18 \ell^2_\Lambda)}{N} \E^{\P^{\smalltext{\hat\balpha}^\tinytext{N}\smalltext{,}\smalltext{N}\smalltext{,}\smalltext{u}}_\smalltext{\omega}} \Bigg[ \int_u^T \mathrm{e}^{\beta t} \sum_{\ell =1}^N \big\|\delta X^{\ell}_{\cdot \land t} \big\|^2_{\infty} \d t \Bigg] \\
\notag&\quad+ \varepsilon_1 \bigg(2+\frac1{\eps_3}\bigg) 6 \ell^2_\Lambda \E^{\P^{\smalltext{\hat\balpha}^\tinytext{N}\smalltext{,}\smalltext{N}\smalltext{,}\smalltext{u}}_\smalltext{\omega}} \bigg[ \int_u^T \mathrm{e}^{\beta t} \Big( \big\|\delta Z^{i,i,N}_t\big\|^2 + \big\|\delta Z^{i,m,i,\star,N}_t\big\|^2 + \big\|\delta Z^{n,i,\star,N}_t\big\|^2 + \big|\aleph^{i,N}_t\big|^2 \Big) \d t \bigg] \\
\notag&\quad+ \varepsilon_1 \bigg(2+\frac1{\eps_3}\bigg) \frac{6 \ell^2_\Lambda}{N} \E^{\P^{\smalltext{\hat\balpha}^\tinytext{N}\smalltext{,}\smalltext{N}\smalltext{,}\smalltext{u}}_\smalltext{\omega}} \Bigg[ \int_u^T \mathrm{e}^{\beta t} \sum_{\ell =1}^N \Big( \big\|\delta Z^{\ell,\ell,N}_t\big\|^2 + \big\|\delta Z^{\ell,m,\ell,\star,N}_t\big\|^2 + \big\|\delta Z^{n,\ell,\star,N}_t\big\|^2 + \big|\aleph^{\ell,N}_t\big|^2 \Big) \d t \Bigg] \\
\notag&\quad+ \bigg(\frac{1}{\varepsilon_2}+\frac1{\eps_4}+\frac{1}{\eps_2\eps_3}\bigg) \big(2\vee (\bar c_{\smallertext{\rm{BMO}}_{\smalltext{[}\smalltext{u}\smalltext{,}\smalltext{T}\smalltext{]}}} \ell^2_{\partial^\smalltext{2}G})\big) c^\star \Bigg(\ell^2_{\varphi_\smalltext{1}} \E^{\P^{\smalltext{\hat\balpha}^\tinytext{N}\smalltext{,}\smalltext{N}\smalltext{,}\smalltext{u}}_\smalltext{\omega}} \Big[ \mathrm{e}^{\beta T} \big\|\delta X^{i}_{\cdot \land T}\big\|^2_\infty \Big] +\frac{\ell^2_{\varphi_\smalltext{2}}}{N}  \E^{\P^{\smalltext{\hat\balpha}^\tinytext{N}\smalltext{,}\smalltext{N}\smalltext{,}\smalltext{u}}_\smalltext{\omega}} \Bigg[ \mathrm{e}^{\beta T} \sum_{\ell =1}^N \big\|\delta X^{\ell}_{\cdot \land T}\big\|^2_\infty \Bigg]\Bigg)\\
	\notag&= \E^{\P^{\smalltext{\hat\balpha}^\tinytext{N}\smalltext{,}\smalltext{N}\smalltext{,}\smalltext{u}}_\smalltext{\omega}} \Bigg[ \mathrm{e}^{\beta T} \Bigg( c_{\eps_{\smalltext{2}\smalltext{,}\smalltext{3}\smalltext{,}\smalltext{4}}} \big\|\delta X^{i}_{\cdot \land T}\big\|^2_\infty +\frac{\overline{c}_{\eps_{\smalltext{2}\smalltext{,}\smalltext{3}\smalltext{,}\smalltext{4}}}}N \sum_{\ell =1}^N \big\|\delta X^{\ell}_{\cdot \land T}\big\|^2_\infty \Bigg) \Bigg]\\
\notag&\quad+ \varepsilon_1 \bigg(2+\frac1{\eps_3}\bigg) (1 + 6 \ell^2_\Lambda) \E^{\P^{\smalltext{\hat\balpha}^\tinytext{N}\smalltext{,}\smalltext{N}\smalltext{,}\smalltext{u}}_\smalltext{\omega}} \bigg[\int_u^T \mathrm{e}^{\beta t} \big\|\delta X^{i}_{\cdot \land t}\big\|^2_\infty \d t \bigg] + \varepsilon_1 \bigg(2+\frac1{\eps_3}\bigg) \frac{(1 + 18 \ell^2_\Lambda)}{N} \E^{\P^{\smalltext{\hat\balpha}^\tinytext{N}\smalltext{,}\smalltext{N}\smalltext{,}\smalltext{u}}_\smalltext{\omega}} \Bigg[ \int_u^T \mathrm{e}^{\beta t} \sum_{\ell =1}^N \big\|\delta X^{\ell}_{\cdot \land t} \big\|^2_{\infty} \d t \Bigg] \\
\notag&\quad+ \varepsilon_1 \bigg(2+\frac1{\eps_3}\bigg) 6 \ell^2_\Lambda \E^{\P^{\smalltext{\hat\balpha}^\tinytext{N}\smalltext{,}\smalltext{N}\smalltext{,}\smalltext{u}}_\smalltext{\omega}} \bigg[ \int_u^T \mathrm{e}^{\beta t} \Big( \big\|\delta Z^{i,i,N}_t\big\|^2 + \big\|\delta Z^{i,m,i,\star,N}_t\big\|^2 + \big\|\delta Z^{n,i,\star,N}_t\big\|^2 + \big|\aleph^{i,N}_t\big|^2 \Big) \d t \bigg] \\
&\quad+ \varepsilon_1 \bigg(2+\frac1{\eps_3}\bigg) \frac{6 \ell^2_\Lambda}{N} \E^{\P^{\smalltext{\hat\balpha}^\tinytext{N}\smalltext{,}\smalltext{N}\smalltext{,}\smalltext{u}}_\smalltext{\omega}} \Bigg[ \int_u^T \mathrm{e}^{\beta t} \sum_{\ell =1}^N \Big( \big\|\delta Z^{\ell,\ell,N}_t\big\|^2 + \big\|\delta Z^{\ell,m,\ell,\star,N}_t\big\|^2 + \big\|\delta Z^{n,\ell,\star,N}_t\big\|^2 + \big|\aleph^{\ell,N}_t\big|^2 \Big) \d t \Bigg],
\end{align}
where
\begin{gather*}
c_{\eps_{\smalltext{2}\smalltext{,}\smalltext{3}\smalltext{,}\smalltext{4}}}\coloneqq \bigg(2+\frac1{\eps_3}\bigg) 2\ell_{g+G,\varphi_\smalltext{1},\varphi_\smalltext{2}}^2 + \bigg(\frac{1}{\varepsilon_2}+\frac1{\eps_4}+\frac{1}{\eps_2\eps_3}\bigg) \big(2\vee (\bar c_{\smallertext{\rm{BMO}}_{\smalltext{[}\smalltext{u}\smalltext{,}\smalltext{T}\smalltext{]}}}\ell^2_{\partial^\smalltext{2}G})\big) c^\star\ell^2_{\varphi_\smalltext{1}},\\
\overline{c}_{\eps_{\smalltext{2}\smalltext{,}\smalltext{3}\smalltext{,}\smalltext{4}}}\coloneqq \bigg(2+\frac1{\eps_3}\bigg) 2\ell_{g+G,\varphi_\smalltext{1},\varphi_\smalltext{2}}^2 + \bigg(\frac{1}{\varepsilon_2}+\frac1{\eps_4}+\frac{1}{\eps_2\eps_3}\bigg) \big(2\vee (\bar c_{\smallertext{\rm{BMO}}_{\smalltext{[}\smalltext{u}\smalltext{,}\smalltext{T}\smalltext{]}}}\ell^2_{\partial^\smalltext{2}G})\big) c^\star\ell^2_{\varphi_\smalltext{2}}.
\end{gather*}

The growth condition on $\aleph^{i,N}$ stated in \Cref{assumpConvThm}.\ref{lipLambda_growthAleph} reads as
\begin{align*}
\big|\aleph^{i,N}_t\big|^2 &\leq R_N^2 \Bigg( 1 + \big\|X^i_{\cdot \land t}\big\|_\infty + \sum_{\ell =1}^N \Big( \big\|Z^{i,\ell,N}_t\big\| + \big\|Z^{i,m,\ell,\star,N}_t\big\| + \big\|Z^{n,\ell,\star,N}_t\big\| \Big) \Bigg)^2 \\
&\leq 8R_N^2 \Bigg( 1 + \big\|X^i_{\cdot \land t}\big\|_\infty^2 + N \sum_{\ell =1}^N \Big( \big\|\widetilde Z^{i,\ell,N}_t\big\|^2  + \big\|\widetilde Z^{i,m,\ell,\star,N}_t\big\|^2 + \big\|\widetilde Z^{n,\ell,\star,N}_t\big\|^2 \Big) \Bigg) \\
&\quad+ 8NR_N^2 \sum_{\ell =1}^N \Big( \big\|\delta Z^{i,\ell,N}_t\big\|^2 + \big\|\delta Z^{i,m,\ell,\star,N}_t\big\|^2 + \big\|\delta Z^{n,\ell,\star,N}_t\big\|^2 \Big),\; i\in\{1,\dots,N\},
\end{align*}
and consequently implies that
\begin{align}\label{eq:controlxhi}
\nonumber&\big|\aleph^{i,N}_t\big|^2 + \frac1N \sum_{\ell =1}^N \big|\aleph^{\ell,N}_t\big|^2\\
\nonumber&\leq 16R_N^2 +8R_N^2\Bigg(\big\|X^i_{\cdot \land t}\big\|_\infty^2+\frac1N \sum_{\ell =1}^N \big\|X^\ell_{\cdot \land t}\big\|_\infty^2\Bigg) +8NR_N^2 \sum_{\ell=1}^N \Big( \big\|\widetilde Z^{i,\ell,N}_t\big\|^2+ \big\|\widetilde Z^{i,m,\ell,\star,N}_t\big\|^2+2 \big\|\widetilde Z^{n,\ell,\star,N}_t\big\|^2\Big)\\
\nonumber&\quad+8R_N^2 \sum_{(k,\ell)\in\{1,\dots,N\}^\smalltext{2}} \Big(\big\|\widetilde Z^{k,\ell,N}_t\big\|^2 +\big\|\widetilde Z^{k,m,\ell,\star,N}_t\big\|^2 + \big\|\delta Z^{k,\ell,N}_t\big\|^2 + \big\|\delta Z^{k,m,\ell,\star,N}_t\big\|^2\Big)\\
&\quad+8NR_N^2 \sum_{\ell =1}^N \Big(\big\|\delta Z^{i,\ell,N}_t\big\|^2 + \big\|\delta Z^{i,m,\ell,\star,N}_t\big\|^2 +2 \big\|\delta Z^{n,\ell,\star,N}_t\big\|^2\Big).
\end{align}

Plugging the above estimate into \eqref{eq:longestim2}, we have
\begin{align}\label{align:deltaYtoSum}
\notag&\bigg( 1 - \varepsilon_3 4c^2_{1,\smallertext{\rm BDG}} - \bigg(\varepsilon_2 +\eps_4+\frac{\eps_2}{\eps_3}\bigg) c_{\smallertext{\rm{BMO}}_{\smalltext{[}\smalltext{u}\smalltext{,}\smalltext{T}\smalltext{]}}} \bigg) \E^{\P^{\smalltext{\hat\balpha}^\tinytext{N}\smalltext{,}\smalltext{N}\smalltext{,}\smalltext{u}}_\smalltext{\omega}} \bigg[ \sup_{t \in [u,T]} \mathrm{e}^{\beta t} \big|\delta Y^{i,N}_t\big|^2 \bigg] +\E^{\P^{\smalltext{\hat\balpha}^\tinytext{N}\smalltext{,}\smalltext{N}\smalltext{,}\smalltext{u}}_\smalltext{\omega}} \Bigg[ \int_u^T \mathrm{e}^{\beta t} \sum_{\ell =1}^N \big\|\delta Z^{i,\ell,N}_t \big\|^2 \d t \Bigg] \\
	\notag&\leq c_{\eps_{\smalltext{2}\smalltext{,}\smalltext{3}\smalltext{,}\smalltext{4}}} \E^{\P^{\smalltext{\hat\balpha}^\tinytext{N}\smalltext{,}\smalltext{N}\smalltext{,}\smalltext{u}}_\smalltext{\omega}} \Big[ \mathrm{e}^{\beta T} \big\|\delta X^{i}_{\cdot \land T}\big\|^2_\infty \Big] +\frac{\overline{c}_{\eps_{\smalltext{2}\smalltext{,}\smalltext{3}\smalltext{,}\smalltext{4}}}}N \E^{\P^{\smalltext{\hat\balpha}^\tinytext{N}\smalltext{,}\smalltext{N}\smalltext{,}\smalltext{u}}_\smalltext{\omega}} \Bigg[\mathrm{e}^{\beta T} \sum_{\ell =1}^N \big\|\delta X^{\ell}_{\cdot \land T}\big\|^2_\infty\Bigg] \\
\notag&\quad+ \eps_1 \bigg(2+\frac1{\eps_3}\bigg) 48 \ell^2_\Lambda R_N^2 \Bigg(\frac{2}{\beta} \mathrm{e}^{\beta T} + \E^{\P^{\smalltext{\hat\balpha}^\tinytext{N}\smalltext{,}\smalltext{N}\smalltext{,}\smalltext{u}}_\smalltext{\omega}} \bigg[\int_u^T\mathrm{e}^{\beta t}\big\|X^i_{\cdot \land t}\big\|_\infty^2\mathrm{d}t\bigg]+\frac1N \E^{\P^{\smalltext{\hat\balpha}^\tinytext{N}\smalltext{,}\smalltext{N}\smalltext{,}\smalltext{u}}_\smalltext{\omega}} \Bigg[\int_u^T\mathrm{e}^{\beta t} \sum_{\ell =1}^N \big\|X^\ell_{\cdot \land t}\big\|_\infty^2\mathrm{d}t\Bigg]\Bigg)\\
\notag&\quad+ \varepsilon_1 \bigg(2+\frac1{\eps_3}\bigg) (1 + 6 \ell^2_\Lambda) \E^{\P^{\smalltext{\hat\balpha}^\tinytext{N}\smalltext{,}\smalltext{N}\smalltext{,}\smalltext{u}}_\smalltext{\omega}} \bigg[\int_u^T \mathrm{e}^{\beta t} \big\|\delta X^{i}_{\cdot \land t}\big\|^2_\infty \d t \bigg] + \varepsilon_1 \bigg(2+\frac1{\eps_3}\bigg) \frac{(1 + 18 \ell^2_\Lambda)}{N} \E^{\P^{\smalltext{\hat\balpha}^\tinytext{N}\smalltext{,}\smalltext{N}\smalltext{,}\smalltext{u}}_\smalltext{\omega}} \Bigg[ \int_u^T \mathrm{e}^{\beta t} \sum_{\ell =1}^N \big\|\delta X^{\ell}_{\cdot \land t}\big\|^2_{\infty} \d t \Bigg] \\
\notag&\quad + \eps_1 \bigg(2+\frac1{\eps_3}\bigg) 48\ell^2_\Lambda NR_N^2\E^{\P^{\smalltext{\hat\balpha}^\tinytext{N}\smalltext{,}\smalltext{N}\smalltext{,}\smalltext{u}}_\smalltext{\omega}} \Bigg[\int_u^T\mathrm{e}^{\beta t} \sum_{\ell =1}^N \Big( \big\|\widetilde Z^{i,\ell,N}_t\big|^2+\big\|\widetilde Z^{i,m,\ell,\star,N}_s\big\|^2+2 \big\|\widetilde Z^{n,\ell,\star,N}_s\big\|^2\Big)\mathrm{d}t\Bigg]\\
\notag&\quad+ \eps_1 \bigg(2+\frac1{\eps_3}\bigg) 48\ell^2_\Lambda R_N^2\E^{\P^{\smalltext{\hat\balpha}^\tinytext{N}\smalltext{,}\smalltext{N}\smalltext{,}\smalltext{u}}_\smalltext{\omega}} \Bigg[\int_u^T\mathrm{e}^{\beta t} \sum_{(k,\ell)\in\{1,\dots,N\}^\smalltext{2}} \Big( \big\|\widetilde Z^{k,\ell,N}_t\big\|^2 +\big\|\widetilde Z^{k,m,\ell,\star,N}_s\big\|^2\Big)\mathrm{d}t\Bigg]\\
\notag&\quad+ \varepsilon_1 \bigg(2+\frac1{\eps_3}\bigg) 6 \ell^2_\Lambda \E^{\P^{\smalltext{\hat\balpha}^\tinytext{N}\smalltext{,}\smalltext{N}\smalltext{,}\smalltext{u}}_\smalltext{\omega}} \bigg[ \int_u^T \mathrm{e}^{\beta t} \Big( \big\|\delta Z^{i,i,N}_t\big\|^2 + \big\|\delta Z^{i,m,i,\star,N}_t\big\|^2 + \big\|\delta Z^{n,i,\star,N}_t\big\|^2 \Big) \d t \bigg] \\
\notag&\quad+ \varepsilon_1 \bigg(2+\frac1{\eps_3}\bigg) \frac{6 \ell^2_\Lambda}{N} \E^{\P^{\smalltext{\hat\balpha}^\tinytext{N}\smalltext{,}\smalltext{N}\smalltext{,}\smalltext{u}}_\smalltext{\omega}} \Bigg[ \int_u^T \mathrm{e}^{\beta t} \sum_{\ell =1}^N \Big( \big\|\delta Z^{\ell,\ell,N}_t\big\|^2 + \big\|\delta Z^{\ell,m,\ell,\star,N}_t\big\|^2 + (1+16N^2R_N^2)\big\|\delta Z^{n,\ell,\star,N}_t\big\|^2  \Big) \d t \Bigg]\\
\notag&\quad + \eps_1 \bigg(2+\frac1{\eps_3}\bigg) 48\ell^2_\Lambda NR_N^2\E^{\P^{\smalltext{\hat\balpha}^\tinytext{N}\smalltext{,}\smalltext{N}\smalltext{,}\smalltext{u}}_\smalltext{\omega}} \Bigg[\int_u^T\mathrm{e}^{\beta t} \sum_{\ell =1}^N \Big(\big\|\delta Z^{i,\ell,N}_t\big\|^2 + \big\|\delta Z^{i,m,\ell,\star,N}_t\big\|^2\Big)\mathrm{d}t\Bigg]\\
&\quad + \eps_1 \bigg(2+\frac1{\eps_3}\bigg) 48\ell^2_\Lambda R_N^2\E^{\P^{\smalltext{\hat\balpha}^\tinytext{N}\smalltext{,}\smalltext{N}\smalltext{,}\smalltext{u}}_\smalltext{\omega}} \Bigg[\int_u^T\mathrm{e}^{\beta t} \sum_{(k,\ell)\in\{1,\dots,N\}^\smalltext{2}} \Big(\big\|\delta Z^{k,\ell,N}_t\|^2 + \big\|\delta Z^{k,m,\ell,\star,N}_t\big\|^2\Big)\mathrm{d}t\Bigg], \; \P\text{\rm--a.e.} \; \omega\in\Omega.
\end{align}

Combining \Cref{align:deltaYtoSum} with the bounds derived in \eqref{align:deltaM_finalEst} and \eqref{align:deltaN_finalEst}, we conclude that, for $\P\text{\rm--a.e.} \; \omega\in\Omega$
\begin{align}\label{align:deltaEverything}
\notag&\bigg( 1 - \varepsilon_3 4c^2_{1,\smallertext{\rm BDG}} - \bigg(\varepsilon_2 +\eps_4+\frac{\eps_2}{\eps_3}\bigg) c_{\smallertext{\rm{BMO}}_{\smalltext{[}\smalltext{u}\smalltext{,}\smalltext{T}\smalltext{]}}} \bigg) \E^{\P^{\smalltext{\hat\balpha}^\tinytext{N}\smalltext{,}\smalltext{N}\smalltext{,}\smalltext{u}}_\smalltext{\omega}} \bigg[ \sup_{t \in [u,T]} \mathrm{e}^{\beta t} \big|\delta Y^{i,N}_t\big|^2 \bigg] +\E^{\P^{\smalltext{\hat\balpha}^\tinytext{N}\smalltext{,}\smalltext{N}\smalltext{,}\smalltext{u}}_\smalltext{\omega}} \Bigg[ \int_u^T \mathrm{e}^{\beta t} \sum_{\ell =1}^N \big\|\delta Z^{i,\ell,N}_t\big\|^2 \d t \Bigg] \\
\notag&\quad +\E^{\P^{\smalltext{\hat\balpha}^\tinytext{N}\smalltext{,}\smalltext{N}\smalltext{,}\smalltext{u}}_\smalltext{\omega}} \Bigg[ \sup_{t \in [u,T]} \mathrm{e}^{\beta t} \big|\delta M^{i,\star,N}_t\big|^2 + \sup_{t \in [u,T]} \mathrm{e}^{\beta t} \big|\delta N^{\star,N}_t\big|^2 + \int_u^T \mathrm{e}^{\beta t} \sum_{\ell =1}^N \Big( \big\|Z^{i,m,\ell,\star,N}_t\big\|^2 + \big\|\delta Z^{n,\ell,\star,N}_t\big\| \Big) \d t \Bigg] \\
	\notag&\leq \big(c_{\eps_{\smalltext{2}\smalltext{,}\smalltext{3}\smalltext{,}\smalltext{4}}}+\ell^2_{\varphi_{\smalltext{1}}} c^\star\big) \E^{\P^{\smalltext{\hat\balpha}^\tinytext{N}\smalltext{,}\smalltext{N}\smalltext{,}\smalltext{u}}_\smalltext{\omega}} \Big[ \mathrm{e}^{\beta T} \big\|\delta X^{i}_{\cdot \land T}\big\|^2_\infty \Big] + \frac{\big(\overline{c}_{\eps_{\smalltext{2}\smalltext{,}\smalltext{3}\smalltext{,}\smalltext{4}}}+\ell^2_{\varphi_{\smalltext{2}}} c^\star\big)}N \E^{\P^{\smalltext{\hat\balpha}^\tinytext{N}\smalltext{,}\smalltext{N}\smalltext{,}\smalltext{u}}_\smalltext{\omega}} \Bigg[\mathrm{e}^{\beta T} \sum_{\ell =1}^N \big\|\delta X^{\ell}_{\cdot \land T}\big\|^2_\infty\Bigg] \\
\notag&\quad+ \eps_1 \bigg(2+\frac1{\eps_3}\bigg) 48 \ell^2_\Lambda R_N^2 \Bigg(\frac{2}{\beta} \mathrm{e}^{\beta T}+ \E^{\P^{\smalltext{\hat\balpha}^\tinytext{N}\smalltext{,}\smalltext{N}\smalltext{,}\smalltext{u}}_\smalltext{\omega}} \bigg[\int_u^T\mathrm{e}^{\beta t}\big\|X^i_{\cdot \land t}\big\|_\infty^2\mathrm{d}t\bigg]+\frac1N \E^{\P^{\smalltext{\hat\balpha}^\tinytext{N}\smalltext{,}\smalltext{N}\smalltext{,}\smalltext{u}}_\smalltext{\omega}} \Bigg[\int_u^T\mathrm{e}^{\beta t} \sum_{\ell =1}^N \big\|X^\ell_{\cdot \land t}\big\|_\infty^2\mathrm{d}t\Bigg]\Bigg)\\
\notag&\quad+ \varepsilon_1 \bigg(2+\frac1{\eps_3}\bigg) (1 + 6 \ell^2_\Lambda) \E^{\P^{\smalltext{\hat\balpha}^\tinytext{N}\smalltext{,}\smalltext{N}\smalltext{,}\smalltext{u}}_\smalltext{\omega}} \bigg[\int_u^T \mathrm{e}^{\beta t} \big\|\delta X^{i}_{\cdot \land t}\big\|^2_\infty \d t \bigg] + \varepsilon_1 \bigg(2+\frac1{\eps_3}\bigg) \frac{(1 + 18 \ell^2_\Lambda)}{N} \E^{\P^{\smalltext{\hat\balpha}^\tinytext{N}\smalltext{,}\smalltext{N}\smalltext{,}\smalltext{u}}_\smalltext{\omega}} \Bigg[ \int_u^T \mathrm{e}^{\beta t} \sum_{\ell =1}^N \big\|\delta X^{\ell}_{\cdot \land t} \big\|^2_{\infty} \d t \Bigg] \\
\notag&\quad + \eps_1 \bigg(2+\frac1{\eps_3}\bigg) 48\ell^2_\Lambda NR_N^2\E^{\P^{\smalltext{\hat\balpha}^\tinytext{N}\smalltext{,}\smalltext{N}\smalltext{,}\smalltext{u}}_\smalltext{\omega}} \Bigg[\int_u^T\mathrm{e}^{\beta t} \sum_{\ell =1}^N \Big(\big\|\widetilde Z^{i,\ell,N}_t\big\|^2+\big\|\widetilde Z^{i,m,\ell,\star,N}_t\big\|^2+\big\|\delta Z^{i,\ell,N}_t\big\|^2+\big\|\delta Z^{i,m,\ell,\star,N}_t\big\|^2\Big)\mathrm{d}t\Bigg]\\
\notag&\quad+ \varepsilon_1 \bigg(2+\frac1{\eps_3}\bigg) 6 \ell^2_\Lambda \E^{\P^{\smalltext{\hat\balpha}^\tinytext{N}\smalltext{,}\smalltext{N}\smalltext{,}\smalltext{u}}_\smalltext{\omega}} \bigg[ \int_u^T \mathrm{e}^{\beta t} \Big( \big\|\delta Z^{i,i,N}_t\big\|^2 + \big\|\delta Z^{i,m,i,\star,N}_t\big\|^2 + \big\|\delta Z^{n,i,\star,N}_t\big\|^2  \Big) \d t \bigg] \\
\notag&\quad+ \varepsilon_1 \bigg(2+\frac1{\eps_3}\bigg) \frac{6 \ell^2_\Lambda}{N} \E^{\P^{\smalltext{\hat\balpha}^\tinytext{N}\smalltext{,}\smalltext{N}\smalltext{,}\smalltext{u}}_\smalltext{\omega}} \Bigg[ \int_u^T \mathrm{e}^{\beta t} \sum_{\ell =1}^N \Big( \big\|\delta Z^{\ell,\ell,N}_t\big\|^2 + \big\|\delta Z^{\ell,m,\ell,\star,N}_t\big\|^2 + \big\|\delta Z^{n,\ell,\star,N}_t\big\|^2 \Big) \d t \Bigg]\\
\notag&\quad + \eps_1 \bigg(2+\frac1{\eps_3}\bigg) 96\ell^2_\Lambda NR_N^2\E^{\P^{\smalltext{\hat\balpha}^\tinytext{N}\smalltext{,}\smalltext{N}\smalltext{,}\smalltext{u}}_\smalltext{\omega}} \Bigg[\int_u^T\mathrm{e}^{\beta t} \sum_{\ell =1}^N \Big( \big\|\widetilde Z^{n,\ell,\star,N}_t\big\|^2 + \big\|\delta Z^{n,\ell,\star,N}_t\big\|^2\Big)\mathrm{d}t\Bigg]\\
&\quad + \eps_1 \bigg(2+\frac1{\eps_3}\bigg) 48\ell^2_\Lambda R_N^2\E^{\P^{\smalltext{\hat\balpha}^\tinytext{N}\smalltext{,}\smalltext{N}\smalltext{,}\smalltext{u}}_\smalltext{\omega}} \Bigg[\int_u^T\mathrm{e}^{\beta t}\sum_{(k,\ell)\in\{1,\dots,N\}^\smalltext{2}} \Big(\big\|\widetilde Z^{k,\ell,N}_t\big\|^2+ \big\|\widetilde Z^{k,m,\ell,\star,N}_t\big\|^2+\big\|\delta Z^{k,\ell,N}_t\big\|^2 + \big\|\delta Z^{k,m,\ell,\star,N}_t\big\|^2\Big)\mathrm{d}t\Bigg].
\end{align}

\textbf{Step 3: estimates for the forward component}

\medskip
An application of It\^o's formula to $\mathrm{e}^{\beta t} \|\delta X^{i}_t\|^2$, for $t \in [u,T]$, yields
\begin{align*}
&\mathrm{e}^{\beta t} \|\delta X^i_t\|^2 \\
&= \int_u^t \beta \mathrm{e}^{\beta s} \|\delta X^i_s\|^2 \d s \\
&\quad+ 2 \int_u^t \mathrm{e}^{\beta s} \delta X^i_s \cdot \Big( \sigma_s(X^i_{\cdot \land s}) b_s\big(X^i_{\cdot \land s},L^N\big(\X^N_{\cdot \land s},\hat\balpha^N_s\big),\hat\alpha^{i,N}_s\big) - \sigma_s(\widetilde X^{i}_{\cdot \land s}) b_s\big(\widetilde X^i_{\cdot \land s},L^N\big(\widetilde \X^N_{\cdot \land s},\widetilde\balpha^N_s\big),\widetilde\alpha^{i,N}_s\big) \Big) \d s \\
&\quad+ 2 \int_u^t \mathrm{e}^{\beta s} \delta X^i_s \cdot  \big( \sigma_s(X^i_{\cdot \land s}) - \sigma_s(\widetilde X^i_{\cdot \land s}) \big) \d \big(W_s^{\hat\balpha^\smalltext{N},N,u,\omega}\big)^i  \\
&\quad+ \int_u^t \mathrm{e}^{\beta s} \mathrm{Tr} \Big[ \big(\sigma_s(X^{i}_{\cdot \land s}) - \sigma_s(\widetilde X^{i}_{\cdot \land s}) \big) \big( \sigma_s(X^{i}_{\cdot \land s}) - \sigma_s(\widetilde X^{i}_{\cdot \land s}) \big)^\top \Big] \d s \\
	&= \int_u^t \beta \mathrm{e}^{\beta s} \|\delta X^{i}_s\|^2 \d s \\
&\quad+ 2 \int_u^t \mathrm{e}^{\beta s} \delta X^{i}_s \cdot \Big( \sigma_s(X^{i}_{\cdot \land s}) b_s\big(X^{i}_{\cdot \land s},L^N\big(\X^N_{\cdot \land s},\hat\balpha^N_s\big),\hat\alpha^{i,N}_s\big) - \sigma_s(\widetilde X^{i}_{\cdot \land s}) b_s\big(\widetilde X^{i}_{\cdot \land s},L^N\big(\X^N_{\cdot \land s},\hat\balpha^N_s\big),\hat\alpha^{i,N}_s\big) \Big) \d s \\
&\quad+ 2 \int_u^t \mathrm{e}^{\beta s} \delta X^{i}_s \cdot \Big( \sigma_s(\widetilde X^{i}_{\cdot \land s}) b_s\big(\widetilde X^{i}_{\cdot \land s},L^N\big(\X^N_{\cdot \land s},\hat\balpha^N_s\big),\hat\alpha^{i,N}_s\big) - \sigma_s(\widetilde X^{i}_{\cdot \land s}) b_s\big(\widetilde X^{i}_{\cdot \land s},L^N\big(\widetilde \X^N_{\cdot \land s},\widetilde\balpha^N_s\big),\widetilde\alpha^{i,N}_s\big)\big) \Big) \d s \\
&\quad+ 2 \int_u^t \mathrm{e}^{\beta s} \delta X^{i}_s \cdot \big( \sigma_s(X^{i}_{\cdot \land s}) - \sigma_s(\widetilde X^{i}_{\cdot \land s}) \big) \d \big(W_s^{\hat\balpha^{\smalltext{N}},N,u,\omega}\big)^i  \\
&\quad+ \int_u^t \mathrm{e}^{\beta s} \mathrm{Tr} \Big[ \big(\sigma_s(X^{i}_{\cdot \land s}) - \sigma_s(\widetilde X^{i}_{\cdot \land s}) \big) \big( \sigma_s(X^{i}_{\cdot \land s}) - \sigma_s(\widetilde X^{i}_{\cdot \land s})\big)^\top\Big]\mathrm{d}s \\
	&\leq \bigg( \beta - 2 K_{\sigma b} + \ell^2_\sigma + \frac{2 \ell^2_{\sigma b}}{\varepsilon_5}\bigg) \int_u^t \mathrm{e}^{\beta s} \big\|\delta X^{i}_{\cdot \land s}\big\|^2_\infty \d s + \varepsilon_5 \int_u^t \mathrm{e}^{\beta s} \Big( \cW^2_2\big(L^N\big(\X^N_{\cdot \land s},\hat\balpha^N_s\big),L^N\big(\widetilde \X^N_{\cdot \land s},\widetilde\balpha^N_s\big)\big) + d^2_A\big(\hat\alpha^{i,N}_s,\widetilde\alpha^{i,N}_s\big) \Big) \d s \\
&\quad+ 2 \int_u^t \mathrm{e}^{\beta s} \delta X^{i}_s \cdot  \big( \sigma_s(X^{i}_{\cdot \land s}) - \sigma_s(\widetilde X^{i}_{\cdot \land s}) \big) \d \big(W_s^{\hat\balpha^{\smalltext{N}},N,u,\omega}\big)^i  \\
	&\leq \bigg( \beta - 2 K_{\sigma b} + \ell^2_\sigma + \frac{2\ell^2_{\sigma b}}{\varepsilon_5} \bigg) \int_u^t \mathrm{e}^{\beta s} \big\|\delta X^{i}_{\cdot \land s}\big\|^2_\infty \d s + \frac{\varepsilon_5}{N} \int_u^t \mathrm{e}^{\beta s} \sum_{\ell =1}^N \big\|\delta X^{\ell}_{\cdot \land s}\big\|^2_\infty \d s\\
	&\quad+ \varepsilon_5 6 \ell_\Lambda^2 \int_u^t \mathrm{e}^{\beta s} \Bigg( \big\|\delta X^{i}_{\cdot \land s}\big\|_\infty^2 + \frac{1}{N} \sum_{\ell =1}^N \big\|\delta X^{\ell}_{\cdot \land s}\big\|_\infty^2 + \big\|\delta Z^{i,i,N}_s\big\|^2 + \big\|\delta Z^{i,m,i,\star,N}_s\big\|^2 + \big\|\delta Z^{n,i,\star,N}_s\big\|^2 + \big|\aleph^{i,N}_s\big|^2 \Bigg) \d s \\
	&\quad+ \varepsilon_5 \frac{6 \ell_\Lambda^2}{N} \int_u^t \mathrm{e}^{\beta s} \sum_{\ell =1}^N \Big( 2 \big\|\delta X^{\ell}_{\cdot \land s}\big\|_\infty^2 + \big\|\delta Z^{\ell,\ell,N}_s\big\|^2 + \big\|\delta Z^{\ell,m,\ell,\star,N}_s\big\|^2 + \big\|\delta Z^{n,\ell,\star,N}_s\big\|^2 + \big|\aleph^{\ell,N}_s\big|^2 \Big) \d s \\
	&\quad+ 2 \int_u^t \mathrm{e}^{\beta s} \delta X^{i}_s \cdot  \big( \sigma_s(X^{i}_{\cdot \land s}) - \sigma_s(\widetilde X^{i}_{\cdot \land s}) \big) \d \big(W_s^{\hat\balpha^{\smalltext{N}},N,u,\omega}\big)^i \\
		&= \bigg( \beta - 2 K_{\sigma b} + \ell^2_\sigma + \frac{2\ell^2_{\sigma b}}{\varepsilon_5} + \varepsilon_5 6 \ell_\Lambda^2 \bigg) \int_u^t \mathrm{e}^{\beta s} \big\|\delta X^{i}_{\cdot \land s}\big\|^2_\infty \d s + \varepsilon_5 \frac{( 1 + 18 \ell^2_\Lambda )}{N} \int_u^t \mathrm{e}^{\beta s} \sum_{\ell =1}^N \|\delta X^{\ell}_{\cdot \land s}\|^2_\infty \d s\\
&\quad+ \varepsilon_5 6 \ell_\Lambda^2 \int_u^t \mathrm{e}^{\beta s} \Big( \big\|\delta Z^{i,i,N}_s\big\|^2 + \big\|\delta Z^{i,m,i,\star,N}_s\big\|^2 + \big\|\delta Z^{n,i,\star,N}_s\big\|^2 + \big|\aleph^{i,N}_s\big|^2 \Big) \d s \\
&\quad+ \varepsilon_5 \frac{6 \ell_\Lambda^2}{N} \int_u^t \mathrm{e}^{\beta s} \sum_{\ell=1}^N \Big( \big\|\delta Z^{\ell,\ell,N}_s\big\|^2 + \big\|\delta Z^{\ell,m,\ell,\star,N}_s\big\|^2 + \big\|\delta Z^{n,\ell,\star,N}_s\big\|^2 + \big|\aleph^{\ell,N}_s\big|^2 \Big) \d s \\
&\quad+ 2 \int_u^t \mathrm{e}^{\beta s} \delta X^{i}_s \cdot \big( \sigma_s(X^{i,N}_{\cdot \land s}) - \sigma_s(\widetilde X^{i}_{\cdot \land s}) \big) \d \big(W_s^{\hat\balpha^\smalltext{N},N,u,\omega}\big)^i, \; t \in [u,T], \; \P^{\hat\balpha^\smalltext{N},N,u}_\omega \text{\rm--a.s.}, \; \text{for} \; \P\text{\rm--a.e.} \; \omega\in\Omega.
\end{align*}
The first inequality follows from the Lipschitz condition in \Cref{assumpConvThm}.\ref{lipSigma} and the dissipativity condition in \Cref{assumpConvThm}.\ref{diss}, together with an application of Young’s inequality for some $\eps_5>0$, while the second follows directly from the definition of the Wasserstein distance for empirical distributions (see, for instance, \citeauthor*{cardaliaguet2018short} \cite[Lemma 5.1.7]{cardaliaguet2018short}) and the Lipschitz-continuity of the function $\Lambda$ stated in \Cref{assumpConvThm}.\ref{lipLambda_growthAleph}. Moreover, the assumptions on $\aleph^{i,N}$---recall \Cref{eq:controlxhi}---imply that
\begin{align*}
\mathrm{e}^{\beta t} \|\delta X^{i}_t\|^2 
	&\leq \bigg( \beta - 2 K_{\sigma b} + \ell^2_\sigma + \frac{2\ell^2_{\sigma b}}{\varepsilon_5} + \varepsilon_5 6 \ell_\Lambda^2 \bigg) \int_u^t \mathrm{e}^{\beta s} \big\|\delta X^{i}_{\cdot \land s}\big\|^2_\infty \d s + \varepsilon_5 \frac{( 1 + 18 \ell^2_\Lambda )}{N} \int_u^t \mathrm{e}^{\beta s} \sum_{\ell =1}^N \big\|\delta X^{\ell}_{\cdot \land s}\big\|^2_\infty \d s \\
&\quad + \eps_5 48 \ell^2_\Lambda R_N^2 \Bigg( \frac{2}{\beta} \mathrm{e}^{\beta t} + \int_u^t \mathrm{e}^{\beta s}\big\|X^i_{\cdot \land s}\big\|_\infty^2 \d s +\frac1N \int_u^t \sum_{\ell =1}^N \mathrm{e}^{\beta s}\big\|X^\ell_{\cdot \land s}\big\|_\infty^2 \d s \Bigg) \\
&\quad + \eps_5 48 \ell^2_\Lambda NR_N^2 \int_u^t \mathrm{e}^{\beta s} \sum_{\ell =1}^N \Big( \big\|\widetilde Z^{i,\ell,N}_s\big\|^2 + \big\|\widetilde Z^{i,m,\ell,\star,N}_s\big\|^2 + 2 \big\|\widetilde Z^{n,\ell,\star,N}_s\big\|^2 \Big) \d s \\
&\quad+ \eps_5 48 \ell^2_\Lambda R_N^2 \int_u^t \mathrm{e}^{\beta s} \sum_{(k,\ell)\in\{1,\dots,N\}^\smalltext{2}} \Big(\big\|\widetilde Z^{k,\ell,N}_s\big\|^2 +\big\|\widetilde Z^{k,m,\ell,\star,N}_s\big\|^2 + \big\|\delta Z^{k,\ell,N}_s\big\|^2 + \big\|\delta Z^{k,m,\ell,\star,N}_s\big\|^2\Big) \d s \\
&\quad+ \varepsilon_5 6 \ell_\Lambda^2 \int_u^t \mathrm{e}^{\beta s} \Big( \big\|\delta Z^{i,i,N}_s\big\|^2 + \big\|\delta Z^{i,m,i,\star,N}_s\big\|^2 + \big\|\delta Z^{n,i,\star,N}_s\big\|^2 \Big) \d s \\
&\quad + \eps_5 48 \ell^2_\Lambda NR_N^2 \int_u^t \mathrm{e}^{\beta s} \sum_{\ell =1}^N \Big( \big\|\delta Z^{i,\ell,N}_s\big\|^2 + \big\|\delta Z^{i,m,\ell,\star,N}_s\big\|^2 + 2 \big\|\delta Z^{n,\ell,\star,N}_s\big\|^2 \Big) \d s \\
&\quad+ \varepsilon_5 \frac{6 \ell_\Lambda^2}{N} \int_u^t \mathrm{e}^{\beta s} \sum_{\ell =1}^N \Big( \big\|\delta Z^{\ell,\ell,N}_s\big\|^2 + \big\|\delta Z^{\ell,m,\ell,\star,N}_s\big\|^2 + \big\|\delta Z^{n,\ell,\star,N}_s\big\|^2 \Big) \d s \\
&\quad+ 2 \int_u^t \mathrm{e}^{\beta s} \delta X^{i}_s \cdot \big( \sigma_s(X^{i}_{\cdot \land s}) - \sigma_s(\widetilde X^{i}_{\cdot \land s}) \big) \d \big(W_s^{\hat\balpha^{\smalltext{N}},N,u,\omega}\big)^i, \; t \in [u,T], \; \P^{\hat\balpha^\smalltext{N},N,u}_\omega \text{\rm--a.s.}, \; \text{for} \; \P\text{\rm--a.e.} \; \omega\in\Omega.
\end{align*}

Assuming that
\[
K_{\sigma b}\geq \frac12\bigg(\beta+\ell^2_\sigma + \frac{2\ell^2_{\sigma b}}{\varepsilon_5} + \varepsilon_5 6 \ell_\Lambda^2\bigg),
\]
and taking the expectation, we deduce that for any $t\in[u,T]$,
\begin{align*}
&\E^{\P^{\smalltext{\hat\balpha}^\tinytext{N}\smalltext{,}\smalltext{N}\smalltext{,}\smalltext{u}}_\smalltext{\omega}} \bigg[ \mathrm{e}^{\beta t}\sup_{r \in [u,t]}  \|\delta X^{i}_r\|^2 \bigg] \\
	&\leq \varepsilon_5\mathrm{e}^{\beta t} \frac{( 1 + 18 \ell^2_\Lambda )}{N} \E^{\P^{\smalltext{\hat\balpha}^\tinytext{N}\smalltext{,}\smalltext{N}\smalltext{,}\smalltext{u}}_\smalltext{\omega}} \Bigg[ \int_u^t \mathrm{e}^{\beta s} \sum_{\ell =1}^N \big\|\delta X^{\ell}_{\cdot \land s}\big\|^2_\infty \d s \Bigg] \\
&\quad + \eps_5 \mathrm{e}^{\beta t}48 \ell^2_\Lambda R_N^2 \Bigg( \frac{2}{\beta} \mathrm{e}^{\beta t} + \E^{\P^{\smalltext{\hat\balpha}^\tinytext{N}\smalltext{,}\smalltext{N}\smalltext{,}\smalltext{u}}_\smalltext{\omega}} \bigg[ \int_u^t\mathrm{e}^{\beta s} \big\|X^i_{\cdot \land s}\big\|_\infty^2 \d s \bigg] +\frac1N \E^{\P^{\smalltext{\hat\balpha}^\tinytext{N}\smalltext{,}\smalltext{N}\smalltext{,}\smalltext{u}}_\smalltext{\omega}} \Bigg[ \int_u^t\mathrm{e}^{\beta s} \sum_{\ell =1}^N \big\|X^\ell_{\cdot \land s}\big\|_\infty^2 \d s \Bigg] \Bigg) \\
&\quad + \eps_5 \mathrm{e}^{\beta t}48 \ell^2_\Lambda NR_N^2 \E^{\P^{\smalltext{\hat\balpha}^\tinytext{N}\smalltext{,}\smalltext{N}\smalltext{,}\smalltext{u}}_\smalltext{\omega}} \Bigg[ \int_u^t \mathrm{e}^{\beta s} \sum_{\ell =1}^N \Big( \big\|\widetilde Z^{i,\ell,N}_s\big\|^2 + \big\|\widetilde Z^{i,m,\ell,\star,N}_s\big\|^2 + 2 \big\|\widetilde Z^{n,\ell,\star,N}_s\big\|^2 \Big) \d s \Bigg] \\
&\quad+ \eps_5 \mathrm{e}^{\beta t}48 \ell^2_\Lambda R_N^2 \E^{\P^{\smalltext{\hat\balpha}^\tinytext{N}\smalltext{,}\smalltext{N}\smalltext{,}\smalltext{u}}_\smalltext{\omega}} \Bigg[ \int_u^t \mathrm{e}^{\beta s} \sum_{(k,\ell)\in\{1,\dots,N\}^\smalltext{2}} \Big(\big\|\widetilde Z^{k,\ell,N}_s\big\|^2 +\big\|\widetilde Z^{k,m,\ell,\star,N}_s\big\|^2 + \big\|\delta Z^{k,\ell,N}_s\big\|^2 + \big\|\delta Z^{k,m,\ell,\star,N}_s\big\|^2\Big) \d s \Bigg] \\
&\quad+ \varepsilon_5 \mathrm{e}^{\beta t}6 \ell_\Lambda^2 \E^{\P^{\smalltext{\hat\balpha}^\tinytext{N}\smalltext{,}\smalltext{N}\smalltext{,}\smalltext{u}}_\smalltext{\omega}} \bigg[ \int_u^t \mathrm{e}^{\beta s} \Big( \big\|\delta Z^{i,i,N}_s\big\|^2 + \big\|\delta Z^{i,m,i,\star,N}_s\big\|^2 + \big\|\delta Z^{n,i,\star,N}_s\big\|^2 \Big) \d s \bigg] \\
&\quad + \eps_5 \mathrm{e}^{\beta t}48 \ell^2_\Lambda NR_N^2 \E^{\P^{\smalltext{\hat\balpha}^\tinytext{N}\smalltext{,}\smalltext{N}\smalltext{,}\smalltext{u}}_\smalltext{\omega}} \Bigg[ \int_u^t \mathrm{e}^{\beta s} \sum_{\ell =1}^N \Big( \big\|\delta Z^{i,\ell,N}_s\big\|^2 + \big\|\delta Z^{i,m,\ell,\star,N}_s\big\|^2 + 2 \big\|\delta Z^{n,\ell,\star,N}_s\big\|^2 \Big) \d s \bigg] \\
&\quad+ \varepsilon_5 \mathrm{e}^{\beta t}\frac{6 \ell_\Lambda^2}{N} \E^{\P^{\smalltext{\hat\balpha}^\tinytext{N}\smalltext{,}\smalltext{N}\smalltext{,}\smalltext{u}}_\smalltext{\omega}} \Bigg[ \int_u^t \mathrm{e}^{\beta s} \sum_{\ell =1}^N \Big( \big\|\delta Z^{\ell,\ell,N}_s\big\|^2 + \big\|\delta Z^{\ell,m,\ell,\star,N}_s\big\|^2 + \big\|\delta Z^{n,\ell,\star,N}_s\big\|^2 \Big) \d s \Bigg] \\
&\quad+ 2\mathrm{e}^{\beta t} \E^{\P^{\smalltext{\hat\balpha}^\tinytext{N}\smalltext{,}\smalltext{N}\smalltext{,}\smalltext{u}}_\smalltext{\omega}} \bigg[ \sup_{r \in [u,t]} \bigg|\int_u^r \mathrm{e}^{\beta s} \delta X^{i}_s \cdot \big( \sigma_s(X^{i}_{\cdot \land s}) - \sigma_s(\widetilde X^{i}_{\cdot \land s}) \big) \d \big(W_s^{\hat\balpha^{\smalltext{N}},N,u,\omega}\big)^i \bigg| \bigg] \\
	&\leq \varepsilon_5\mathrm{e}^{\beta t} \frac{( 1 + 18 \ell^2_\Lambda )}{N} \E^{\P^{\smalltext{\hat\balpha}^\tinytext{N}\smalltext{,}\smalltext{N}\smalltext{,}\smalltext{u}}_\smalltext{\omega}} \Bigg[ \int_u^t \mathrm{e}^{\beta s} \sum_{\ell =1}^N \big\|\delta X^{\ell}_{\cdot \land s}\big\|^2_\infty \d s \Bigg] \\
&\quad + \eps_5 \mathrm{e}^{\beta t}48 \ell^2_\Lambda R_N^2 \Bigg( \frac{2}{\beta} \mathrm{e}^{\beta t} + \E^{\P^{\smalltext{\hat\balpha}^\tinytext{N}\smalltext{,}\smalltext{N}\smalltext{,}\smalltext{u}}_\smalltext{\omega}} \bigg[ \int_u^t\mathrm{e}^{\beta s} \big\|X^i_{\cdot \land s}\big\|_\infty^2 \d s \bigg] +\frac1N \E^{\P^{\smalltext{\hat\balpha}^\tinytext{N}\smalltext{,}\smalltext{N}\smalltext{,}\smalltext{u}}_\smalltext{\omega}} \Bigg[ \int_u^t \mathrm{e}^{\beta s}\sum_{\ell =1}^N \big\|X^\ell_{\cdot \land s}\big\|_\infty^2 \d s \Bigg] \Bigg) \\
&\quad + \eps_5 \mathrm{e}^{\beta t}48 \ell^2_\Lambda NR_N^2 \E^{\P^{\smalltext{\hat\balpha}^\tinytext{N}\smalltext{,}\smalltext{N}\smalltext{,}\smalltext{u}}_\smalltext{\omega}} \Bigg[ \int_u^t \mathrm{e}^{\beta s} \sum_{\ell =1}^N \Big( \big\|\widetilde Z^{i,\ell,N}_s\big\|^2 + \big\|\widetilde Z^{i,m,\ell,\star,N}_s\big\|^2 + 2 \big\|\widetilde Z^{n,\ell,\star,N}_s\big\|^2 \Big) \d s \Bigg] \\
&\quad+ \eps_5 \mathrm{e}^{\beta t}48 \ell^2_\Lambda R_N^2 \E^{\P^{\smalltext{\hat\balpha}^\tinytext{N}\smalltext{,}\smalltext{N}\smalltext{,}\smalltext{u}}_\smalltext{\omega}} \Bigg[ \int_u^t \mathrm{e}^{\beta s} \sum_{(k,\ell)\in\{1,\dots,N\}^\smalltext{2}} \Big(\big\|\widetilde Z^{k,\ell,N}_s\big\|^2 +\big\|\widetilde Z^{k,m,\ell,\star,N}_s\big\|^2 + \big\|\delta Z^{k,\ell,N}_s\big\|^2 + \big\|\delta Z^{k,m,\ell,\star,N}_s\big\|^2\Big) \d s \Bigg] \\
&\quad+ \varepsilon_5 \mathrm{e}^{\beta t}6 \ell_\Lambda^2 \E^{\P^{\smalltext{\hat\balpha}^\tinytext{N}\smalltext{,}\smalltext{N}\smalltext{,}\smalltext{u}}_\smalltext{\omega}} \bigg[ \int_u^t \mathrm{e}^{\beta s} \Big( \big\|\delta Z^{i,i,N}_s\big\|^2 + \big\|\delta Z^{i,m,i,\star,N}_s\big\|^2 + \big\|\delta Z^{n,i,\star,N}_s\big\|^2 \Big) \d s \bigg] \\
&\quad + \eps_5 \mathrm{e}^{\beta t}48 \ell^2_\Lambda NR_N^2 \E^{\P^{\smalltext{\hat\balpha}^\tinytext{N}\smalltext{,}\smalltext{N}\smalltext{,}\smalltext{u}}_\smalltext{\omega}} \Bigg[ \int_u^t \mathrm{e}^{\beta s} \sum_{\ell =1}^N \Big( \big\|\delta Z^{i,\ell,N}_s\big\|^2 + \big\|\delta Z^{i,m,\ell,\star,N}_s\big\|^2 + 2 \big\|\delta Z^{n,\ell,\star,N}_s\big\|^2 \Big) \d s \Bigg] \\
&\quad+ \varepsilon_5\mathrm{e}^{\beta t} \frac{6 \ell_\Lambda^2}{N} \E^{\P^{\smalltext{\hat\balpha}^\tinytext{N}\smalltext{,}\smalltext{N}\smalltext{,}\smalltext{u}}_\smalltext{\omega}} \Bigg[ \int_u^t \mathrm{e}^{\beta s} \sum_{\ell =1}^N \Big( \big\|\delta Z^{\ell,\ell,N}_s\big\|^2 + \big\|\delta Z^{\ell,m,\ell,\star,N}_s\big\|^2 + \big\|\delta Z^{n,\ell,\star,N}_s\big\|^2 \Big) \d s \Bigg] \\
&\quad+ \varepsilon_6 \mathrm{e}^{2\beta t}c^2_{1,{\smallertext{\rm BDG}}} \ell^2_\sigma \E^{\P^{\smalltext{\hat\balpha}^\tinytext{N}\smalltext{,}\smalltext{N}\smalltext{,}\smalltext{u}}_\smalltext{\omega}} \bigg[ \sup_{r \in [u,t]} \mathrm{e}^{\beta r} \|\delta X^{i}_r\|^2 \bigg] + \frac{1}{\varepsilon_6} \E^{\P^{\smalltext{\hat\balpha}^\tinytext{N}\smalltext{,}\smalltext{N}\smalltext{,}\smalltext{u}}_\smalltext{\omega}} \bigg[ \int_u^t \mathrm{e}^{\beta s} \big\|\delta X^{i}_{\cdot \land s}\big\|^2_\infty \d s \bigg], \; \P\text{\rm--a.e.} \; \omega\in\Omega,
\end{align*}
where the last inequality follows from Burkholder--Davis--Gundy's inequality, combined with the Lipschitz condition stated in \Cref{assumpConvThm}.\ref{lipSigma}, and Young’s inequality for some $\varepsilon_6 > 0$. Consequently, for any $\eps_6 \in (0,1/(c^2_{1,{\smallertext{\rm BDG}}} \mathrm{e}^{2\beta t}\ell^2_\sigma))$, for $\P\text{\rm--a.e.} \; \omega\in\Omega$,
\begin{align*}
 &(1-\varepsilon_6 \mathrm{e}^{2\beta t}c^2_{1,{\smallertext{\rm BDG}}} \ell^2_\sigma) \E^{\P^{\smalltext{\hat\balpha}^\tinytext{N}\smalltext{,}\smalltext{N}\smalltext{,}\smalltext{u}}_\smalltext{\omega}} \bigg[ \mathrm{e}^{\beta t} \big\|\delta X^{i}_{\cdot \land t}\big\|^2 \bigg] \\
	&\leq \varepsilon_5\mathrm{e}^{\beta t} \frac{( 1 + 18 \ell^2_\Lambda )}{N} \E^{\P^{\smalltext{\hat\balpha}^\tinytext{N}\smalltext{,}\smalltext{N}\smalltext{,}\smalltext{u}}_\smalltext{\omega}} \Bigg[ \int_u^t \mathrm{e}^{\beta s} \sum_{\ell =1}^N \big\|\delta X^{\ell}_{\cdot \land s}\big\|^2_\infty \d s \Bigg] \\
&\quad + \eps_5 \mathrm{e}^{\beta t}48 \ell^2_\Lambda R_N^2 \Bigg( \frac{2}{\beta} \mathrm{e}^{\beta t} + \E^{\P^{\smalltext{\hat\balpha}^\tinytext{N}\smalltext{,}\smalltext{N}\smalltext{,}\smalltext{u}}_\smalltext{\omega}} \bigg[ \int_u^t \mathrm{e}^{\beta s} \big\|X^i_{\cdot \land s}\big\|_\infty^2 \d s \bigg] +\frac1N \E^{\P^{\smalltext{\hat\balpha}^\tinytext{N}\smalltext{,}\smalltext{N}\smalltext{,}\smalltext{u}}_\smalltext{\omega}} \Bigg[ \int_u^t \mathrm{e}^{\beta s} \sum_{\ell =1}^N \big\|X^\ell_{\cdot \land s}\big\|_\infty^2 \d s \Bigg] \Bigg) \\
&\quad + \eps_5 \mathrm{e}^{\beta t}48 \ell^2_\Lambda NR_N^2 \E^{\P^{\smalltext{\hat\balpha}^\tinytext{N}\smalltext{,}\smalltext{N}\smalltext{,}\smalltext{u}}_\smalltext{\omega}} \Bigg[ \int_u^t \mathrm{e}^{\beta s} \sum_{\ell =1}^N \Big( \big\|\widetilde Z^{i,\ell,N}_s\big\|^2 + \big\|\widetilde Z^{i,m,\ell,\star,N}_s\big\|^2 + 2 \big\|\widetilde Z^{n,\ell,\star,N}_s\big\|^2 \Big) \d s \Bigg] \\
&\quad+ \eps_5 \mathrm{e}^{\beta t}48 \ell^2_\Lambda R_N^2 \E^{\P^{\smalltext{\hat\balpha}^\tinytext{N}\smalltext{,}\smalltext{N}\smalltext{,}\smalltext{u}}_\smalltext{\omega}} \Bigg[ \int_u^t \mathrm{e}^{\beta s} \sum_{(k,\ell)\in\{1,\dots,N\}^\smalltext{2}} \Big(\big\|\widetilde Z^{k,\ell,N}_s\big\|^2 +\big\|\widetilde Z^{k,m,\ell,\star,N}_s\big\|^2 + \big\|\delta Z^{k,\ell,N}_s\big\|^2 + \big\|\delta Z^{k,m,\ell,\star,N}_s\big\|^2\Big) \d s \Bigg] \\
&\quad+ \varepsilon_5 \mathrm{e}^{\beta t}6 \ell_\Lambda^2 \E^{\P^{\smalltext{\hat\balpha}^\tinytext{N}\smalltext{,}\smalltext{N}\smalltext{,}\smalltext{u}}_\smalltext{\omega}} \bigg[ \int_u^t \mathrm{e}^{\beta s} \Big( \big\|\delta Z^{i,i,N}_s\big\|^2 + \big\|\delta Z^{i,m,i,\star,N}_s\big\|^2 + \big\|\delta Z^{n,i,\star,N}_s\big\|^2 \Big) \d s \bigg] \\
&\quad + \eps_5 \mathrm{e}^{\beta t}48 \ell^2_\Lambda NR_N^2 \E^{\P^{\smalltext{\hat\balpha}^\tinytext{N}\smalltext{,}\smalltext{N}\smalltext{,}\smalltext{u}}_\smalltext{\omega}} \Bigg[ \int_u^t \mathrm{e}^{\beta s} \sum_{\ell =1}^N \Big( \big\|\delta Z^{i,\ell,N}_s\big\|^2 + \big\|\delta Z^{i,m,\ell,\star,N}_s\big\|^2 + 2 \big\|\delta Z^{n,\ell,\star,N}_s\big\|^2 \Big) \d s \Bigg] \\
&\quad+ \varepsilon_5 \mathrm{e}^{\beta t}\frac{6 \ell_\Lambda^2}{N} \E^{\P^{\smalltext{\hat\balpha}^\tinytext{N}\smalltext{,}\smalltext{N}\smalltext{,}\smalltext{u}}_\smalltext{\omega}} \Bigg[ \int_u^t \mathrm{e}^{\beta s} \sum_{\ell =1}^N \Big( \big\|\delta Z^{\ell,\ell,N}_s\big\|^2 + \big\|\delta Z^{\ell,m,\ell,\star,N}_s\big\|^2 + \big\|\delta Z^{n,\ell,\star,N}_s\big\|^2 \Big) \d s \Bigg] + \frac{1}{\varepsilon_6} \E^{\P^{\smalltext{\hat\balpha}^\tinytext{N}\smalltext{,}\smalltext{N}\smalltext{,}\smalltext{u}}_\smalltext{\omega}} \bigg[ \int_u^t \mathrm{e}^{\beta s} \big\|\delta X^{i}_{\cdot \land s}\big\|^2_\infty \d s \bigg].
\end{align*}

If we define 
\begin{align*}
c_{\eps_{\smalltext{6}}}(t) \coloneqq \mathrm{exp}\bigg(\beta t+\frac{ t}{\varepsilon_6 (1-\varepsilon_6 \mathrm{e}^{2\beta t}c^2_{1,{\smallertext{\rm BDG}}} \ell^2_\sigma)}  \bigg)\big(1-\varepsilon_6 \mathrm{e}^{2\beta t}c^2_{1,{\smallertext{\rm BDG}}} \ell^2_\sigma\big)^{-1},
\end{align*}
Gr\"onwall's inequality implies that
\begin{align}\label{gronwallineq1}
\notag&\E^{\P^{\smalltext{\hat\balpha}^\tinytext{N}\smalltext{,}\smalltext{N}\smalltext{,}\smalltext{u}}_\smalltext{\omega}} \bigg[ \mathrm{e}^{\beta t} \big\|\delta X^{i}_{\cdot \land t}\big\|^2 \bigg] \\
	\notag&\leq \varepsilon_5 \frac{( 1 + 18 \ell^2_\Lambda ) c_{\eps_{\smalltext{6}}}(t)}{N} \E^{\P^{\smalltext{\hat\balpha}^\tinytext{N}\smalltext{,}\smalltext{N}\smalltext{,}\smalltext{u}}_\smalltext{\omega}} \Bigg[ \int_u^t \mathrm{e}^{\beta s} \sum_{\ell =1}^N \big\|\delta X^{\ell}_{\cdot \land s}\big\|^2_\infty \d s \Bigg] \\
\notag&\quad + \eps_5 48 \ell^2_\Lambda c_{\eps_{\smalltext{6}}}(t) R_N^2 \Bigg( \frac{2}{\beta} \mathrm{e}^{\beta t} + \E^{\P^{\smalltext{\hat\balpha}^\tinytext{N}\smalltext{,}\smalltext{N}\smalltext{,}\smalltext{u}}_\smalltext{\omega}} \bigg[ \int_u^t \mathrm{e}^{\beta s}\big\|X^i_{\cdot \land s}\big\|_\infty^2 \d s \bigg] +\frac1N \E^{\P^{\smalltext{\hat\balpha}^\tinytext{N}\smalltext{,}\smalltext{N}\smalltext{,}\smalltext{u}}_\smalltext{\omega}} \Bigg[ \int_u^t \mathrm{e}^{\beta s}\sum_{\ell =1}^N \big\|X^\ell_{\cdot \land s}\big\|_\infty^2 \d s \Bigg] \Bigg) \\
\notag&\quad + \eps_5 48 \ell^2_\Lambda c_{\eps_{\smalltext{6}}}(t) NR_N^2 \E^{\P^{\smalltext{\hat\balpha}^\tinytext{N}\smalltext{,}\smalltext{N}\smalltext{,}\smalltext{u}}_\smalltext{\omega}} \Bigg[ \int_u^t \mathrm{e}^{\beta s} \sum_{\ell =1}^N \Big( \big\|\widetilde Z^{i,\ell,N}_s\big\|^2 + \big\|\widetilde Z^{i,m,\ell,\star,N}_s\big\|^2 + 2 \big\|\widetilde Z^{n,\ell,\star,N}_s\big\|^2 \Big) \d s \bigg] \\
\notag&\quad+ \eps_5 48 \ell^2_\Lambda c_{\eps_{\smalltext{6}}}(t) R_N^2 \E^{\P^{\smalltext{\hat\balpha}^\tinytext{N}\smalltext{,}\smalltext{N}\smalltext{,}\smalltext{u}}_\smalltext{\omega}} \Bigg[ \int_u^t \mathrm{e}^{\beta s} \sum_{(k,\ell)\in\{1,\dots,N\}^\smalltext{2}} \Big(\big\|\widetilde Z^{k,\ell,N}_s\big\|^2 +\big\|\widetilde Z^{k,m,\ell,\star,N}_s\big\|^2 + \big\|\delta Z^{k,\ell,N}_s\big\|^2 + \big\|\delta Z^{k,m,\ell,\star,N}_s\big\|^2\Big) \d s \Bigg] \\
\notag&\quad+ \varepsilon_5 6 \ell_\Lambda^2 c_{\eps_{\smalltext{6}}}(t) \E^{\P^{\smalltext{\hat\balpha}^\tinytext{N}\smalltext{,}\smalltext{N}\smalltext{,}\smalltext{u}}_\smalltext{\omega}} \bigg[ \int_u^t \mathrm{e}^{\beta s} \Big( \big\|\delta Z^{i,i,N}_s\big\|^2 + \big\|\delta Z^{i,m,i,\star,N}_s\big\|^2 + \big\|\delta Z^{n,i,\star,N}_s\big\|^2 \Big) \d s \bigg] \\
\notag&\quad + \eps_5 48 \ell^2_\Lambda c_{\eps_{\smalltext{6}}}(t) NR_N^2 \E^{\P^{\smalltext{\hat\balpha}^\tinytext{N}\smalltext{,}\smalltext{N}\smalltext{,}\smalltext{u}}_\smalltext{\omega}} \Bigg[ \int_u^t \mathrm{e}^{\beta s} \sum_{\ell =1}^N \Big( \big\|\delta Z^{i,\ell,N}_s\big\|^2 + \big\|\delta Z^{i,m,\ell,\star,N}_s\big\|^2 + 2 \big\|\delta Z^{n,\ell,\star,N}_s\big\|^2 \Big) \d s \Bigg] \\
&\quad+ \varepsilon_5 \frac{6 \ell_\Lambda^2 c_{\eps_{\smalltext{6}}}(t)}{N} \E^{\P^{\smalltext{\hat\balpha}^\tinytext{N}\smalltext{,}\smalltext{N}\smalltext{,}\smalltext{u}}_\smalltext{\omega}} \Bigg[ \int_u^t \mathrm{e}^{\beta s} \sum_{\ell =1}^N \Big( \big\|\delta Z^{\ell,\ell,N}_s\big\|^2 + \big\|\delta Z^{\ell,m,\ell,\star,N}_s\big\|^2 + \big\|\delta Z^{n,\ell,\star,N}_s\big\|^2 \Big) \d s \Bigg], \; \P\text{\rm--a.e.} \; \omega\in\Omega.
\end{align}

Summing with respect to $i\in\{1,\dots,N\}$ yields
\begin{align*}
&\E^{\P^{\smalltext{\hat\balpha}^\tinytext{N}\smalltext{,}\smalltext{N}\smalltext{,}\smalltext{u}}_\smalltext{\omega}} \Bigg[\mathrm{e}^{\beta t} \sum_{\ell =1}^N \big\|\delta X^{\ell}_{\cdot \land t}\big\|^2_\infty \Bigg] \\
\notag&\leq  \varepsilon_5 \big( 1 + 18 \ell^2_\Lambda \big) c_{\eps_{\smalltext{6}}}(t) \E^{\P^{\smalltext{\hat\balpha}^\tinytext{N}\smalltext{,}\smalltext{N}\smalltext{,}\smalltext{u}}_\smalltext{\omega}} \Bigg[ \int_u^t \mathrm{e}^{\beta s} \sum_{\ell =1}^N \big\|\delta X^{\ell}_{\cdot \land s}\big\|^2_\infty \d s \Bigg]\\
&\quad+\eps_5 96 \ell^2_\Lambda c_{\eps_{\smalltext{6}}}(t) R_N^2 \Bigg( \frac{N}{\beta} \mathrm{e}^{\beta t} + \E^{\P^{\smalltext{\hat\balpha}^\tinytext{N}\smalltext{,}\smalltext{N}\smalltext{,}\smalltext{u}}_\smalltext{\omega}} \Bigg[ \int_u^t \mathrm{e}^{\beta s} \sum_{\ell =1}^N \big\|X^\ell_{\cdot \land s}\big\|_\infty^2 \d s \Bigg] \Bigg) \\
&\quad+ \eps_5 96 \ell^2_\Lambda c_{\eps_{\smalltext{6}}}(t) NR_N^2 \E^{\P^{\smalltext{\hat\balpha}^\tinytext{N}\smalltext{,}\smalltext{N}\smalltext{,}\smalltext{u}}_\smalltext{\omega}} \Bigg[ \int_u^t \mathrm{e}^{\beta s}\Bigg( \sum_{(k,\ell)\in\{1,\dots,N\}^\smalltext{2}} \Big(\big\|\widetilde Z^{k,\ell,N}_s\big\|^2 +\big\|\widetilde Z^{k,m,\ell,\star,N}_s\big\|^2\Big)+N\sum_{\ell =1}^N \big\|\widetilde Z^{n,\ell,\star,N}_s\big\|^2\Bigg) \d s \Bigg]\\
&\quad + \eps_5 96 \ell^2_\Lambda c_{\eps_{\smalltext{6}}}(t) NR_N^2 \E^{\P^{\smalltext{\hat\balpha}^\tinytext{N}\smalltext{,}\smalltext{N}\smalltext{,}\smalltext{u}}_\smalltext{\omega}} \Bigg[ \int_u^t \mathrm{e}^{\beta s}\Bigg( \sum_{(\ell,k)\in\{1,\dots,N\}^\smalltext{2}} \Big( \big|\delta Z^{k,\ell,N}_s\|^2 + \big|\delta Z^{k,m,\ell,\star,N}_s\big|^2\Big) + N \sum_{\ell =1}^N \big\|\delta Z^{n,\ell,\star,N}_s\big\|^2\Bigg) \d s \Bigg] \\
&\quad+ \varepsilon_5 12 \ell_\Lambda^2 c_{\eps_{\smalltext{6}}}(t) \E^{\P^{\smalltext{\hat\balpha}^\tinytext{N}\smalltext{,}\smalltext{N}\smalltext{,}\smalltext{u}}_\smalltext{\omega}} \Bigg[ \int_u^t \mathrm{e}^{\beta s} \sum_{\ell =1}^N \Big( \big\|\delta Z^{\ell,\ell,N}_s\big\|^2 + \big\|\delta Z^{\ell,m,\ell,\star,N}_s\big\|^2 + \big\|\delta Z^{n,\ell,\star,N}_s\big\|^2 \Big) \d s \Bigg], \; \P\text{\rm--a.e.} \; \omega\in\Omega.
\end{align*}
Applying Gr\"onwall's inequality once more, and introducing the constant
\[
 c_{\eps_{\smalltext{5}\smalltext{,}\smalltext{6}}}(t)\coloneqq \eps_5 12 \ell^2_\Lambda c_{\eps_{\smalltext{6}}}(t)\exp\big( \varepsilon_5 \big( 1 + 18 \ell^2_\Lambda \big) c_{\eps_{\smalltext{6}}}(T)T\big),
\]
we get
\begin{align}\label{align:GronwallSumDeltaX2}
&\notag\frac1N\E^{\P^{\smalltext{\hat\balpha}^\tinytext{N}\smalltext{,}\smalltext{N}\smalltext{,}\smalltext{u}}_\smalltext{\omega}} \Bigg[\mathrm{e}^{\beta t}\sum_{\ell =1}^N \big\|\delta X^{\ell}_{\cdot \land t}\big\|^2_\infty \Bigg] \\
	\notag&\leq c_{\eps_{\smalltext{5}\smalltext{,}\smalltext{6}}}(t) 8R_N^2 \Bigg( \frac{ \mathrm{e}^{\beta t}}{\beta} + \frac 1N\E^{\P^{\smalltext{\hat\balpha}^\tinytext{N}\smalltext{,}\smalltext{N}\smalltext{,}\smalltext{u}}_\smalltext{\omega}} \bigg[ \int_u^t \mathrm{e}^{\beta s} \big\|X^\ell_{\cdot \land s}\big\|_\infty^2 \d s \bigg] \Bigg) \\
\notag&\quad+ \frac{c_{\eps_{\smalltext{5}\smalltext{,}\smalltext{6}}}(t)}N \E^{\P^{\smalltext{\hat\balpha}^\tinytext{N}\smalltext{,}\smalltext{N}\smalltext{,}\smalltext{u}}_\smalltext{\omega}} \Bigg[ \int_u^t \mathrm{e}^{\beta s} \sum_{\ell =1}^N \Big( \big\|\delta Z^{\ell,\ell,N}_s\big\|^2 + \big\|\delta Z^{\ell,m,\ell,\star,N}_s\big\|^2 + \big\|\delta Z^{n,\ell,\star,N}_s\big\|^2 \Big) \d s \Bigg] \\
\notag&\quad + 8c_{\eps_{\smalltext{5}\smalltext{,}\smalltext{6}}}(t) R_N^2 \E^{\P^{\smalltext{\hat\balpha}^\tinytext{N}\smalltext{,}\smalltext{N}\smalltext{,}\smalltext{u}}_\smalltext{\omega}} \Bigg[ \int_u^t \mathrm{e}^{\beta s} \sum_{(\ell,k)\in\{1,\dots,N\}^\smalltext{2}} \Big( \big\|\widetilde Z^{\ell,k,N}_s\big\|^2 + \big\|\widetilde Z^{\ell,m,k,\star,N}_s\big\|^2 + \big\|\delta Z^{\ell,k,N}_s\big\|^2 + \big\|\delta Z^{\ell,m,k,\star,N}_s\big\|^2 \Big) \d s \Bigg] \\
&\quad + 8c_{\eps_{\smalltext{5}\smalltext{,}\smalltext{6}}}(t) NR_N^2 \E^{\P^{\smalltext{\hat\balpha}^\tinytext{N}\smalltext{,}\smalltext{N}\smalltext{,}\smalltext{u}}_\smalltext{\omega}} \Bigg[ \int_u^t \mathrm{e}^{\beta s} \sum_{\ell =1}^N \Big( \big\|\widetilde Z^{n,\ell,\star,N}_s\big\|^2 + \big\|\delta Z^{n,\ell,\star,N}_s\big\|^2 \Big) \d s \Bigg], \; \P\text{\rm--a.e.} \; \omega\in\Omega,
\end{align}
which in turn implies
\begin{align}\label{align:GronwallSumDeltaX}
&\notag\frac1N\E^{\P^{\smalltext{\hat\balpha}^\tinytext{N}\smalltext{,}\smalltext{N}\smalltext{,}\smalltext{u}}_\smalltext{\omega}} \Bigg[\int_u^t\mathrm{e}^{\beta s}\sum_{\ell =1}^N \big\|\delta X^{\ell}_{\cdot \land s}\big\|^2_\infty\mathrm{d}s \Bigg] \\
	\notag&\leq c_{\eps_{\smalltext{5}\smalltext{,}\smalltext{6}}}(t) 8R_N^2 \Bigg( \frac{\mathrm{e}^{\beta t}}{\beta^2} + \frac tN\E^{\P^{\smalltext{\hat\balpha}^\tinytext{N}\smalltext{,}\smalltext{N}\smalltext{,}\smalltext{u}}_\smalltext{\omega}} \bigg[ \int_u^t \mathrm{e}^{\beta s} \sum_{\ell =1}^N \big\|X^\ell_{\cdot \land s}\big\|_\infty^2 \d s \bigg] \Bigg) \\
\notag&\quad+ \frac{tc_{\eps_{\smalltext{5}\smalltext{,}\smalltext{6}}}(t)}N \E^{\P^{\smalltext{\hat\balpha}^\tinytext{N}\smalltext{,}\smalltext{N}\smalltext{,}\smalltext{u}}_\smalltext{\omega}} \Bigg[ \int_u^t \mathrm{e}^{\beta s} \sum_{\ell=1}^N \Big( \big\|\delta Z^{\ell,\ell,N}_s\big\|^2 + \big\|\delta Z^{\ell,m,\ell,\star,N}_s\big\|^2 + \big\|\delta Z^{n,\ell,\star,N}_s\big\|^2 \Big) \d s \Bigg] \\
\notag&\quad + 8tc_{\eps_{\smalltext{5}\smalltext{,}\smalltext{6}}}(t) R_N^2 \E^{\P^{\smalltext{\hat\balpha}^\tinytext{N}\smalltext{,}\smalltext{N}\smalltext{,}\smalltext{u}}_\smalltext{\omega}} \Bigg[ \int_u^t \mathrm{e}^{\beta s} \sum_{(\ell,k)\in\{1,\dots,N\}^\smalltext{2}}  \Big( \big\|\widetilde Z^{\ell,k,N}_s\big\|^2 + \big\|\widetilde Z^{\ell,m,k,\star,N}_s\big\|^2 + \big\|\delta Z^{\ell,k,N}_s\big\|^2 + \big\|\delta Z^{\ell,m,k,\star,N}_s\big\|^2 \Big) \d s \Bigg] \\
&\quad + 8tc_{\eps_{\smalltext{5}\smalltext{,}\smalltext{6}}}(t) NR_N^2 \E^{\P^{\smalltext{\hat\balpha}^\tinytext{N}\smalltext{,}\smalltext{N}\smalltext{,}\smalltext{u}}_\smalltext{\omega}} \Bigg[ \int_u^t \mathrm{e}^{\beta s} \sum_{\ell =1}^N \Big( \big\|\widetilde Z^{n,\ell,\star,N}_s\big\|^2 + \big\|\delta Z^{n,\ell,\star,N}_s\big\|^2 \Big) \d s \Bigg], \; \P\text{\rm--a.e.} \; \omega\in\Omega.
\end{align}

We define
\[
\bar{c}^1_{\eps_{\smalltext{5}\smalltext{,}\smalltext{6}}}(t)\coloneqq \eps_5 8 c_{\eps_\smalltext{6}}(t)\big(6\ell^2_\Lambda +(1+18\ell^2_\Lambda)tc_{\eps_{\smalltext{5}\smalltext{,}\smalltext{6}}}(t)\big),\; \bar{c}^2_{\eps_{\smalltext{5}\smalltext{,}\smalltext{6}}}(t)\coloneqq \eps_5 8 c_{\eps_\smalltext{6}}(t)\big(12\ell^2_\Lambda +(1+18\ell^2_\Lambda)tc_{\eps_{\smalltext{5}\smalltext{,}\smalltext{6}}}(t)\big),\; \bar{c}^3_{\eps_{\smalltext{5}\smalltext{,}\smalltext{6}}}(t)\coloneqq  \varepsilon_5 6 \ell_\Lambda^2 c_{\eps_{\smalltext{6}}}(t).
\]
Plugging the estimate in \eqref{align:GronwallSumDeltaX} back into \Cref{gronwallineq1}, we have
\begin{align}\label{align:GronwallDeltaX}
\notag&\E^{\P^{\smalltext{\hat\balpha}^\tinytext{N}\smalltext{,}\smalltext{N}\smalltext{,}\smalltext{u}}_\smalltext{\omega}} \bigg[ \mathrm{e}^{\beta t}\big\|\delta X^{i}_{\cdot \land t}\big\|^2 \bigg] \\
\notag&\leq \eps_5 \frac8{\beta^2} c_{\eps_\smalltext{6}}(t)R_N^2\mathrm{e}^{\beta t}\big(12\beta\ell^2_\Lambda +c_{\eps_{\smalltext{5}\smalltext{,}\smalltext{6}}}(t)(1+18\ell^2_\Lambda)\big)+8\bar{c}^3_{\eps_{\smalltext{5}\smalltext{,}\smalltext{6}}}(t)R^2_N\E^{\P^{\smalltext{\hat\balpha}^\tinytext{N}\smalltext{,}\smalltext{N}\smalltext{,}\smalltext{u}}_\smalltext{\omega}} \bigg[ \int_u^t \mathrm{e}^{\beta s}\big\|X^i_{\cdot \land s}\big\|_\infty^2 \d s \bigg]\\
\notag&\quad+\bar{c}^1_{\eps_{\smalltext{5}\smalltext{,}\smalltext{6}}}(t)\frac{R_N^2}N\E^{\P^{\smalltext{\hat\balpha}^\tinytext{N}\smalltext{,}\smalltext{N}\smalltext{,}\smalltext{u}}_\smalltext{\omega}} \Bigg[ \int_u^t \mathrm{e}^{\beta s}\sum_{\ell =1} \big\|X^\ell_{\cdot \land s}\big\|_\infty^2 \d s \Bigg]\\
\notag&\quad +8\bar{c}^3_{\eps_{\smalltext{5}\smalltext{,}\smalltext{6}}}(t) NR_N^2 \E^{\P^{\smalltext{\hat\balpha}^\tinytext{N}\smalltext{,}\smalltext{N}\smalltext{,}\smalltext{u}}_\smalltext{\omega}} \Bigg[ \int_u^t \mathrm{e}^{\beta s} \sum_{\ell =1}^N \Big( \big\|\widetilde Z^{i,\ell,N}_s\big\|^2 + \big\|\widetilde Z^{i,m,\ell,\star,N}_s\big\|^2 +\big\|\delta Z^{i,\ell,N}_s\big\|^2 + \big\|\delta Z^{i,m,\ell,\star,N}_s\big\|^2 \Big) \d s \Bigg] \\
\notag&\quad + \bar{c}^2_{\eps_{\smalltext{5}\smalltext{,}\smalltext{6}}}(t) NR_N^2\E^{\P^{\smalltext{\hat\balpha}^\tinytext{N}\smalltext{,}\smalltext{N}\smalltext{,}\smalltext{u}}_\smalltext{\omega}} \Bigg[ \int_u^t \mathrm{e}^{\beta s} \sum_{\ell =1}^N \Big(\big\|\widetilde Z^{n,\ell,\star,N}_s\big\|^2+\big\|\delta Z^{n,\ell,\star,N}_s\big\|^2\Big)\mathrm{d}s\Bigg]\\
\notag&\quad +\bar{c}^1_{\eps_{\smalltext{5}\smalltext{,}\smalltext{6}}}(t) R_N^2\E^{\P^{\smalltext{\hat\balpha}^\tinytext{N}\smalltext{,}\smalltext{N}\smalltext{,}\smalltext{u}}_\smalltext{\omega}} \bigg[ \int_u^t \mathrm{e}^{\beta s} \sum_{(k,\ell)\in\{1,\dots,N\}^\smalltext{2}} \Big(\big\|\widetilde Z^{k,\ell,N}_s\big\|^2 +\big\|\widetilde Z^{k,m,\ell,\star,N}_s\big\|^2 + \big\|\delta Z^{k,\ell,N}_s\big\|^2 + \big\|\delta Z^{k,m,\ell,\star,N}_s\big\|^2\Big) \d s \Bigg] \\
\notag&\quad+ \bar{c}^3_{\eps_{\smalltext{5}\smalltext{,}\smalltext{6}}}(t) \E^{\P^{\smalltext{\hat\balpha}^\tinytext{N}\smalltext{,}\smalltext{N}\smalltext{,}\smalltext{u}}_\smalltext{\omega}} \bigg[ \int_u^t \mathrm{e}^{\beta s} \Big( \big\|\delta Z^{i,i,N}_s\big\|^2 + \big\|\delta Z^{i ,m,i,\star,N}_s\big\|^2 + \big\|\delta Z^{n,i,\star,N}_s\big\|^2 \Big) \d s \bigg] \\
&\quad+\frac{\bar{c}^1_{\eps_{\smalltext{5}\smalltext{,}\smalltext{6}}}(t)}{N} \E^{\P^{\smalltext{\hat\balpha}^\tinytext{N}\smalltext{,}\smalltext{N}\smalltext{,}\smalltext{u}}_\smalltext{\omega}} \Bigg[ \int_u^t \mathrm{e}^{\beta s} \sum_{\ell =1}^N \Big( \big\|\delta Z^{\ell,\ell,N}_s\big\|^2 + \big\|\delta Z^{\ell,m,\ell,\star,N}_s\big\|^2 + \big\|\delta Z^{n,\ell,\star,N}_s\big\|^2 \Big) \d s \Bigg], \; \P\text{\rm--a.e.} \; \omega\in\Omega,
\end{align}
and consequently,
\begin{align}\label{align:GronwallDeltaX2}
\notag&\E^{\P^{\smalltext{\hat\balpha}^\tinytext{N}\smalltext{,}\smalltext{N}\smalltext{,}\smalltext{u}}_\smalltext{\omega}} \bigg[\int_u^t \mathrm{e}^{\beta s}\big\|\delta X^{i}_{\cdot \land s}\big\|^2 \d s \bigg] \\
\notag&\leq \eps_5 \frac8{\beta^3}c_{\eps_\smalltext{6}}(t)R_N^2\mathrm{e}^{\beta t}\big(12\beta\ell^2_\Lambda +c_{\eps_{\smalltext{5}\smalltext{,}\smalltext{6}}}(t)(1+18\ell^2_\Lambda)\big)+8t\bar{c}^3_{\eps_{\smalltext{5}\smalltext{,}\smalltext{6}}}(t)R^2_N\E^{\P^{\smalltext{\hat\balpha}^\tinytext{N}\smalltext{,}\smalltext{N}\smalltext{,}\smalltext{u}}_\smalltext{\omega}} \bigg[ \int_u^t \mathrm{e}^{\beta s}\big\|X^i_{\cdot \land s}\big\|_\infty^2 \d s \bigg]\\
\notag&\quad+t\bar{c}^1_{\eps_{\smalltext{5}\smalltext{,}\smalltext{6}}}(t)\frac{R_N^2}N\E^{\P^{\smalltext{\hat\balpha}^\tinytext{N}\smalltext{,}\smalltext{N}\smalltext{,}\smalltext{u}}_\smalltext{\omega}} \Bigg[ \int_u^t \mathrm{e}^{\beta s}\sum_{\ell =1}^N\big\|X^\ell_{\cdot \land s}\big\|_\infty^2 \d s \Bigg]\\
\notag&\quad +8t\bar{c}^3_{\eps_{\smalltext{5}\smalltext{,}\smalltext{6}}}(t) NR_N^2 \E^{\P^{\smalltext{\hat\balpha}^\tinytext{N}\smalltext{,}\smalltext{N}\smalltext{,}\smalltext{u}}_\smalltext{\omega}} \Bigg[ \int_u^t \mathrm{e}^{\beta s} \sum_{\ell =1}^N \Big( \big\|\widetilde Z^{i,\ell,N}_s\big\|^2 + \big\|\widetilde Z^{i,m,\ell,\star,N}_s\big\|^2 +\big\|\delta Z^{i,\ell,N}_s\big\|^2 + \big\|\delta Z^{i,m,\ell,\star,N}_s\big\|^2 \Big) \d s \Bigg] \\
\notag&\quad + t\bar{c}^2_{\eps_{\smalltext{5}\smalltext{,}\smalltext{6}}}(t) NR_N^2\E^{\P^{\smalltext{\hat\balpha}^\tinytext{N}\smalltext{,}\smalltext{N}\smalltext{,}\smalltext{u}}_\smalltext{\omega}} \Bigg[ \int_u^t \mathrm{e}^{\beta s} \sum_{\ell =1}^N \Big(\big\|\widetilde Z^{n,\ell,\star,N}_s\big\|^2+\big\|\delta Z^{n,\ell,\star,N}_s\big\|^2\Big)\mathrm{d}s\Bigg]\\
\notag&\quad +t\bar{c}^1_{\eps_{\smalltext{5}\smalltext{,}\smalltext{6}}}(t) R_N^2\E^{\P^{\smalltext{\hat\balpha}^\tinytext{N}\smalltext{,}\smalltext{N}\smalltext{,}\smalltext{u}}_\smalltext{\omega}} \Bigg[ \int_u^t \mathrm{e}^{\beta s} \sum_{(k,\ell)\in\{1,\dots,N\}^\smalltext{2}} \Big(\big\|\widetilde Z^{k,\ell,N}_s\big\|^2 +\big\|\widetilde Z^{k,m,\ell,\star,N}_s\big\|^2 + \big\|\delta Z^{k,\ell,N}_s\big\|^2 + \big\|\delta Z^{k,m,\ell,\star,N}_s\big\|^2\Big) \d s \Bigg] \\
\notag&\quad+ t\bar{c}^3_{\eps_{\smalltext{5}\smalltext{,}\smalltext{6}}}(t) \E^{\P^{\smalltext{\hat\balpha}^\tinytext{N}\smalltext{,}\smalltext{N}\smalltext{,}\smalltext{u}}_\smalltext{\omega}} \bigg[ \int_u^t \mathrm{e}^{\beta s} \Big( \big\|\delta Z^{i,i,N}_s\big\|^2 + \big\|\delta Z^{i ,m,i,\star,N}_s\big\|^2 + \big\|\delta Z^{n,i,\star,N}_s\big\|^2 \Big) \d s \bigg] \\
&\quad+t\frac{\bar{c}^1_{\eps_{\smalltext{5}\smalltext{,}\smalltext{6}}}(t)}{N} \E^{\P^{\smalltext{\hat\balpha}^\tinytext{N}\smalltext{,}\smalltext{N}\smalltext{,}\smalltext{u}}_\smalltext{\omega}} \Bigg[ \int_u^t \mathrm{e}^{\beta s} \sum_{\ell =1}^N \Big( \big\|\delta Z^{\ell,\ell,N}_s\big\|^2 + \big\|\delta Z^{\ell,m,\ell,\star,N}_s\big\|^2 + \big\|\delta Z^{n,\ell,\star,N}_s\big\|^2 \Big) \d s \Bigg], \; \P\text{\rm--a.e.} \; \omega\in\Omega.
\end{align}

\medskip
\textbf{Step 4: combining all the estimates}

\medskip
We define 
\begin{align*}
{c}^1_{\eps_{\smalltext{1}\smalltext{,}\smalltext{2}\smalltext{,}\smalltext{3}\smalltext{,}\smalltext{4}\smalltext{,}\smalltext{5}\smalltext{,}\smalltext{6}}}&\coloneqq 8\frac{\mathrm{e}^{\beta T}}{\beta^2}\Bigg(\eps_1\bigg(2+\frac1{\eps_3}\bigg)\bigg(12\beta\ell^2_\Lambda+\eps_5(1+6\ell^2_\Lambda)\frac{c_{\eps_\smalltext{6}}(T)}{\beta}\big(12\beta\ell^2_\Lambda +c_{\eps_{\smalltext{5}\smalltext{,}\smalltext{6}}}(T)(1+18\ell^2_\Lambda)\big)+(1+18\ell^2_\Lambda)c_{\eps_{\smalltext{5}\smalltext{,}\smalltext{6}}}(T)\bigg)\\&\quad+\eps_5\big(c_{\eps_{\smalltext{2}\smalltext{,}\smalltext{3}\smalltext{,}\smalltext{4}}}+\ell^2_{\varphi_\smalltext{1}}c^\star\big)c_{\eps_\smalltext{6}}(T)\big(12\beta\ell^2_\Lambda + c_{\eps_{\smalltext{5}\smalltext{,}\smalltext{6}}}(T)(1+18\ell^2_\Lambda)\big)+\beta\big(\bar{c}_{\eps_{\smalltext{2}\smalltext{,}\smalltext{3}\smalltext{,}\smalltext{4}}}+\ell^2_{\varphi_\smalltext{2}}c^\star\big)	 c_{\eps_{\smalltext{5}\smalltext{,}\smalltext{6}}}(T)\Bigg),\\
{c}^2_{\eps_{\smalltext{1}\smalltext{,}\smalltext{2}\smalltext{,}\smalltext{3}\smalltext{,}\smalltext{4}\smalltext{,}\smalltext{5}\smalltext{,}\smalltext{6}}}&\coloneqq 8\Bigg(\eps_1\bigg(2+\frac1{\eps_3}\bigg)\bigg(6\ell^2_\Lambda+(1+6\ell^2_\Lambda)T\bar{c}^3_{\eps_{\smalltext{5}\smalltext{,}\smalltext{6}}}(T)\bigg)+\bar{c}^3_{\eps_{\smalltext{5}\smalltext{,}\smalltext{6}}}(T)\big(c_{\eps_{\smalltext{2}\smalltext{,}\smalltext{3}\smalltext{,}\smalltext{4}}}+\ell^2_{\varphi_\smalltext{1}}c^\star\big)\Bigg),\\
{c}^3_{\eps_{\smalltext{1}\smalltext{,}\smalltext{2}\smalltext{,}\smalltext{3}\smalltext{,}\smalltext{4}\smalltext{,}\smalltext{5}\smalltext{,}\smalltext{6}}}&\coloneqq \eps_1\bigg(2+\frac1{\eps_3}\bigg)\big(48\ell^2_\Lambda+(1+6\ell_\Lambda^2)T\bar{c}^1_{\eps_{\smalltext{5}\smalltext{,}\smalltext{6}}}(T)+(1+18\ell^2_\Lambda)8Tc_{\eps_{\smalltext{5}\smalltext{,}\smalltext{6}}}(T)\big)+\big(c_{\eps_{\smalltext{2}\smalltext{,}\smalltext{3}\smalltext{,}\smalltext{4}}}+\ell^2_{\varphi_\smalltext{1}}c^\star\big)\bar{c}^1_{\eps_{\smalltext{5}\smalltext{,}\smalltext{6}}}(T)\\
&\quad+8\big(\bar{c}_{\eps_{\smalltext{2}\smalltext{,}\smalltext{3}\smalltext{,}\smalltext{4}}}+\ell^2_{\varphi_\smalltext{2}}c^\star\big)c_{\eps_{\smalltext{5}\smalltext{,}\smalltext{6}}}(T),\\
{c}^4_{\eps_{\smalltext{1}\smalltext{,}\smalltext{2}\smalltext{,}\smalltext{3}\smalltext{,}\smalltext{4}\smalltext{,}\smalltext{5}\smalltext{,}\smalltext{6}}}&\coloneqq\eps_1\bigg(2+\frac1{\eps_3}\bigg)\big(48\ell^2_\Lambda+(1+6\ell^2_\Lambda)8T\bar{c}^3_{\eps_{\smalltext{5}\smalltext{,}\smalltext{6}}}(T)\big)+8\bar{c}^3_{\eps_{\smalltext{5}\smalltext{,}\smalltext{6}}}(T)\big(c_{\eps_{\smalltext{2}\smalltext{,}\smalltext{3}\smalltext{,}\smalltext{4}}}+\ell^2_{\varphi_\smalltext{1}}c^\star\big),\\
{c}^5_{\eps_{\smalltext{1}\smalltext{,}\smalltext{2}\smalltext{,}\smalltext{3}\smalltext{,}\smalltext{4}\smalltext{,}\smalltext{5}\smalltext{,}\smalltext{6}}}&\coloneqq \eps_1\bigg(2+\frac1{\eps_3}\bigg)\big(48\ell^2_\Lambda+(1+6\ell^2_\Lambda)T\bar{c}^1_{\eps_{\smalltext{5}\smalltext{,}\smalltext{6}}}(T)+(1+18\ell^2_\Lambda)8T{c}_{\eps_{\smalltext{5}\smalltext{,}\smalltext{6}}}(T)\big)+\big(c_{\eps_{\smalltext{2}\smalltext{,}\smalltext{3}\smalltext{,}\smalltext{4}}}+\ell^2_{\varphi_\smalltext{1}}c^\star\big)\bar{c}^1_{\eps_{\smalltext{5}\smalltext{,}\smalltext{6}}}(T)\\
&\quad+8\big(\bar{c}_{\eps_{\smalltext{2}\smalltext{,}\smalltext{3}\smalltext{,}\smalltext{4}}}+\ell^2_{\varphi_\smalltext{2}}c^\star\big){c}_{\eps_{\smalltext{5}\smalltext{,}\smalltext{6}}}(T),\\
{c}^6_{\eps_{\smalltext{1}\smalltext{,}\smalltext{2}\smalltext{,}\smalltext{3}\smalltext{,}\smalltext{4}\smalltext{,}\smalltext{5}\smalltext{,}\smalltext{6}}}&\coloneqq \eps_1\bigg(2+\frac1{\eps_3}\bigg)\big(96\ell^2_\Lambda+(1+6\ell^2_\Lambda)T\bar{c}^2_{\eps_{\smalltext{5}\smalltext{,}\smalltext{6}}}(T)+(1+18\ell^2_\Lambda)8T{c}_{\eps_{\smalltext{5}\smalltext{,}\smalltext{6}}}(T)\big)+\big(c_{\eps_{\smalltext{2}\smalltext{,}\smalltext{3}\smalltext{,}\smalltext{4}}}+\ell^2_{\varphi_\smalltext{1}}c^\star\big)\bar{c}^2_{\eps_{\smalltext{5}\smalltext{,}\smalltext{6}}}(T)\\
&\quad+8\big(\bar{c}_{\eps_{\smalltext{2}\smalltext{,}\smalltext{3}\smalltext{,}\smalltext{4}}}+\ell^2_{\varphi_\smalltext{2}}c^\star\big){c}_{\eps_{\smalltext{5}\smalltext{,}\smalltext{6}}}(T),\\
{c}^7_{\eps_{\smalltext{1}\smalltext{,}\smalltext{2}\smalltext{,}\smalltext{3}\smalltext{,}\smalltext{4}\smalltext{,}\smalltext{5}\smalltext{,}\smalltext{6}}}&\coloneqq \eps_1\bigg(2+\frac1{\eps_3}\bigg)\big(6\ell^2_\Lambda+(1+6\ell^2_\Lambda)T\bar{c}^1_{\eps_{\smalltext{5}\smalltext{,}\smalltext{6}}}(T)+(1+18\ell^2_\Lambda)T{c}_{\eps_{\smalltext{5}\smalltext{,}\smalltext{6}}}(T)\big)+\big(c_{\eps_{\smalltext{2}\smalltext{,}\smalltext{3}\smalltext{,}\smalltext{4}}}+\ell^2_{\varphi_\smalltext{1}}c^\star\big)\bar{c}^1_{\eps_{\smalltext{5}\smalltext{,}\smalltext{6}}}(T)\\
&\quad+\big(\bar{c}_{\eps_{\smalltext{2}\smalltext{,}\smalltext{3}\smalltext{,}\smalltext{4}}}+\ell^2_{\varphi_\smalltext{2}}c^\star\big){c}_{\eps_{\smalltext{5}\smalltext{,}\smalltext{6}}}(T),\\
{c}^8_{\eps_{\smalltext{1}\smalltext{,}\smalltext{2}\smalltext{,}\smalltext{3}\smalltext{,}\smalltext{4}\smalltext{,}\smalltext{5}\smalltext{,}\smalltext{6}}}&\coloneqq\eps_1\bigg(2+\frac1{\eps_3}\bigg)\big(6\ell^2_\Lambda+(1+6\ell^2_\Lambda)T\bar{c}^3_{\eps_{\smalltext{5}\smalltext{,}\smalltext{6}}}(T)\big)+\bar{c}^3_{\eps_{\smalltext{5}\smalltext{,}\smalltext{6}}}(T)\big(c_{\eps_{\smalltext{2}\smalltext{,}\smalltext{3}\smalltext{,}\smalltext{4}}}+\ell^2_{\varphi_\smalltext{1}}c^\star\big).
\end{align*}

Returning to \Cref{align:deltaEverything} and using the estimates derived in \eqref{align:GronwallSumDeltaX2}, \eqref{align:GronwallSumDeltaX}, \eqref{align:GronwallDeltaX} and \eqref{align:GronwallDeltaX2}, we have 
\begin{align*}
\notag&\bigg( 1 - \varepsilon_3 4c^2_{1,\smallertext{\rm BDG}} - \bigg(\varepsilon_2 +\eps_4+\frac{\eps_2}{\eps_3}\bigg) c_{\smallertext{\rm{BMO}}_{\smalltext{[}\smalltext{u}\smalltext{,}\smalltext{T}\smalltext{]}}} \bigg) \E^{\P^{\smalltext{\hat\balpha}^\tinytext{N}\smalltext{,}\smalltext{N}\smalltext{,}\smalltext{u}}_\smalltext{\omega}} \bigg[ \sup_{t \in [u,T]} \mathrm{e}^{\beta t} \big|\delta Y^{i,N}_t\big|^2 \bigg] +\E^{\P^{\smalltext{\hat\balpha}^\tinytext{N}\smalltext{,}\smalltext{N}\smalltext{,}\smalltext{u}}_\smalltext{\omega}} \Bigg[ \int_u^T \mathrm{e}^{\beta t} \sum_{\ell =1}^N \big\|\delta Z^{i,\ell,N}_t\big\|^2 \d t \Bigg] \\
\notag&\quad +\E^{\P^{\smalltext{\hat\balpha}^\tinytext{N}\smalltext{,}\smalltext{N}\smalltext{,}\smalltext{u}}_\smalltext{\omega}} \Bigg[ \sup_{t \in [u,T]} \mathrm{e}^{\beta t} \big|\delta M^{i,\star,N}_t\big|^2 + \sup_{t \in [u,T]} \mathrm{e}^{\beta t} \big|\delta N^{\star,N}_t\big|^2 + \int_u^T \mathrm{e}^{\beta t} \sum_{\ell =1}^N \Big( \big\|Z^{i,m,\ell,\star,N}_t\big\|^2 + \big\|\delta Z^{n,\ell,\star,N}_t\big\| \Big) \d t \Bigg] \\
	\notag&\leq  {c}^1_{\eps_{\smalltext{1}\smalltext{,}\smalltext{2}\smalltext{,}\smalltext{3}\smalltext{,}\smalltext{4}\smalltext{,}\smalltext{5}\smalltext{,}\smalltext{6}}}R_N^2+{c}^2_{\eps_{\smalltext{1}\smalltext{,}\smalltext{2}\smalltext{,}\smalltext{3}\smalltext{,}\smalltext{4}\smalltext{,}\smalltext{5}\smalltext{,}\smalltext{6}}}R_N^2\E^{\P^{\smalltext{\hat\balpha}^\tinytext{N}\smalltext{,}\smalltext{N}\smalltext{,}\smalltext{u}}_\smalltext{\omega}} \bigg[\int_u^T\mathrm{e}^{\beta t}\big\| X^{i}_{\cdot\wedge t}\big\|^2_\infty\mathrm{d}t\bigg] + {c}^3_{\eps_{\smalltext{1}\smalltext{,}\smalltext{2}\smalltext{,}\smalltext{3}\smalltext{,}\smalltext{4}\smalltext{,}\smalltext{5}\smalltext{,}\smalltext{6}}}\frac{R_N^2}{N}\E^{\P^{\smalltext{\hat\balpha}^\tinytext{N}\smalltext{,}\smalltext{N}\smalltext{,}\smalltext{u}}_\smalltext{\omega}} \Bigg[\int_u^T\mathrm{e}^{\beta t}\sum_{\ell =1}^N\big\| X^{\ell}_{\cdot\wedge t}\big\|^2_\infty\mathrm{d}t\Bigg]\\
	&\quad+{c}^4_{\eps_{\smalltext{1}\smalltext{,}\smalltext{2}\smalltext{,}\smalltext{3}\smalltext{,}\smalltext{4}\smalltext{,}\smalltext{5}\smalltext{,}\smalltext{6}}}NR_N^2\E^{\P^{\smalltext{\hat\balpha}^\tinytext{N}\smalltext{,}\smalltext{N}\smalltext{,}\smalltext{u}}_\smalltext{\omega}} \Bigg[\int_u^T\mathrm{e}^{\beta t}\sum_{\ell=1}^N\Big( \big\|\widetilde Z^{i,\ell,N}_t\big\|^2+ \big\|\widetilde Z^{i,m,\ell,\star,N}_t\big\|^2+ \big\|\delta Z^{i,\ell,N}_t\big\|^2+\big\|\delta Z^{i,m,\ell,\star,N}_t\big\|^2\Big)\mathrm{d}t\Bigg]\\
	&\quad+{c}^5_{\eps_{\smalltext{1}\smalltext{,}\smalltext{2}\smalltext{,}\smalltext{3}\smalltext{,}\smalltext{4}\smalltext{,}\smalltext{5}\smalltext{,}\smalltext{6}}}R_N^2\E^{\P^{\smalltext{\hat\balpha}^\tinytext{N}\smalltext{,}\smalltext{N}\smalltext{,}\smalltext{u}}_\smalltext{\omega}} \Bigg[ \int_u^T \mathrm{e}^{\beta t} \sum_{(k,\ell)\in\{1,\dots,N\}^\smalltext{2}} \Big(\big\|\widetilde Z^{k,\ell,N}_t\big\|^2 +\big\|\widetilde Z^{k,m,\ell,\star,N}_t\big\|^2 + \big\|\delta Z^{k,\ell,N}_t\big\|^2 + \big\|\delta Z^{k,m,\ell,\star,N}_t\big\|^2\Big) \d t \Bigg]\\
	&\quad+ {c}^6_{\eps_{\smalltext{1}\smalltext{,}\smalltext{2}\smalltext{,}\smalltext{3}\smalltext{,}\smalltext{4}\smalltext{,}\smalltext{5}\smalltext{,}\smalltext{6}}}NR_N^2\E^{\P^{\smalltext{\hat\balpha}^\tinytext{N}\smalltext{,}\smalltext{N}\smalltext{,}\smalltext{u}}_\smalltext{\omega}} \Bigg[ \int_u^T \mathrm{e}^{\beta t} \sum_{\ell =1}^N \Big(\big\|\widetilde Z^{n,\ell,\star,N}_t\big\|^2+\big\|\delta Z^{n,\ell,\star,N}_t\big\|^2\Big)\mathrm{d}t\Bigg]\\
	&\quad+\frac{{c}^7_{\eps_{\smalltext{1}\smalltext{,}\smalltext{2}\smalltext{,}\smalltext{3}\smalltext{,}\smalltext{4}\smalltext{,}\smalltext{5}\smalltext{,}\smalltext{6}}}}{N} \E^{\P^{\smalltext{\hat\balpha}^\tinytext{N}\smalltext{,}\smalltext{N}\smalltext{,}\smalltext{u}}_\smalltext{\omega}} \Bigg[ \int_u^T \mathrm{e}^{\beta t} \sum_{\ell =1}^N \Big( \big\|\delta Z^{\ell,\ell,N}_t\big\|^2 + \big\|\delta Z^{\ell,m,\ell,\star,N}_t\big\|^2 + \big\|\delta Z^{n,\ell,\star,N}_t\big\|^2 \Big) \d t \Bigg]\\
	&\quad+{c}^8_{\eps_{\smalltext{1}\smalltext{,}\smalltext{2}\smalltext{,}\smalltext{3}\smalltext{,}\smalltext{4}\smalltext{,}\smalltext{5}\smalltext{,}\smalltext{6}}}\E^{\P^{\smalltext{\hat\balpha}^\tinytext{N}\smalltext{,}\smalltext{N}\smalltext{,}\smalltext{u}}_\smalltext{\omega}} \bigg[ \int_u^T \mathrm{e}^{\beta t} \Big( \big\|\delta Z^{i,i,N}_t\big\|^2 + \big\|\delta Z^{i ,m,i,\star,N}_t\big\|^2 + \big\|\delta Z^{n,i,\star,N}_t\big\|^2 \Big) \d t \bigg], \; \P\text{\rm--a.e.} \; \omega\in\Omega.
	\end{align*}
	
We continue and deduce that
\begin{align}\label{eq:estim678}
\notag&\E^{\P^{\smalltext{\hat\balpha}^\tinytext{N}\smalltext{,}\smalltext{N}\smalltext{,}\smalltext{u}}_\smalltext{\omega}} \Bigg[ \bigg( 1 - \varepsilon_3 4c^2_{1,\smallertext{\rm BDG}} - \bigg(\varepsilon_2 +\eps_4+\frac{\eps_2}{\eps_3}\bigg) c_{\smallertext{\rm{BMO}}_{\smalltext{[}\smalltext{u}\smalltext{,}\smalltext{T}\smalltext{]}}} \bigg) \sup_{t \in [u,T]} \mathrm{e}^{\beta t} \big|\delta Y^{i,N}_t\big|^2   + \sup_{t \in [u,T]} \mathrm{e}^{\beta t} \big|\delta M^{i,\star,N}_t\big|^2 + \sup_{t \in [u,T]} \mathrm{e}^{\beta t} \big|\delta N^{\star,N}_t\big|^2 \Bigg] \\
\notag&\quad +\big(1-{c}^8_{\eps_{\smalltext{1}\smalltext{,}\smalltext{2}\smalltext{,}\smalltext{3}\smalltext{,}\smalltext{4}\smalltext{,}\smalltext{5}\smalltext{,}\smalltext{6}}}-{c}^4_{\eps_{\smalltext{1}\smalltext{,}\smalltext{2}\smalltext{,}\smalltext{3}\smalltext{,}\smalltext{4}\smalltext{,}\smalltext{5}\smalltext{,}\smalltext{6}}} NR_N^2\big) \E^{\P^{\smalltext{\hat\balpha}^\tinytext{N}\smalltext{,}\smalltext{N}\smalltext{,}\smalltext{u}}_\smalltext{\omega}} \Bigg[ \int_u^T \mathrm{e}^{\beta t} \sum_{\ell =1}^N \Big( \big\|\delta Z^{i,\ell,N}_t\big\|^2 + \big\|\delta Z^{i,m,\ell,\star,N}_t\big\|^2 \Big) \d t \Bigg] \\
\notag&\quad +\bigg(1-{c}^8_{\eps_{\smalltext{1}\smalltext{,}\smalltext{2}\smalltext{,}\smalltext{3}\smalltext{,}\smalltext{4}\smalltext{,}\smalltext{5}\smalltext{,}\smalltext{6}}}-\frac{c^7_{\eps_{\smalltext{1}\smalltext{,}\smalltext{2}\smalltext{,}\smalltext{3}\smalltext{,}\smalltext{4}\smalltext{,}\smalltext{5}\smalltext{,}\smalltext{6}}}}{N}-c^6_{\eps_{\smalltext{1}\smalltext{,}\smalltext{2}\smalltext{,}\smalltext{3}\smalltext{,}\smalltext{4}\smalltext{,}\smalltext{5}\smalltext{,}\smalltext{6}}} N R_N^2\bigg) \E^{\P^{\smalltext{\hat\balpha}^\tinytext{N}\smalltext{,}\smalltext{N}\smalltext{,}\smalltext{u}}_\smalltext{\omega}} \Bigg[ \int_u^T \mathrm{e}^{\beta t} \sum_{\ell =1}^N \big\|\delta Z^{n,\ell,\star,N}_t\big\|^2 \d t \Bigg] \\
	\notag&\leq  {c}^1_{\eps_{\smalltext{1}\smalltext{,}\smalltext{2}\smalltext{,}\smalltext{3}\smalltext{,}\smalltext{4}\smalltext{,}\smalltext{5}\smalltext{,}\smalltext{6}}}R_N^2+{c}^2_{\eps_{\smalltext{1}\smalltext{,}\smalltext{2}\smalltext{,}\smalltext{3}\smalltext{,}\smalltext{4}\smalltext{,}\smalltext{5}\smalltext{,}\smalltext{6}}}R_N^2\E^{\P^{\smalltext{\hat\balpha}^\tinytext{N}\smalltext{,}\smalltext{N}\smalltext{,}\smalltext{u}}_\smalltext{\omega}} \bigg[\int_u^T\mathrm{e}^{\beta t}\big\| X^{i}_{\cdot\wedge t}\big\|^2_\infty\mathrm{d}t\bigg] + {c}^3_{\eps_{\smalltext{1}\smalltext{,}\smalltext{2}\smalltext{,}\smalltext{3}\smalltext{,}\smalltext{4}\smalltext{,}\smalltext{5}\smalltext{,}\smalltext{6}}}\frac{R_N^2}{N}\E^{\P^{\smalltext{\hat\balpha}^\tinytext{N}\smalltext{,}\smalltext{N}\smalltext{,}\smalltext{u}}_\smalltext{\omega}} \Bigg[\int_u^T\mathrm{e}^{\beta t}\sum_{\ell =1}^N\big\| X^{\ell}_{\cdot\wedge t}\big\|^2_\infty\mathrm{d}t\Bigg]\\
	\notag&\quad+NR_N^2\E^{\P^{\smalltext{\hat\balpha}^\tinytext{N}\smalltext{,}\smalltext{N}\smalltext{,}\smalltext{u}}_\smalltext{\omega}} \Bigg[{c}^4_{\eps_{\smalltext{1}\smalltext{,}\smalltext{2}\smalltext{,}\smalltext{3}\smalltext{,}\smalltext{4}\smalltext{,}\smalltext{5}\smalltext{,}\smalltext{6}}}\int_u^T\mathrm{e}^{\beta t}\sum_{\ell =1}^N\Big( \big\|\widetilde Z^{i,\ell,N}_t\big\|^2+ \big\|\widetilde Z^{i,m,\ell,\star,N}_t\big\|^2\Big)\mathrm{d}t+ {c}^6_{\eps_{\smalltext{1}\smalltext{,}\smalltext{2}\smalltext{,}\smalltext{3}\smalltext{,}\smalltext{4}\smalltext{,}\smalltext{5}\smalltext{,}\smalltext{6}}} \int_u^T \mathrm{e}^{\beta t} \sum_{\ell =1}^N \big\|\widetilde Z^{n,\ell,\star,N}_t\big\|^2\mathrm{d}t\Bigg]\\
	\notag&\quad+{c}^5_{\eps_{\smalltext{1}\smalltext{,}\smalltext{2}\smalltext{,}\smalltext{3}\smalltext{,}\smalltext{4}\smalltext{,}\smalltext{5}\smalltext{,}\smalltext{6}}}R_N^2\E^{\P^{\smalltext{\hat\balpha}^\tinytext{N}\smalltext{,}\smalltext{N}\smalltext{,}\smalltext{u}}_\smalltext{\omega}} \Bigg[ \int_u^T \mathrm{e}^{\beta t} \sum_{(k,\ell)\in\{1,\dots,N\}^\smalltext{2}} \Big(\big\|\widetilde Z^{k,\ell,N}_t\big\|^2 +\big\|\widetilde Z^{k,m,\ell,\star,N}_t\big\|^2 + \big\|\delta Z^{k,\ell,N}_t\big\|^2 + \big\|\delta Z^{k,m,\ell,\star,N}_t\big\|^2\Big) \d t \Bigg]\\
	&\quad+\frac{{c}^7_{\eps_{\smalltext{1}\smalltext{,}\smalltext{2}\smalltext{,}\smalltext{3}\smalltext{,}\smalltext{4}\smalltext{,}\smalltext{5}\smalltext{,}\smalltext{6}}}}{N} \E^{\P^{\smalltext{\hat\balpha}^\tinytext{N}\smalltext{,}\smalltext{N}\smalltext{,}\smalltext{u}}_\smalltext{\omega}} \Bigg[ \int_u^T \mathrm{e}^{\beta t} \sum_{\ell =1}^N \Big( \big\|\delta Z^{\ell,\ell,N}_t\big\|^2 + \big\|\delta Z^{\ell,m,\ell,\star,N}_t\big\|^2 \Big) \d t \Bigg], \; \P\text{\rm--a.e.} \; \omega\in\Omega.
	\end{align}

Summing over $i \in \{1,\ldots,N\}$, we have
\begin{align}\label{eq:sumZN}
\notag&\big(1-{c}^8_{\eps_{\smalltext{1}\smalltext{,}\smalltext{2}\smalltext{,}\smalltext{3}\smalltext{,}\smalltext{4}\smalltext{,}\smalltext{5}\smalltext{,}\smalltext{6}}} -{c}^7_{\eps_{\smalltext{1}\smalltext{,}\smalltext{2}\smalltext{,}\smalltext{3}\smalltext{,}\smalltext{4}\smalltext{,}\smalltext{5}\smalltext{,}\smalltext{6}}} - (c^4_{\eps_{\smalltext{1}\smalltext{,}\smalltext{2}\smalltext{,}\smalltext{3}\smalltext{,}\smalltext{4}\smalltext{,}\smalltext{5}\smalltext{,}\smalltext{6}}} +c^5_{\eps_{\smalltext{1}\smalltext{,}\smalltext{2}\smalltext{,}\smalltext{3}\smalltext{,}\smalltext{4}\smalltext{,}\smalltext{5}\smalltext{,}\smalltext{6}}})NR_N^2 \big) \E^{\P^{\smalltext{\hat\balpha}^\tinytext{N}\smalltext{,}\smalltext{N}\smalltext{,}\smalltext{u}}_\smalltext{\omega}} \Bigg[ \int_u^T \mathrm{e}^{\beta t} \sum_{(k,\ell) \in \{1,\ldots,N\}^\smalltext{2}} \Big( \big\|\delta Z^{k,\ell,N}_t\big\|^2 + \big\|\delta Z^{k,m,\ell,\star,N}_t \big\|^2 \Big) \d t \Bigg] \\
		\notag&\leq  {c}^1_{\eps_{\smalltext{1}\smalltext{,}\smalltext{2}\smalltext{,}\smalltext{3}\smalltext{,}\smalltext{4}\smalltext{,}\smalltext{5}\smalltext{,}\smalltext{6}}}NR_N^2+\big({c}^2_{\eps_{\smalltext{1}\smalltext{,}\smalltext{2}\smalltext{,}\smalltext{3}\smalltext{,}\smalltext{4}\smalltext{,}\smalltext{5}\smalltext{,}\smalltext{6}}} + {c}^3_{\eps_{\smalltext{1}\smalltext{,}\smalltext{2}\smalltext{,}\smalltext{3}\smalltext{,}\smalltext{4}\smalltext{,}\smalltext{5}\smalltext{,}\smalltext{6}}}\big)R_N^2\E^{\P^{\smalltext{\hat\balpha}^\tinytext{N}\smalltext{,}\smalltext{N}\smalltext{,}\smalltext{u}}_\smalltext{\omega}} \Bigg[\int_u^T\mathrm{e}^{\beta t}\sum_{\ell =1}^N\big\| X^{\ell}_{\cdot\wedge t}\big\|^2_\infty\mathrm{d}t\Bigg]\\
	\notag&\quad+\big({c}^4_{\eps_{\smalltext{1}\smalltext{,}\smalltext{2}\smalltext{,}\smalltext{3}\smalltext{,}\smalltext{4}\smalltext{,}\smalltext{5}\smalltext{,}\smalltext{6}}}+{c}^5_{\eps_{\smalltext{1}\smalltext{,}\smalltext{2}\smalltext{,}\smalltext{3}\smalltext{,}\smalltext{4}\smalltext{,}\smalltext{5}\smalltext{,}\smalltext{6}}}\big)NR_N^2\E^{\P^{\smalltext{\hat\balpha}^\tinytext{N}\smalltext{,}\smalltext{N}\smalltext{,}\smalltext{u}}_\smalltext{\omega}} \Bigg[\int_u^T\mathrm{e}^{\beta t}\sum_{(k,\ell)\in\{1,\dots,N\}^\smalltext{2}}\Big( \big\|\widetilde Z^{k,\ell,N}_t\big\|^2+ \big\|\widetilde Z^{k,m,\ell,\star,N}_t\big\|^2\Big)\mathrm{d}t\Bigg]\\
	&\quad+ {c}^6_{\eps_{\smalltext{1}\smalltext{,}\smalltext{2}\smalltext{,}\smalltext{3}\smalltext{,}\smalltext{4}\smalltext{,}\smalltext{5}\smalltext{,}\smalltext{6}}}N^2R_N^2\E^{\P^{\smalltext{\hat\balpha}^\tinytext{N}\smalltext{,}\smalltext{N}\smalltext{,}\smalltext{u}}_\smalltext{\omega}} \Bigg[ \int_u^T \mathrm{e}^{\beta t} \sum_{\ell =1}^N \big\|\widetilde Z^{n,\ell,\star,N}_t\big\|^2\mathrm{d}t\Bigg], \; \P\text{\rm--a.e.} \; \omega\in\Omega.
	\end{align}
	
Before continuing, we need the following lemma, whose proof is relegated to \Cref{appendix:auxResult}.
\begin{lemma}\label{lemma:momBoundAuxiliarySystem}
There are constants---defined explicitly in the proof---such that for $\P\text{\rm--a.e.} \; \omega\in\Omega$
\begin{gather*}
\E^{\P^{\smalltext{\hat\balpha}^\tinytext{N}\smalltext{,}\smalltext{N}\smalltext{,}\smalltext{u}}_\smalltext{\omega}} \bigg[\int_u^T\mathrm{e}^{\beta t} \big\| X^{i}_{\cdot\wedge t}\big\|^2_\infty\mathrm{d}t \bigg] \leq c^2_{2} \|X^{i}_u(\omega)\|^2 + \bar{c}^2_{2}, \; \E^{\P^{\smalltext{\hat\balpha}^\tinytext{N}\smalltext{,}\smalltext{N}\smalltext{,}\smalltext{u}}_\smalltext{\omega}} \Bigg[\int_u^T\mathrm{e}^{\beta t} \sum_{\ell =1}^N \big\| X^{\ell}_{\cdot\wedge t}\big\|^2_\infty\mathrm{d}t \Bigg] \leq c^2_{2} \sum_{\ell =1}^N \|X^{\ell}_u(\omega)\|^2 + \bar{c}^2_{2} N,\\
\E^{\P^{\smalltext{\hat\balpha}^\tinytext{N}\smalltext{,}\smalltext{N}\smalltext{,}\smalltext{u}}_\smalltext{\omega}} \Bigg[ \int_u^T \mathrm{e}^{\beta t} \sum_{\ell =1}^N \big\|\widetilde Z^{i,\ell,N}_t\big\|^2 \d t \Bigg]\leq \bar{c}^1_{\eps_{{\smalltext{7}\smalltext{,}\smalltext{8}\smalltext{,}\smalltext{9}}}}+\bar{c}^2_{\eps_{{\smalltext{7}\smalltext{,}\smalltext{8}\smalltext{,}\smalltext{9}}}}\Bigg(\|X^{i}_u(\omega)\|^{2\bar{p}}+ \frac{1}{N} \sum_{\ell =1}^N \| X^{\ell}_{u}(\omega)\|^{2\bar{p}} \Bigg),\\
\E^{\P^{\smalltext{\hat\balpha}^\tinytext{N}\smalltext{,}\smalltext{N}\smalltext{,}\smalltext{u}}_\smalltext{\omega}} \Bigg[ \int_u^T \mathrm{e}^{\beta t} \sum_{\ell =1}^N \big\|\widetilde Z^{i,m,\ell,\star,N}_t\big\|^2 \d t \Bigg]  \leq \mathrm{e}^{\beta T} c_{\varphi_\smalltext{1}},\;  \E^{\P^{\smalltext{\hat\balpha}^\tinytext{N}\smalltext{,}\smalltext{N}\smalltext{,}\smalltext{u}}_\smalltext{\omega}} \Bigg[ \int_u^T \mathrm{e}^{\beta t} \sum_{\ell =1}^N \big\|\widetilde Z^{n,\ell,\star,N}_t\big\|^2 \d t \Bigg]  \leq \mathrm{e}^{\beta T} c_{\varphi_\smalltext{2}}.
\end{gather*}
\end{lemma}

Therefore, the bounds listed above can be plugged into \Cref{eq:sumZN}, from which we deduce that, for $\P\text{\rm--a.e.} \; \omega\in\Omega$
\begin{align*}
\notag&\big(1-{c}^8_{\eps_{\smalltext{1}\smalltext{,}\smalltext{2}\smalltext{,}\smalltext{3}\smalltext{,}\smalltext{4}\smalltext{,}\smalltext{5}\smalltext{,}\smalltext{6}}} -{c}^7_{\eps_{\smalltext{1}\smalltext{,}\smalltext{2}\smalltext{,}\smalltext{3}\smalltext{,}\smalltext{4}\smalltext{,}\smalltext{5}\smalltext{,}\smalltext{6}}} - (c^4_{\eps_{\smalltext{1}\smalltext{,}\smalltext{2}\smalltext{,}\smalltext{3}\smalltext{,}\smalltext{4}\smalltext{,}\smalltext{5}\smalltext{,}\smalltext{6}}} +c^5_{\eps_{\smalltext{1}\smalltext{,}\smalltext{2}\smalltext{,}\smalltext{3}\smalltext{,}\smalltext{4}\smalltext{,}\smalltext{5}\smalltext{,}\smalltext{6}}})NR_N^2 \big) \E^{\P^{\smalltext{\hat\balpha}^\tinytext{N}\smalltext{,}\smalltext{N}\smalltext{,}\smalltext{u}}_\smalltext{\omega}} \Bigg[ \int_u^T \mathrm{e}^{\beta t} \sum_{(k,\ell) \in \{1,\ldots,N\}^\smalltext{2}} \Big( \big\|\delta Z^{k,\ell,N}_t\big\|^2 + \big\|\delta Z^{k,m,\ell,\star,N}_t \big\|^2 \Big) \d t \Bigg] \\
		\notag&\leq  {c}^1_{\eps_{\smalltext{1}\smalltext{,}\smalltext{2}\smalltext{,}\smalltext{3}\smalltext{,}\smalltext{4}\smalltext{,}\smalltext{5}\smalltext{,}\smalltext{6}}}NR_N^2+\big({c}^2_{\eps_{\smalltext{1}\smalltext{,}\smalltext{2}\smalltext{,}\smalltext{3}\smalltext{,}\smalltext{4}\smalltext{,}\smalltext{5}\smalltext{,}\smalltext{6}}} + {c}^3_{\eps_{\smalltext{1}\smalltext{,}\smalltext{2}\smalltext{,}\smalltext{3}\smalltext{,}\smalltext{4}\smalltext{,}\smalltext{5}\smalltext{,}\smalltext{6}}}\big)NR_N^2\Bigg(\frac{c^2_{2}}{N} \sum_{\ell=1}^N \big\|X^{\ell}_u(\omega)\big\|^2 + \bar{c}^2_{2} \Bigg)\\
\notag	&\quad+\big({c}^4_{\eps_{\smalltext{1}\smalltext{,}\smalltext{2}\smalltext{,}\smalltext{3}\smalltext{,}\smalltext{4}\smalltext{,}\smalltext{5}\smalltext{,}\smalltext{6}}}+{c}^5_{\eps_{\smalltext{1}\smalltext{,}\smalltext{2}\smalltext{,}\smalltext{3}\smalltext{,}\smalltext{4}\smalltext{,}\smalltext{5}\smalltext{,}\smalltext{6}}}\big)N^2R_N^2\Bigg(\mathrm{e}^{\beta T}c_{\varphi_\smalltext{1}}+\bar{c}^1_{\eps_{{\smalltext{7}\smalltext{,}\smalltext{8}\smalltext{,}\smalltext{9}}}}+2\frac{\bar{c}^2_{\eps_{{\smalltext{7}\smalltext{,}\smalltext{8}\smalltext{,}\smalltext{9}}}}}{N} \sum_{\ell=1}^N \| X^{\ell}_{u}(\omega)\|^{2\bar{p}} \Bigg)+ \mathrm{e}^{\beta T}c_{\varphi_\smalltext{2}}{c}^6_{\eps_{\smalltext{1}\smalltext{,}\smalltext{2}\smalltext{,}\smalltext{3}\smalltext{,}\smalltext{4}\smalltext{,}\smalltext{5}\smalltext{,}\smalltext{6}}}N^2R_N^2\\
	&=\Bigg({c}^1_{\eps_{\smalltext{1}\smalltext{,}\smalltext{2}\smalltext{,}\smalltext{3}\smalltext{,}\smalltext{4}\smalltext{,}\smalltext{5}\smalltext{,}\smalltext{6}\smalltext{,}\smalltext{7}\smalltext{,}\smalltext{8}\smalltext{,}\smalltext{9}}}+\frac{\bar{c}^1_{\eps_{\smalltext{1}\smalltext{,}\smalltext{2}\smalltext{,}\smalltext{3}\smalltext{,}\smalltext{4}\smalltext{,}\smalltext{5}\smalltext{,}\smalltext{6}\smalltext{,}\smalltext{7}\smalltext{,}\smalltext{8}\smalltext{,}\smalltext{9}}}}{N} \sum_{\ell=1}^N \| X^{\ell}_{u}(\omega)\|^{2} \Bigg)NR_N^2+\Bigg({c}^2_{\eps_{\smalltext{1}\smalltext{,}\smalltext{2}\smalltext{,}\smalltext{3}\smalltext{,}\smalltext{4}\smalltext{,}\smalltext{5}\smalltext{,}\smalltext{6}\smalltext{,}\smalltext{7}\smalltext{,}\smalltext{8}\smalltext{,}\smalltext{9}}}+\frac{\bar{c}^2_{\eps_{\smalltext{1}\smalltext{,}\smalltext{2}\smalltext{,}\smalltext{3}\smalltext{,}\smalltext{4}\smalltext{,}\smalltext{5}\smalltext{,}\smalltext{6}\smalltext{,}\smalltext{7}\smalltext{,}\smalltext{8}\smalltext{,}\smalltext{9}}}}{N} \sum_{\ell=1}^N \| X^{\ell}_{u}(\omega)\|^{2\bar{p}} \Bigg)N^2R_N^2,
	\end{align*}
	where
	\begin{gather*}
	{c}^1_{\eps_{\smalltext{1}\smalltext{,}\smalltext{2}\smalltext{,}\smalltext{3}\smalltext{,}\smalltext{4}\smalltext{,}\smalltext{5}\smalltext{,}\smalltext{6}\smalltext{,}\smalltext{7}\smalltext{,}\smalltext{8}\smalltext{,}\smalltext{9}}}\coloneqq {c}^1_{\eps_{\smalltext{1}\smalltext{,}\smalltext{2}\smalltext{,}\smalltext{3}\smalltext{,}\smalltext{4}\smalltext{,}\smalltext{5}\smalltext{,}\smalltext{6}}} + \big( {c}^2_{\eps_{\smalltext{1}\smalltext{,}\smalltext{2}\smalltext{,}\smalltext{3}\smalltext{,}\smalltext{4}\smalltext{,}\smalltext{5}\smalltext{,}\smalltext{6}}} + {c}^3_{\eps_{\smalltext{1}\smalltext{,}\smalltext{2}\smalltext{,}\smalltext{3}\smalltext{,}\smalltext{4}\smalltext{,}\smalltext{5}\smalltext{,}\smalltext{6}}} \big) \bar{c}^2_{2} , \; \bar{c}^1_{\eps_{\smalltext{1}\smalltext{,}\smalltext{2}\smalltext{,}\smalltext{3}\smalltext{,}\smalltext{4}\smalltext{,}\smalltext{5}\smalltext{,}\smalltext{6}\smalltext{,}\smalltext{7}\smalltext{,}\smalltext{8}\smalltext{,}\smalltext{9}}}\coloneqq \big( {c}^2_{\eps_{\smalltext{1}\smalltext{,}\smalltext{2}\smalltext{,}\smalltext{3}\smalltext{,}\smalltext{4}\smalltext{,}\smalltext{5}\smalltext{,}\smalltext{6}}} + {c}^3_{\eps_{\smalltext{1}\smalltext{,}\smalltext{2}\smalltext{,}\smalltext{3}\smalltext{,}\smalltext{4}\smalltext{,}\smalltext{5}\smalltext{,}\smalltext{6}}} \big) {c}^2_{2}, \\
	{c}^2_{\eps_{\smalltext{1}\smalltext{,}\smalltext{2}\smalltext{,}\smalltext{3}\smalltext{,}\smalltext{4}\smalltext{,}\smalltext{5}\smalltext{,}\smalltext{6}\smalltext{,}\smalltext{7}\smalltext{,}\smalltext{8}\smalltext{,}\smalltext{9}}}\coloneqq \mathrm{e}^{\beta T}c_{\varphi_\smalltext{2}}{c}^6_{\eps_{\smalltext{1}\smalltext{,}\smalltext{2}\smalltext{,}\smalltext{3}\smalltext{,}\smalltext{4}\smalltext{,}\smalltext{5}\smalltext{,}\smalltext{6}}} + \big({c}^4_{\eps_{\smalltext{1}\smalltext{,}\smalltext{2}\smalltext{,}\smalltext{3}\smalltext{,}\smalltext{4}\smalltext{,}\smalltext{5}\smalltext{,}\smalltext{6}}}+{c}^5_{\eps_{\smalltext{1}\smalltext{,}\smalltext{2}\smalltext{,}\smalltext{3}\smalltext{,}\smalltext{4}\smalltext{,}\smalltext{5}\smalltext{,}\smalltext{6}}}\big) \big( \mathrm{e}^{\beta T}c_{\varphi_\smalltext{1}}+\bar{c}^1_{\eps_{{\smalltext{7}\smalltext{,}\smalltext{8}\smalltext{,}\smalltext{9}}}} \big) , \; \bar{c}^2_{\eps_{\smalltext{1}\smalltext{,}\smalltext{2}\smalltext{,}\smalltext{3}\smalltext{,}\smalltext{4}\smalltext{,}\smalltext{5}\smalltext{,}\smalltext{6}\smalltext{,}\smalltext{7}\smalltext{,}\smalltext{8}\smalltext{,}\smalltext{9}}}\coloneqq 2 \big({c}^4_{\eps_{\smalltext{1}\smalltext{,}\smalltext{2}\smalltext{,}\smalltext{3}\smalltext{,}\smalltext{4}\smalltext{,}\smalltext{5}\smalltext{,}\smalltext{6}}}+{c}^5_{\eps_{\smalltext{1}\smalltext{,}\smalltext{2}\smalltext{,}\smalltext{3}\smalltext{,}\smalltext{4}\smalltext{,}\smalltext{5}\smalltext{,}\smalltext{6}}}\big) \bar{c}^2_{\eps_{{\smalltext{7}\smalltext{,}\smalltext{8}\smalltext{,}\smalltext{9}}}}.
	\end{gather*}
And thus, for $\P\text{\rm--a.e.} \; \omega\in\Omega$,
	\begin{align*}
&\E^{\P^{\smalltext{\hat\balpha}^\tinytext{N}\smalltext{,}\smalltext{N}\smalltext{,}\smalltext{u}}_\smalltext{\omega}} \Bigg[ \int_u^T \mathrm{e}^{\beta s} \sum_{(k,\ell) \in \{1,\ldots,N\}^\smalltext{2}} \Big( \big\|\delta Z^{k,\ell,N}_s\big\|^2 + \big\|\delta Z^{k,m,\ell,\star,N}_s \big\|^2 \Big) \d s \Bigg]\\
&\leq \Bigg({c}^3_{\eps_{\smalltext{1}\smalltext{,}\smalltext{2}\smalltext{,}\smalltext{3}\smalltext{,}\smalltext{4}\smalltext{,}\smalltext{5}\smalltext{,}\smalltext{6}\smalltext{,}\smalltext{7}\smalltext{,}\smalltext{8}\smalltext{,}\smalltext{9}}}+\frac{\bar{c}^3_{\eps_{\smalltext{1}\smalltext{,}\smalltext{2}\smalltext{,}\smalltext{3}\smalltext{,}\smalltext{4}\smalltext{,}\smalltext{5}\smalltext{,}\smalltext{6}\smalltext{,}\smalltext{7}\smalltext{,}\smalltext{8}\smalltext{,}\smalltext{9}}}}{N} \sum_{\ell=1}^N \| X^{\ell}_{u}(\omega)\|^{2} \Bigg)NR_N^2+\Bigg({c}^4_{\eps_{\smalltext{1}\smalltext{,}\smalltext{2}\smalltext{,}\smalltext{3}\smalltext{,}\smalltext{4}\smalltext{,}\smalltext{5}\smalltext{,}\smalltext{6}\smalltext{,}\smalltext{7}\smalltext{,}\smalltext{8}\smalltext{,}\smalltext{9}}}+\frac{\bar{c}^4_{\eps_{\smalltext{1}\smalltext{,}\smalltext{2}\smalltext{,}\smalltext{3}\smalltext{,}\smalltext{4}\smalltext{,}\smalltext{5}\smalltext{,}\smalltext{6}\smalltext{,}\smalltext{7}\smalltext{,}\smalltext{8}\smalltext{,}\smalltext{9}}}}{N} \sum_{\ell=1}^N \| X^{\ell,N}_{u}(\omega)\|^{2\bar{p}} \Bigg)N^2R_N^2,
	\end{align*}
	where, for $i\in\{3,4\}$,
	\begin{gather*}
	{c}^i_{\eps_{\smalltext{1}\smalltext{,}\smalltext{2}\smalltext{,}\smalltext{3}\smalltext{,}\smalltext{4}\smalltext{,}\smalltext{5}\smalltext{,}\smalltext{6}\smalltext{,}\smalltext{7}\smalltext{,}\smalltext{8}\smalltext{,}\smalltext{9}}}\coloneqq \frac{{c}^{i-2}_{\eps_{\smalltext{1}\smalltext{,}\smalltext{2}\smalltext{,}\smalltext{3}\smalltext{,}\smalltext{4}\smalltext{,}\smalltext{5}\smalltext{,}\smalltext{6}\smalltext{,}\smalltext{7}\smalltext{,}\smalltext{8}\smalltext{,}\smalltext{9}}}}{1-{c}^8_{\eps_{\smalltext{1}\smalltext{,}\smalltext{2}\smalltext{,}\smalltext{3}\smalltext{,}\smalltext{4}\smalltext{,}\smalltext{5}\smalltext{,}\smalltext{6}}} -{c}^7_{\eps_{\smalltext{1}\smalltext{,}\smalltext{2}\smalltext{,}\smalltext{3}\smalltext{,}\smalltext{4}\smalltext{,}\smalltext{5}\smalltext{,}\smalltext{6}}} - (c^4_{\eps_{\smalltext{1}\smalltext{,}\smalltext{2}\smalltext{,}\smalltext{3}\smalltext{,}\smalltext{4}\smalltext{,}\smalltext{5}\smalltext{,}\smalltext{6}}} +c^5_{\eps_{\smalltext{1}\smalltext{,}\smalltext{2}\smalltext{,}\smalltext{3}\smalltext{,}\smalltext{4}\smalltext{,}\smalltext{5}\smalltext{,}\smalltext{6}}})NR_N^2 },\\	
	\bar{c}^i_{\eps_{\smalltext{1}\smalltext{,}\smalltext{2}\smalltext{,}\smalltext{3}\smalltext{,}\smalltext{4}\smalltext{,}\smalltext{5}\smalltext{,}\smalltext{6}\smalltext{,}\smalltext{7}\smalltext{,}\smalltext{8}\smalltext{,}\smalltext{9}}}\coloneqq \frac{\bar{c}^{i-2}_{\eps_{\smalltext{1}\smalltext{,}\smalltext{2}\smalltext{,}\smalltext{3}\smalltext{,}\smalltext{4}\smalltext{,}\smalltext{5}\smalltext{,}\smalltext{6}\smalltext{,}\smalltext{7}\smalltext{,}\smalltext{8}\smalltext{,}\smalltext{9}}}}{1-{c}^8_{\eps_{\smalltext{1}\smalltext{,}\smalltext{2}\smalltext{,}\smalltext{3}\smalltext{,}\smalltext{4}\smalltext{,}\smalltext{5}\smalltext{,}\smalltext{6}}} -{c}^7_{\eps_{\smalltext{1}\smalltext{,}\smalltext{2}\smalltext{,}\smalltext{3}\smalltext{,}\smalltext{4}\smalltext{,}\smalltext{5}\smalltext{,}\smalltext{6}}} - (c^4_{\eps_{\smalltext{1}\smalltext{,}\smalltext{2}\smalltext{,}\smalltext{3}\smalltext{,}\smalltext{4}\smalltext{,}\smalltext{5}\smalltext{,}\smalltext{6}}} +c^5_{\eps_{\smalltext{1}\smalltext{,}\smalltext{2}\smalltext{,}\smalltext{3}\smalltext{,}\smalltext{4}\smalltext{,}\smalltext{5}\smalltext{,}\smalltext{6}}})NR_N^2 }.
	\end{gather*}
Now, we apply all of the above in \Cref{eq:estim678} to deduce that, for $\P\text{\rm--a.e.} \; \omega\in\Omega$,
\begin{align*}
\notag&\E^{\P^{\smalltext{\hat\balpha}^\tinytext{N}\smalltext{,}\smalltext{N}\smalltext{,}\smalltext{u}}_\smalltext{\omega}} \Bigg[ \bigg( 1 - \varepsilon_3 4c^2_{1,\smallertext{\rm BDG}} - \bigg(\varepsilon_2 +\eps_4+\frac{\eps_2}{\eps_3}\bigg) c_{\smallertext{\rm{BMO}}_{\smalltext{[}\smalltext{u}\smalltext{,}\smalltext{T}\smalltext{]}}} \bigg) \sup_{t \in [u,T]} \mathrm{e}^{\beta t} \big|\delta Y^{i,N}_t\big|^2   + \sup_{t \in [u,T]} \mathrm{e}^{\beta t} \big|\delta M^{i,\star,N}_t\big|^2 + \sup_{t \in [u,T]} \mathrm{e}^{\beta t} \big|\delta N^{\star,N}_t\big|^2 \Bigg] \\
\notag&\quad +\big(1-{c}^8_{\eps_{\smalltext{1}\smalltext{,}\smalltext{2}\smalltext{,}\smalltext{3}\smalltext{,}\smalltext{4}\smalltext{,}\smalltext{5}\smalltext{,}\smalltext{6}}}-{c}^4_{\eps_{\smalltext{1}\smalltext{,}\smalltext{2}\smalltext{,}\smalltext{3}\smalltext{,}\smalltext{4}\smalltext{,}\smalltext{5}\smalltext{,}\smalltext{6}}} NR_N^2\big) \E^{\P^{\smalltext{\hat\balpha}^\tinytext{N}\smalltext{,}\smalltext{N}\smalltext{,}\smalltext{u}}_\smalltext{\omega}} \Bigg[ \int_u^T \mathrm{e}^{\beta t} \sum_{\ell=1}^N \Big( \big\|\delta Z^{i,\ell,N}_t\big\|^2 + \big\|\delta Z^{i,m,\ell,\star,N}_t\big\|^2 \Big) \d t \Bigg] \\
\notag&\quad +\bigg(1-{c}^8_{\eps_{\smalltext{1}\smalltext{,}\smalltext{2}\smalltext{,}\smalltext{3}\smalltext{,}\smalltext{4}\smalltext{,}\smalltext{5}\smalltext{,}\smalltext{6}}}-\frac{c^7_{\eps_{\smalltext{1}\smalltext{,}\smalltext{2}\smalltext{,}\smalltext{3}\smalltext{,}\smalltext{4}\smalltext{,}\smalltext{5}\smalltext{,}\smalltext{6}}}}{N}-c^6_{\eps_{\smalltext{1}\smalltext{,}\smalltext{2}\smalltext{,}\smalltext{3}\smalltext{,}\smalltext{4}\smalltext{,}\smalltext{5}\smalltext{,}\smalltext{6}}} N R_N^2\bigg) \E^{\P^{\smalltext{\hat\balpha}^\tinytext{N}\smalltext{,}\smalltext{N}\smalltext{,}\smalltext{u}}_\smalltext{\omega}} \Bigg[ \int_u^T \mathrm{e}^{\beta t} \sum_{\ell=1}^N \big\|\delta Z^{n,\ell,\star,N}_t\big\|^2 \d t \Bigg] \\
	\notag&\leq  {c}^1_{\eps_{\smalltext{1}\smalltext{,}\smalltext{2}\smalltext{,}\smalltext{3}\smalltext{,}\smalltext{4}\smalltext{,}\smalltext{5}\smalltext{,}\smalltext{6}}}R_N^2+{c}^2_{\eps_{\smalltext{1}\smalltext{,}\smalltext{2}\smalltext{,}\smalltext{3}\smalltext{,}\smalltext{4}\smalltext{,}\smalltext{5}\smalltext{,}\smalltext{6}}}R_N^2\big( c^2_{2}  \big\|X^{i}_u(\omega)\big\|^2 + \bar{c}^2_{2} \big) + {c}^3_{\eps_{\smalltext{1}\smalltext{,}\smalltext{2}\smalltext{,}\smalltext{3}\smalltext{,}\smalltext{4}\smalltext{,}\smalltext{5}\smalltext{,}\smalltext{6}}}{R_N^2}\Bigg( \frac{c^2_{2}}{N} \sum_{\ell=1}^N \big\|X^{\ell}_u(\omega)\big\|^2 + \bar{c}^2_{2} \Bigg)\\
	\notag&\quad+NR_N^2\Bigg({c}^4_{\eps_{\smalltext{1}\smalltext{,}\smalltext{2}\smalltext{,}\smalltext{3}\smalltext{,}\smalltext{4}\smalltext{,}\smalltext{5}\smalltext{,}\smalltext{6}}}\Bigg(\mathrm{e}^{\beta T}c_{\varphi_\smalltext{1}}+ \bar{c}^1_{\eps_{{\smalltext{7}\smalltext{,}\smalltext{8}\smalltext{,}\smalltext{9}}}}+\bar{c}^2_{\eps_{{\smalltext{7}\smalltext{,}\smalltext{8}\smalltext{,}\smalltext{9}}}}\Bigg(\big\|X^{i}_u(\omega)\big\|^{2\bar{p}}+ \frac{1}{N} \sum_{\ell =1}^N \| X^{\ell}_{u}(\omega)\|^{2\bar{p}} \Bigg)\Bigg)+ \mathrm{e}^{\beta T}{c}^6_{\eps_{\smalltext{1}\smalltext{,}\smalltext{2}\smalltext{,}\smalltext{3}\smalltext{,}\smalltext{4}\smalltext{,}\smalltext{5}\smalltext{,}\smalltext{6}}}c_{\varphi_\smalltext{2}}\Bigg)\\
	\notag&\quad+{c}^5_{\eps_{\smalltext{1}\smalltext{,}\smalltext{2}\smalltext{,}\smalltext{3}\smalltext{,}\smalltext{4}\smalltext{,}\smalltext{5}\smalltext{,}\smalltext{6}}}NR_N^2\Bigg(\bar{c}^1_{\eps_{{\smalltext{7}\smalltext{,}\smalltext{8}\smalltext{,}\smalltext{9}}}}+2\frac{\bar{c}^2_{\eps_{{\smalltext{7}\smalltext{,}\smalltext{8}\smalltext{,}\smalltext{9}}}}}{N}\sum_{\ell=1}^N \| X^{\ell}_{u}(\omega)\|^{2\bar{p}} +\mathrm{e}^{\beta T}c_{\varphi_\smalltext{1}}+\Bigg({c}^3_{\eps_{\smalltext{1}\smalltext{,}\smalltext{2}\smalltext{,}\smalltext{3}\smalltext{,}\smalltext{4}\smalltext{,}\smalltext{5}\smalltext{,}\smalltext{6}\smalltext{,}\smalltext{7}\smalltext{,}\smalltext{8}\smalltext{,}\smalltext{9}}}+\frac{\bar{c}^3_{\eps_{\smalltext{1}\smalltext{,}\smalltext{2}\smalltext{,}\smalltext{3}\smalltext{,}\smalltext{4}\smalltext{,}\smalltext{5}\smalltext{,}\smalltext{6}\smalltext{,}\smalltext{7}\smalltext{,}\smalltext{8}\smalltext{,}\smalltext{9}}}}{N} \sum_{\ell=1}^N \| X^{\ell}_{u}(\omega)\|^{2} \Bigg)R_N^2\Bigg)\\
	&\quad+{c}^5_{\eps_{\smalltext{1}\smalltext{,}\smalltext{2}\smalltext{,}\smalltext{3}\smalltext{,}\smalltext{4}\smalltext{,}\smalltext{5}\smalltext{,}\smalltext{6}}}N^2R_N^4\Bigg({c}^4_{\eps_{\smalltext{1}\smalltext{,}\smalltext{2}\smalltext{,}\smalltext{3}\smalltext{,}\smalltext{4}\smalltext{,}\smalltext{5}\smalltext{,}\smalltext{6}\smalltext{,}\smalltext{7}\smalltext{,}\smalltext{8}\smalltext{,}\smalltext{9}}}+\frac{\bar{c}^4_{\eps_{\smalltext{1}\smalltext{,}\smalltext{2}\smalltext{,}\smalltext{3}\smalltext{,}\smalltext{4}\smalltext{,}\smalltext{5}\smalltext{,}\smalltext{6}\smalltext{,}\smalltext{7}\smalltext{,}\smalltext{8}\smalltext{,}\smalltext{9}}}}{N} \sum_{\ell=1}^N \| X^{\ell}_{u}(\omega)\|^{2\bar{p}} \Bigg)\\
	&\quad+{c}^7_{\eps_{\smalltext{1}\smalltext{,}\smalltext{2}\smalltext{,}\smalltext{3}\smalltext{,}\smalltext{4}\smalltext{,}\smalltext{5}\smalltext{,}\smalltext{6}}}R_N^2\Bigg( {c}^3_{\eps_{\smalltext{1}\smalltext{,}\smalltext{2}\smalltext{,}\smalltext{3}\smalltext{,}\smalltext{4}\smalltext{,}\smalltext{5}\smalltext{,}\smalltext{6}\smalltext{,}\smalltext{7}\smalltext{,}\smalltext{8}\smalltext{,}\smalltext{9}}}+\frac{\bar{c}^3_{\eps_{\smalltext{1}\smalltext{,}\smalltext{2}\smalltext{,}\smalltext{3}\smalltext{,}\smalltext{4}\smalltext{,}\smalltext{5}\smalltext{,}\smalltext{6}\smalltext{,}\smalltext{7}\smalltext{,}\smalltext{8}\smalltext{,}\smalltext{9}}}}{N} \sum_{\ell=1}^N \| X^{\ell}_{u}(\omega)\|^{2} +\Bigg({c}^4_{\eps_{\smalltext{1}\smalltext{,}\smalltext{2}\smalltext{,}\smalltext{3}\smalltext{,}\smalltext{4}\smalltext{,}\smalltext{5}\smalltext{,}\smalltext{6}\smalltext{,}\smalltext{7}\smalltext{,}\smalltext{8}\smalltext{,}\smalltext{9}}}+\frac{\bar{c}^4_{\eps_{\smalltext{1}\smalltext{,}\smalltext{2}\smalltext{,}\smalltext{3}\smalltext{,}\smalltext{4}\smalltext{,}\smalltext{5}\smalltext{,}\smalltext{6}\smalltext{,}\smalltext{7}\smalltext{,}\smalltext{8}\smalltext{,}\smalltext{9}}}}{N} \sum_{\ell=1}^N \| X^{\ell}_{u}(\omega)\|^{2\bar{p}} \Bigg)N\Bigg)\\
	&=R_N^2\Bigg({c}^5_{\eps_{\smalltext{1}\smalltext{,}\smalltext{2}\smalltext{,}\smalltext{3}\smalltext{,}\smalltext{4}\smalltext{,}\smalltext{5}\smalltext{,}\smalltext{6}\smalltext{,}\smalltext{7}\smalltext{,}\smalltext{8}\smalltext{,}\smalltext{9}}}+\bar{c}^5_{\eps_{\smalltext{1}\smalltext{,}\smalltext{2}\smalltext{,}\smalltext{3}\smalltext{,}\smalltext{4}\smalltext{,}\smalltext{5}\smalltext{,}\smalltext{6}\smalltext{,}\smalltext{7}\smalltext{,}\smalltext{8}\smalltext{,}\smalltext{9}}}\| X^{i}_{u}(\omega)\|^{2}+\frac{\tilde{c}^5_{\eps_{\smalltext{1}\smalltext{,}\smalltext{2}\smalltext{,}\smalltext{3}\smalltext{,}\smalltext{4}\smalltext{,}\smalltext{5}\smalltext{,}\smalltext{6}\smalltext{,}\smalltext{7}\smalltext{,}\smalltext{8}\smalltext{,}\smalltext{9}}}}{N}\sum_{\ell=1}^N\| X^{\ell}_{u}(\omega)\|^{2}\Bigg)\\
	&\quad+NR_N^2\Bigg({c}^6_{\eps_{\smalltext{1}\smalltext{,}\smalltext{2}\smalltext{,}\smalltext{3}\smalltext{,}\smalltext{4}\smalltext{,}\smalltext{5}\smalltext{,}\smalltext{6}\smalltext{,}\smalltext{7}\smalltext{,}\smalltext{8}\smalltext{,}\smalltext{9}}}+\bar{c}^6_{\eps_{\smalltext{1}\smalltext{,}\smalltext{2}\smalltext{,}\smalltext{3}\smalltext{,}\smalltext{4}\smalltext{,}\smalltext{5}\smalltext{,}\smalltext{6}\smalltext{,}\smalltext{7}\smalltext{,}\smalltext{8}\smalltext{,}\smalltext{9}}}\| X^{i}_{u}(\omega)\|^{2\bar{p}}+\frac{\tilde{c}^6_{\eps_{\smalltext{1}\smalltext{,}\smalltext{2}\smalltext{,}\smalltext{3}\smalltext{,}\smalltext{4}\smalltext{,}\smalltext{5}\smalltext{,}\smalltext{6}\smalltext{,}\smalltext{7}\smalltext{,}\smalltext{8}\smalltext{,}\smalltext{9}}}}{N}\sum_{\ell=1}^N\| X^{\ell}_{u}(\omega)\|^{2\bar{p}}\Bigg)\\
	&\quad+NR_N^4\Bigg({c}^7_{\eps_{\smalltext{1}\smalltext{,}\smalltext{2}\smalltext{,}\smalltext{3}\smalltext{,}\smalltext{4}\smalltext{,}\smalltext{5}\smalltext{,}\smalltext{6}\smalltext{,}\smalltext{7}\smalltext{,}\smalltext{8}\smalltext{,}\smalltext{9}}}+\frac{\tilde{c}^7_{\eps_{\smalltext{1}\smalltext{,}\smalltext{2}\smalltext{,}\smalltext{3}\smalltext{,}\smalltext{4}\smalltext{,}\smalltext{5}\smalltext{,}\smalltext{6}\smalltext{,}\smalltext{7}\smalltext{,}\smalltext{8}\smalltext{,}\smalltext{9}}}}{N}\sum_{\ell=1}^N\| X^{\ell}_{u}(\omega)\|^{2}\Bigg)+N^2R_N^4\Bigg({c}^8_{\eps_{\smalltext{1}\smalltext{,}\smalltext{2}\smalltext{,}\smalltext{3}\smalltext{,}\smalltext{4}\smalltext{,}\smalltext{5}\smalltext{,}\smalltext{6}\smalltext{,}\smalltext{7}\smalltext{,}\smalltext{8}\smalltext{,}\smalltext{9}}}+\frac{\tilde{c}^8_{\eps_{\smalltext{1}\smalltext{,}\smalltext{2}\smalltext{,}\smalltext{3}\smalltext{,}\smalltext{4}\smalltext{,}\smalltext{5}\smalltext{,}\smalltext{6}\smalltext{,}\smalltext{7}\smalltext{,}\smalltext{8}\smalltext{,}\smalltext{9}}}}{N}\sum_{\ell=1}^N\| X^{\ell}_{u}(\omega)\|^{2}\Bigg), 
	\end{align*}
	where
	\begin{gather*}
	{c}^5_{\eps_{\smalltext{1}\smalltext{,}\smalltext{2}\smalltext{,}\smalltext{3}\smalltext{,}\smalltext{4}\smalltext{,}\smalltext{5}\smalltext{,}\smalltext{6}\smalltext{,}\smalltext{7}\smalltext{,}\smalltext{8}\smalltext{,}\smalltext{9}}} \coloneqq {c}^1_{\eps_{\smalltext{1}\smalltext{,}\smalltext{2}\smalltext{,}\smalltext{3}\smalltext{,}\smalltext{4}\smalltext{,}\smalltext{5}\smalltext{,}\smalltext{6}}} + \big( {c}^2_{\eps_{\smalltext{1}\smalltext{,}\smalltext{2}\smalltext{,}\smalltext{3}\smalltext{,}\smalltext{4}\smalltext{,}\smalltext{5}\smalltext{,}\smalltext{6}}} + {c}^3_{\eps_{\smalltext{1}\smalltext{,}\smalltext{2}\smalltext{,}\smalltext{3}\smalltext{,}\smalltext{4}\smalltext{,}\smalltext{5}\smalltext{,}\smalltext{6}}} \big) \bar{c}^2_{2} + {c}^7_{\eps_{\smalltext{1}\smalltext{,}\smalltext{2}\smalltext{,}\smalltext{3}\smalltext{,}\smalltext{4}\smalltext{,}\smalltext{5}\smalltext{,}\smalltext{6}}} {c}^3_{\eps_{\smalltext{1}\smalltext{,}\smalltext{2}\smalltext{,}\smalltext{3}\smalltext{,}\smalltext{4}\smalltext{,}\smalltext{5}\smalltext{,}\smalltext{6}\smalltext{,}\smalltext{7}\smalltext{,}\smalltext{8}\smalltext{,}\smalltext{9}}}, \;
	\bar{c}^5_{\eps_{\smalltext{1}\smalltext{,}\smalltext{2}\smalltext{,}\smalltext{3}\smalltext{,}\smalltext{4}\smalltext{,}\smalltext{5}\smalltext{,}\smalltext{6}\smalltext{,}\smalltext{7}\smalltext{,}\smalltext{8}\smalltext{,}\smalltext{9}}} \coloneqq {c}^2_{\eps_{\smalltext{1}\smalltext{,}\smalltext{2}\smalltext{,}\smalltext{3}\smalltext{,}\smalltext{4}\smalltext{,}\smalltext{5}\smalltext{,}\smalltext{6}}} {c}^2_{2} , \\
	\tilde{c}^5_{\eps_{\smalltext{1}\smalltext{,}\smalltext{2}\smalltext{,}\smalltext{3}\smalltext{,}\smalltext{4}\smalltext{,}\smalltext{5}\smalltext{,}\smalltext{6}\smalltext{,}\smalltext{7}\smalltext{,}\smalltext{8}\smalltext{,}\smalltext{9}}} \coloneqq {c}^3_{\eps_{\smalltext{1}\smalltext{,}\smalltext{2}\smalltext{,}\smalltext{3}\smalltext{,}\smalltext{4}\smalltext{,}\smalltext{5}\smalltext{,}\smalltext{6}}} {c}^2_{2} + {c}^7_{\eps_{\smalltext{1}\smalltext{,}\smalltext{2}\smalltext{,}\smalltext{3}\smalltext{,}\smalltext{4}\smalltext{,}\smalltext{5}\smalltext{,}\smalltext{6}}}  \bar{c}^3_{\eps_{\smalltext{1}\smalltext{,}\smalltext{2}\smalltext{,}\smalltext{3}\smalltext{,}\smalltext{4}\smalltext{,}\smalltext{5}\smalltext{,}\smalltext{6}\smalltext{,}\smalltext{7}\smalltext{,}\smalltext{8}\smalltext{,}\smalltext{9}}} , \\
	{c}^6_{\eps_{\smalltext{1}\smalltext{,}\smalltext{2}\smalltext{,}\smalltext{3}\smalltext{,}\smalltext{4}\smalltext{,}\smalltext{5}\smalltext{,}\smalltext{6}\smalltext{,}\smalltext{7}\smalltext{,}\smalltext{8}\smalltext{,}\smalltext{9}}} \coloneqq {c}^4_{\eps_{\smalltext{1}\smalltext{,}\smalltext{2}\smalltext{,}\smalltext{3}\smalltext{,}\smalltext{4}\smalltext{,}\smalltext{5}\smalltext{,}\smalltext{6}}} \big(\mathrm{e}^{\beta T}c_{\varphi_\smalltext{1}}+ \bar{c}^1_{\eps_{{\smalltext{7}\smalltext{,}\smalltext{8}\smalltext{,}\smalltext{9}}}} \big) + \mathrm{e}^{\beta T}{c}^6_{\eps_{\smalltext{1}\smalltext{,}\smalltext{2}\smalltext{,}\smalltext{3}\smalltext{,}\smalltext{4}\smalltext{,}\smalltext{5}\smalltext{,}\smalltext{6}}}c_{\varphi_\smalltext{2}} + {c}^5_{\eps_{\smalltext{1}\smalltext{,}\smalltext{2}\smalltext{,}\smalltext{3}\smalltext{,}\smalltext{4}\smalltext{,}\smalltext{5}\smalltext{,}\smalltext{6}}} \big( \bar{c}^1_{\eps_{{\smalltext{7}\smalltext{,}\smalltext{8}\smalltext{,}\smalltext{9}}}} +\mathrm{e}^{\beta T}c_{\varphi_\smalltext{1}} \big) + {c}^7_{\eps_{\smalltext{1}\smalltext{,}\smalltext{2}\smalltext{,}\smalltext{3}\smalltext{,}\smalltext{4}\smalltext{,}\smalltext{5}\smalltext{,}\smalltext{6}}} {c}^4_{\eps_{\smalltext{1}\smalltext{,}\smalltext{2}\smalltext{,}\smalltext{3}\smalltext{,}\smalltext{4}\smalltext{,}\smalltext{5}\smalltext{,}\smalltext{6}\smalltext{,}\smalltext{7}\smalltext{,}\smalltext{8}\smalltext{,}\smalltext{9}}} , \\
	\bar{c}^6_{\eps_{\smalltext{1}\smalltext{,}\smalltext{2}\smalltext{,}\smalltext{3}\smalltext{,}\smalltext{4}\smalltext{,}\smalltext{5}\smalltext{,}\smalltext{6}\smalltext{,}\smalltext{7}\smalltext{,}\smalltext{8}\smalltext{,}\smalltext{9}}} \coloneqq {c}^4_{\eps_{\smalltext{1}\smalltext{,}\smalltext{2}\smalltext{,}\smalltext{3}\smalltext{,}\smalltext{4}\smalltext{,}\smalltext{5}\smalltext{,}\smalltext{6}}} \bar{c}^2_{\eps_{{\smalltext{7}\smalltext{,}\smalltext{8}\smalltext{,}\smalltext{9}}}} , \\
	\tilde{c}^6_{\eps_{\smalltext{1}\smalltext{,}\smalltext{2}\smalltext{,}\smalltext{3}\smalltext{,}\smalltext{4}\smalltext{,}\smalltext{5}\smalltext{,}\smalltext{6}\smalltext{,}\smalltext{7}\smalltext{,}\smalltext{8}\smalltext{,}\smalltext{9}}} \coloneqq {c}^4_{\eps_{\smalltext{1}\smalltext{,}\smalltext{2}\smalltext{,}\smalltext{3}\smalltext{,}\smalltext{4}\smalltext{,}\smalltext{5}\smalltext{,}\smalltext{6}}} \bar{c}^2_{\eps_{{\smalltext{7}\smalltext{,}\smalltext{8}\smalltext{,}\smalltext{9}}}} + 2 {c}^5_{\eps_{\smalltext{1}\smalltext{,}\smalltext{2}\smalltext{,}\smalltext{3}\smalltext{,}\smalltext{4}\smalltext{,}\smalltext{5}\smalltext{,}\smalltext{6}}} \bar{c}^2_{\eps_{{\smalltext{7}\smalltext{,}\smalltext{8}\smalltext{,}\smalltext{9}}}} + {c}^7_{\eps_{\smalltext{1}\smalltext{,}\smalltext{2}\smalltext{,}\smalltext{3}\smalltext{,}\smalltext{4}\smalltext{,}\smalltext{5}\smalltext{,}\smalltext{6}}} \bar{c}^4_{\eps_{\smalltext{1}\smalltext{,}\smalltext{2}\smalltext{,}\smalltext{3}\smalltext{,}\smalltext{4}\smalltext{,}\smalltext{5}\smalltext{,}\smalltext{6}\smalltext{,}\smalltext{7}\smalltext{,}\smalltext{8}\smalltext{,}\smalltext{9}}} , \\
	{c}^7_{\eps_{\smalltext{1}\smalltext{,}\smalltext{2}\smalltext{,}\smalltext{3}\smalltext{,}\smalltext{4}\smalltext{,}\smalltext{5}\smalltext{,}\smalltext{6}\smalltext{,}\smalltext{7}\smalltext{,}\smalltext{8}\smalltext{,}\smalltext{9}}} \coloneqq {c}^5_{\eps_{\smalltext{1}\smalltext{,}\smalltext{2}\smalltext{,}\smalltext{3}\smalltext{,}\smalltext{4}\smalltext{,}\smalltext{5}\smalltext{,}\smalltext{6}}} {c}^3_{\eps_{\smalltext{1}\smalltext{,}\smalltext{2}\smalltext{,}\smalltext{3}\smalltext{,}\smalltext{4}\smalltext{,}\smalltext{5}\smalltext{,}\smalltext{6}\smalltext{,}\smalltext{7}\smalltext{,}\smalltext{8}\smalltext{,}\smalltext{9}}} , \;
	\tilde{c}^7_{\eps_{\smalltext{1}\smalltext{,}\smalltext{2}\smalltext{,}\smalltext{3}\smalltext{,}\smalltext{4}\smalltext{,}\smalltext{5}\smalltext{,}\smalltext{6}\smalltext{,}\smalltext{7}\smalltext{,}\smalltext{8}\smalltext{,}\smalltext{9}}} \coloneqq {c}^5_{\eps_{\smalltext{1}\smalltext{,}\smalltext{2}\smalltext{,}\smalltext{3}\smalltext{,}\smalltext{4}\smalltext{,}\smalltext{5}\smalltext{,}\smalltext{6}}} \bar{c}^3_{\eps_{\smalltext{1}\smalltext{,}\smalltext{2}\smalltext{,}\smalltext{3}\smalltext{,}\smalltext{4}\smalltext{,}\smalltext{5}\smalltext{,}\smalltext{6}\smalltext{,}\smalltext{7}\smalltext{,}\smalltext{8}\smalltext{,}\smalltext{9}}} , \\
	{c}^8_{\eps_{\smalltext{1}\smalltext{,}\smalltext{2}\smalltext{,}\smalltext{3}\smalltext{,}\smalltext{4}\smalltext{,}\smalltext{5}\smalltext{,}\smalltext{6}\smalltext{,}\smalltext{7}\smalltext{,}\smalltext{8}\smalltext{,}\smalltext{9}}} \coloneqq {c}^5_{\eps_{\smalltext{1}\smalltext{,}\smalltext{2}\smalltext{,}\smalltext{3}\smalltext{,}\smalltext{4}\smalltext{,}\smalltext{5}\smalltext{,}\smalltext{6}}} {c}^4_{\eps_{\smalltext{1}\smalltext{,}\smalltext{2}\smalltext{,}\smalltext{3}\smalltext{,}\smalltext{4}\smalltext{,}\smalltext{5}\smalltext{,}\smalltext{6}\smalltext{,}\smalltext{7}\smalltext{,}\smalltext{8}\smalltext{,}\smalltext{9}}} , \;
	\tilde{c}^8_{\eps_{\smalltext{1}\smalltext{,}\smalltext{2}\smalltext{,}\smalltext{3}\smalltext{,}\smalltext{4}\smalltext{,}\smalltext{5}\smalltext{,}\smalltext{6}\smalltext{,}\smalltext{7}\smalltext{,}\smalltext{8}\smalltext{,}\smalltext{9}}} \coloneqq {c}^5_{\eps_{\smalltext{1}\smalltext{,}\smalltext{2}\smalltext{,}\smalltext{3}\smalltext{,}\smalltext{4}\smalltext{,}\smalltext{5}\smalltext{,}\smalltext{6}}} \bar{c}^4_{\eps_{\smalltext{1}\smalltext{,}\smalltext{2}\smalltext{,}\smalltext{3}\smalltext{,}\smalltext{4}\smalltext{,}\smalltext{5}\smalltext{,}\smalltext{6}\smalltext{,}\smalltext{7}\smalltext{,}\smalltext{8}\smalltext{,}\smalltext{9}}} .
	\end{gather*}
	
\medskip 
To complete the first part of the proof, which concerns the convergence of the $N$-player system described in \eqref{align:systemNplayerGame_rcpd} to the intermediate system introduced in \eqref{align:intermediateSystemNplayerGame}, we now verify that all the conditions that have been either explicitly or implicitly used through all the previous steps are indeed satisfied. This can be achieved by appropriately choosing the parameters $\varepsilon_i>0$ for $i=1,\ldots,9$, and $\beta>0$, and by requiring that the dissipativity constant $K_{\sigma b}$ is sufficiently large, so that all of the following conditions are satisfied
\begin{gather}\label{1}
\beta \geq \max\bigg\{\frac{3 \ell^2_f}{\eps_1},\ell_f^2(1+c_A)^2+ 2\ell^2_f \bigg\},\\
\label{2}
K_{\sigma b}\geq \frac12\bigg(\beta+\ell^2_\sigma + \frac{2\ell^2_{\sigma b}}{\varepsilon_5} + \varepsilon_5 6 \ell_\Lambda^2\bigg),\\
\label{3}
1-\varepsilon_6 \mathrm{e}^{2\beta T}c^2_{1,{\smallertext{\rm BDG}}} \ell^2_\sigma>0,\\
\label{4}
1 - \varepsilon_3 4c^2_{1,\smallertext{\rm BDG}} - \bigg(\varepsilon_2 +\eps_4+\frac{\eps_2}{\eps_3}\bigg) c_{\smallertext{\rm{BMO}}_{\smalltext{[}\smalltext{u}\smalltext{,}\smalltext{T}\smalltext{]}}} >0,\\
\label{5}
1-{c}^8_{\eps_{\smalltext{1}\smalltext{,}\smalltext{2}\smalltext{,}\smalltext{3}\smalltext{,}\smalltext{4}\smalltext{,}\smalltext{5}\smalltext{,}\smalltext{6}}} -{c}^7_{\eps_{\smalltext{1}\smalltext{,}\smalltext{2}\smalltext{,}\smalltext{3}\smalltext{,}\smalltext{4}\smalltext{,}\smalltext{5}\smalltext{,}\smalltext{6}}} - (c^4_{\eps_{\smalltext{1}\smalltext{,}\smalltext{2}\smalltext{,}\smalltext{3}\smalltext{,}\smalltext{4}\smalltext{,}\smalltext{5}\smalltext{,}\smalltext{6}}} +c^5_{\eps_{\smalltext{1}\smalltext{,}\smalltext{2}\smalltext{,}\smalltext{3}\smalltext{,}\smalltext{4}\smalltext{,}\smalltext{5}\smalltext{,}\smalltext{6}}})NR_N^2>0,\\
\label{6}
1 - \varepsilon_7 4c^2_{1,\smallertext{\rm BDG}} - \varepsilon_8 c^\prime_{\smallertext{\rm{BMO}}_{\smalltext{[}\smalltext{u}\smalltext{,}\smalltext{T}\smalltext{]}}}>0,\\
\label{7}
1 - \frac{c_{\eps_{\smalltext{7}\smalltext{,}\smalltext{8}\smalltext{,}\smalltext{9}}}}{\eps_7}>0,\\
\label{8}
1-{c}^8_{\eps_{\smalltext{1}\smalltext{,}\smalltext{2}\smalltext{,}\smalltext{3}\smalltext{,}\smalltext{4}\smalltext{,}\smalltext{5}\smalltext{,}\smalltext{6}}}-\frac{c^7_{\eps_{\smalltext{1}\smalltext{,}\smalltext{2}\smalltext{,}\smalltext{3}\smalltext{,}\smalltext{4}\smalltext{,}\smalltext{5}\smalltext{,}\smalltext{6}}}}{N}-c^6_{\eps_{\smalltext{1}\smalltext{,}\smalltext{2}\smalltext{,}\smalltext{3}\smalltext{,}\smalltext{4}\smalltext{,}\smalltext{5}\smalltext{,}\smalltext{6}}} N R_N^2>0,
\end{gather}
where
\begin{align*}
c_{\eps_{\smalltext{7}\smalltext{,}\smalltext{8}\smalltext{,}\smalltext{9}}}&= \frac{ \varepsilon_9 c^\prime_{\smallertext{\rm{BMO}}_{\smalltext{[}\smalltext{u}\smalltext{,}\smalltext{T}\smalltext{]}}} }{1 - \varepsilon_7 4c^2_{1,\smallertext{\rm BDG}} - \varepsilon_8 c^\prime_{\smallertext{\rm{BMO}}_{\smalltext{[}\smalltext{u}\smalltext{,}\smalltext{T}\smalltext{]}}}},\\
c_{\eps_{\smalltext{2}\smalltext{,}\smalltext{3}\smalltext{,}\smalltext{4}}}&= \bigg(2+\frac1{\eps_3}\bigg) 2\ell_{g+G,\varphi_\smalltext{1},\varphi_\smalltext{2}}^2 + \bigg(\frac{1}{\varepsilon_2}+\frac1{\eps_4}+\frac{1}{\eps_2\eps_3}\bigg) \big(2\vee (\bar c_{\smallertext{\rm{BMO}}_{\smalltext{[}\smalltext{u}\smalltext{,}\smalltext{T}\smalltext{]}}}\ell^2_{\partial^\smalltext{2}G})\big) c^\star\ell^2_{\varphi_\smalltext{1}},\\
\overline{c}_{\eps_{\smalltext{2}\smalltext{,}\smalltext{3}\smalltext{,}\smalltext{4}}}&= \bigg(2+\frac1{\eps_3}\bigg) 2\ell_{g+G,\varphi_\smalltext{1},\varphi_\smalltext{2}}^2 + \bigg(\frac{1}{\varepsilon_2}+\frac1{\eps_4}+\frac{1}{\eps_2\eps_3}\bigg) \big(2\vee (\bar c_{\smallertext{\rm{BMO}}_{\smalltext{[}\smalltext{u}\smalltext{,}\smalltext{T}\smalltext{]}}}\ell^2_{\partial^\smalltext{2}G})\big) c^\star\ell^2_{\varphi_\smalltext{2}},\\
c_{\eps_{\smalltext{6}}}(T) &= \mathrm{exp}\bigg(\beta T+\frac{ T}{\varepsilon_6 (1-\varepsilon_6 \mathrm{e}^{2\beta T}c^2_{1,{\smallertext{\rm BDG}}} \ell^2_\sigma)}  \bigg)\big(1-\varepsilon_6 \mathrm{e}^{2\beta T}c^2_{1,{\smallertext{\rm BDG}}} \ell^2_\sigma\big)^{-1},\\
c_{\eps_{\smalltext{5}\smalltext{,}\smalltext{6}}}(T)&= \eps_5 12 \ell^2_\Lambda c_{\eps_{\smalltext{6}}}(T)\exp\big( \varepsilon_5 \big( 1 + 18 \ell^2_\Lambda \big) c_{\eps_{\smalltext{6}}}(T)T\big),\\
\bar{c}^1_{\eps_{\smalltext{5}\smalltext{,}\smalltext{6}}}(T)&= \eps_5 8 c_{\eps_\smalltext{6}}(T)\big(6\ell^2_\Lambda +(1+18\ell^2_\Lambda)Tc_{\eps_{\smalltext{5}\smalltext{,}\smalltext{6}}}(T)\big),\\
\bar{c}^2_{\eps_{\smalltext{5}\smalltext{,}\smalltext{6}}}(T)&= \eps_5 8 c_{\eps_\smalltext{6}}(t)\big(12\ell^2_\Lambda +(1+18\ell^2_\Lambda)Tc_{\eps_{\smalltext{5}\smalltext{,}\smalltext{6}}}(T)\big),\\
\bar{c}^3_{\eps_{\smalltext{5}\smalltext{,}\smalltext{6}}}(T) &= \varepsilon_5 6 \ell_\Lambda^2 c_{\eps_{\smalltext{6}}}(T),\\
{c}^4_{\eps_{\smalltext{1}\smalltext{,}\smalltext{2}\smalltext{,}\smalltext{3}\smalltext{,}\smalltext{4}\smalltext{,}\smalltext{5}\smalltext{,}\smalltext{6}}}&=\eps_1\bigg(2+\frac1{\eps_3}\bigg)\big(48\ell^2_\Lambda+(1+6\ell^2_\Lambda)8T\bar{c}^3_{\eps_{\smalltext{5}\smalltext{,}\smalltext{6}}}(T)\big)+8\bar{c}^3_{\eps_{\smalltext{5}\smalltext{,}\smalltext{6}}}(T)\big(c_{\eps_{\smalltext{2}\smalltext{,}\smalltext{3}\smalltext{,}\smalltext{4}}}+\ell^2_{\varphi_\smalltext{1}}c^\star\big),\\
{c}^5_{\eps_{\smalltext{1}\smalltext{,}\smalltext{2}\smalltext{,}\smalltext{3}\smalltext{,}\smalltext{4}\smalltext{,}\smalltext{5}\smalltext{,}\smalltext{6}}}&= \eps_1\bigg(2+\frac1{\eps_3}\bigg)\big(48\ell^2_\Lambda+(1+6\ell^2_\Lambda)T\bar{c}^1_{\eps_{\smalltext{5}\smalltext{,}\smalltext{6}}}(T)+(1+18\ell^2_\Lambda)8T{c}_{\eps_{\smalltext{5}\smalltext{,}\smalltext{6}}}(T)\big)+\big(c_{\eps_{\smalltext{2}\smalltext{,}\smalltext{3}\smalltext{,}\smalltext{4}}}+\ell^2_{\varphi_\smalltext{1}}c^\star\big)\bar{c}^1_{\eps_{\smalltext{5}\smalltext{,}\smalltext{6}}}(T)\\
&\quad+8\big(\bar{c}_{\eps_{\smalltext{2}\smalltext{,}\smalltext{3}\smalltext{,}\smalltext{4}}}+\ell^2_{\varphi_\smalltext{2}}c^\star\big){c}_{\eps_{\smalltext{5}\smalltext{,}\smalltext{6}}}(T),\\
{c}^6_{\eps_{\smalltext{1}\smalltext{,}\smalltext{2}\smalltext{,}\smalltext{3}\smalltext{,}\smalltext{4}\smalltext{,}\smalltext{5}\smalltext{,}\smalltext{6}}}&= \eps_1\bigg(2+\frac1{\eps_3}\bigg)\big(96\ell^2_\Lambda+(1+6\ell^2_\Lambda)T\bar{c}^2_{\eps_{\smalltext{5}\smalltext{,}\smalltext{6}}}(T)+(1+18\ell^2_\Lambda)8T{c}_{\eps_{\smalltext{5}\smalltext{,}\smalltext{6}}}(T)\big)+\big(c_{\eps_{\smalltext{2}\smalltext{,}\smalltext{3}\smalltext{,}\smalltext{4}}}+\ell^2_{\varphi_\smalltext{1}}c^\star\big)\bar{c}^2_{\eps_{\smalltext{5}\smalltext{,}\smalltext{6}}}(T)\\
&\quad+8\big(\bar{c}_{\eps_{\smalltext{2}\smalltext{,}\smalltext{3}\smalltext{,}\smalltext{4}}}+\ell^2_{\varphi_\smalltext{2}}c^\star\big){c}_{\eps_{\smalltext{5}\smalltext{,}\smalltext{6}}}(T),\\
{c}^7_{\eps_{\smalltext{1}\smalltext{,}\smalltext{2}\smalltext{,}\smalltext{3}\smalltext{,}\smalltext{4}\smalltext{,}\smalltext{5}\smalltext{,}\smalltext{6}}}&=\eps_1\bigg(2+\frac1{\eps_3}\bigg)\big(6\ell^2_\Lambda+(1+6\ell^2_\Lambda)T\bar{c}^1_{\eps_{\smalltext{5}\smalltext{,}\smalltext{6}}}(T)+(1+18\ell^2_\Lambda)T{c}_{\eps_{\smalltext{5}\smalltext{,}\smalltext{6}}}(T)\big)+\big(c_{\eps_{\smalltext{2}\smalltext{,}\smalltext{3}\smalltext{,}\smalltext{4}}}+\ell^2_{\varphi_\smalltext{1}}c^\star\big)\bar{c}^1_{\eps_{\smalltext{5}\smalltext{,}\smalltext{6}}}(T)\\
&\quad+\big(\bar{c}_{\eps_{\smalltext{2}\smalltext{,}\smalltext{3}\smalltext{,}\smalltext{4}}}+\ell^2_{\varphi_\smalltext{2}}c^\star\big){c}_{\eps_{\smalltext{5}\smalltext{,}\smalltext{6}}}(T), \\
{c}^8_{\eps_{\smalltext{1}\smalltext{,}\smalltext{2}\smalltext{,}\smalltext{3}\smalltext{,}\smalltext{4}\smalltext{,}\smalltext{5}\smalltext{,}\smalltext{6}}} &=  \eps_1\bigg(2+\frac1{\eps_3}\bigg)\big(6\ell^2_\Lambda+(1+6\ell^2_\Lambda) T \bar{c}^3_{\eps_{\smalltext{5}\smalltext{,}\smalltext{6}}}(T) + \varepsilon_5 6 \ell_\Lambda^2 c_{\eps_{\smalltext{6}}}(T) \big(c_{\eps_{\smalltext{2}\smalltext{,}\smalltext{3}\smalltext{,}\smalltext{4}}}+\ell^2_{\varphi_\smalltext{1}}c^\star\big).
\end{align*}
We observe that, for $i \in\{ 4,\ldots,8\}$, each constant ${c}^i_{\eps_{\smalltext{1}\smalltext{,}\smalltext{2}\smalltext{,}\smalltext{3}\smalltext{,}\smalltext{4}\smalltext{,}\smalltext{5}\smalltext{,}\smalltext{6}}}$ is a linear combination of the terms $\eps_1 f(\ell_\Lambda,T,\varepsilon_3,\varepsilon_4,\varepsilon_2,\varepsilon_5,\varepsilon_6)$ and $\eps_5 g(\ell_\Lambda, \ell^2_{\varphi_\smalltext{1}}, \ell^2_{\varphi_\smalltext{2}}, T, \varepsilon_3,\varepsilon_4,\varepsilon_2,\varepsilon_5,\varepsilon_6)$, for some well-defined positive functions $f:(\R^\star_+)^7 \longrightarrow \R^\star_+$ and $g:(\R^\star_+)^9 \longrightarrow \R^\star_+$. Moreover, the function $g$ satisfies
\begin{align*}
\lim_{\varepsilon_\smalltext{5} \rightarrow 0} g(\ell_\Lambda, \ell^2_{\varphi_\smalltext{1}}, \ell^2_{\varphi_\smalltext{2}}, T, \varepsilon_3,\varepsilon_4,\varepsilon_2,\varepsilon_5,\varepsilon_6) = \hat{g}(\ell_\Lambda, \ell_{\varphi_\smalltext{1}}, \ell_{\varphi_\smalltext{2}}, T, \varepsilon_3,\varepsilon_4,\varepsilon_2,\varepsilon_6) \in (0,\infty).
\end{align*}
Then, for all of this to work, we can do the following:
\begin{enumerate}[label={$(\roman*)$}]
\item start by fixing $\eps_6$ small enough for \Cref{3} to be satisfied;
\item second, fix $\eps_7$ and $\eps_8$ small enough for \Cref{6} to be satisfied;
\item we can then fix $\eps_9$ small enough for \Cref{7} to be satisfied;
\item next, fix first $\eps_2$ small enough, and subsequently $\eps_3$ and $\eps_4$ small enough for \Cref{4} to be satisfied;
\item afterwards, make $\eps_5$ and $\eps_1$ small enough for \Cref{5} and \Cref{8} to be satisfied;
\item finally, we fix $\beta$ large enough for \Cref{1} to be satisfied, and ultimately $K_{\sigma b}$ large enough for \Cref{2} to be satisfied.
\end{enumerate}

\medskip
We conclude that there exists some constant $C>0$, which we omit explicitly here for notational simplicity, since it depends on all the constants listed above, and thus only on the parameters of the game and not on $N$, such that
\begin{align}\label{align:firstPart_finalEst}
\notag&\E^{\P^{\smalltext{\hat\balpha}^\tinytext{N}\smalltext{,}\smalltext{N}\smalltext{,}\smalltext{u}}_\smalltext{\omega}} \bigg[ \sup_{t \in [u,T]} \mathrm{e}^{\beta t} \big|\delta Y^{i,N}_t\big|^2   + \sup_{t \in [u,T]} \mathrm{e}^{\beta t} \big|\delta M^{i,\star,N}_t\big|^2 + \sup_{t \in [u,T]} \mathrm{e}^{\beta t} \big|\delta N^{\star,N}_t\big|^2 \bigg] \\
\notag&\quad + \E^{\P^{\smalltext{\hat\balpha}^\tinytext{N}\smalltext{,}\smalltext{N}\smalltext{,}\smalltext{u}}_\smalltext{\omega}} \Bigg[ \int_u^T \mathrm{e}^{\beta t} \sum_{\ell=1}^N \Big( \big\|\delta Z^{i,\ell,N}_t\big\|^2 + \big\|\delta Z^{i,m,\ell,\star,N}_t\big\|^2 + \big\|\delta Z^{n,\ell,\star,N}_t\big\|^2\Big) \d t \Bigg] \\
	\notag&\leq C R_N^2\Bigg(1+\| X^{i}_{u}(\omega)\|^{2}+\frac1N\sum_{\ell =1}^N\| X^{\ell}_{u}(\omega)\|^{2}\Bigg)+C NR_N^2\Bigg(1+\| X^{i}_{u}(\omega)\|^{2\bar{p}}+\frac1N\sum_{\ell =1}^N\| X^{\ell}_{u}(\omega)\|^{2\bar{p}}\Bigg)\\
	&\quad+C NR_N^4\Bigg(1+\frac1N\sum_{\ell =1}^N\| X^{\ell}_{u}(\omega)\|^{2}\Bigg) (1+N), \; \P\text{\rm--a.e.} \; \omega\in\Omega.
\end{align}

The proof of the first part is thus complete, since for $p \in \{1,\bar{p}\}$, we have
\begin{align}\label{align:LLN}
\lim_{N \rightarrow \infty} \frac1N\sum_{\ell =1}^N\| X^{\ell}_{u}\|^{2p} \in [0,\infty), \; \P\text{\rm--a.s.}
\end{align}
This follows from the strong law of large numbers (see, for instance, \cite[Theorem 5.23]{kallenberg2002foundations}), since the sequence $(X^i_u)_{i \in \mathbb{N}}$ consists of $\P$--i.i.d. random variables, as stated in \Cref{assumpConvThm}.\ref{growth_f_gG}, so that
\begin{align}\label{align:LLN_withLIM}
\lim_{N \rightarrow \infty} \frac1N\sum_{\ell =1}^N \|X^{\ell}_{u}\|^{2p} = \E^\P \big[ \| X^{1}_{u} \|^{2p} \big] \in [0,+\infty), \; \P\text{\rm--a.s.}
\end{align}

\subsubsection{Convergence to infinitely many identical copies of the mean-field game}\label{convTOmeanFieldGames}

In \Cref{subsubsection:fromNplayerGameToAuxiliary}, we establish the convergence of the $N$-player game to the intermediate system defined in \eqref{align:intermediateSystemNplayerGame}. In this section, we show that this intermediate system converges to the mean-field system described in \eqref{align:systemMeanFieldGame_rcpd}, although not directly. More precisely, we introduce a second auxiliary FBSDE system, which coincides with the one in \eqref{align:intermediateSystemNplayerGame} except that it is not formulated under the probability measure $\P^{\hat\balpha^\smalltext{N},N}$ with the corresponding fixed Brownian motions $((W^{\hat\balpha^{\smalltext{N}},N})^i)_{i \{1,\ldots,N\}}$. Instead, we construct an FBSDE system defined under the probability measure $\P^{\hat\alpha,N}$ and driven by the Brownian motions $((W^{\hat\alpha,N})^i)_{i \{1,\ldots,N\}}$. For $\P\text{\rm--a.e.} \; \omega\in\Omega$,
\begin{align}\label{align:intermediateSystemNplayerGame_rcpdMeanField}
\notag \overline X^i_t & = X^i_u(\omega) + \int_u^t \sigma_s(\overline X^i_{\cdot \land s}) b_s\big(\overline X^i_{\cdot \land s},L^N\big(\overline \X^N_{\cdot \land s},\overline\balpha^N_s\big),\overline\alpha^{i,N}_s\big) \d s + \int_u^t \sigma_s(\overline X^i_{\cdot \land s}) \d \big(W_s^{\hat\alpha,N,u,\omega}\big)^i, \; t \in [u,T], \; \P^{\hat\alpha,N,u}_\omega \text{\rm--a.s.}, \\
\notag\overline Y^{i,N}_t & = g\big(\overline X^i_{\cdot \land T}, L^N\big(\overline \X^N_{\cdot \land T})\big) + G\big(\varphi_1(\overline Xi_{\cdot \land T}),\varphi_2\big(L^N\big(\overline \X^N_{\cdot \land T}\big)\big)\big) \\
\notag &\quad + \int_t^T f_s\big(\overline X^i_{\cdot \land s},L^N\big(\overline \X^N_{\cdot \land s},\overline\balpha^N_s\big),\overline\alpha^{i,N}_s\big) \d s - \int_t^T \partial^2_{m,n}G\big(\overline M^{i,\star,N}_s,\overline N^{\star,N}_s\big) \sum_{\ell =1}^N \overline Z^{i,m,\ell,\star,N}_s \cdot \overline Z^{n,\ell,\star,N}_s \d s \\
\notag &\quad - \frac{1}{2} \int_t^T \partial^2_{m,m}G\big(\overline M^{i,\star,N}_s,\overline N^{\star,N}_s\big) \sum_{\ell =1}^N \big\|\overline Z^{i,m,\ell,\star,N}_s\big\|^2 \d s - \frac{1}{2} \int_t^T \partial^2_{n,n}G\big(\overline M^{i,\star,N}_s,\overline N^{\star,N}_s\big) \sum_{\ell =1}^N \big\|\overline Z^{n,\ell,\star,N}_s\big\|^2 \d s \\
\notag &\quad - \int_t^T \sum_{\ell =1}^N \overline Z^{i,\ell,N}_s \cdot \d \big(W_s^{\hat\alpha,N,u,\omega}\big)^\ell, \; t \in [u,T], \; \P^{\hat\alpha,N,u}_\omega \text{\rm--a.s.}, \\
\notag \overline M^{i,\star,N}_t & = \varphi_1(\overline X^i_{\cdot \land T}) - \int_t^T \sum_{\ell =1}^N \overline Z^{i,m,\ell,\star,N}_s \cdot \d \big(W_s^{\hat\balpha,N,u,\omega}\big)^\ell, \; t \in [u,T], \; \P^{\hat\alpha,N,u}_\omega \text{\rm--a.s.}, \\
\notag\overline N^{\star,N}_t & = \varphi_2\big(L^N(\overline \X^N_{\cdot \land T})\big) - \int_t^T \sum_{\ell =1}^N \overline Z^{n,\ell,\star,N}_s \cdot \d \big(W_s^{\hat\alpha,N,u,\omega}\big)^\ell, \; t \in [u,T], \; \P^{\hat\alpha,N,u}_\omega \text{\rm--a.s.}, \\
\overline \alpha^{i,N}_t &\coloneqq \Lambda_t\big(\overline X^{i,N}_{\cdot \land t}, L^N(\overline \X^N_{\cdot \land t}), \overline Z^{i,i,N}_t, \overline Z^{i,m,i,\star,N}_t, \overline Z^{n,i,\star,N}_t, 0\big), \; \mathrm{d}t\otimes\P^{\hat\alpha,N,u}_\omega \text{\rm--a.e.}
\end{align}
\Cref{assumpConvThm}.\ref{auxiliarySystem} implies that a Yamada--Watanabe-type result holds (one may adapt, for instance, the argument used in the proof of \citeauthor*{carmona2018probabilisticII} \cite[Theorem 1.33]{carmona2018probabilisticII} to the non-Markovian setting). Namely, for any $u \in [0,T]$, for $\P\text{\rm--a.e.} \; \omega\in\Omega$, and for $\mathrm{d}t$--a.e. in $[u,T]$, we have
\begin{align}\label{YW_lawEq}
\notag&\P^{\hat\balpha^\smalltext{N},N,u}_\omega \circ \big(\widetilde X^{i}_t, \widetilde Y^{i,N}_t, \widetilde M^{i,\star,N}_t, \widetilde N^{\star,N}_t, \widetilde \Z^{i,N}_t, \widetilde \Z^{i,m,\star,N}_t, \widetilde \Z^{n,\star,N}_t \big)^{-1} \\
& = \P^{\hat\alpha,N,u}_\omega \circ \big(\overline{X}^{i}_t, \overline Y^{i,N}_t, \overline M^{i,\star,N}_t, \overline N^{\star,N}_t, \overline \Z^{i,N}_t, \overline \Z^{i,m,\star,N}_t, \overline \Z^{n,\star,N}_t \big)^{-1}.
\end{align}
Consequently
\begin{align*}
\widetilde Y^{i,N}_u(\omega) = \E^{\P^{\smalltext{\hat\balpha}^\tinytext{N}\smalltext{,}\smalltext{N}\smalltext{,}\smalltext{u}}_\smalltext{\omega}} \big[ \widetilde Y^{i,N}_u \big] = \E^{\P^{\hat\balpha,N,u}_\omega} \big[ \overline Y^{i,N}_u \big] = \overline Y^{i,N}_u(\omega), \; \text{for} \; \P\text{\rm--a.e.} \; \omega \in \Omega, \; \text{for any} \; u \in [0,T].
\end{align*}

From this equality, we deduce that it suffices to prove the convergence of the newly introduced auxiliary system to the mean-field system. To do so, we rely on arguments similar to those used previously; we therefore outline only the main steps, emphasising the key differences so as to avoid unnecessary repetition. As a first step, we define, for $(i,\ell) \in \{1,\ldots,N\}^2$, the processes
\begin{gather*}
\delta X^{i}_t \coloneqq \overline X^{i}_t - X^i_t, \; \delta Y^{i,N}_t \coloneqq \overline Y^{i,N}_t - Y^i_t, \; \delta M^{i,\star,N}_t \coloneqq \overline M^{i,\star,N}_t - M^{i,\star}_t, \; \delta N^{\star,N}_t \coloneqq \overline N^{\star,N}_t - N^{\star}_t \\
\delta Z^{i,\ell,N}_t \coloneqq \overline Z^{i,\ell,N}_t - Z^{i,i}_t \mathbf{1}_{\{i = \ell\}}, \; \delta Z^{i,m,\ell,\star,N}_t \coloneqq \overline Z^{i,m,\ell,\star,N}_t - Z^{i,m,i,\star}_t \mathbf{1}_{\{i = \ell\}}, \; \delta Z^{n,\ell,\star,N}_t \coloneqq Z^{n,\ell,\star,N}_t - \mathbf{0}, \; t \in [u,T].
\end{gather*}
We then introduce a constant $\beta > 0$, whose value will be fixed at the end of the section. Throughout the analysis, we keep the same notation as in \Cref{subsubsection:fromNplayerGameToAuxiliary}, further highlighting the analogies between the two parts.

\medskip
\textbf{Step 1: estimates for the backward components}

\medskip
By applying Itô's formula to the processes $\mathrm{e}^{\beta t} \big|\delta M^{i,\star,N}_t\big|^2$, for each $i \in \{1, \ldots, N\}$, and $\mathrm{e}^{\beta t} \big|\delta N^{\star,N}_t\big|^2$, $t \in [u, T]$, and using the Lipschitz-continuity of $\varphi_1$ and $\varphi_2$ as stated in \Cref{assumpConvThm}.\ref{lip_gGf}, together with the Burkholder--Davis--Gundy inequality with constant $c_{1,{\smallertext{\rm BDG}}}$, we have
\begin{align}\label{align:deltaM_secondPart}
\E^{\P^{\smalltext{\hat\alpha}\smalltext{,}\smalltext{N}\smalltext{,}\smalltext{u}}_\smalltext{\omega}} \Bigg[ \sup_{t \in [u,T]} \mathrm{e}^{\beta t} \big|\delta M^{i,\star,N}_t\big|^2 + \int_u^T \mathrm{e}^{\beta t} \sum_{\ell =1}^N \big\|\delta Z^{i,m,\ell,\star,N}_t\big\|^2 \d t \Bigg] \leq \ell^2_{\varphi_{\smalltext{1}}}c^\star \E^{\P^{\smalltext{\hat\alpha}\smalltext{,}\smalltext{N}\smalltext{,}\smalltext{u}}_\smalltext{\omega}} \Big[ \mathrm{e}^{\beta T} \big\|\delta X^{i}_{\cdot \land T}\big\|^2_\infty \Big], \; \P\text{\rm--a.e.} \; \omega\in\Omega,
\end{align}
and, using the triangle inequality for the Wasserstein distance,
\begin{align}\label{align:deltaN_secondPart}
\notag &\E^{\P^{\smalltext{\hat\alpha}\smalltext{,}\smalltext{N}\smalltext{,}\smalltext{u}}_\smalltext{\omega}} \Bigg[ \sup_{t \in [u,T]} \mathrm{e}^{\beta t} \big|\delta N^{\star,N}_t\big|^2 + \int_u^T \mathrm{e}^{\beta t} \sum_{\ell =1}^N \big\|\delta Z^{n,\ell,\star,N}_t\big\|^2 \d t \Bigg] \\
&\leq 2 \ell^2_{\varphi_{\smalltext{2}}}c^\star \E^{\P^{\smalltext{\hat\alpha}\smalltext{,}\smalltext{N}\smalltext{,}\smalltext{u}}_\smalltext{\omega}} \Bigg[ \mathrm{e}^{\beta T} \Bigg( \frac1N \sum_{\ell =1}^N \big\|\delta X^{\ell}_{\cdot \land T}\big\|^2_\infty + \cW_2^2\big(L^N\big(\X^N_{\cdot \land T}\big),\cL_{\hat\alpha}(X_{\cdot \land T})\big) \Bigg) \Bigg], \; \P\text{\rm--a.e.} \; \omega\in\Omega,
\end{align}
where $c^\star$ is defined as in \Cref{eq:cStarConst}.

\medskip
Again we apply It\^o's formula to the process $\mathrm{e}^{\beta t} \big|\delta Y^{i,\star,N}_t\big|^2$, for $t \in [u, T]$, and use the Lipschitz-continuity of $(g+G)(\varphi_1,\varphi_2)$, $f$ and $\Lambda$ as assumed in \Cref{assumpConvThm}.\ref{lip_gGf} and \Cref{assumpConvThm}.\ref{lipLambda_growthAleph}, Young's inequality with some constant $\eps_1 \geq 3 \ell^2_f/\beta$, the boundedness of the processes $\partial_{m,m}^2 G(\overline M^{i,\star,N}, \overline N^{\star,N})$, $\partial_{m,n}^2 G(\overline M^{i,\star,N}, \overline N^{\star,N})$ and $\partial_{n,n}^2 G(\overline M^{i,\star,N}, \overline N^{\star,N})$, as well as the triangle inequality for the Wasserstein distance. This yields
\begin{align}\label{align:deltaY_secondPart}
\notag&\mathrm{e}^{\beta t} \big|\delta Y^{i,N}_t\big|^2 + \int_t^T \mathrm{e}^{\beta s} \sum_{\ell =1}^N \big\|\delta Z^{i,\ell,N}_s\big\|^2 \d s \\
	\notag&\leq 2 \ell_{g+G,\varphi_\smalltext{1},\varphi_\smalltext{2}}^2 \mathrm{e}^{\beta T} \Bigg( \big\|\delta X^i_{\cdot \land T}\big\|^2_\infty + \frac{2}{N} \sum_{\ell =1}^N \big\|\delta X^{\ell}_{\cdot \land T}\big\|^2_\infty + 2 \cW_2^2\big(L^N\big(\X^N_{\cdot \land T}\big),\cL_{\hat\alpha}(X_{\cdot \land T})\big) \Bigg) \\
\notag&\quad+ \varepsilon_1 \int_t^T \mathrm{e}^{\beta s} \big\|\delta X^{i}_{\cdot \land s}\big\|^2_\infty \d s + \frac{\varepsilon_1 2}{N} \int_t^T \mathrm{e}^{\beta s} \sum_{\ell =1}^N \|\delta X^{\ell}_{\cdot \land s} \|^2_{\infty} \d s \\
\notag&\quad+ \varepsilon_1 30\ell^2_\Lambda \int_t^T \mathrm{e}^{\beta s} \cW_2^2\big(L^N\big(\X^N_{\cdot \land s}\big),\cL_{\hat\alpha}(X_{\cdot \land s})\big) \d s + \varepsilon_1 2 \int_t^T \mathrm{e}^{\beta s} \cW_2^2\big(L^N\big(\X^N_{\cdot \land s},\hat\alpha_s\big),\cL_{\hat\alpha}(X_{\cdot \land s},\hat\alpha_s)\big) \d s \\
\notag&\quad+ \varepsilon_1 5 \ell^2_\Lambda \int_t^T \mathrm{e}^{\beta s} \Bigg( \big\|\delta X^{i}_{\cdot \land s}\big\|^2_\infty + \frac{2}{N} \sum_{\ell =1}^N \big\|\delta X^{\ell}_{\cdot \land s}\big\|^2_\infty + \big\|\delta Z^{i,i,N}_s\big\|^2 + \big\|\delta Z^{i,m,i,\star,N}_s\big\|^2 + \big\|\delta Z^{n,i,\star,N}_s\big\|^2 \Bigg) \d s \\
\notag&\quad+ \frac{\varepsilon_1 10 \ell^2_\Lambda}{N} \int_t^T \mathrm{e}^{\beta s} \sum_{\ell =1}^N \Big( 3 \big\|\delta X^{\ell}_{\cdot \land s}\big\|^2_\infty + \big\|\delta Z^{\ell,\ell,N}_s\big\|^2 + \big\|\delta Z^{\ell,m,\ell,\star,N}_s\big\|^2 + \big\|\delta Z^{n,\ell,\star,N}_s\big\|^2 \Big) \d s \\
\notag&\quad+ 2 c_{\partial^\smalltext{2}G} \int_t^T \mathrm{e}^{\beta s} \bigg| \d \bigg\langle \int_0^\cdot \delta Y^{i,N}_{r} \d \overline M^{i,\star,N}_{r} , \delta N^{\star,N} \bigg\rangle_s \bigg| + c_{\partial^\smalltext{2}G} \int_t^T \mathrm{e}^{\beta s} \bigg| \d \bigg\langle \int_0^\cdot \delta Y^{i,N}_{r} \d \overline M^{i,\star,N}_{r} , \delta M^{i,\star,N} \bigg\rangle_s \bigg| \\
\notag&\quad+ c_{\partial^\smalltext{2}G} \int_t^T \mathrm{e}^{\beta s} \bigg| \d \bigg\langle \int_0^\cdot \delta Y^{i,N}_{r} \d M^{i,\star}_{r} , \delta^i\! M^{i,\star,N} \bigg\rangle_s \bigg| + c_{\partial^\smalltext{2}G} \int_t^T \mathrm{e}^{\beta s} \bigg| \d \bigg\langle \int_0^\cdot \delta Y^{i,N}_{r} \d \overline N^{\star,N}_{r} , \delta N^{\star,N} \bigg\rangle_s \bigg| \\
\notag&\quad+ \int_t^T \mathrm{e}^{\beta s} \bigg| \d \bigg\langle \int_0^\cdot \delta Y^{i,N}_{r} \d M^{i,\star}_{r} , \int_0^\cdot \Big( \partial^2_{m,m}G\big(\overline M^{i,\star,N}_{r},\overline N^{\star,N}_{r}\big) - \partial^2_{m,m}G\big(M^{i,\star}_{r},N^{\star}_{r}\big) \Big) \d M^{i,\star}_{r} \bigg\rangle_s \bigg| \\
&\quad- 2 \int_t^T \mathrm{e}^{\beta s} \delta Y^{i,N}_s\sum_{\ell =1}^N \delta Z^{i,\ell,N}_s \cdot \d \big(W_s^{\hat\alpha,N,u,\omega}\big)^\ell, \; t \in [u,T], \; \P^{\hat\alpha,N,u}_\omega \text{\rm--a.s.}, \; \text{for} \; \P\text{\rm--a.e.} \; \omega\in\Omega,
\end{align}
where, in the eighth line, we introduce the process
\begin{align*}
\delta^i\! \!M^{i,\star,N}_t \coloneqq \delta M^{i,\star,N}_u + \int_u^t \delta Z^{i,m,i,\star,N}_s \cdot \d \big(W_s^{\hat\alpha,N,u,\omega}\big)^i, \; t \in [u,T].
\end{align*}
We define the constants
\begin{gather*}
c_{\smallertext{\rm{BMO}}_{\smalltext{[}\smalltext{u}\smalltext{,}\smalltext{T}\smalltext{]}}} \coloneqq 3 c^2_{\partial^\smalltext{2}G} \big\|\overline M^{i,\star,N}\big\|^2_{\smallertext{\rm{BMO}}_{\smalltext{[}\smalltext{u}\smalltext{,}\smalltext{T}\smalltext{]}}} + (1+c^2_{\partial^\smalltext{2}G}) \big\|M^{i,\star}\big\|^2_{\smallertext{\rm{BMO}}_{\smalltext{[}\smalltext{u}\smalltext{,}\smalltext{T}\smalltext{]}}} + c^2_{\partial^\smalltext{2}G} \big\|\overline N^{\star,N}\big\|^2_{\smallertext{\rm{BMO}}_{\smalltext{[}\smalltext{u}\smalltext{,}\smalltext{T}\smalltext{]}}} , \; \bar c_{\smallertext{\rm{BMO}}_{\smalltext{[}\smalltext{u}\smalltext{,}\smalltext{T}\smalltext{]}}} \coloneqq \big\|\overline M^{i,\star,N}\big\|^2_{\smallertext{\rm{BMO}}_{\smalltext{[}\smalltext{u}\smalltext{,}\smalltext{T}\smalltext{]}}} .
\end{gather*}

By following the exact same steps that led to the estimates in \eqref{eq:intDeltaZ}, and observing that the process $\langle \delta M^{i,\star,N} \rangle - \langle \delta^i\! M^{i,\star,N} \rangle$ is non-negative, we deduce from \eqref{align:deltaY_secondPart} that, for some $\eps_2 > 0$, 
\begin{align}\label{eq:intDeltaZ_secondPart}
\notag&\E^{\P^{\smalltext{\hat\alpha}\smalltext{,}\smalltext{N}\smalltext{,}\smalltext{u}}_\smalltext{\omega}} \Bigg[ \int_u^T \mathrm{e}^{\beta t} \sum_{\ell =1}^N \big\|\delta Z^{i,\ell,N}_t\big\|^2 \d t \Bigg] \\
		\notag&\leq 2\ell_{g+G,\varphi_\smalltext{1},\varphi_\smalltext{2}}^2 \E^{\P^{\smalltext{\hat\alpha}\smalltext{,}\smalltext{N}\smalltext{,}\smalltext{u}}_\smalltext{\omega}} \Bigg[ \mathrm{e}^{\beta T} \Bigg( \big\|\delta X^{i}_{\cdot \land T}\big\|^2_\infty + \frac{2}{N} \sum_{\ell =1}^N \big\|\delta X^{\ell}_{\cdot \land T}\big\|^2_\infty + 2 \cW_2^2\big(L^N\big(\X^N_{\cdot \land T}\big),\cL_{\hat\alpha}(X_{\cdot \land T})\big) \Bigg) \Bigg] \\
\notag&\quad+ \varepsilon_1 (1 + 5 \ell^2_\Lambda) \E^{\P^{\smalltext{\hat\alpha}\smalltext{,}\smalltext{N}\smalltext{,}\smalltext{u}}_\smalltext{\omega}} \bigg[\int_u^T \mathrm{e}^{\beta t} \big\|\delta X^{i}_{\cdot \land t}\big\|^2_\infty \d t \bigg]  + \frac{\varepsilon_1 2 (1 + 20 \ell^2_\Lambda)}{N} \E^{\P^{\smalltext{\hat\alpha}\smalltext{,}\smalltext{N}\smalltext{,}\smalltext{u}}_\smalltext{\omega}} \Bigg[ \int_u^T \mathrm{e}^{\beta t} \sum_{\ell =1}^N \big\|\delta X^{\ell}_{\cdot \land t} \big\|^2_{\infty} \d t \Bigg] \\
\notag&\quad+ \varepsilon_1 2 \E^{\P^{\smalltext{\hat\alpha}\smalltext{,}\smalltext{N}\smalltext{,}\smalltext{u}}_\smalltext{\omega}} \bigg[ \int_u^T \mathrm{e}^{\beta s} \Big( 15 \ell^2_\Lambda \cW_2^2\big(L^N\big(\X^N_{\cdot \land t}\big),\cL_{\hat\alpha}(X_{\cdot \land t})\big) + \cW_2^2\big(L^N\big(\X^N_{\cdot \land t},\hat\alpha_t\big),\cL_{\hat\alpha}(X_{\cdot \land t},\hat\alpha_t)\big) \Big) \d t \bigg] \\
\notag&\quad+ \varepsilon_1 5 \ell^2_\Lambda \E^{\P^{\smalltext{\hat\alpha}\smalltext{,}\smalltext{N}\smalltext{,}\smalltext{u}}_\smalltext{\omega}} \bigg[ \int_u^T \mathrm{e}^{\beta t} \Big( \big\|\delta Z^{i,i,N}_t\big\|^2 + \big\|\delta Z^{i,m,i,\star,N}_t\big\|^2 + \big\|\delta Z^{n,i,\star,N}_t\big\|^2 \Big) \d t \bigg] \\
\notag&\quad+ \frac{\varepsilon_1 10 \ell^2_\Lambda}{N} \E^{\P^{\smalltext{\hat\alpha}\smalltext{,}\smalltext{N}\smalltext{,}\smalltext{u}}_\smalltext{\omega}} \Bigg[ \int_u^T \mathrm{e}^{\beta t} \sum_{\ell =1}^N \Big( \big\|\delta Z^{\ell,\ell,N}_t\big\|^2 + \big\|\delta Z^{\ell,m,\ell,\star,N}_t\big\|^2 + \big\|\delta Z^{n,\ell,\star,N}_t\big\|^2 \Big) \d t \Bigg] \\
\notag&\quad+ \frac{1}{\varepsilon_2} \E^{\P^{\smalltext{\hat\alpha}\smalltext{,}\smalltext{N}\smalltext{,}\smalltext{u}}_\smalltext{\omega}}  \Bigg[ \int_u^T \mathrm{e}^{\beta t} \sum_{\ell =1}^N \big\|\delta Z^{i,m,\ell,\star,N}_t\big\|^2 \d t \Bigg] + \frac{3}{\varepsilon_22} \E^{\P^{\smalltext{\hat\alpha}\smalltext{,}\smalltext{N}\smalltext{,}\smalltext{u}}_\smalltext{\omega}}  \Bigg[ \int_u^T \mathrm{e}^{\beta t} \sum_{\ell =1}^N \big\|\delta Z^{n,\ell,\star,N}_t\big\|^2 \d t \Bigg] \\	
&\notag\quad+ \frac{\bar c_{\smallertext{\rm{BMO}}_{\smalltext{[}\smalltext{u}\smalltext{,}\smalltext{T}\smalltext{]}}} \ell^2_{\partial^\smalltext{2}G} }{\varepsilon_2} \E^{\P^{\smalltext{\hat\alpha}\smalltext{,}\smalltext{N}\smalltext{,}\smalltext{u}}_\smalltext{\omega}} \bigg[ \sup_{t \in [u,T]} \mathrm{e}^{\beta t} \Big(\big| \delta M^{i,\star,N}_t\big|^2 + \big|\delta N^{\star,N}_t \big|^2\Big) \bigg] \\
&\quad+ \varepsilon_2  c_{\smallertext{\rm{BMO}}_{\smalltext{[}\smalltext{u}\smalltext{,}\smalltext{T}\smalltext{]}}} \E^{\P^{\smalltext{\hat\alpha}\smalltext{,}\smalltext{N}\smalltext{,}\smalltext{u}}_\smalltext{\omega}} \bigg[ \sup_{t \in [u,T]} \mathrm{e}^{\beta t} \big|\delta Y^{i,N}_t\big|^2 \bigg], \; \P\text{\rm--a.e.} \; \omega\in\Omega.
\end{align}

Analogously, repeating the same steps as in the derivation of \eqref{eq:supDeltaY}, we find that, for some $\varepsilon_3 > 0$ and $\varepsilon_4 > 0$,
\begin{align}\label{eq:supDeltaY_secondPart}
\notag&\E^{\P^{\smalltext{\hat\alpha}\smalltext{,}\smalltext{N}\smalltext{,}\smalltext{u}}_\smalltext{\omega}} \bigg[  \sup_{t \in [u,T]} \mathrm{e}^{\beta t} \big|\delta Y^{i,N}_t\big|^2 \bigg]\\
		\notag&\leq 2\ell_{g+G,\varphi_\smalltext{1},\varphi_\smalltext{2}}^2 \E^{\P^{\smalltext{\hat\alpha}\smalltext{,}\smalltext{N}\smalltext{,}\smalltext{u}}_\smalltext{\omega}} \Bigg[ \mathrm{e}^{\beta T} \Bigg( \big\|\delta X^{i}_{\cdot \land T}\big\|^2_\infty + \frac{2}{N} \sum_{\ell =1}^N \big\|\delta X^{\ell}_{\cdot \land T}\big\|^2_\infty + 2 \cW_2^2\big(L^N\big(\X^N_{\cdot \land T}\big),\cL_{\hat\alpha}(X_{\cdot \land T})\big) \Bigg) \Bigg] \\
\notag&\quad+ \varepsilon_1 (1 + 5 \ell^2_\Lambda) \E^{\P^{\smalltext{\hat\alpha}\smalltext{,}\smalltext{N}\smalltext{,}\smalltext{u}}_\smalltext{\omega}} \bigg[\int_u^T \mathrm{e}^{\beta t} \big\|\delta X^{i}_{\cdot \land t}\big\|^2_\infty \d t \bigg]  + \frac{\varepsilon_1 2 (1 + 20 \ell^2_\Lambda)}{N} \E^{\P^{\smalltext{\hat\alpha}\smalltext{,}\smalltext{N}\smalltext{,}\smalltext{u}}_\smalltext{\omega}} \Bigg[ \int_u^T \mathrm{e}^{\beta t} \sum_{\ell =1}^N \big\|\delta X^{\ell}_{\cdot \land t} \big\|^2_{\infty} \d t \Bigg] \\
\notag&\quad+ \varepsilon_1 2 \E^{\P^{\smalltext{\hat\alpha}\smalltext{,}\smalltext{N}\smalltext{,}\smalltext{u}}_\smalltext{\omega}} \bigg[ \int_u^T \mathrm{e}^{\beta s} \Big( 15 \ell^2_\Lambda \cW_2^2\big(L^N\big(\X^N_{\cdot \land t}\big),\cL_{\hat\alpha}(X_{\cdot \land t})\big) + \cW_2^2\big(L^N\big(\X^N_{\cdot \land t},\hat\alpha_t\big),\cL_{\hat\alpha}(X_{\cdot \land t},\hat\alpha_t)\big) \Big) \d t \bigg] \\
\notag&\quad+ \varepsilon_1 5 \ell^2_\Lambda \E^{\P^{\smalltext{\hat\alpha}\smalltext{,}\smalltext{N}\smalltext{,}\smalltext{u}}_\smalltext{\omega}} \bigg[ \int_u^T \mathrm{e}^{\beta t} \Big( \big\|\delta Z^{i,i,N}_t\big\|^2 + \big\|\delta Z^{i,m,i,\star,N}_t\big\|^2 + \big\|\delta Z^{n,i,\star,N}_t\big\|^2 \Big) \d t \bigg] \\
\notag&\quad+ \frac{\varepsilon_1 10 \ell^2_\Lambda}{N} \E^{\P^{\smalltext{\hat\alpha}\smalltext{,}\smalltext{N}\smalltext{,}\smalltext{u}}_\smalltext{\omega}} \Bigg[ \int_u^T \mathrm{e}^{\beta t} \sum_{\ell =1}^N \Big( \big\|\delta Z^{\ell,\ell,N}_t\big\|^2 + \big\|\delta Z^{\ell,m,\ell,\star,N}_t\big\|^2 + \big\|\delta Z^{n,\ell,\star,N}_t\big\|^2 \Big) \d t \Bigg] \\
\notag&\quad+ \frac{1}{\varepsilon_4} \E^{\P^{\smalltext{\hat\alpha}\smalltext{,}\smalltext{N}\smalltext{,}\smalltext{u}}_\smalltext{\omega}} \Bigg[ \int_u^T \mathrm{e}^{\beta t} \sum_{\ell =1}^N \big\|\delta Z^{i,m,\ell,\star,N}_t\big\|^2 \d t \Bigg] + \frac{3}{\varepsilon_42} \E^{\P^{\smalltext{\hat\alpha}\smalltext{,}\smalltext{N}\smalltext{,}\smalltext{u}}_\smalltext{\omega}} \Bigg[ \int_u^T \mathrm{e}^{\beta t} \sum_{\ell =1}^N \big\|\delta Z^{n,\ell,\star,N}_t\big\|^2 \d t \Bigg] \\	
\notag&\quad+ \frac{\bar c_{\smallertext{\rm{BMO}}_{\smalltext{[}\smalltext{u}\smalltext{,}\smalltext{T}\smalltext{]}}} \ell^2_{\partial^\smalltext{2}G} }{\varepsilon_4} \E^{\P^{\smalltext{\hat\alpha}\smalltext{,}\smalltext{N}\smalltext{,}\smalltext{u}}_\smalltext{\omega}} \bigg[ \sup_{t \in [u,T]} \mathrm{e}^{\beta t} \Big(\big| \delta M^{i,\star,N}_t\big|^2 + \big|\delta N^{\star,N}_t \big|^2\Big) \bigg] + \varepsilon_4 c_{\smallertext{\rm{BMO}}_{\smalltext{[}\smalltext{u}\smalltext{,}\smalltext{T}\smalltext{]}}} \E^{\P^{\smalltext{\hat\alpha}\smalltext{,}\smalltext{N}\smalltext{,}\smalltext{u}}_\smalltext{\omega}} \bigg[ \sup_{t \in [u,T]} \mathrm{e}^{\beta t} \big|\delta Y^{i,N}_t\big|^2 \bigg] \\
&\quad+ \varepsilon_3 4 c^2_{1,{\smallertext{\rm BDG}}} \E^{\P^{\smalltext{\hat\alpha}\smalltext{,}\smalltext{N}\smalltext{,}\smalltext{u}}_\smalltext{\omega}} \bigg[ \sup_{t \in [u,T]} \mathrm{e}^{\beta t} \big|\delta Y^{i,N}_t\big|^2 \bigg] + \frac{1}{\varepsilon_3} \E^{\P^{\smalltext{\hat\alpha}\smalltext{,}\smalltext{N}\smalltext{,}\smalltext{u}}_\smalltext{\omega}} \Bigg[ \int_u^T \mathrm{e}^{\beta t} \sum_{\ell =1}^N \big\|\delta Z^{i,\ell,N}_t\big\|^2 \d t \Bigg], \; \P\text{\rm--a.e.} \; \omega\in\Omega.
\end{align}

We introduce the constants
\begin{gather*}
c_{\eps_{\smalltext{2}\smalltext{,}\smalltext{3}\smalltext{,}\smalltext{4}}}\coloneqq \bigg(2+\frac1{\eps_3}\bigg) 2\ell_{g+G,\varphi_\smalltext{1},\varphi_\smalltext{2}}^2 + \bigg(\frac{1}{\varepsilon_2}+\frac1{\eps_4}+\frac{1}{\eps_2\eps_3}\bigg) \bigg(\frac{3}{2}\vee (\bar c_{\smallertext{\rm{BMO}}_{\smalltext{[}\smalltext{u}\smalltext{,}\smalltext{T}\smalltext{]}}}\ell^2_{\partial^\smalltext{2}G})\bigg) c^\star\ell^2_{\varphi_\smalltext{1}},\\
\overline{c}_{\eps_{\smalltext{2}\smalltext{,}\smalltext{3}\smalltext{,}\smalltext{4}}}\coloneqq \bigg(2+\frac1{\eps_3}\bigg) 4\ell_{g+G,\varphi_\smalltext{1},\varphi_\smalltext{2}}^2 + \bigg(\frac{1}{\varepsilon_2}+\frac1{\eps_4}+\frac{1}{\eps_2\eps_3}\bigg) \bigg(\frac{3}{2}\vee (\bar c_{\smallertext{\rm{BMO}}_{\smalltext{[}\smalltext{u}\smalltext{,}\smalltext{T}\smalltext{]}}}\ell^2_{\partial^\smalltext{2}G})\bigg) 2 c^\star\ell^2_{\varphi_\smalltext{2}}.
\end{gather*}
Combining the last two inequalities, \eqref{eq:intDeltaZ_secondPart} and \eqref{eq:supDeltaY_secondPart}, together with the estimates obtained in \eqref{align:deltaM_secondPart} and \eqref{align:deltaN_secondPart}, it follows that
\begin{align}\label{eq:longestim2_secondPart}
\notag&\bigg( 1 - \varepsilon_3 4c^2_{1,\smallertext{\rm BDG}} - \bigg(\varepsilon_2 +\eps_4+\frac{\eps_2}{\eps_3}\bigg) c_{\smallertext{\rm{BMO}}_{\smalltext{[}\smalltext{u}\smalltext{,}\smalltext{T}\smalltext{]}}} \bigg) \E^{\P^{\smalltext{\hat\alpha}\smalltext{,}\smalltext{N}\smalltext{,}\smalltext{u}}_\smalltext{\omega}} \bigg[ \sup_{t \in [u,T]} \mathrm{e}^{\beta t} \big|\delta Y^{i,N}_t\big|^2 \bigg] + \E^{\P^{\smalltext{\hat\alpha}\smalltext{,}\smalltext{N}\smalltext{,}\smalltext{u}}_\smalltext{\omega}} \Bigg[ \int_u^T \mathrm{e}^{\beta t} \sum_{\ell=1}^N \big\|\delta Z^{i,\ell,N}_t\big\|^2 \d t \Bigg] \\
\notag&\leq \E^{\P^{\smalltext{\hat\alpha}\smalltext{,}\smalltext{N}\smalltext{,}\smalltext{u}}_\smalltext{\omega}} \Bigg[ \mathrm{e}^{\beta T} \Bigg( c_{\eps_{\smalltext{2}\smalltext{,}\smalltext{3}\smalltext{,}\smalltext{4}}} \big\|\delta X^{i}_{\cdot \land T}\big\|^2_\infty + \overline{c}_{\eps_{\smalltext{2}\smalltext{,}\smalltext{3}\smalltext{,}\smalltext{4}}} \Bigg( \frac{1}{N} \sum_{\ell =1}^N \big\|\delta X^{\ell}_{\cdot \land T}\big\|^2_\infty + \cW_2^2\big(L^N\big(\X^N_{\cdot \land T}\big),\cL_{\hat\alpha}(X_{\cdot \land T})\big) \Bigg) \Bigg) \Bigg]\\
\notag&\quad+ \varepsilon_1 \bigg( 2 + \frac{1}{\eps_3} \bigg) (1 + 5 \ell^2_\Lambda) \E^{\P^{\smalltext{\hat\alpha}\smalltext{,}\smalltext{N}\smalltext{,}\smalltext{u}}_\smalltext{\omega}} \bigg[\int_u^T \mathrm{e}^{\beta t} \big\|\delta X^{i}_{\cdot \land t}\big\|^2_\infty \d t \bigg]  + \eps_1  \bigg( 2 + \frac{1}{\eps_3} \bigg) \frac{ 2 (1 + 20 \ell^2_\Lambda)}{N} \E^{\P^{\smalltext{\hat\alpha}\smalltext{,}\smalltext{N}\smalltext{,}\smalltext{u}}_\smalltext{\omega}} \Bigg[ \int_u^T \mathrm{e}^{\beta t} \sum_{\ell =1}^N \big\|\delta X^{\ell}_{\cdot \land t} \big\|^2_{\infty} \d t \Bigg] \\
\notag&\quad+ \varepsilon_1 \bigg( 2 + \frac{1}{\eps_3} \bigg) 2 \E^{\P^{\smalltext{\hat\alpha}\smalltext{,}\smalltext{N}\smalltext{,}\smalltext{u}}_\smalltext{\omega}} \bigg[ \int_u^T \mathrm{e}^{\beta s} \Big( 15 \ell^2_\Lambda \cW_2^2\big(L^N\big(\X^N_{\cdot \land t}\big),\cL_{\hat\alpha}(X_{\cdot \land t})\big) + \cW_2^2\big(L^N\big(\X^N_{\cdot \land t},\hat\alpha_t\big),\cL_{\hat\alpha}(X_{\cdot \land t},\hat\alpha_t)\big) \Big) \d t \bigg] \\
\notag&\quad+ \varepsilon_1  \bigg( 2 + \frac{1}{\eps_3} \bigg) 5 \ell^2_\Lambda \E^{\P^{\smalltext{\hat\alpha}\smalltext{,}\smalltext{N}\smalltext{,}\smalltext{u}}_\smalltext{\omega}} \bigg[ \int_u^T \mathrm{e}^{\beta t} \Big( \big\|\delta Z^{i,i,N}_t\big\|^2 + \big\|\delta Z^{i,m,i,\star,N}_t\big\|^2 + \big\|\delta Z^{n,i,\star,N}_t\big\|^2 \Big) \d t \bigg] \\
&\quad+ \varepsilon_1  \bigg( 2 + \frac{1}{\eps_3} \bigg) \frac{10 \ell^2_\Lambda}{N} \E^{\P^{\smalltext{\hat\alpha}\smalltext{,}\smalltext{N}\smalltext{,}\smalltext{u}}_\smalltext{\omega}} \Bigg[ \int_u^T \mathrm{e}^{\beta t} \sum_{\ell =1}^N \Big( \big\|\delta Z^{\ell,\ell,N}_t\big\|^2 + \big\|\delta Z^{\ell,m,\ell,\star,N}_t\big\|^2 + \big\|\delta Z^{n,\ell,\star,N}_t\big\|^2 \Big) \d t \Bigg], \; \P\text{\rm--a.e.} \; \omega\in\Omega.
\end{align}

\medskip
\textbf{Step 2: estimates for the forward component}

\medskip
Repeating the same computations carried out in \textbf{Step 3} of \Cref{subsubsection:fromNplayerGameToAuxiliary}, that is, applying It\^o's formula to $\mathrm{e}^{\beta t} \|\delta X^{i}_t\|^2$, $t \in [u,T]$, and using the Lipschitz condition in \Cref{assumpConvThm}.\ref{lipSigma}, the dissipativity condition in \Cref{assumpConvThm}.\ref{diss}, the Lipschitz-continuity of the function $\Lambda$ from \Cref{assumpConvThm}.\ref{lipLambda_growthAleph}, together with Young’s inequality for some $\eps_5>0$ and the triangle inequality of the Wasserstein distance, it follows that 
\begin{align*}
\mathrm{e}^{\beta t} \|\delta X^i_t\|^2 
	&\leq \bigg( \beta - 2 K_{\sigma b} + \ell^2_\sigma + \frac{2\ell^2_{\sigma b}}{\varepsilon_5} + \varepsilon_5 5 \ell_\Lambda^2 \bigg) \int_u^t \mathrm{e}^{\beta s} \big\|\delta X^{i}_{\cdot \land s}\big\|^2_\infty \d s + \varepsilon_5 \frac{2(1+20\ell^2_\Lambda)}{N} \int_u^t \mathrm{e}^{\beta s} \sum_{\ell =1}^N \big\|\delta X^{\ell}_{\cdot \land s}\big\|^2_\infty \d s\\
	&\quad+ \varepsilon_5 2 \int_u^t \mathrm{e}^{\beta s} \Big(\cW_2^2\big(L^N\big(\X^N_{\cdot \land s}\big),\cL_{\hat\alpha}(X_{\cdot \land s})\big) + 15 \ell^2_\Lambda \cW_2^2\big(L^N\big(\X^N_{\cdot \land s},\hat\alpha_s\big),\cL_{\hat\alpha}(X_{\cdot \land s},\hat\alpha_s)\big) \Big) \d s \\
	&\quad+ \varepsilon_5 5 \ell_\Lambda^2 \int_u^t \mathrm{e}^{\beta s} \Big( \big\|\delta Z^{i,i,N}_s\big\|^2 + \big\|\delta Z^{i,m,i,\star,N}_s\big\|^2 + \big\|\delta Z^{n,i,\star,N}_s\big\|^2 \Big) \d s \\
	&\quad+ \varepsilon_5 \frac{10 \ell_\Lambda^2}{N} \int_u^t \mathrm{e}^{\beta s} \sum_{\ell =1}^N \Big( \big\|\delta Z^{\ell,\ell,N}_s\big\|^2 + \big\|\delta Z^{\ell,m,\ell,\star,N}_s\big\|^2 + \big\|\delta Z^{n,\ell,\star,N}_s\big\|^2 \Big) \d s \\
	&\quad+ 2 \int_u^t \mathrm{e}^{\beta s} \delta X^{i}_s \cdot  \big( \sigma_s(X^{i}_{\cdot \land s}) - \sigma_s(\widetilde X^{i}_{\cdot \land s}) \big) \d \big(W_s^{\hat\alpha,N,u,\omega}\big)^i, \; t \in [u,T], \; \P^{\hat\alpha,N,u}_\omega \text{\rm--a.s.}, \; \text{for} \; \P\text{\rm--a.e.} \; \omega\in\Omega.
\end{align*}

Assuming that $K_{\sigma b} \geq (\beta + \ell^2_\sigma + 2\ell^2_{\sigma b}/\varepsilon_5 + 5 \ell^2_\Lambda \varepsilon_5)/2$, the application of Burkholder--Davis--Gundy's inequality, Young's inequality for some $\eps_6 \in (0,1/(c^2_{1,{\smallertext{\rm BDG}}} \mathrm{e}^{2\beta t}\ell^2_\sigma))$, and subsequently Gr\"onwall's inequality yields, for any $t \in [u,T]$, that
\begin{align*}
&\E^{\P^{\smalltext{\hat\balpha}\smalltext{,}\smalltext{N}\smalltext{,}\smalltext{u}}_\smalltext{\omega}} \bigg[ \mathrm{e}^{\beta t} \sup_{r \in [u,t]} \|\delta X^{i}_r\|^2 \bigg] \\
	&\leq \eps_5 \frac{2(1 + 20 \ell^2_\Lambda)}{N} c_{\eps_{\smalltext{6}}}(t) \E^{\P^{\smalltext{\hat\alpha}\smalltext{,}\smalltext{N}\smalltext{,}\smalltext{u}}_\smalltext{\omega}} \Bigg[ \int_u^t \mathrm{e}^{\beta s} \sum_{\ell =1}^N \big\|\delta X^{\ell}_{\cdot \land s}\big\|^2_\infty \d s \Bigg] \\
&\quad+ \varepsilon_5 2 c_{\eps_{\smalltext{6}}}(t) \E^{\P^{\smalltext{\hat\alpha}\smalltext{,}\smalltext{N}\smalltext{,}\smalltext{u}}_\smalltext{\omega}} \bigg[ \int_u^t \mathrm{e}^{\beta s} \Big(\cW_2^2\big(L^N\big(\X^N_{\cdot \land s}\big),\cL_{\hat\alpha}(X_{\cdot \land s})\big) + 15 \ell^2_\Lambda \cW_2^2\big(L^N\big(\X^N_{\cdot \land s},\hat\alpha_s\big),\cL_{\hat\alpha}(X_{\cdot \land s},\hat\alpha_s)\big) \Big) \d s \bigg] \\
&\quad + \eps_5 5 \ell^2_\Lambda c_{\eps_{\smalltext{6}}}(t) \E^{\P^{\smalltext{\hat\alpha}\smalltext{,}\smalltext{N}\smalltext{,}\smalltext{u}}_\smalltext{\omega}} \bigg[ \int_u^t \mathrm{e}^{\beta s} \Big( \big\|\delta Z^{i,i,N}_s\big\|^2 + \big\|\delta Z^{i,m,i,\star,N}_s\big\|^2 + \big\|\delta Z^{n,i,\star,N}_s\big\|^2 \Big) \d s \bigg] \\
&\quad + \eps_5 \frac{10 \ell_\Lambda^2 c_{\eps_{\smalltext{6}}}(t)}{N} \E^{\P^{\smalltext{\hat\alpha}\smalltext{,}\smalltext{N}\smalltext{,}\smalltext{u}}_\smalltext{\omega}}\Bigg[ \int_u^t \mathrm{e}^{\beta s} \sum_{\ell =1}^N \Big( \big\|\delta Z^{\ell,\ell,N}_s\big\|^2 + \big\|\delta Z^{\ell,m,\ell,\star,N}_s\big\|^2 + \big\|\delta Z^{n,\ell,\star,N}_s\big\|^2 \Big) \d s \Bigg], \; \P\text{\rm--a.e.} \; \omega\in\Omega,
\end{align*}
where 
\begin{align*}
c_{\eps_{\smalltext{6}}}(t) \coloneqq \mathrm{exp}\bigg(\beta t+\frac{ t}{\varepsilon_6 (1-\varepsilon_6 \mathrm{e}^{2\beta t}c^2_{1,{\smallertext{\rm BDG}}} \ell^2_\sigma)}  \bigg)\big(1-\varepsilon_6 \mathrm{e}^{2\beta t}c^2_{1,{\smallertext{\rm BDG}}} \ell^2_\sigma\big)^{-1}.
\end{align*}
Moreover, if we define $c_{\eps_{\smalltext{5}\smalltext{,}\smalltext{6}}}(t) \coloneqq \eps_5 c_{\eps_{\smalltext{6}}}(t)\exp\big( \varepsilon_5 2 \big( 1 + 20 \ell^2_\Lambda \big) c_{\eps_{\smalltext{6}}}(T)T\big),$
then, applying Gr\"onwall's inequality once more, we have
\begin{align*}
&\frac1N\E^{\P^{\smalltext{\hat\balpha}\smalltext{,}\smalltext{N}\smalltext{,}\smalltext{u}}_\smalltext{\omega}} \Bigg[ \mathrm{e}^{\beta t} \sum_{\ell =1}^N \|\delta X^{\ell}_{\cdot \land t}\|^2_\infty \Bigg] \\
	&\leq 2 c_{\eps_{\smalltext{5}\smalltext{,}\smalltext{6}}}(t) \E^{\P^{\smalltext{\hat\alpha}\smalltext{,}\smalltext{N}\smalltext{,}\smalltext{u}}_\smalltext{\omega}} \bigg[ \int_u^t \mathrm{e}^{\beta s} \Big(\cW_2^2\big(L^N\big(\X^N_{\cdot \land s}\big),\cL_{\hat\alpha}(X_{\cdot \land s})\big) + 15 \ell^2_\Lambda \cW_2^2\big(L^N\big(\X^N_{\cdot \land s},\hat\alpha_s\big),\cL_{\hat\alpha}(X_{\cdot \land s},\hat\alpha_s)\big) \Big) \d s \bigg] \\
&\quad + \frac{15 \ell_\Lambda^2 c_{\eps_{\smalltext{5}\smalltext{,}\smalltext{6}}}(t)}{N} \E^{\P^{\smalltext{\hat\alpha}\smalltext{,}\smalltext{N}\smalltext{,}\smalltext{u}}_\smalltext{\omega}}\Bigg[ \int_u^t \mathrm{e}^{\beta s} \sum_{\ell =1}^N \Big( \big\|\delta Z^{\ell,\ell,N}_s\big\|^2 + \big\|\delta Z^{\ell,m,\ell,\star,N}_s\big\|^2 + \big\|\delta Z^{n,\ell,\star,N}_s\big\|^2 \Big) \d s \Bigg], \; \P\text{\rm--a.e.} \; \omega\in\Omega,
\end{align*}
from which we deduce that, for $\P\text{\rm--a.e.} \; \omega\in\Omega$,
\begin{align}\label{align:forwardMeanField_final}
\notag&\E^{\P^{\smalltext{\hat\balpha}\smalltext{,}\smalltext{N}\smalltext{,}\smalltext{u}}_\smalltext{\omega}} \Big[ \mathrm{e}^{\beta t} \|\delta X^{i}_{\cdot \land t}\|^2_\infty \Big] \\
\notag	&\leq \eps_5 2 c_{\eps_{\smalltext{6}}}(t) (1+2(1 + 20 \ell^2_\Lambda) c_{\eps_{\smalltext{5}\smalltext{,}\smalltext{6}}}(t)) \E^{\P^{\smalltext{\hat\alpha}\smalltext{,}\smalltext{N}\smalltext{,}\smalltext{u}}_\smalltext{\omega}} \bigg[ \int_u^t \mathrm{e}^{\beta s} \cW_2^2\big(L^N\big(\X^N_{\cdot \land s}\big),\cL_{\hat\alpha}(X_{\cdot \land s})\big)\mathrm{d}s\bigg]\\
\notag	&\quad +\eps_5 2 c_{\eps_{\smalltext{6}}}(t) (1+2(1 + 20 \ell^2_\Lambda) c_{\eps_{\smalltext{5}\smalltext{,}\smalltext{6}}}(t)) \E^{\P^{\smalltext{\hat\alpha}\smalltext{,}\smalltext{N}\smalltext{,}\smalltext{u}}_\smalltext{\omega}} \bigg[ \int_u^t \mathrm{e}^{\beta s}  15 \ell^2_\Lambda \cW_2^2\big(L^N\big(\X^N_{\cdot \land s},\hat\alpha_s\big),\cL_{\hat\alpha}(X_{\cdot \land s},\hat\alpha_s)\big) \d s \bigg] \\
\notag&\quad + \eps_5 5 \ell^2_\Lambda c_{\eps_{\smalltext{6}}}(t) \E^{\P^{\smalltext{\hat\alpha}\smalltext{,}\smalltext{N}\smalltext{,}\smalltext{u}}_\smalltext{\omega}} \bigg[ \int_u^t \mathrm{e}^{\beta s} \Big( \big\|\delta Z^{i,i,N}_s\big\|^2 + \big\|\delta Z^{i,m,i,\star,N}_s\big\|^2 + \big\|\delta Z^{n,i,\star,N}_s\big\|^2 \Big) \d s \bigg] \\
&\quad + \eps_5 \frac{10 \ell_\Lambda^2 c_{\eps_{\smalltext{6}}}(t)}{N} ( 1 + 3 (1+20\ell^2_\Lambda) c_{\eps_{\smalltext{5}\smalltext{,}\smalltext{6}}}(t)) \E^{\P^{\smalltext{\hat\alpha}\smalltext{,}\smalltext{N}\smalltext{,}\smalltext{u}}_\smalltext{\omega}}\Bigg[ \int_u^t \mathrm{e}^{\beta s} \sum_{\ell =1}^N \Big( \big\|\delta Z^{\ell,\ell,N}_s\big\|^2 + \big\|\delta Z^{\ell,m,\ell,\star,N}_s\big\|^2 + \big\|\delta Z^{n,\ell,\star,N}_s\big\|^2 \Big) \d s \Bigg].
\end{align}

\medskip
\textbf{Step 3: all estimates combined}

\medskip
We introduce
\begin{align*}
{c}^1_{\eps_{\smalltext{1}\smalltext{,}\smalltext{2}\smalltext{,}\smalltext{3}\smalltext{,}\smalltext{4}\smalltext{,}\smalltext{5}\smalltext{,}\smalltext{6}}}&\coloneqq \varepsilon_5 2 c_{\eps_{\smalltext{6}}}(T) (1+2(1+20\ell^2_\Lambda) c_{\eps_{\smalltext{5}\smalltext{,}\smalltext{6}}}(T)) \bigg( c_{\eps_{\smalltext{2}\smalltext{,}\smalltext{3}\smalltext{,}\smalltext{4}}} + \ell^2_{\varphi_{\smalltext{1}}}c^\star + \varepsilon_1 \bigg( 2 + \frac{1}{\eps_3} \bigg) (1 + 5 \ell^2_\Lambda) T \bigg) \\
&\quad+ 2 c_{\eps_{\smalltext{5}\smalltext{,}\smalltext{6}}}(T) \bigg(\overline{c}_{\eps_{\smalltext{2}\smalltext{,}\smalltext{3}\smalltext{,}\smalltext{4}}} + 2 \ell^2_{\varphi_{\smalltext{2}}}c^\star + \varepsilon_1  \bigg( 2 + \frac{1}{\eps_3} \bigg) 2 (1 + 20 \ell^2_\Lambda) T \bigg) + \varepsilon_1  \bigg( 2 + \frac{1}{\eps_3} \bigg) 2, \\
{c}^2_{\eps_{\smalltext{1}\smalltext{,}\smalltext{2}\smalltext{,}\smalltext{3}\smalltext{,}\smalltext{4}\smalltext{,}\smalltext{5}\smalltext{,}\smalltext{6}}}&\coloneqq \varepsilon_5 5 \ell^2_\Lambda c_{\eps_{\smalltext{6}}}(T) \bigg( c_{\eps_{\smalltext{2}\smalltext{,}\smalltext{3}\smalltext{,}\smalltext{4}}} + \ell^2_{\varphi_{\smalltext{1}}}c^\star + \varepsilon_1 \bigg( 2 + \frac{1}{\eps_3} \bigg) (1 + 5 \ell^2_\Lambda) T \bigg) + \varepsilon_1  \bigg( 2 + \frac{1}{\eps_3} \bigg) 5 \ell^2_\Lambda, \\
{c}^3_{\eps_{\smalltext{1}\smalltext{,}\smalltext{2}\smalltext{,}\smalltext{3}\smalltext{,}\smalltext{4}\smalltext{,}\smalltext{5}\smalltext{,}\smalltext{6}}}&\coloneqq \eps_5 10 \ell^2_\Lambda c_{\eps_{\smalltext{6}}}(T) (1+3(1+20\ell^2_\Lambda) c_{\eps_{\smalltext{5}\smalltext{,}\smalltext{6}}}(T)) \bigg( c_{\eps_{\smalltext{2}\smalltext{,}\smalltext{3}\smalltext{,}\smalltext{4}}} + \ell^2_{\varphi_{\smalltext{1}}}c^\star + \varepsilon_1 \bigg( 2 + \frac{1}{\eps_3} \bigg) (1 + 5 \ell^2_\Lambda) T \bigg) \\
&\quad+ 15\ell^2_\Lambda c_{\eps_{\smalltext{5}\smalltext{,}\smalltext{6}}}(T) \bigg(\overline{c}_{\eps_{\smalltext{2}\smalltext{,}\smalltext{3}\smalltext{,}\smalltext{4}}} + 2 \ell^2_{\varphi_{\smalltext{2}}}c^\star + \varepsilon_1  \bigg( 2 + \frac{1}{\eps_3} \bigg) 2 (1 + 20 \ell^2_\Lambda) T \bigg) + \varepsilon_1  \bigg( 2 + \frac{1}{\eps_3} \bigg) 10 \ell^2_\Lambda.
\end{align*}

Combining the previous estimates for the forward component with those obtained in \eqref{align:deltaM_secondPart}, \eqref{align:deltaN_secondPart}, and \eqref{eq:longestim2_secondPart}, it follows that
\begin{align*}
\notag&\E^{\P^{\smalltext{\hat\alpha}\smalltext{,}\smalltext{N}\smalltext{,}\smalltext{u}}_\smalltext{\omega}} \Bigg[ \bigg( 1 - \varepsilon_3 4c^2_{1,\smallertext{\rm BDG}} - \bigg(\varepsilon_2 +\eps_4+\frac{\eps_2}{\eps_3}\bigg) c_{\smallertext{\rm{BMO}}_{\smalltext{[}\smalltext{u}\smalltext{,}\smalltext{T}\smalltext{]}}} \bigg) \sup_{t \in [u,T]} \mathrm{e}^{\beta t} \big|\delta Y^{i,N}_t\big|^2 + \sup_{t \in [u,T]} \mathrm{e}^{\beta t} \big|\delta M^{i,\star,N}_t\big|^2 + \sup_{t \in [u,T]} \mathrm{e}^{\beta t} \big|\delta N^{\star,N}_t\big|^2 \Bigg] \\
\notag&\quad + \E^{\P^{\smalltext{\hat\alpha}\smalltext{,}\smalltext{N}\smalltext{,}\smalltext{u}}_\smalltext{\omega}} \Bigg[ \int_u^T \mathrm{e}^{\beta t} \sum_{\ell =1}^N \Big( \big\|\delta Z^{i,\ell,N}_t\big\|^2 + \big\|\delta Z^{i,m,\ell,\star,N}_t\big\|^2 + \big\|\delta Z^{n,\ell,\star,N}_t\big\|^2 \Big) \d t \Bigg] \\
	\notag&\leq ( \overline{c}_{\eps_{\smalltext{2}\smalltext{,}\smalltext{3}\smalltext{,}\smalltext{4}}} + 2 \ell^2_{\varphi_{\smalltext{2}}}c^\star ) \E^{\P^{\smalltext{\hat\alpha}\smalltext{,}\smalltext{N}\smalltext{,}\smalltext{u}}_\smalltext{\omega}} \Big[ \mathrm{e}^{\beta T} \cW_2^2\big(L^N\big(\X^N_{\cdot \land T}\big),\cL_{\hat\alpha}(X_{\cdot \land T})\big) \Big] \\
\notag&\quad + {c}^1_{\eps_{\smalltext{1}\smalltext{,}\smalltext{2}\smalltext{,}\smalltext{3}\smalltext{,}\smalltext{4}\smalltext{,}\smalltext{5}\smalltext{,}\smalltext{6}}} \E^{\P^{\smalltext{\hat\alpha}\smalltext{,}\smalltext{N}\smalltext{,}\smalltext{u}}_\smalltext{\omega}} \bigg[ \int_u^T \mathrm{e}^{\beta t} \Big(\cW_2^2\big(L^N\big(\X^N_{\cdot \land t}\big),\cL_{\hat\alpha}(X_{\cdot \land t})\big) + 15 \ell^2_\Lambda \cW_2^2\big(L^N\big(\X^N_{\cdot \land t},\hat\alpha_t\big),\cL_{\hat\alpha}(X_{\cdot \land t},\hat\alpha_t)\big) \Big) \d t \bigg] \\
\notag&\quad+ {c}^2_{\eps_{\smalltext{1}\smalltext{,}\smalltext{2}\smalltext{,}\smalltext{3}\smalltext{,}\smalltext{4}\smalltext{,}\smalltext{5}\smalltext{,}\smalltext{6}}} \E^{\P^{\smalltext{\hat\alpha}\smalltext{,}\smalltext{N}\smalltext{,}\smalltext{u}}_\smalltext{\omega}} \bigg[ \int_u^T \mathrm{e}^{\beta t} \Big( \big\|\delta Z^{i,i,N}_t\big\|^2 + \big\|\delta Z^{i,m,i,\star,N}_t\big\|^2 + \big\|\delta Z^{n,i,\star,N}_t\big\|^2 \Big) \d t \bigg] \\
&\quad+ \frac{{c}^3_{\eps_{\smalltext{1}\smalltext{,}\smalltext{2}\smalltext{,}\smalltext{3}\smalltext{,}\smalltext{4}\smalltext{,}\smalltext{5}\smalltext{,}\smalltext{6}}}}{N} \E^{\P^{\smalltext{\hat\alpha}\smalltext{,}\smalltext{N}\smalltext{,}\smalltext{u}}_\smalltext{\omega}} \Bigg[ \int_u^T \mathrm{e}^{\beta t} \sum_{\ell =1}^N \Big( \big\|\delta Z^{\ell,\ell,N}_t\big\|^2 + \big\|\delta Z^{\ell,m,\ell,\star,N}_t\big\|^2 + \big\|\delta Z^{n,\ell,\star,N}_t\big\|^2 \Big) \d t \Bigg], \; \P\text{\rm--a.e.} \; \omega\in\Omega.
\end{align*}

All the constants mentioned above are independent of both $N \in \N^\star$ and $\omega \in \Omega$, as discussed in \Cref{subsubsection:fromNplayerGameToAuxiliary}. Therefore, following the same reasoning, we can choose the parameters $\varepsilon_i>0$ for $i\in\{1,\ldots,6\}$, and $\beta>0$, and require the dissipativity constant $K_{\sigma b}$ to be sufficiently large so that all the conditions stated throughout the proof are satisfied. These conditions are
\begin{gather*}
\beta \geq \max\bigg\{\frac{3 \ell^2_f}{\eps_1},\ell_f^2(1+c_A)^2+ 2\ell^2_f \bigg\},\; K_{\sigma b}\geq \frac12\bigg(\beta+\ell^2_\sigma + \frac{2\ell^2_{\sigma b}}{\varepsilon_5} + \varepsilon_5 5 \ell_\Lambda^2\bigg),\;
1-\varepsilon_6 \mathrm{e}^{2\beta T}c^2_{1,{\smallertext{\rm BDG}}} \ell^2_\sigma>0,\\ 1 - \varepsilon_3 4c^2_{1,\smallertext{\rm BDG}} - \bigg(\varepsilon_2 +\eps_4+\frac{\eps_2}{\eps_3}\bigg) c_{\smallertext{\rm{BMO}}_{\smalltext{[}\smalltext{u}\smalltext{,}\smalltext{T}\smalltext{]}}} >0,\; 1-c^2_{\eps_{\smalltext{1}\smalltext{,}\smalltext{2}\smalltext{,}\smalltext{3}\smalltext{,}\smalltext{4}\smalltext{,}\smalltext{5}\smalltext{,}\smalltext{6}}}-c^3_{\eps_{\smalltext{1}\smalltext{,}\smalltext{2}\smalltext{,}\smalltext{3}\smalltext{,}\smalltext{4}\smalltext{,}\smalltext{5}\smalltext{,}\smalltext{6}}} > 0.
\end{gather*}

Consequently, as in the derivation of the final inequality in \eqref{align:firstPart_finalEst}, we can conclude that, for $\P\text{\rm--a.e.} \; \omega\in\Omega$,
\begin{align*}
\notag&\E^{\P^{\smalltext{\hat\alpha}\smalltext{,}\smalltext{N}\smalltext{,}\smalltext{u}}_\smalltext{\omega}} \Bigg[ \bigg( 1 - \varepsilon_3 4c^2_{1,\smallertext{\rm BDG}} - \bigg(\varepsilon_2 +\eps_4+\frac{\eps_2}{\eps_3}\bigg) c_{\smallertext{\rm{BMO}}_{\smalltext{[}\smalltext{u}\smalltext{,}\smalltext{T}\smalltext{]}}} \bigg) \sup_{t \in [u,T]} \mathrm{e}^{\beta t} \big|\delta Y^{i,N}_t\big|^2 + \sup_{t \in [u,T]} \mathrm{e}^{\beta t} \big|\delta M^{i,\star,N}_t\big|^2 + \sup_{t \in [u,T]} \mathrm{e}^{\beta t} \big|\delta N^{\star,N}_t\big|^2 \Bigg] \\
\notag&\quad + \big(1-c^2_{\eps_{\smalltext{1}\smalltext{,}\smalltext{2}\smalltext{,}\smalltext{3}\smalltext{,}\smalltext{4}\smalltext{,}\smalltext{5}\smalltext{,}\smalltext{6}}}\big) \E^{\P^{\smalltext{\hat\alpha}\smalltext{,}\smalltext{N}\smalltext{,}\smalltext{u}}_\smalltext{\omega}} \Bigg[ \int_u^T \mathrm{e}^{\beta t} \sum_{\ell =1}^N \Big( \big\|\delta Z^{i,\ell,N}_t\big\|^2 + \big\|\delta Z^{i,m,\ell,\star,N}_t\big\|^2 + \big\|\delta Z^{n,\ell,\star,N}_t\big\|^2 \Big) \d t \Bigg] \\
	\notag&\leq ( \overline{c}_{\eps_{\smalltext{2}\smalltext{,}\smalltext{3}\smalltext{,}\smalltext{4}}} + 2 \ell^2_{\varphi_{\smalltext{2}}}c^\star ) \E^{\P^{\smalltext{\hat\alpha}\smalltext{,}\smalltext{N}\smalltext{,}\smalltext{u}}_\smalltext{\omega}} \Big[ \mathrm{e}^{\beta T} \cW_2^2\big(L^N\big(\X^N_{\cdot \land T}\big),\cL_{\hat\alpha}(X_{\cdot \land T})\big) \Big]+{c}^1_{\eps_{\smalltext{1}\smalltext{,}\smalltext{2}\smalltext{,}\smalltext{3}\smalltext{,}\smalltext{4}\smalltext{,}\smalltext{5}\smalltext{,}\smalltext{6}}} \bigg( 1 + \frac{{c}^3_{\eps_{\smalltext{1}\smalltext{,}\smalltext{2}\smalltext{,}\smalltext{3}\smalltext{,}\smalltext{4}\smalltext{,}\smalltext{5}\smalltext{,}\smalltext{6}}}}{1-c^2_{\eps_{\smalltext{1}\smalltext{,}\smalltext{2}\smalltext{,}\smalltext{3}\smalltext{,}\smalltext{4}\smalltext{,}\smalltext{5}\smalltext{,}\smalltext{6}}}-c^3_{\eps_{\smalltext{1}\smalltext{,}\smalltext{2}\smalltext{,}\smalltext{3}\smalltext{,}\smalltext{4}\smalltext{,}\smalltext{5}\smalltext{,}\smalltext{6}}}} \bigg)  \\
&\quad \times\E^{\P^{\smalltext{\hat\alpha}\smalltext{,}\smalltext{N}\smalltext{,}\smalltext{u}}_\smalltext{\omega}} \bigg[ \int_u^T \mathrm{e}^{\beta t} \Big(\cW_2^2\big(L^N\big(\X^N_{\cdot \land t}\big),\cL_{\hat\alpha}(X_{\cdot \land t})\big) + 15 \ell^2_\Lambda \cW_2^2\big(L^N\big(\X^N_{\cdot \land t},\hat\alpha_t\big),\cL_{\hat\alpha}(X_{\cdot \land t},\hat\alpha_t)\big) \Big) \d t \bigg].
\end{align*}

Therefore, there exists a constant $C>0$, independent of $N$, such that
\begin{align}\label{align:finalEstimatesWithConstants_secondPart}
\notag&\E^{\P^{\smalltext{\hat\alpha}\smalltext{,}\smalltext{N}\smalltext{,}\smalltext{u}}_\smalltext{\omega}} \bigg[ \sup_{t \in [u,T]} \mathrm{e}^{\beta t} \big|\delta Y^{i,N}_t\big|^2 + \sup_{t \in [u,T]} \mathrm{e}^{\beta t} \big|\delta M^{i,\star,N}_t\big|^2 + \sup_{t \in [u,T]} \mathrm{e}^{\beta t} \big|\delta N^{\star,N}_t\big|^2 \bigg] \\
\notag&\quad +  \E^{\P^{\smalltext{\hat\alpha}\smalltext{,}\smalltext{N}\smalltext{,}\smalltext{u}}_\smalltext{\omega}} \Bigg[ \int_u^T \mathrm{e}^{\beta t} \sum_{\ell =1}^N \Big( \big\|\delta Z^{i,\ell,N}_t\big\|^2 + \big\|\delta Z^{i,m,\ell,\star,N}_t\big\|^2 + \big\|\delta Z^{n,\ell,\star,N}_t\big\|^2 \Big) \d t \Bigg] \\
&\leq C \sup_{ t \in [u,T]} \E^{\P^{\smalltext{\hat\alpha}\smalltext{,}\smalltext{N}\smalltext{,}\smalltext{u}}_\smalltext{\omega}} \Big[ \cW_2^2\big(L^N\big(\X^N_{\cdot \land t}\big),\cL_{\hat\alpha}(X_{\cdot \land t})\big) + \cW_2^2\big(L^N(\hat\alpha_t),\cL_{\hat\alpha}(\hat\alpha_t)\big) \Big], \; \P\text{\rm--a.e.} \; \omega\in\Omega.
\end{align}

\subsubsection{The convergence of the equilibria}\label{subsubsection:conEquilibria}

This section is devoted to proving the convergence of a sub-game--perfect Nash equilibrium to the unique sub-game--perfect mean-field equilibrium. Let us fix some $u \in [0,T]$. The triangular inequality, together with the characterisation of the two equilibria via the function $\Lambda$, as given in \eqref{align:systemNplayerGame} and \eqref{align:systemMeanFieldGame}, yields
\begin{align*}
&\int_u^T \cW^2_2\big(\P^{{\hat\balpha}^\smalltext{N}{,}{N}{,}{u}}_{\omega} \circ ( \hat\alpha^{i,N}_t )^{-1}, \P^{{\hat\alpha}{,}{N}{,}{u}}_{\omega} \circ ( \hat\alpha^{i}_t )^{-1} \big) \d t \\
&\leq \int_u^T \cW^2_2\big(\P^{{\hat\balpha}^\smalltext{N}{,}{N}{,}{u}}_{\omega} \circ ( \hat\alpha^{i,N}_t )^{-1}, \P^{{\hat\balpha}^\smalltext{N}{,}{N}{,}{u}}_{\omega} \circ ( \widetilde \alpha^{i,N}_t )^{-1} \big) \d t  +\int_u^T \cW^2_2\big(\P^{{\hat\balpha}^\smalltext{N}{,}{N}{,}{u}}_{\omega} \circ ( \widetilde \alpha^{i,N}_t )^{-1}, \P^{{\hat\alpha}{,}{N}{,}{u}}_{\omega} \circ ( \hat\alpha^{i}_t )^{-1} \big) \d t \\
&\leq 6 \ell_\Lambda^2 \E^{\P^{\smalltext{\hat\balpha}^\tinytext{N}\smalltext{,}\smalltext{N}\smalltext{,}\smalltext{u}}_\smalltext{\omega}} \Bigg[ \int_u^T \Bigg( \big\|\delta \widetilde X^i_{\cdot \land t}\big\|^2_\infty + \frac{1}{N} \sum_{\ell=1}^N \big\|\delta \widetilde X^{\ell,N}_{\cdot \land t}\big\|^2_\infty + \big\|\delta \widetilde Z^{i,i,N}_t\big\|^2 + \big\|\delta \widetilde Z^{i,m,i,\star,N}_t\big\|^2 + \big\|\delta \widetilde Z^{n,i,\star,N}_t\big\|^2 + \big|\aleph^{i,N}_t\big|^2 \Bigg) \d t \Bigg] \\
&\quad + 5 \ell_\Lambda^2 \E^{\P^{\smalltext{\hat\alpha}\smalltext{,}\smalltext{N}\smalltext{,}\smalltext{u}}_\smalltext{\omega}} \Bigg[ \int_u^T \Bigg( \big\|\delta \overline X^i_{\cdot \land t}\big\|^2_\infty + 2 \cW_2^2\big(L^N\big(\X^N_{\cdot \land t}\big),\cL_{\hat\alpha}(X_{\cdot \land t})\big) + \big\|\delta \overline Z^{i,i,N}_t\big\|^2 + \big\|\delta \overline Z^{i,m,i,\star,N}_t\big\|^2 + \big\|\delta \overline Z^{n,i,\star,N}_t\big\|^2 \Bigg) \d t \bigg] \\
&\quad + 10 \ell_\Lambda^2 \E^{\P^{\smalltext{\hat\alpha}\smalltext{,}\smalltext{N}\smalltext{,}\smalltext{u}}_\smalltext{\omega}} \bigg[ \int_u^T \cW_2^2\big(L^N\big(\X^N_{\cdot \land t}\big),\cL_{\hat\alpha}(X_{\cdot \land t})\big) \d t \bigg], \; \P\text{\rm--a.e.} \; \omega\in\Omega,
\end{align*}
where the last inequality follows from the equality in law established in \eqref{YW_lawEq} and the Lipschitz-continuity of $\Lambda$, as stated in \Cref{assumpConvThm}.\ref{lipLambda_growthAleph}. Here, for notational convenience, we use the tilde superscript to denote the differences between the processes associated with the $N$-player game and those of the first auxiliary system, and the overline superscript to denote the differences between the processes associated with the second auxiliary system and the mean-field system. Then, combining the estimates obtained above---namely \eqref{align:GronwallDeltaX2} together with \eqref{align:firstPart_finalEst}, and \eqref{align:forwardMeanField_final} together with \eqref{align:finalEstimatesWithConstants_secondPart}---we conclude that there exists a constant $C>0$ such that, for $\P\text{\rm--a.e.} \; \omega\in\Omega$,
\begin{align*}
&\int_u^T \cW^2_2\big(\P^{{\hat\balpha}^\smalltext{N}{,}{N}{,}{u}}_{\omega} \circ ( \hat\alpha^{i,N}_t )^{-1}, \P^{{\hat\alpha}{,}{N}{,}{u}}_{\omega} \circ ( \hat\alpha^{i}_t )^{-1} \big) \d t \\
&\leq C R_N^2\Bigg(1+\| X^{i}_{u}(\omega)\|^{2}+\frac1N\sum_{\ell =1}^N\| X^{\ell}_{u}(\omega)\|^{2}\Bigg)+C NR_N^2\Bigg(1+\| X^{i}_{u}(\omega)\|^{2\bar{p}}+\frac1N\sum_{\ell =1}^N\| X^{\ell}_{u}(\omega)\|^{2\bar{p}}\Bigg)\\
	&\quad+C NR_N^4\Bigg(1+\frac1N\sum_{\ell =1}^N\| X^{\ell}_{u}(\omega)\|^{2}\Bigg) (1+N) + C \sup_{ t \in [u,T]} \E^{\P^{\smalltext{\hat\alpha}\smalltext{,}\smalltext{N}\smalltext{,}\smalltext{u}}_\smalltext{\omega}} \Big[ \cW_2^2\big(L^N\big(\X^N_{\cdot \land t}\big),\cL_{\hat\alpha}(X_{\cdot \land t})\big) + \cW_2^2\big(L^N(\hat\alpha_t),\cL_{\hat\alpha}(\hat\alpha_t)\big) \Big].
\end{align*}
Therefore, the proof is complete, thanks to the strong law of large numbers stated in \eqref{align:LLN} or, equivalently, in \eqref{align:LLN_withLIM}.

{\small \bibliography{bibliographyDylan}}

\begin{appendix}

\section{Martingale representation under initial enlargement}\label{appendix:martingaleRep}

\begin{proof}[Proof of \Cref{lemma:martingaleRepresentation}]
We prove the result in the case where $M$ is an $(\F^i,\P)$-martingale, the $(\F_N,\P)$-martingale case follows analogously. Without loss of generality, we may assume that $M$ is null at zero. The proof is structured in several steps, each building upon and generalising the previous one; these steps make use of the properties of the two filtrations $\F^i$ and $(\F^i)^{\P_{\smalltext{+}}}$, as well as the boundedness or unboundedness of the martingale $M$. 

\medskip
We first work under the assumption that $M$ is of the form $M_t \coloneqq \E^\P \big[ \eta f(X^i_0) \big| \cF^i_t \big], \; t \in [0,T],$ where $\eta$ is a bounded $\cG^i_T$-measurable random variable and $f$ is a bounded Borel-measurable function. Then, for each $t \in [0,T]$, the $\P$-independence assumption implies that $M_t = f(X^i_0) \E^\P \big[ \eta \big| \cG^i_t \big]$, $\P$--a.s. Therefore, by \cite[Theorem 9.7.4]{weizsaecker1990stochastic} there exists a process $Z \in \L^2_{\mathrm{loc}}(\G^i,\P)$ such that 
\begin{align*}
M_t = f(X^i_0) \int_0^t Z_s \cdot \d W^i_s, \; \P \text{\rm --a.s.}, \; t \in [0,T].
\end{align*}
We note that $(f(X^i_0) Z) \in \L^2_{\mathrm{loc}}(\F^i,\P)$. We can conclude that the martingale representation property holds for all bounded $(\F^i,\P)$-martingales, since those of the form considered above generate all bounded $(\F^i,\P)$-martingales as a direct application of the monotone class theorem, given that 
\begin{align*}
\cF^i_T = \sigma\big(\eta f(X^i_0): \eta \; \text{is a bounded $\cG^i_T$-measurable random variable}, \; f \; \text{is a bounded Borel-measurable function}\big).
\end{align*}

\medskip
We now show that the martingale representation property also holds for bounded $((\F^i)^{\P_{\smalltext{+}}},\P)$-martingales. Let $M$ be a bounded $((\F^i)^{\P_{\smalltext{+}}},\P)$-martingale null at zero. We define 
\begin{align*}
\widetilde{M}_t \coloneqq \E^\P \big[M_T \big|\cF^i_t \big], \; t \in [0,T].
\end{align*}
By construction, $\widetilde{M}$ is a bounded $(\F^i,\P)$-martingale. Therefore, by the martingale representation property proved in the previous step, there exists a unique process $Z \in \L^2_{\mathrm{loc}}(\F^i,\P)$ such that 
\begin{align*}
\widetilde{M}_t = \int_0^t Z_s \cdot \d W^i_s, \; \P \text{\rm --a.s.}, \; t \in [0,T].
\end{align*}
Applying the backward martingale convergence theorem of \citeauthor*{dellacherie1982probabilities} \cite[Theorem V.33]{dellacherie1982probabilities}, we deduce that, for $t\in[0,T]$,
\begin{align*}
M_t = \E^\P \big[M_T \big| (\cF^i_t)^{\P_{\smalltext{+}}} \big] = \E^\P \big[M_T \big|\cF_{t+} \big] = \lim_{u \searrow t} \E^\P \big[M_T \big|\cF_u \big] = \lim_{u \searrow t} \widetilde{M}_u = \lim_{u \searrow t} \int_0^u Z_s \cdot \d W^i_s = \int_0^t Z_s \cdot \d W^i_s, \; \P \text{\rm --a.s.}
\end{align*}

\medskip
By \citeauthor*{he1992semimartingale} \cite[Theorem 13.4]{he1992semimartingale}, the martingale representation property extends to all $((\F^i)^{\P_{\smalltext{+}}},\P)$-martingales null at zero. We conclude the proof by fixing a general $(\F^i,\P)$-martingale $M$ null at zero. Then, there exists a right-continuous process $\overline{M}$ such that 
\begin{align*}
\overline{M}_t = \E^\P \big[M_T \big| (\cF^i_t)^{\P_{\smalltext{+}}} \big], \; \P \text{\rm --a.s.}, \; t \in [0,T].
\end{align*}
The martingale representation property for $((\F^i)^{\P_{\smalltext{+}}},\P)$-martingales ensures there exists a unique $Z \in \L^2_{\mathrm{loc}}((\F^i)^{\P_{\smalltext{+}}},\P)$ such that 
\begin{align*}
\overline{M}_t = \int_0^t \overline{Z}_s \cdot \d W^i_s, \; \P \text{\rm --a.s.}, \; t \in [0,T].
\end{align*}
We can then select a process $Z \in \L^2_{\mathrm{loc}}(\F^i,\P)$, $\P$-indistinguishable from $\overline{Z}$ (see \citeauthor*{dellacherie1978probabilities} \cite[Page 134]{dellacherie1978probabilities}), so that 
\begin{align*}
\overline{M}_t = \int_0^t Z_s \cdot \d W^i_s, \; \P \text{\rm --a.s.}, \; t \in [0,T].
\end{align*}
Moreover, since 
\begin{align*}
M_t = \E^\P \big[M_T \big| \cF^i_t \big] = \E^\P \big[ \overline{M}_t \big| \cF^i_t \big] = \int_0^t Z_s \cdot \d W^i_s, \; \P \text{\rm --a.s.}, \; t \in [0,T],
\end{align*}
we conclude that the $(\F^i,\P)$-martingale $M$ admits the martingale representation property.
\end{proof}

\section{The dynamic programming principle}\label{section:DPP}

This section is dedicated to proving an extended version of the dynamic programming principle, following the approach of \cite[Theorem 7.3]{hernandez2023me}, or equivalently, \cite[Theorem 2.4.3]{hernandez2021general}. Without loss of generality, we focus on player 1. To unify the analysis of the $N$-player game and its mean field counterpart within a common framework, we introduce a probability measure $\Q$ defined on $(\Omega,\cF)$, which may differ from the original measure $\P$ but is assumed to be equivalent to it. We assume that the state process $X^1$ satisfies the SDE
\begin{equation*}
X^1_t = X^1_0 + \int_0^t \sigma^1_s(X^1_{\cdot \land s}) \d W^{\Q,1}_s, \; t \in [0,T], \; \P\text{\rm--a.s.},
\end{equation*}
where $W^{\Q,1}$ is a $(\G^1,\Q)$-Brownian motion. Let $\widetilde{\F}$ denote a filtration, which can be either $\F_N$ or $\F^1$, and define $\widetilde{\cA}$ as the set of $\widetilde{\F}$-predictable, $A$-valued control processes. We also consider a bounded function $\widetilde{b}^1: \Omega \times [0,T] \times \cC_m \times A \longrightarrow \R^d$, which we assume to be ${\rm Prog}(\widetilde{\F})\otimes\cB(\cC_m)\otimes\cB(A)$-measurable. Then, given a control $\alpha \in \widetilde{A}$, we define a new probability measure $\Q^\alpha$ via Girsanov's theorem with Radon–Nikod\'ym derivative
\begin{equation*}
\frac{\d \Q^{\alpha}}{\d \Q} \coloneqq \cE\bigg( \int_0^\cdot \widetilde{b}^1_s\big(X^1_{\cdot \land s},\alpha_s\big) \cdot \d W^{Q,1}_s \bigg)_T.
\end{equation*}
Following the notation introduced in \Cref{section:rcpd}, we consider a family of r.c.p.d.s\ $(\Q_\omega^{\alpha,\tau})_{\omega \in \Omega}$ of $\Q^\alpha$ with respect to $\widetilde{\cF}_\tau$, for any stopping time $\tau \in \cT_{0,T}(\widetilde{\F})$. Within this set-up, we define a generic payoff function of the form
\begin{align*}
\widetilde{J}(t,\omega,\alpha) \coloneqq \E^{\Q^{\smalltext{\alpha}\smalltext{,}\smalltext{t}}_\smalltext{\omega}} \bigg[ \int_t^T \widetilde{f}_s\big(X^1_{\cdot \land s}, \alpha_s \big) \d s + \widetilde{g}\big(X^1_{\cdot \land T}\big) \bigg] + G\Big( \E^{\Q^{\smalltext{\alpha}\smalltext{,}\smalltext{t}}_\smalltext{\omega}} \big[ \widetilde{\varphi}\big(X^1_{\cdot \land T}\big) \big] \Big), \; (t,\omega,\alpha) \in [0,T] \times \Omega \times \widetilde{\cA}.
\end{align*}
Here, the functions $\widetilde{f}: \Omega \times [0,T] \times \cC_m \times A \longrightarrow \R$, $\widetilde{g}: \Omega \times \cC_m \longrightarrow \R$, $G: \R^v \longrightarrow \R$ and $\widetilde{\varphi}: \Omega \times \cC_m \longrightarrow \R^v$ are assumed to satisfy the following conditions:
\begin{enumerate}[label={$(\roman*)$}]
\item the function $\Omega \times [0,T] \times \cC_m \times A \ni (\omega,t,x,a) \longmapsto \widetilde{f}_t(x,a)$ is ${\rm Prog}(\widetilde{\F})\otimes\cB(\cC_m)\otimes\cB(A)$-measurable;
\item for each component $i \in \{1,\ldots,v\}$, where $v \in \N^\star$, the functions $\Omega \times \cC_m \ni (\omega,x) \longmapsto \widetilde{g}(x)$ and $\Omega \times \cC_m \ni (\omega,x) \longmapsto \widetilde{\varphi}^i(x)$ are $\cF \otimes \cB([0,T])$-measurable;
\item the function $\R^v \ni m^\star \longmapsto G(m^\star)$ is Borel-measurable.
\end{enumerate}

We say that $\alpha^\star \in \widetilde{\cA}$ is a sub-game--perfect equilibrium if $\ell_\eps >0$ for any $\varepsilon>0$, where $\ell_\eps$ is defined in \Cref{def:NashEquilibrium}, and analogously in \Cref{def:meanFieldEquilibrium}, that is,
\begin{align*}
\ell_\varepsilon \coloneqq \inf \Big\{ \ell >0 : \exists (t,\alpha) \in [0,T] \times \widetilde{\cA}, \; \P \big[\big\{\omega \in \Omega: \widetilde{J}(t,\omega,\alpha^\star) < \widetilde{J}(t,\omega,\alpha \otimes_{t+\ell} \alpha^\star) - \varepsilon \ell \big\}\big] >0 \Big\}.
\end{align*}
Throughout this section, we assume that such a sub-game--perfect equilibrium exists, and we fix one such strategy $\alpha^\star \in \widetilde{\cA}$. By the definition of the payoff, it is clear that for each fixed pair $(t,\alpha) \in [0,T] \times \widetilde{A}$, the function $\Omega \ni \omega \longmapsto \widetilde{J}(t,\omega,\alpha)$ is $\cF$-measurable. Hence, we can define the value process $\widetilde{V}$ associated with the strategy $\alpha^\star$ as
\begin{equation*}
\widetilde{V}_t \coloneqq \widetilde{J}(t,\cdot,\alpha^\star), \; t \in [0,T].
\end{equation*}
Moreover, the function $\Omega \ni \omega \longmapsto \E^{\Q^{\smalltext{\alpha}^{\tinytext{\star}}\smalltext{,}\smalltext{t}}_\smalltext{\omega}} \big[ \widetilde{\varphi}\big(X^1_{\cdot \land T}\big) \big]$ is also $\cF$-measurable. Consequently, we can define the $\R^v$-valued process
\begin{align*}
\M^\star_t \coloneqq \big(M^{\star,1}_t,\ldots,M^{\star,v}_t\big) \coloneqq \E^{\Q^{\smalltext{\alpha}^{\tinytext{\star}}, \smalltext{t}}_\cdot} \big[ \widetilde{\varphi}\big(X^1_{\cdot \land T}\big) \big] \coloneqq \Big(\E^{\Q^{\smalltext{\alpha}^{\tinytext{\star}}, \smalltext{t}}_\cdot} \big[ \widetilde{\varphi}^1\big(X^1_{\cdot \land T}\big) \big], \ldots, \E^{\Q^{\smalltext{\alpha}^{\tinytext{\star}}, \smalltext{t}}_\cdot} \big[ \widetilde{\varphi}^v\big(X^1_{\cdot \land T}\big) \big]\Big), \; t \in [0,T].
\end{align*}

\begin{assumption}
\begin{enumerate}[label={$(\roman*)$}]\label{assumpDPP}
\item\label{boundedPhisMart} The function $\Omega \times \cC_m \ni (\omega,x) \longmapsto \widetilde{\varphi}(x)$ is bounded$;$
\item\label{GLipcont} the function $\R^v \ni m^\star \coloneqq ((m^\star)^1,\ldots,(m^\star)^v) \longmapsto G(m^\star)$ is twice continuously differentiable with Lipschitz-continuous first- and second-order derivatives $\partial_{m^{\smalltext{i}}}G(m^\star)$, $\partial^2_{m^{\smalltext{i}},m^{\smalltext{j}}}G(m^\star)$, for any $(i,j) \in \{1,\ldots,v\}^2;$
\item\label{modulusCondExpMart} there exists a constant $c>0$ and a modulus of continuity $\rho$ such that
\begin{align*}
\E^{\Q^{\smalltext{\alpha}\smalltext{,}\smalltext{t}}_{\smalltext{\cdot}}} \Bigg[ \sum_{i =1}^v \Big| \E^{\Q^{\smalltext{\alpha} \smalltext{,}\smalltext{\tilde{t}}}_\smalltext{\cdot}} \big[ M^{\star,i}_{{t}^\prime} \big] - M^{\star,i}_{\tilde{t}} \Big|^2 \Bigg] \leq c |t^\prime-\tilde{t}| \rho \big(|t^\prime-\tilde{t}|\big), \; \P\text{\rm--a.s.}, \; (\alpha,t,\tilde{t},t^\prime) \in \widetilde{\cA} \times [0,T] \times [t,T] \times [\tilde{t},T].
\end{align*}
\end{enumerate}
\end{assumption}

\begin{theorem}\label{thm:DPP}
Let {\rm\Cref{assumpDPP}} hold. Given the set-up introduced so far, let $\alpha^\star \in \widetilde{\cA}$ be a sub-game--perfect equilibrium. Then, for any $(t,\tilde{t}) \in [0,T] \times [t,T]$, it holds that
\begin{align*}
\widetilde{V}_t = \esssup_{\alpha \in \widetilde{\cA}} \E^{\Q^{\smalltext{\alpha}\smalltext{,}\smalltext{t}}_{\smalltext{\cdot}}} \Bigg[& \widetilde{V}_{\tilde{t}} + \int_t^{\tilde{t}} \widetilde{f}_s\big(X^1_{\cdot \land s}, \alpha_s \big) \d s - \frac{1}{2} \int_t^{\tilde{t}} \sum_{(i,j) \in \{1,\ldots,v\}^2} \partial^2_{m^{\smalltext{i}},m^{\smalltext{j}}} G\big(\M^{\star}_s\big) \d \big[M^{\star,i},M^{\star,j}\big]_s \Bigg], \; \P\text{\rm--a.s.}
\end{align*}
\end{theorem}

\begin{remark}
First, we recall that {\rm\Cref{assumpDPP}.\ref{boundedPhisMart}} implies that the process $\M^{\star}$ is an $(\widetilde{\F},\Q^{\alpha^{\smalltext{\star}}})$--martingale. By the martingale representation property stated in {\rm\Cref{lemma:martingaleRepresentation}}, it admits a $\P$-modification that is right-continuous and $\P\text{\rm--a.s.}$ continuous. With a slight abuse of notation, we continue to denote this $\P$-modification by $\M^\star$.

\medskip
It is worth emphasising that, for the martingale property to hold, it actually suffices to assume that $\M^\star$ is $\Q^{\alpha^{\smalltext{\star}}}$-integrable. However, the stronger boundedness assumption on $\M^\star$ is used at two key points in the proof of {\rm\Cref{thm:DPP}}, specifically where we examine the convergence of the terms we will denote by $J^1$ and $J^3$. In both cases, boundedness plays a crucial role: it allows us both to control the square of the quadratic variation of $\M^\star$, and to apply the dominated convergence theorem. A localisation argument would not suffice in this context, as we are unable to interchange the limits corresponding to the vanishing mesh size and the sequence of stopping times approaching the terminal time $T$.

\medskip
Finally, as observed in {\rm\cite[Remark 7.1]{hernandez2023me}}, the non-linear dependence of the payoff on the expected value makes it necessary to have access to the quadratic variations $[M^{\star,i}, M^{\star,j}]$, for all $(i,j) \in \{1,\ldots,v\}^2$. We simply write $[M^{\star,i}, M^{\star,j}]$ without specifying the underlying probability measure since these quantities are invariant under changes of measure, as we work with equivalent measures $($see, for instance, {\rm\cite[Theorem III.3.13]{jacod2003limit}}$)$.
\end{remark}
 
Before proving the result stated in \Cref{thm:DPP}, we introduce an intermediate step that will be useful for the proof. Our approach follows the methodology outlined in \cite[Lemma 2.10.1, Proposition 2.10.2 and Theorem 2.4.3]{hernandez2021general}. This intermediate result investigates the local behaviour of the value function of the game by leveraging the notion of $\varepsilon$-optimality of a sub-game--perfect equilibrium.  

\begin{lemma}\label{lemma:DPP_sum}
Fix an arbitrary $\varepsilon>0$ and two times $(t,\tilde{t}) \in [0,T] \times [t,T]$. We consider a partition $\Pi^\ell \coloneqq (t^{{\ell}}_k)_{k\in\{0,\ldots,n^\smalltext{\ell}\}}$ of the interval $[t,\tilde{t}]$ with mesh size smaller than $\ell \in (0,\ell_\varepsilon)$, and such that $t^\ell_0 = t$ and $t^\ell_{n^\smalltext{\ell}} = \tilde{t}$. Then, it holds that
\begin{align}\label{align:DPP_sum}
\widetilde{V}_t &\geq \esssup_{\alpha \in \widetilde{\cA}} \E^{\Q^{\smalltext{\alpha} \smalltext{,} \smalltext{t}}_{\smalltext{\cdot}}} \Bigg[ \widetilde{V}_{\tilde{t}} + \int_{t}^{\tilde{t}} \widetilde{f}_s\big(X^1_{\cdot \land s}, {\alpha}_s\big) \d s + \sum_{k=0}^{n^\smalltext{\ell}-1} \bigg( G\Big(\E^{\Q^{\smalltext{\alpha}\smalltext{,}\smalltext{t}^\tinytext{\ell}_\tinytext{k}}_\smalltext{\cdot}}\Big[\M^{\star}_{t^\smalltext{\ell}_{\smalltext{k}\smalltext{+}\smalltext{1}}}\Big]\Big) - G\Big(\M^{\star}_{t^\smalltext{\ell}_{\smalltext{k}\smalltext{+}\smalltext{1}}}\Big) \bigg) \Bigg] - n^\ell \varepsilon \ell, \; \P\text{\rm--a.s.}
\end{align}
\end{lemma}

\begin{proof}
We begin by proving the result in a simplified case where, instead of a general partition, we focus on only two time points. Specifically, let $\varepsilon>0$. The definition of the sub-game--perfect equilibrium $\alpha^\star \in \widetilde{\cA}$ implies that $\ell_\eps >0$. Accordingly, we fix $(\ell,t,\tilde{t}) \in (0,\ell_\varepsilon) \times [0,T] \times [t,t+\ell)$. Then, for some $\alpha \in \widetilde{\cA}$, it holds that
\begin{align*}
\widetilde{V}_t = \widetilde{J}(t,\cdot,\alpha^\star) 
	&\geq \widetilde{J}\big(t,\cdot,\alpha \otimes_{{\tilde{t}}} \alpha^\star\big) - \varepsilon \ell \\
	&= \E^{\Q^{\smalltext{\alpha} \smalltext{\otimes}_{\tinytext{\tilde{t}}} \smalltext{\alpha}^\tinytext{\star} \smalltext{,} \smalltext{t}}_\smalltext{\cdot}} \bigg[ \int_t^T \widetilde{f}_s\big(X^1_{\cdot \land s}, (\alpha \otimes_{{\tilde{t}}} \alpha^\star)_s\big) \d s + \widetilde{g}\big(X^1_{\cdot \land T}\big) \bigg] + G\Big(\E^{\Q^{\smalltext{\alpha} \smalltext{\otimes}_{\tinytext{\tilde{t}}} \smalltext{\alpha}^\tinytext{\star} \smalltext{,} \smalltext{t}}_\smalltext{\cdot}} \big[ \widetilde{\varphi}(X^1_{\cdot \land T}) \big]\Big) - \varepsilon \ell \\
	&= \E^{\Q^{\smalltext{\alpha} \smalltext{\otimes}_{\tinytext{\tilde{t}}} \smalltext{\alpha}^\tinytext{\star} \smalltext{,} \smalltext{t}}_\smalltext{\cdot}} \bigg[ \widetilde{V}_{\tilde{t}} + \int_t^{\tilde{t}} \widetilde{f}_s\big(X^1_{\cdot \land s}, \alpha_s\big) \d s + G\Big(\E^{\Q^{\smalltext{\alpha} \smalltext{\otimes}_{\tinytext{\tilde{t}}} \smalltext{\alpha}^\tinytext{\star} \smalltext{,} \smalltext{t}}_\smalltext{\cdot}} \big[ \widetilde{\varphi}(X^1_{\cdot \land T}) \big]\Big) -  G\Big(\E^{\Q^{\smalltext{\alpha}^\tinytext{\star} \smalltext{,} \smalltext{\tilde{t}}}_{\cdot}} \big[ \widetilde{\varphi}(X^1_{\cdot \land T}) \big]\Big) \bigg] - \varepsilon \ell \\
	&= \E^{\Q^{\smalltext{\alpha} \smalltext{\otimes}_{\tinytext{\tilde{t}}} \smalltext{\alpha}^\tinytext{\star} \smalltext{,} \smalltext{t}}_\smalltext{\cdot}} \bigg[ \widetilde{V}_{\tilde{t}} + \int_t^{\tilde{t}} \widetilde{f}_s\big(X^1_{\cdot \land s}, \alpha_s\big) \d s +  G\Big( \E^{\Q^{\smalltext{\alpha} \smalltext{,} \smalltext{t}}_\smalltext{\cdot}} \big[ \M^{\star}_{\tilde{t}} \big] \Big) - G\big(\M^{\star}_{\tilde{t}}\big) \bigg] - \varepsilon \ell \\
	&= \E^{\Q^{\smalltext{\alpha} \smalltext{,} \smalltext{t}}_\smalltext{\cdot}} \bigg[ \widetilde{V}_{\tilde{t}} + \int_t^{\tilde{t}} \widetilde{f}_s\big(X^1_{\cdot \land s}, \alpha_s\big) \d s + G\Big( \E^{\Q^{\smalltext{\alpha} \smalltext{,} \smalltext{t}}_\smalltext{\cdot}} \big[ \M^{\star}_{\tilde{t}} \big] \Big) - G\big(\M^{\star}_{\tilde{t}}\big) \bigg] - \varepsilon \ell, \; \P\text{\rm--a.s.}
\end{align*}
The arbitrariness of $\alpha \in \widetilde{\cA}$ implies that 
\begin{align*}
\widetilde{V}_t \geq \esssup_{\alpha \in \widetilde{\cA}} \E^{\Q^{\smalltext{\alpha} \smalltext{,} \smalltext{t}}_\smalltext{\cdot}} \bigg[ \widetilde{V}_{\tilde{t}} + \int_t^{\tilde{t}} \widetilde{f}_s\big(X^1_{\cdot \land s}, \alpha_s\big) \d s + G\Big( \E^{\Q^{\smalltext{\alpha} \smalltext{,} \smalltext{t}}_\smalltext{\cdot}} \big[ \M^{\star}_{\tilde{t}} \big] \Big) - G\big(\M^{\star}_{\tilde{t}}\big) \bigg] - \varepsilon \ell, \; \P\text{\rm--a.s.}
\end{align*}

\medskip
We now extend the previous result by considering a partition rather than just two points. Specifically, let $\Pi^\ell \coloneqq (t^{{\ell}}_k)_{k\in\{0,\ldots,n^\smalltext{\ell}\}}$ be a partition of $[t,\tilde{t}]$ with mesh size smaller than $\ell \in (0,\ell_\varepsilon)$, such that $t^\ell_0 = t$ and $t^\ell_{n^\smalltext{\ell}} = \tilde{t}$. Then, $\P$--a.s., it holds that
\begin{align*}
\widetilde{V}_t &\geq \esssup_{\alpha \in \widetilde{\cA}} \E^{\Q^{\smalltext{\alpha} \smalltext{,} \smalltext{t}^\tinytext{\ell}_\tinytext{0}}_{\smalltext{\cdot}}} \bigg[ \widetilde{V}_{t^{\smalltext{\ell}}_{\smalltext{1}}} + \int_{t^\smalltext{\ell}_\smalltext{0}}^{t^\smalltext{\ell}_\smalltext{1}} \widetilde{f}_s\big(X^1_{\cdot \land s}, \alpha_s\big) \d s + G\Big(\E^{\P^{\smalltext{\alpha}\smalltext{,}\smalltext{t}^\tinytext{\ell}_\tinytext{0}}_\smalltext{\cdot}} \big[ \M^{\star}_{t^\smalltext{\ell}_\smalltext{1}} \big] \Big) - G\big(\M^{\star}_{t^\smalltext{\ell}_\smalltext{1}}\big) \bigg] - \varepsilon \ell \\
	&\geq \esssup_{\alpha \in \widetilde{\cA}} \E^{\Q^{\smalltext{\alpha} \smalltext{,} \smalltext{t}^\tinytext{\ell}_\tinytext{0}}_{\smalltext{\cdot}}} \Bigg[ \esssup_{\tilde{\alpha} \in \widetilde{\cA}} \E^{\Q^{\smalltext{\tilde{\alpha}} \smalltext{,} \smalltext{t}^\tinytext{\ell}_\tinytext{1}}_{\smalltext{\cdot}}} \bigg[ \widetilde{V}_{t^{\smalltext{\ell}}_{\smalltext{2}}} + \int_{t^\smalltext{\ell}_\smalltext{1}}^{t^\smalltext{\ell}_\smalltext{2}} \widetilde{f}_s\big(X^1_{\cdot \land s}, \tilde{\alpha}_s\big) \d s + G\Big(\E^{\P^{\smalltext{\tilde{\alpha}}\smalltext{,}\smalltext{t}^\tinytext{\ell}_\tinytext{1}}_\smalltext{\cdot}} \big[ \M^{\star}_{t^\smalltext{\ell}_\smalltext{2}} \big]\Big) - G\big(\M^{\star}_{t^\smalltext{\ell}_\smalltext{2}}\big) \bigg] \\
	&\qquad \qquad \qquad \quad \; + \int_{t^\smalltext{\ell}_\smalltext{0}}^{t^\smalltext{\ell}_\smalltext{1}} \widetilde{f}_s\big(X^1_{\cdot \land s}, \alpha_s\big) \d s + G\Big(\E^{\P^{\smalltext{\alpha}\smalltext{,}\smalltext{t}^\tinytext{\ell}_\tinytext{0}}_\smalltext{\cdot}} \big[ \M^{\star}_{t^\smalltext{\ell}_\smalltext{1}} \big]\Big) - G\big(\M^{\star}_{t^\smalltext{\ell}_\smalltext{1}}\big) \Bigg] - 2 \varepsilon \ell \\
	&\geq \esssup_{\alpha \in \widetilde{\cA}} \E^{\Q^{\smalltext{\alpha} \smalltext{,} \smalltext{t}^\tinytext{\ell}_\tinytext{0}}_{\smalltext{\cdot}}} \Bigg[ \widetilde{V}_{t^{\smalltext{\ell}}_{\smalltext{1}}} + \int_{t^\smalltext{\ell}_\smalltext{0}}^{t^\smalltext{\ell}_\smalltext{2}} \widetilde{f}_s\big(X^1_{\cdot \land s}, {\alpha}_s\big) \d s + G\Big(\E^{\P^{\smalltext{{\alpha}}\smalltext{,}\smalltext{t}^\tinytext{\ell}_\tinytext{1}}_\smalltext{\cdot}} \big[ \M^{\star}_{t^\smalltext{\ell}_\smalltext{2}} \big] \Big) - G\big(\M^{\star}_{t^\smalltext{\ell}_\smalltext{2}}\big) + G\Big(\E^{\P^{\smalltext{\alpha}\smalltext{,}\smalltext{t}^\tinytext{\ell}_\tinytext{0}}_\smalltext{\cdot}} \big[ \M^{\star}_{t^\smalltext{\ell}_\smalltext{1}} \big] \Big) - G\big(\M^{\star}_{t^\smalltext{\ell}_\smalltext{1}}\big) \Bigg] - 2 \varepsilon \ell \\
	&\geq \esssup_{\alpha \in \widetilde{\cA}} \E^{\Q^{\smalltext{\alpha} \smalltext{,} \smalltext{t}}_{\smalltext{\cdot}}} \Bigg[ \widetilde{V}_{\tilde{t}} + \int_{t}^{\tilde{t}} \widetilde{f}_s\big(X^1_{\cdot \land s}, {\alpha}_s\big) \d s + \sum_{k=0}^{n^\smalltext{\ell}-1}  \bigg( G\Big(\E^{\Q^{\smalltext{\alpha}\smalltext{,}\smalltext{t}^\tinytext{\ell}_\tinytext{k}}_\smalltext{\cdot}} \big[ \M^{\star}_{t^\smalltext{\ell}_{\smalltext{k}\smalltext{+}\smalltext{1}}} \big] \Big) - G\big(\M^{\star}_{t^\smalltext{\ell}_{\smalltext{k}\smalltext{+}\smalltext{1}}}\big) \bigg) \Bigg] - n^\ell \varepsilon \ell.
\end{align*}
Here, the last inequality holds by a simple iteration over the countable index set $\{0,\ldots,n^\ell-1\}$.
\end{proof}

Having established the preliminary estimate for a fixed partition of $[t,\tilde{t}]$, we proceed to derive the dynamic programming principle by passing to the limit as the partition mesh size tends to zero.

\begin{proof}[Proof of \Cref{thm:DPP}]
We prove the result by verifying both inequalities separately, starting with the inequality in which the left-hand side is greater than or equal to the right-hand side. This follows by taking the limit in the inequality stated in \eqref{align:DPP_sum}, following the approach of \cite[Theorem 2.4.3]{hernandez2021general}. To make the argument precise, we recall that for any $\varepsilon > 0$, the definition of sub-game--perfect equilibrium guarantees the existence of a corresponding $\ell_\eps >0$. As in \Cref{lemma:DPP_sum} or equivalently in \Cref{align:DPP_sum}, we consider a partition $\Pi^\ell \coloneqq (t^{{\ell}}_k)_{k\in\{0,\ldots,n^\smalltext{\ell}\}}$ of the interval $[t,\tilde{t}]$ with mesh size smaller than $\ell \in (0,\ell_\varepsilon)$, and such that $t^\ell_0 = t$ and $t^\ell_{n^\smalltext{\ell}} = \tilde{t}$. This gives us
\begin{align}\label{align:sumDPPtoLimit}
\widetilde{V}_t &\geq \esssup_{\alpha \in \widetilde{\cA}} \E^{\Q^{\smalltext{\alpha} \smalltext{,} \smalltext{t}}_{\smalltext{\cdot}}} \Bigg[ \widetilde{V}_{\tilde{t}} + \int_{t}^{\tilde{t}} \widetilde{f}_s\big(X^1_{\cdot \land s}, {\alpha}_s\big) \d s + \sum_{k=0}^{n^\smalltext{\ell}-1}  \bigg( G\Big(\E^{\Q^{\smalltext{\alpha}\smalltext{,}\smalltext{t}^\tinytext{\ell}_\tinytext{k}}_\smalltext{\cdot}} \big[ \M^{\star}_{t^\smalltext{\ell}_{\smalltext{k}\smalltext{+}\smalltext{1}}} \big] \Big) - G\big(\M^{\star}_{t^\smalltext{\ell}_{\smalltext{k}\smalltext{+}\smalltext{1}}}\big) \bigg) \Bigg] - n^\ell \varepsilon \ell, \; \P\text{\rm--a.s.}
\end{align}

The next step is to consider the limits $\varepsilon \longrightarrow 0$ and $\ell \longrightarrow 0$. It is important to observe that taking the limit $\varepsilon \longrightarrow 0$ alone may not suffice to conclude the proof. This is due to the behaviour of the sequence $(\ell_\varepsilon)_{\varepsilon >0}$, which, although bounded and non-decreasing in $\varepsilon$ (as observed in \cite[Remark 2.7]{hernandez2023me}), is not guaranteed to converge to zero as $\varepsilon \longrightarrow 0$. This leads to two possible scenarii: either $\ell_\varepsilon \longrightarrow 0$, or $\ell_\varepsilon \longrightarrow \tilde{\ell}_0 >0$ as $\varepsilon \longrightarrow 0$. In the first case, we consider the partition $\Pi^\ell \coloneqq (t^{{\ell}}_k)_{k\in\{0,\ldots,n^\smalltext{\ell}\}}$ with $n_\ell \coloneqq \lceil (\tilde{t} - t) /\ell \rceil$. It follows that taking the limit $\varepsilon \longrightarrow 0$ is sufficient, since it also ensures that the mesh size of the partition becomes arbitrarily small due to $\ell_\varepsilon \longrightarrow 0$, and the error term $n^\ell \varepsilon \ell$ goes to zero. In the second case, however, taking only $\varepsilon \rightarrow 0$ is not enough to control the mesh size of the partition. To address this, we proceed as follows: we fix some $\ell < \tilde{\ell}_0$ at the beginning of the argument, ensuring that the chosen partition is independent of $\varepsilon$. We then let $\varepsilon \longrightarrow 0$, and only afterwards take the limit $\ell \longrightarrow 0$. In what follows, we therefore consider only the limit $\ell \longrightarrow 0$, keeping in mind that the limit of the essential supremum is certainly an upper bound for the essential supremum of the limit.

\medskip
In the computations that follow, we fix an admissible strategy $\alpha \in \widetilde{\cA}$ and a state $\omega \in \Omega$, although this will be left implicit. Moreover, all equalities are understood to hold pointwise in $\omega$, unless otherwise stated. We also recall that \Cref{assumpDPP}.\ref{boundedPhisMart} ensures that $\M^{\star}$ is a $(\widetilde{\F},\Q^{\alpha^{\smalltext{\star}}})$-martingale, and by \Cref{lemma:martingaleRepresentation}, it admits a $\P$-modification that is right-continuous and $\P\text{--a.s.}$ continuous, which we continue to denote by $\M^{\star}$ with a slight abuse of notation. Rather than analysing the sum of differences in \eqref{align:sumDPPtoLimit} directly, we decompose it into a sum of intermediate terms, each of which is more tractable to study in the limit. To this end, we define the increments
\begin{align*}
\Delta M^{\star,i}_{t^\smalltext{\ell}_{\smalltext{k}\smalltext{+}\smalltext{1}}} \coloneqq M^{\star,i}_{t^\smalltext{\ell}_{\smalltext{k}\smalltext{+}\smalltext{1}}} - M^{\star,i}_{t^\smalltext{\ell}_{\smalltext{k}}}, \; (i,k) \in \{1,\ldots, v\} \times \{0, \ldots, n^\ell-1\}.
\end{align*}
We rewrite the sum in \eqref{align:sumDPPtoLimit} as
\begin{align*}
\sum_{k=0}^{n^\smalltext{\ell}-1}  \bigg( G\Big(\E^{\Q^{\smalltext{\alpha}\smalltext{,}\smalltext{t}^\tinytext{\ell}_\tinytext{k}}_\smalltext{\cdot}} \big[ \M^{\star}_{t^\smalltext{\ell}_{\smalltext{k}\smalltext{+}\smalltext{1}}} \big] \Big) - G\big(\M^{\star}_{t^\smalltext{\ell}_{\smalltext{k}\smalltext{+}\smalltext{1}}}\big) \bigg) = J^1 + J^2 + J^3,
\end{align*}
where 
\begin{align*}
J^1 \coloneqq \sum_{k=0}^{n^\smalltext{\ell}-1} \bigg(& G\big(\M^{\star}_{t^\smalltext{\ell}_{\smalltext{k}}}\big) - G\big(\M^{\star}_{t^\smalltext{\ell}_{\smalltext{k}\smalltext{+}\smalltext{1}}}\big) + \sum_{i \in \{1,\ldots,v\}} \partial_{m^{\smalltext{i}}} G\big(\M^{\star}_{t^\smalltext{\ell}_{\smalltext{k}}}\big) \Delta M^{\star,i}_{t^\smalltext{\ell}_{\smalltext{k}\smalltext{+}\smalltext{1}}} + \frac{1}{2} \sum_{(i,j) \in \{1,\ldots,v\}^2} \partial^2_{m^{\smalltext{i}},m^{\smalltext{j}}} G\big(\M^{\star}_{t^\smalltext{\ell}_{\smalltext{k}}}\big) \Delta M^{\star,i}_{t^\smalltext{\ell}_{\smalltext{k}\smalltext{+}\smalltext{1}}} \Delta M^{\star,j}_{t^\smalltext{\ell}_{\smalltext{k}\smalltext{+}\smalltext{1}}} \bigg), \\
J^2 \coloneqq \sum_{k=0}^{n^\smalltext{\ell}-1} \bigg(& G\Big(\E^{\Q^{\smalltext{\alpha}\smalltext{,}\smalltext{t}^\tinytext{\ell}_\tinytext{k}}_\smalltext{\cdot}} \big[ \M^{\star}_{t^\smalltext{\ell}_{\smalltext{k}\smalltext{+}\smalltext{1}}} \big] \Big) - G\big(\M^{\star}_{t^\smalltext{\ell}_{\smalltext{k}}}\big) - \sum_{i \in \{1,\ldots,v\}} \partial_{m^{\smalltext{i}}} G\big(\M^{\star}_{t^\smalltext{\ell}_{\smalltext{k}}}\big) \Delta M^{\star,i}_{t^\smalltext{\ell}_{\smalltext{k}\smalltext{+}\smalltext{1}}} \bigg), \\
J^3 \coloneqq \sum_{k=0}^{n^\smalltext{\ell}-1} \bigg(& -\frac{1}{2} \sum_{(i,j) \in \{1,\ldots,v\}^2} \partial^2_{m^{\smalltext{i}},m^{\smalltext{j}}} G\big(\M^{\star}_{t^\smalltext{\ell}_{\smalltext{k}}}\big) \Delta M^{\star,i}_{t^\smalltext{\ell}_{\smalltext{k}\smalltext{+}\smalltext{1}}} \Delta M^{\star,j}_{t^\smalltext{\ell}_{\smalltext{k}\smalltext{+}\smalltext{1}}} \bigg).
\end{align*}

In order to study $J^1$, we introduce some functions $\vartheta^{k,i} : \Omega \longrightarrow [0,1]$, and define
\begin{align*}
M^{\vartheta,\star,i}_{t^\smalltext{\ell}_{\smalltext{k}},t^\smalltext{\ell}_{\smalltext{k}\smalltext{+}\smalltext{1}}} \coloneqq \vartheta^{k,i} M^{\star,i}_{t^\smalltext{\ell}_{\smalltext{k}\smalltext{+}\smalltext{1}}} + \big(1-\vartheta^{k,i}\big) M^{\star,i}_{t^\smalltext{\ell}_{\smalltext{k}}}, \; (i,k) \in \{1,\ldots,v\} \times \{0, \ldots, n^\ell-1\}.
\end{align*}
The second-order Taylor expansion yields
\begin{align*}
J^1 = \frac{1}{2} \sum_{k=0}^{n^\smalltext{\ell}-1} \sum_{(i,j) \in \{1,\ldots,v\}^2} \Big( \partial^2_{m^{\smalltext{i}},m^{\smalltext{j}}} G\big(\M^{\star}_{t^\smalltext{\ell}_{\smalltext{k}}}\big) - \partial^2_{m^{\smalltext{i}},m^{\smalltext{j}}} G\big(\M^{\vartheta,\star}_{t^\smalltext{\ell}_{\smalltext{k}},t^\smalltext{\ell}_{\smalltext{k}\smalltext{+}\smalltext{1}}}\big) \Big) \Delta M^{\star,i}_{t^\smalltext{\ell}_{\smalltext{k}\smalltext{+}\smalltext{1}}} \Delta M^{\star,j}_{t^\smalltext{\ell}_{\smalltext{k}\smalltext{+}\smalltext{1}}}.
\end{align*}
We notice we may choose $(\vartheta^{k,i})_{i \in \{1,\ldots,v\}}$ so that each term in the sum over $k \in \{0, \ldots, n^\ell-1\}$ is $\cF_{t^\smalltext{\ell}_{\smalltext{k}}}$-measurable. It holds that 
\begin{align*}
\Big| \E^{\Q^{\smalltext{\alpha} \smalltext{,} \smalltext{t}}_{\smalltext{\cdot}}} \big[ J^1 \big] \Big| &= \Bigg| \E^{\Q^{\smalltext{\alpha} \smalltext{,} \smalltext{t}}_{\smalltext{\cdot}}} \Bigg[ \frac{1}{2} \sum_{k=0}^{n^\smalltext{\ell}-1} \sum_{(i,j) \in \{1,\ldots,v\}^2} \Big( \partial^2_{m^{\smalltext{i}},m^{\smalltext{j}}} G\big(\M^{\star}_{t^\smalltext{\ell}_{\smalltext{k}}}\big) - \partial^2_{m^{\smalltext{i}},m^{\smalltext{j}}} G\big(\M^{\vartheta,\star}_{t^\smalltext{\ell}_{\smalltext{k}},t^\smalltext{\ell}_{\smalltext{k}\smalltext{+}\smalltext{1}}}\big) \Big) \Delta M^{\star,i}_{t^\smalltext{\ell}_{\smalltext{k}\smalltext{+}\smalltext{1}}} \Delta M^{\star,j}_{t^\smalltext{\ell}_{\smalltext{k}\smalltext{+}\smalltext{1}}} \Bigg] \Bigg| \\
	&\leq \ell_{\partial^\smalltext{2} G} \E^{\Q^{\smalltext{\alpha} \smalltext{,} \smalltext{t}}_{\smalltext{\cdot}}} \Bigg[ \max_{k \in \{ 0, \ldots, n^\smalltext{\ell}-1\}} \Bigg\{ \sum_{j =1}^v \big|\Delta M^{\star,j}_{t^\smalltext{\ell}_{\smalltext{k}\smalltext{+}\smalltext{1}}}\big| \Bigg\} \sum_{k=0}^{n^\smalltext{\ell}-1} \sum_{i =1}^v \big|\Delta M^{\star,i}_{t^\smalltext{\ell}_{\smalltext{k}\smalltext{+}\smalltext{1}}}\big|^2 \Bigg] \\
	&\leq \ell_{\partial^\smalltext{2} G} \Bigg( \E^{\Q^{\smalltext{\alpha} \smalltext{,} \smalltext{t}}_{\smalltext{\cdot}}} \Bigg[ \max_{k \in \{ 0, \ldots, n^\smalltext{\ell}-1\}} \Bigg\{ \sum_{j =1}^v \big|\Delta M^{\star,j}_{t^\smalltext{\ell}_{\smalltext{k}\smalltext{+}\smalltext{1}}}\big| \Bigg\}^2 \Bigg] \Bigg)^{\frac{1}{2}} \sum_{i =1}^v \Bigg( \E^{\Q^{\smalltext{\alpha} \smalltext{,} \smalltext{t}}_{\smalltext{\cdot}}} \Bigg[ \Bigg( \sum_{k=0}^{n^\smalltext{\ell}-1} \big|\Delta M^{\star,i}_{t^\smalltext{\ell}_{\smalltext{k}\smalltext{+}\smalltext{1}}}\big|^2 \Bigg)^2 \Bigg] \Bigg)^{\frac{1}{2}}, \; \P\text{\rm--a.s.}
\end{align*}
The first inequality follows from the Lipschitz-continuity of the functions $\partial^2_{m^{\smalltext{i}},m^{\smalltext{j}}} G$, for all $(i,j) \in \{1,\ldots,v\}^2$, while the last one is a direct consequence of the Cauchy-Schwarz inequality. To compute the limit as $\ell \rightarrow 0$, we fix an index $i \in \{1,\ldots,v\}$. We observe that \Cref{assumpDPP}.\ref{boundedPhisMart} implies that $\M^{\star}$ is a bounded $(\widetilde{\F},\Q^{\alpha^{\smalltext{\star}}})$-martingale, and we denote the bound by $c_\star>0$. Then, proceeding as in the proof of \citeauthor*{karatzas1991brownian} \cite[Lemma 1.5.9]{karatzas1991brownian}, we have
\begin{align*}
&\Bigg( \E^{\Q^{\smalltext{\alpha} \smalltext{,} \smalltext{t}}_{\smalltext{\cdot}}} \Bigg[ \max_{k \in \{ 0, \ldots, n^\smalltext{\ell}-1\}} \Bigg\{ \sum_{j =1}^v \big|\Delta M^{\star,j}_{t^\smalltext{\ell}_{\smalltext{k}\smalltext{+}\smalltext{1}}}\big| \Bigg\}^2 \Bigg] \Bigg)^{\frac{1}{2}} \Bigg( \E^{\Q^{\smalltext{\alpha} \smalltext{,} \smalltext{t}}_{\smalltext{\cdot}}} \Bigg[ \Bigg( \sum_{k=0}^{n^\smalltext{\ell}-1} \big|\Delta M^{\star,i}_{t^\smalltext{\ell}_{\smalltext{k}\smalltext{+}\smalltext{1}}}\big|^2 \Bigg)^2 \Bigg] \Bigg)^{\frac{1}{2}}\\
	&= \Bigg( \E^{\Q^{\smalltext{\alpha} \smalltext{,} \smalltext{t}}_{\smalltext{\cdot}}} \Bigg[ \max_{k \in \{ 0, \ldots, n^\smalltext{\ell}-1\}} \Bigg\{ \sum_{j =1}^v \big|\Delta M^{\star,j}_{t^\smalltext{\ell}_{\smalltext{k}\smalltext{+}\smalltext{1}}}\big| \Bigg\}^2 \Bigg] \Bigg)^{\frac{1}{2}} \Bigg( \E^{\Q^{\smalltext{\alpha} \smalltext{,} \smalltext{t}}_{\smalltext{\cdot}}} \Bigg[ \sum_{k=0}^{n^\smalltext{\ell}-1} \big|\Delta M^{\star,i}_{t^\smalltext{\ell}_{\smalltext{k}\smalltext{+}\smalltext{1}}}\big|^4 + 2 \sum_{k=0}^{n^\smalltext{\ell}-2} \sum_{\tilde{k}=k+1}^{n^\smalltext{\ell}-1} \big|\Delta M^{\star,i}_{t^\smalltext{\ell}_{\smalltext{k}\smalltext{+}\smalltext{1}}}\big|^2 \big|\Delta M^{\star,i}_{t^\smalltext{\ell}_{\smalltext{\tilde{k}}\smalltext{+}\smalltext{1}}}\big|^2 \Bigg] \Bigg)^{\frac{1}{2}} \\
	&\leq 2 \sqrt{2} c_\star^2 \Bigg( \E^{\Q^{\smalltext{\alpha} \smalltext{,} \smalltext{t}}_{\smalltext{\cdot}}} \Bigg[ \max_{k \in \{ 0, \ldots, n^\smalltext{\ell}-1\}} \Bigg\{ \sum_{j =1}^v \big|\Delta M^{\star,j}_{t^\smalltext{\ell}_{\smalltext{k}\smalltext{+}\smalltext{1}}}\big| \Bigg\}^2 \Bigg] \Bigg)^{\frac{1}{2}} \xrightarrow{\ell \rightarrow 0} 0, \; \P\text{\rm--a.s.}
\end{align*}
The limit follows from the dominated convergence theorem and the uniform continuity on $[0,T]$ of the paths of $\M^{\star}$, $\P$--a.s.

\medskip
We consider the term $J^2$. Analogously to the study of $J^1$, we introduce auxiliary functions $\vartheta^{k,i}: \Omega \longrightarrow [0,1]$, and we define
\begin{align*}
M^{\vartheta,\star,i}_{t^\smalltext{\ell}_{\smalltext{k}}} \coloneqq \vartheta^{k,i} \E^{\Q^{\smalltext{\alpha}\smalltext{,}\smalltext{t}^\tinytext{\ell}_\tinytext{k}}_\smalltext{\cdot}} \Big[ M^{\star,i}_{t^\smalltext{\ell}_{\smalltext{k}\smalltext{+}\smalltext{1}}} \Big] + (1-\vartheta^{k,i}) M^{\star,i}_{t^\smalltext{\ell}_{\smalltext{k}}}, \; (i,k) \in \{1, \ldots, v\} \times \{0, \ldots, n^\ell-1\}.
\end{align*}
A first-order Taylor expansion yields
\begin{align*}
\Big| \E^{\Q^{\smalltext{\alpha} \smalltext{,} \smalltext{t}}_{\smalltext{\cdot}}} \big[ J^2 \big] \Big| &= \Bigg| \E^{\Q^{\smalltext{\alpha}\smalltext{,}\smalltext{t}}_\smalltext{\cdot}} \Bigg[ \sum_{k=0}^{n^\smalltext{\ell}-1} \sum_{i=1}^v \bigg( \partial_{m^{\smalltext{i}}} G\big(\M^{\vartheta,\star}_{t^\smalltext{\ell}_{\smalltext{k}}}\big) \Big( \E^{\Q^{\smalltext{\alpha} \smalltext{,}\smalltext{t}^\tinytext{\ell}_\tinytext{k}}_\smalltext{\cdot}} \Big[ M^{\star,i}_{t^\smalltext{\ell}_{\smalltext{k}\smalltext{+}\smalltext{1}}} \Big] - M^{\star,i}_{t^\smalltext{\ell}_{\smalltext{k}}} \Big) - \partial_{m^{\smalltext{i}}} G\big(\M^{\star}_{t^\smalltext{\ell}_{\smalltext{k}}}\big) \Delta M^{\star,i}_{t^\smalltext{\ell}_{\smalltext{k}\smalltext{+}\smalltext{1}}} \bigg) \Bigg] \Bigg| \\
	&= \Bigg| \E^{\Q^{\smalltext{\alpha}\smalltext{,}\smalltext{t}}_\smalltext{\cdot}} \Bigg[ \sum_{k=0}^{n^\smalltext{\ell}-1} \sum_{i=1}^v \bigg( \big( \partial_{m^{\smalltext{i}}} G\big(\M^{\vartheta,\star}_{t^\smalltext{\ell}_{\smalltext{k}}}\big) - \partial_{m^{\smalltext{i}}} G\big(\M^{\star}_{t^\smalltext{\ell}_{\smalltext{k}}}\big) \Big) \Big( \E^{\Q^{\smalltext{\alpha} \smalltext{,}\smalltext{t}^\tinytext{\ell}_\tinytext{k}}_\smalltext{\cdot}} \Big[ M^{\star,i}_{t^\smalltext{\ell}_{\smalltext{k}\smalltext{+}\smalltext{1}}} \Big] - M^{\star,i}_{t^\smalltext{\ell}_{\smalltext{k}}} \Big)  \Bigg] \Bigg| \\
	&\leq \ell_{\partial G} \E^{\Q^{\smalltext{\alpha} \smalltext{,}\smalltext{t}}_\smalltext{\cdot}} \Bigg[ \sum_{k=0}^{n^\smalltext{\ell}-1} \sum_{i=1}^v \Big| \E^{\Q^{\smalltext{\alpha} \smalltext{,}\smalltext{t}^\tinytext{\ell}_\tinytext{k}}_\smalltext{\cdot}} \Big[ M^{\star,i}_{t^\smalltext{\ell}_{\smalltext{k}\smalltext{+}\smalltext{1}}} \Big] - M^{\star,i}_{t^\smalltext{\ell}_{\smalltext{k}}} \Big|^2 \Bigg] \xrightarrow{\ell \rightarrow 0} 0, \; \P\text{\rm--a.s.}
\end{align*}
The inequality follows from the assumption that the functions $\partial_{m^{\smalltext{i}}}G$ are Lipschitz-continuous, $i \in \{1,\ldots,v\}$, while the convergence directly follows from \Cref{assumpDPP}.\ref{modulusCondExpMart}.

\medskip
We turn to the remaining term, namely $J^3$. Arguing as in the case of $J^1$ following the approach of \cite[Lemma 1.5.9]{karatzas1991brownian}, we use the continuity of the functions $\partial^2_{m^{\smalltext{i}},m^{\smalltext{j}}} G$, for all $(i,j) \in \{1,\ldots,v\}^2$, together with the boundedness of the $(\widetilde{\F},\Q^{\alpha^{\smalltext{\star}}})$-martingale $\M^{\star}$, to conclude that there exist two constants $c_{\partial^\smalltext{2} G} >0$ and $c_\star>0$ such that
\begin{align*}
\Big| \E^{\Q^{\smalltext{\alpha} \smalltext{,} \smalltext{t}}_{\smalltext{\cdot}}} \big[ J^3 \big] \Big| &\leq c_{\partial^\smalltext{2} G} \E^{\Q^{\smalltext{\alpha} \smalltext{,}\smalltext{t}}_\smalltext{\cdot}} \Bigg[ \sum_{k=0}^{n^\smalltext{\ell}-1} \sum_{i=1}^v \big|\Delta M^{\star,i}_{t^\smalltext{\ell}_{\smalltext{k}\smalltext{+}\smalltext{1}}}\big|^2 \Bigg] \leq c_{\partial^\smalltext{2} G} c_\star v, \; \P\text{\rm--a.s.}
\end{align*}
Applying the dominated convergence theorem together with \cite[Lemma 1.5.8]{karatzas1991brownian}, we obtain that
\begin{align*}
\E^{\Q^{\smalltext{\alpha} \smalltext{,} \smalltext{t}}_{\smalltext{\omega}}} \big[ J^3 \big] \xrightarrow{\ell \rightarrow 0} \E^{\Q^{\smalltext{\alpha} \smalltext{,}\smalltext{t}}_\smalltext{\omega}} \Bigg[ - \frac{1}{2} \int_t^{\tilde{t}} \sum_{(i,j) \in \{1,\ldots,v\}^2} \partial^2_{m^{\smalltext{i}},m^{\smalltext{j}}} G\big(\M^{\star}_s\big) \d \big[M^{\star,i},M^{\star,j}\big]_s \Bigg], \; \P\text{\rm--a.s.}
\end{align*}

\medskip
To conclude the proof, it remains to verify that the right-hand side is greater than or equal to the left-hand side. To this end, we observe that for any $(t,\tilde{t}) \in [0,T] \times [t,T]$, It\^o's formula and the $(\widetilde{\F},\Q^{\alpha^{\smalltext{\star}}})$-martingale property of $\M^{\star}$ imply
\begin{align*}
\widetilde{V}_t \coloneqq \widetilde{J}(t,\cdot,\alpha^\star) 
& \coloneqq \E^{\Q^{\smalltext{\alpha}^{\tinytext{\star}}\smalltext{,}\smalltext{t}}_\smalltext{\cdot}} \bigg[ \int_t^T \widetilde{f}_s\big(X^1_{\cdot \land s}, \alpha^\star_s \big) \d s + \widetilde{g}\big(X^1_{\cdot \land T}\big) \bigg] + G\Big( \E^{\Q^{\smalltext{\alpha}^{\tinytext{\star}}\smalltext{,}\smalltext{t}}_\smalltext{\cdot}} \big[ \widetilde{\varphi}\big(X^1_{\cdot \land T}\big) \big] \Big) \\
& \; = \E^{\Q^{\smalltext{\alpha}^{\tinytext{\star}}\smalltext{,}\smalltext{t}}_\smalltext{\cdot}} \bigg[ \widetilde{V}_{\tilde{t}} + \int_t^{\tilde{t}} \widetilde{f}_s\big(X^1_{\cdot \land s}, \alpha^\star_s \big) \d s - G(\M^\star_{\tilde{t}}) + G(\M^\star_t)\bigg] \\
& \; = \E^{\Q^{\smalltext{\alpha}^{\tinytext{\star}}\smalltext{,}\smalltext{t}}_\smalltext{\cdot}} \Bigg[ \widetilde{V}_{\tilde{t}} + \int_t^{\tilde{t}} \widetilde{f}_s\big(X^1_{\cdot \land s}, \alpha^\star_s \big) \d s - \frac{1}{2} \int_t^{\tilde{t}} \sum_{(i,j) \in \{1,\ldots,v\}^2} \partial^2_{m^{\smalltext{i}},m^{\smalltext{j}}} G\big(\M^{\star}_s\big) \d \big[M^{\star,i},M^{\star,j}\big]_s \Bigg] \\
& \; \leq \esssup_{\alpha \in \widetilde{\cA}} \E^{\Q^{\smalltext{\alpha}\smalltext{,}\smalltext{t}}_{\smalltext{\cdot}}} \Bigg[ \widetilde{V}_{\tilde{t}} + \int_t^{\tilde{t}} \widetilde{f}_s\big(X^1_{\cdot \land s}, \alpha_s \big) \d s - \frac{1}{2} \int_t^{\tilde{t}} \sum_{(i,j) \in \{1,\ldots,v\}^2} \partial^2_{m^{\smalltext{i}},m^{\smalltext{j}}} G\big(\M^{\star}_s\big) \d \big[M^{\star,i},M^{\star,j}\big]_s \Bigg], \; \P\text{\rm--a.s.}
\end{align*}
Since $\M^{\star}$ is bounded by \Cref{assumpDPP}.\ref{boundedPhisMart}, and $\partial_{m^{\smalltext{i}}}G$ is continuous by \Cref{assumpDPP}.\ref{GLipcont}, $\int_0^\cdot \partial_{m^{\smalltext{i}}} G(\M^{\star}_t) \d M^{\star,i}_t$ is a $(\widetilde{\F},\Q^{\alpha^{\smalltext{\star}}})$-martingale for all $i \in \{1,\ldots,v\}$. This justifies the last equality and completes the proof.
\end{proof}

\section{Characterisation through BSDEs}\label{section:characterisation}

The extended dynamic programming principle proved in \Cref{section:DPP} naturally leads to a system of BSDEs. This system plays a central role: as we will show in this section, its well-posedness is both necessary and sufficient to characterise each sub-game--perfect Nash equilibrium and the associated value process for each player. 

\begin{proof}[Proof of \Cref{necessity_toBSDE}]
We begin by fixing a sub-game--perfect Nash equilibrium $\hat{\balpha}^N \in \mathcal{N}\!\mathcal{A}_{s,N}$, and consider a player index $i \in \{1,\ldots,N\}$. Under \Cref{assumpDPP_NplayerGame}.\ref{boundedPhisMart_NplayerGame}, the processes $M^{i,\star,N}$ and $N^{i,\star,N}$ are $(\F_N,\P^{\hat\balpha^{\smalltext{N}}})$--martingales and, consequently, admit $\P$-modifications that are right-continuous and $\P\text{--a.s.}$ continuous. As a result, the martingale representation property stated in \Cref{lemma:martingaleRepresentation} guarantees the existence of $Z^{i,m,\ell,\star,N}$ and $Z^{i,n,\ell,\star,N}$ in $\L^2_{\mathrm{loc}}(\F_N,\P^{\hat\balpha^{\smalltext{N}}})$, for $\ell \in \{1,\ldots,N\}$, such that
\begin{gather*}
M^{i,\star,N}_t = M^{i,\star,N}_0 + \int_0^t \sum_{\ell =1}^N Z^{i,m,\ell,\star,N}_s \cdot \d \big(W_s^{\hat\balpha^{\smalltext{N}},N}\big)^{\ell}, \; t \in [0,T], \; \P\text{\rm--a.s.}, \\
N^{i,\star,N}_t = N^{i,\star,N}_0 + \int_0^t \sum_{\ell =1}^N Z^{i,n,\ell,\star,N}_s \cdot \d \big(W_s^{\hat\balpha^{\smalltext{N}},N}\big)^{\ell}, \; t \in [0,T], \; \P\text{\rm--a.s.}
\end{gather*}
Equivalently, for any admissible strategy $\alpha \in \cA_N$, we can write
\begin{align*}
M^{i,\star,N}_t &= \varphi^i_1(X^i_{\cdot \land T}) + \int_t^T Z^{i,m,i,\star,N}_s \cdot \Big(b^i_s\big(X^i_{\cdot \land s},L^N\big(\X^N_{\cdot \land s},\hat\balpha^N_s\big),\hat\alpha^{i,N}_s\big) - b^i_s\big(X^i_{\cdot \land s},L^N\big(\X^N_{\cdot \land s},(\alpha \otimes_{i} \hat\balpha^{N,\smallertext{-}i})_s\big),\alpha_s\big)\Big) \d s \\
&\quad - \int_t^T \sum_{\ell =1}^N Z^{i,m,\ell,\star,N}_s \cdot \d \big(W^{{\alpha \otimes_\smalltext{i} \hat\balpha^{\smalltext{N}\smalltext{,}\smalltext{-}\smalltext{i}}},N}_s\big)^\ell, \; t \in [0,T], \; \P\text{\rm--a.s.}, \\
N^{i,\star,N}_t &= \varphi^i_2\big(L^N\big(\X^N_{\cdot \land T}\big)\big) + \int_t^T Z^{i,n,i,\star,N}_s \cdot \Big(b^i_s\big(X^i_{\cdot \land s},L^N\big(\X^N_{\cdot \land s},\hat\balpha^N_s\big),\hat\alpha^{i,N}_s\big) - b^i_s\big(X^i_{\cdot \land s},L^N\big(\X^N_{\cdot \land s},(\alpha \otimes_{i} \hat\balpha^{N,\smallertext{-}i})_s\big),\alpha_s\big)\Big) \d s \\
&\quad - \int_t^T \sum_{\ell =1}^N Z^{i,n,\ell,\star,N}_s \cdot \d \big(W^{{\alpha \otimes_\smalltext{i} \hat\balpha^{\smalltext{N}\smalltext{,}\smalltext{-}\smalltext{i}}},N}_s\big)^\ell, \; t \in [0,T], \; \P\text{\rm--a.s.}
\end{align*}
Hence, it follows directly that the random variables
\begin{align*}
\sup_{t \in [0,T]} \big| M^{i,\star,N}_t\big| \; \text{and} \; \sup_{t \in [0,T]} \big| N^{i,\star,N}_t \big|
\end{align*}
are bounded, due to \Cref{assumpDPP_NplayerGame}.\ref{boundedPhisMart_NplayerGame}. Furthermore, since each drift function $b^i$ is assumed to be bounded, we may apply, for instance, \citeauthor*{zhang2017backward} \cite[Theorem 7.2.1]{zhang2017backward} to conclude that 
\begin{align*}
\esssup_{\balpha \in \cA^\smalltext{N}_\smalltext{N}} \sup_{\tau \in \cT_{0,T}} \E^{\P^{\smalltext{\balpha}\smalltext{,}\smalltext{N}\smalltext{,}\smalltext{\tau}}_{\smalltext{\cdot}}} \Bigg[ \int_\tau^T \sum_{\ell =1}^N \Big( \big\|Z^{i,m,\ell,\star,N}_t\big\|^2 + \big\|Z^{i,n,\ell,\star,N}_t\big\|^2 \Big) \d t \Bigg] <+\infty.
\end{align*}

\medskip
Following an argument analogous to \cite[Proposition 2.6]{possamai2025non}, and thus making use of the extended dynamic programming principle in \Cref{thm:DPP_NplayerGame}, it can be shown that
\[
\hat\balpha^N_t \in \cO_{N}\big(t, \X^N_{\cdot \land t}, \Z^N_t, \M^{\star,N}_t, \N^{\star,N}_t, \Z^{m,\star,N}_t, \Z^{n,\star,N}_t\big),\; \d t \otimes \d \P\text{\rm--a.e.} \; (t,\omega) \in [0,T] \times \Omega,
\]
where, for each $i \in \{1,\ldots,N\}$, the pair $(Y^{i,N},\Z^{i,N})$ satisfies the BSDE
\begin{align*}
Y^{i,N}_t &= g^i\big(X^i_{\cdot \land T}, L^N\big(\X^N_{\cdot \land T}\big)\big) + G^i\big(\varphi^i_1(X^i_{\cdot \land T}),\varphi^i_2\big(L^N\big(\X^N_{\cdot \land T}\big)\big)\big) \\
&\quad+ \int_t^T f^i_s\big(X^i_{\cdot \land s},L^N(\X^N_{\cdot \land s},\hat\balpha^N_s),\hat\alpha^{i,N}_s\big) \d s - \int_t^T \partial^2_{m,n}G^i\big(M^{i,\star,N}_s,N^{i,\star,N}_s\big) \sum_{\ell =1}^N Z^{i,m,\ell,\star,N}_s \cdot Z^{i,n,\ell,\star,N}_s \d s \\
&\quad- \frac{1}{2} \int_t^T \partial^2_{m,m}G^i\big(M^{i,\star,N}_s,N^{i,\star,N}_s\big) \sum_{\ell =1}^N \big\|Z^{i,m,\ell,\star,N}_s\big\|^2 \d s - \frac{1}{2} \int_t^T \partial^2_{n,n}G^i\big(M^{i,\star,N}_s,N^{i,\star,N}_s\big) \sum_{\ell =1}^N \big\|Z^{i,n,\ell,\star,N}_s\big\|^2 \d s \\
\notag &\quad- \int_t^T \sum_{\ell =1}^N Z^{i,\ell,N}_s \cdot \d \big(W_s^{\hat\balpha^\smalltext{N},N}\big)^\ell, \; t \in [0,T], \; \P\text{\rm--a.s.}
\end{align*}
By \cite[Theorem 4.2]{briand2003p}, recalling that \Cref{assumpDPP_NplayerGame}.\ref{Glip_NplayerGame} and \Cref{assumption:regularityForNecessity} hold, together with the estimates for $(\M^{\star,N}, \N^{\star,N}, \Z^{m,\star,N}, \Z^{n,\star,N})$ derived above, there exists some $p \geq 1$ such that
\begin{align*}
\sup_{\balpha \in \cA^\smalltext{N}_\smalltext{N}} \E^{\P^{\smalltext{\balpha}}} \Bigg[ \sup_{t \in [0,T]} \big|Y^{i,N}_t\big|^p + \bigg(\int_0^T \sum_{\ell =1}^N \big\|{Z}_t^{i,\ell,N}\big\|^2 \d t\bigg)^\frac{p}{2} \Bigg] <+\infty.
\end{align*}
Moreover, using the definition of the probability measure $\P^{\hat\balpha^{\smalltext{N}},N}$, it is straightforward to verify that $V^{i,N}_t = Y^{i,N}_t$, $\P\text{\rm--a.s.}$, for any $t \in [0,T]$. 
\end{proof}

We now proceed to prove the sufficiency of the BSDE system.
\begin{proof}[Proof of \Cref{sufficiency_toBSDE}]
For any fixed index $i \in \{1,\ldots,N\}$, one immediately obtains
\begin{gather*}
M_t^{i,\star,N} = \E^{\P^{\smalltext{\hat\balpha}^\tinytext{N}}} \big[ \varphi^i_1(X^i_{\cdot \land T}) \big|\cF_{N,t} \big] = \E^{\P^{\smalltext{\hat\balpha}^\tinytext{N}\smalltext{,}\smalltext{N}\smalltext{,}\smalltext{t}}_\smalltext{\cdot}} \big[ \varphi^i_1(X^i_{\cdot \land T}) \big], \; \P\text{\rm--a.s.}, \; t \in [0,T], \\
N_t^{i,\star,N} = \E^{\P^{\smalltext{\hat\balpha}^\tinytext{N}}} \big[ \varphi^i_2\big(L^N\big(\X^N_{\cdot \land T}\big)\big) \big|\cF_{N,t} \big] = \E^{\P^{\smalltext{\hat\balpha}^\tinytext{N}\smalltext{,}\smalltext{N}\smalltext{,}\smalltext{t}}_\smalltext{\cdot}} \big[ \varphi^i_2\big(L^N\big(\X^N_{\cdot \land T}\big)\big) \big], \; \P\text{\rm--a.s.}, \; t \in [0,T].
\end{gather*}
On the other hand, It\^o's formula implies that
\begin{align*}
G^i\big(\varphi^i_1(X^i_{\cdot \land T}),\varphi^i_2\big(L^N\big(\X^N_{\cdot \land T}\big)\big)\big) &= G^i\big(M^{i,\star,N}_t,N^{i,\star,N}_t\big) + \int_t^T \partial_m G^i\big(M^{i,\star,N}_s,N^{i,\star,N}_s\big) \d M^{i,\star,N}_s \\
&\quad+ \int_t^T \partial_n G^i\big(M^{i,\star,N}_s,N^{i,\star,N}_s\big) \d N^{i,\star,N}_s \\
&\quad + \frac{1}{2} \int_t^T \partial^2_{m,m}G^i\big(M^{i,\star,N}_s,N^{i,\star,N}_s\big) \sum_{\ell =1}^N \big\|Z^{i,m,\ell,\star,N}_s\big\|^2 \d s \\
&\quad + \int_t^T \partial^2_{m,n}G^i\big(M^{i,\star,N}_s,N^{i,\star,N}_s\big) \sum_{\ell =1}^N Z^{i,m,\ell,\star,N}_s \cdot Z^{i,n,\ell,\star,N}_s \d s \\
&\quad + \frac{1}{2} \int_t^T \partial^2_{n,n}G^i\big(M^{i,\star,N}_s,N^{i,\star,N}_s\big) \sum_{\ell=1}^N \big\|Z^{i,n,\ell,\star,N}_s\big\|^ 2 \d s, \; t \in [0,T], \; \P\text{\rm--a.s.}
\end{align*}
Consequently, we can deduce that $J^i(t,\cdot,\hat\alpha^{i,N};\hat\balpha^{N,\smallertext{-}i}) = Y^{i,N}_t$, $\P\text{\rm--a.s.}$, for any $t \in [0,T]$, since 
\begin{align*}
Y^{i,N}_t &= \E^{\P^{\smalltext{\hat\balpha}^\tinytext{N}\smalltext{,}\smalltext{N}\smalltext{,}\smalltext{t}}_\smalltext{\cdot}} \Bigg[ g^i\big(X^i_{\cdot \land T}, L^N\big(\X^N_{\cdot \land T}\big)\big) + G^i\big(\varphi^i_1(X^i_{\cdot \land T}),\varphi^i_2\big(L^N\big(\X^N_{\cdot \land T}\big)\big)\big) + \int_t^T f^i_s\big( X^i_{\cdot \land s}, L^N\big(\X^N_{\cdot \land s},\hat\balpha^N_s\big), \hat\alpha^{i,N}_s\big) \d s \\
&\quad - \int_t^T  \partial^2_{m,n}G^i\big(M^{i,\star,N}_s,N^{i,\star,N}_s\big) \sum_{\ell =1}^N Z^{i,m,\ell,\star,N}_s \cdot Z^{i,n,\ell,\star,N}_s \d s  - \frac{1}{2} \int_t^T \bigg( \partial^2_{m,m}G^i\big(M^{i,\star,N}_s,N^{i,\star,N}_s\big) \sum_{\ell =1}^N \big\|Z^{i,m,\ell,\star,N}_s\big\|^2 \\
&\quad + \partial^2_{n,n}G^i\big(M^{i,\star,N}_s,N^{i,\star,N}_s\big) \sum_{\ell =1}^N  \big\|Z^{i,n,\ell,\star,N}_s\big\|^ 2 \bigg) \d s \Bigg] \\
&= \E^{\P^{\smalltext{\hat\balpha}^\tinytext{N}\smalltext{,}\smalltext{N}\smalltext{,}\smalltext{t}}_\smalltext{\cdot}} \bigg[ g^i\big(X^i_{\cdot \land T}, L^N\big(\X^N_{\cdot \land T}\big)\big) + \int_t^T f^i_s\big( X^i_{\cdot \land s}, L^N\big(\X^N_{\cdot \land s},\hat\balpha^N_s\big), \hat\alpha^{i,N}_s\big) \d s \bigg] + G^i\big(M^{i,\star,N}_t,N^{i,\star,N}_t\big) \\
&= \E^{\P^{\smalltext{\hat\balpha}^\tinytext{N}\smalltext{,}\smalltext{N}\smalltext{,}\smalltext{t}}_\smalltext{\cdot}} \bigg[ g^i\big(X^i_{\cdot \land T}, L^N\big(\X^N_{\cdot \land T}\big)\big) + \int_t^T f^i_s\big( X^i_{\cdot \land s}, L^N\big(\X^N_{\cdot \land s},\hat\balpha^N_s\big), \hat\alpha^{i,N}_s\big) \d s \bigg] \\
&\quad+ G^i\Big(\E^{\P^{\smalltext{\hat\balpha}^\tinytext{N}\smalltext{,}\smalltext{N}\smalltext{,}\smalltext{t}}_\smalltext{\cdot}} \big[ \varphi^i_1(X^i_{\cdot \land T}) \big],\E^{\P^{\smalltext{\hat\balpha}^\tinytext{N}\smalltext{,}\smalltext{N}\smalltext{,}\smalltext{t}}_\smalltext{\cdot}} \big[ \varphi^i_2\big(L^N\big(\X^N_{\cdot \land T}\big)\big) \big]\Big), \; \P\text{\rm--a.s.}, \; t \in [0,T],
\end{align*}
where the first equality holds because the stochastic integrals 
\[
\int_0^\cdot \partial_{m} G\big(M^{i,\star,N}_t,N^{i,\star,N}_t\big) \d M^{i,\star,N}_t\; \text{and}\; \int_0^\cdot \partial_{n} G\big(M^{i,\star,N}_t,N^{i,\star,N}_t\big) \d N^{i,\star,N}_t
\]
are $(\F_N,\P^{\hat\balpha^{\smalltext{N}}})$-martingales, as ensured by \Cref{assumpDPP_NplayerGame}.\ref{Glip_NplayerGame} together with the estimates for $(M^{i,\star,N},N^{i,\star,N})$ stated in \Cref{align:estimatesBSDEsupSuff}.

\medskip
To complete the proof, it remains to verify that the constructed strategy $\hat\balpha^N_t := \mathfrak{a}^N(t,\Z^N_t,\M^{\star,N}_t,\N^{\star,N}_t,\Z^{m,\star,N}_t,\Z^{n,\star,N}_t)$, for $t \in [0,T]$, is indeed a sub-game--perfect Nash equilibrium, i.e., $\hat\balpha^N \in \mathcal{NA}_{s,N}$. To this end, let $\varepsilon >0$ be fixed. We then select a player index $i \in \{1,\ldots, N\}$, along with an admissible strategy $\alpha \in \cA_N$, and consider some $\ell \in (0, \ell_\varepsilon)$, where $\ell_\varepsilon >0$ will be specified later. To avoid further complicating the notation, we define $\alpha^{i,t,\ell} \coloneqq \alpha \otimes_{t+\ell} \hat\alpha^{i,N}$, $t \in [0,T]$. We have
\begin{align*}
& J^i(t,\cdot,\hat\alpha^{i,N};\hat\balpha^{N,\smallertext{-}i}) - J^i(t,\cdot, \alpha^{i,t,\ell};\hat\balpha^{N,\smallertext{-}i}) \\
&= \E^{\P^{\smalltext{\hat\balpha}^{\tinytext{N}}\smalltext{,}\smalltext{N}\smalltext{,}\smalltext{t}}_{\smalltext{\cdot}}} \bigg[ g^i\big(X^i_{\cdot \land T}, L^N\big(\X^N_{\cdot \land T}\big)\big) + \int_t^T f_s\big( X^i_{\cdot \land s}, L^N\big(\X^N_{\cdot \land s},\hat\balpha^{N}_s\big), \hat\alpha^{i,N}_s\big) \d s \bigg] \\
\\&\quad+ G^i\Big(\E^{\P^{\smalltext{\hat\balpha}^\tinytext{N}\smalltext{,}\smalltext{N}\smalltext{,}\smalltext{t}}_\smalltext{\cdot}} \big[ \varphi^i_1(X^i_{\cdot \land T}) \big],\E^{\P^{\smalltext{\hat\balpha}^\tinytext{N}\smalltext{,}\smalltext{N}\smalltext{,}\smalltext{t}}_\smalltext{\cdot}} \big[ \varphi^i_2\big(L^N\big(\X^N_{\cdot \land T}\big)\big) \big]\Big) \\
	&\quad- \E^{\P^{\smalltext{\alpha}^{\tinytext{i}\tinytext{,}\tinytext{t}\tinytext{,}\tinytext{\ell}} \smalltext{\otimes}_{\tinytext{i}} \smalltext{\hat\balpha}^{\tinytext{N}\tinytext{,}\tinytext{-i}}\smalltext{,}\smalltext{N}\smalltext{,}\smalltext{t}}_{\smalltext{\cdot}}} \bigg[ g^i\big(X^i_{\cdot \land T}, L^N\big(\X^N_{\cdot \land T}\big)\big) + \int_t^T f_s\big( X^i_{\cdot \land s}, L^N\big(\X^N_{\cdot \land s},(\alpha^{i,t,\ell} \otimes_i \balpha^{N,\smallertext{-}i})_s\big), \alpha^{i,t,\ell}_s\big) \Big) \d s \bigg] \\
	&\quad- G^i\bigg( \E^{\P^{\smalltext{\alpha}^{\tinytext{i}\tinytext{,}\tinytext{t}\tinytext{,}\tinytext{\ell}} \smalltext{\otimes}_{\tinytext{i}} \smalltext{\hat\balpha}^{\tinytext{N}\tinytext{,}\tinytext{-i}}\smalltext{,}\smalltext{N}\smalltext{,}\smalltext{t}}_{\smalltext{\cdot}}} \Big[ \varphi_1^i(X^i_{\cdot \land T}) \Big], \E^{\P^{\smalltext{\alpha}^{\tinytext{i}\tinytext{,}\tinytext{t}\tinytext{,}\tinytext{\ell}} \smalltext{\otimes}_{\tinytext{i}} \smalltext{\hat\balpha}^{\tinytext{N}\tinytext{,}\tinytext{-i}}\smalltext{,}\smalltext{N}\smalltext{,}\smalltext{t}}_{\smalltext{\cdot}}} \Big[ \varphi^i_2\big(L^N(\X^N_{\cdot \land T}\big) \Big] \bigg) \\
&= \E^{\P^{\smalltext{\alpha}^{\tinytext{i}\tinytext{,}\tinytext{t}\tinytext{,}\tinytext{\ell}} \smalltext{\otimes}_{\tinytext{i}} \smalltext{\hat\balpha}^{\tinytext{N}\tinytext{,}\tinytext{-i}}\smalltext{,}\smalltext{N}\smalltext{,}\smalltext{t}}_{\smalltext{\cdot}}} \Bigg[ \int_t^T \Big( f_s\big( X^i_{\cdot \land s}, L^N\big(\X^N_{\cdot \land s},\hat\balpha^{N}_s\big), \hat\alpha^{i,N}_s\big) - f_s\big( X^i_{\cdot \land s}, L^N\big(\X^N_{\cdot \land s},(\alpha^{i,t,\ell} \otimes_i \balpha^{N,\smallertext{-}i})_s\big), \alpha^{i,t,\ell}_s\big) \Big) \d s \\
	&\quad+ \int_t^T Z^{i,i,N}_s \cdot \Big( b^i_s\big(X^i_{\cdot \land s},L^N\big(\X^N_{\cdot \land s},\hat\balpha^N_s\big),\hat\alpha^{i,N}_s\big) - b^i_s\big(X^i_{\cdot \land s}, L^N\big(\X^N_{\cdot \land s},(\alpha^{i,t,\ell} \otimes_i \hat\balpha^{N,\smallertext{-}i})_s\big), \alpha^{i,t,\ell}_s \big) \Big) \d s \\
	&\quad+ \int_t^T \sum_{j \in \{1,\ldots,N\}\setminus\{i\}} Z^{i,j,N}_s \cdot \Big( b^j_s\big(X^j_{\cdot \land s},L^N\big(\X^N_{\cdot \land s},\hat\balpha^N_s\big),\hat\alpha^{\ell,N}_s\big) - b^j_s\big(X^j_{\cdot \land s}, L^N\big(\X^N_{\cdot \land s},(\alpha^{i,t,\ell} \otimes_i \hat\balpha^{N,\smallertext{-}i})_s\big), \hat\alpha^{j,N}_s\big) \Big) \d s \bigg] \\
	&\quad+ G^i\big(\varphi^i_1(X^i_{\cdot \land T}),\varphi^i_2\big(L^N\big(\X^N_{\cdot \land T}\big)\big)\big) - \int_t^T \partial^2_{m,n} G^i\big(M^{i,\star,N}_s,N^{i,\star,N}_s\big) \sum_{\ell =1}^N Z^{i,m,\ell,\star,N}_s \cdot Z^{i,n,\ell,\star,N}_s \d s \\
	&\quad- \frac{1}{2} \int_t^T \bigg(\partial^2_{m,m} G^i\big(M^{i,\star,N}_s,N_s^{i,\star,N}\big) \sum_{\ell =1}^N \big\|Z^{i,m,\ell,\star,N}_s\big\|^2 + \partial^2_{n,n} G^i\big(M^{i,\star,N}_s,N^{i,\star,N}_s\big) \sum_{\ell =1}^N \big\|Z^{i,n,\ell,\star,N}_s\big\|^2 \bigg) \d s \Bigg] \\
	&\quad- G^i\bigg( \E^{\P^{\smalltext{\alpha}^{\tinytext{i}\tinytext{,}\tinytext{t}\tinytext{,}\tinytext{\ell}} \smalltext{\otimes}_{\tinytext{i}} \smalltext{\hat\balpha}^{\tinytext{N}\tinytext{,}\tinytext{-i}}\smalltext{,}\smalltext{N}\smalltext{,}\smalltext{t}}_{\smalltext{\cdot}}} \Big[ \varphi_1^i(X^i_{\cdot \land T}) \Big], \E^{\P^{\smalltext{\alpha}^{\tinytext{i}\tinytext{,}\tinytext{t}\tinytext{,}\tinytext{\ell}} \smalltext{\otimes}_{\tinytext{i}} \smalltext{\hat\balpha}^{\tinytext{N}\tinytext{,}\tinytext{-i}}\smalltext{,}\smalltext{N}\smalltext{,}\smalltext{t}}_{\smalltext{\cdot}}} \Big[ \varphi^i_2\big(L^N(\X^N_{\cdot \land T}\big) \Big] \bigg) \\
&\geq \E^{\P^{\smalltext{\alpha}^{\tinytext{i}\tinytext{,}\tinytext{t}\tinytext{,}\tinytext{\ell}} \smalltext{\otimes}_{\tinytext{i}} \smalltext{\hat\balpha}^{\tinytext{N}\tinytext{,}\tinytext{-i}}\smalltext{,}\smalltext{N}\smalltext{,}\smalltext{t}}_{\smalltext{\cdot}}} \Bigg[ G^i\big(\varphi^i_1(X^i_{\cdot \land T}),\varphi^i_2\big(L^N\big(\X^N_{\cdot \land T}\big)\big)\big) - \int_t^T \partial^2_{m,n} G^i\big(M^{i,\star,N}_s,N^{i,\star,N}_s\big) \sum_{\ell =1}^N Z^{i,m,\ell,\star,N}_s \cdot Z^{i,n,\ell,\star,N}_s \d s \\
	&\quad- \frac{1}{2} \int_t^T \bigg( \partial^2_{m,m} G^i\big(M^{i,\star,N}_s,N_s^{i,\star,N}\big) \sum_{\ell =1}^N \big\|Z^{i,m,\ell,\star,N}_s\big\|^2 + \partial^2_{n,n} G^i\big(M^{i,\star,N}_s,N^{i,\star,N}_s\big) \sum_{\ell =1}^N \big\|Z^{i,n,\ell,\star,N}_s\big\|^2 \bigg) \d s \Bigg] \\
	&\quad- G^i\bigg( \E^{\P^{\smalltext{\alpha}^{\tinytext{i}\tinytext{,}\tinytext{t}\tinytext{,}\tinytext{\ell}} \smalltext{\otimes}_{\tinytext{i}} \smalltext{\hat\balpha}^{\tinytext{N}\tinytext{,}\tinytext{-i}}\smalltext{,}\smalltext{N}\smalltext{,}\smalltext{t}}_{\smalltext{\cdot}}} \Big[ \varphi_1^i(X^i_{\cdot \land T}) \Big], \E^{\P^{\smalltext{\alpha}^{\tinytext{i}\tinytext{,}\tinytext{t}\tinytext{,}\tinytext{\ell}} \smalltext{\otimes}_{\tinytext{i}} \smalltext{\hat\balpha}^{\tinytext{N}\tinytext{,}\tinytext{-i}}\smalltext{,}\smalltext{N}\smalltext{,}\smalltext{t}}_{\smalltext{\cdot}}} \Big[ \varphi^i_2\big(L^N(\X^N_{\cdot \land T}\big) \Big] \bigg), \; \P\text{\rm--a.s.}, \; t \in [0,T],
\end{align*}
where the inequality follows from the property $\hat\balpha^N_t(\omega) \in \cO_{N}\big(t, \X^N_{\cdot \land t}(\omega), \Z^N_t(\omega), \M^{\star,N}_t(\omega), \N^{\star,N}_t(\omega), \Z^{m,\star,N}_t(\omega), \Z^{n,\star,N}_t(\omega)\big)$, for any $(t,\omega) \in [0,T] \times \Omega$. Since the estimates in \Cref{align:estimatesBSDEsupSuff} are satisfied, the processes
\begin{align*}
M_t^{i,\ell,N} \coloneqq \E^{\P^{\smalltext{\alpha}^{\tinytext{i}\tinytext{,}\tinytext{t}\tinytext{,}\tinytext{\ell}} \smalltext{\otimes}_{\tinytext{i}} \smalltext{\hat\balpha}^{\tinytext{N}\tinytext{,}\tinytext{-i}}\smalltext{,}\smalltext{t}}_{\smalltext{\cdot}}} \big[ \varphi^i_1(X^i_{\cdot \land T}) \big] \; \text{and} \; N_t^{i,\ell,N} \coloneqq \E^{\P^{\smalltext{\alpha}^{\tinytext{i}\tinytext{,}\tinytext{t}\tinytext{,}\tinytext{\ell}} \smalltext{\otimes}_{\tinytext{i}} \smalltext{\hat\balpha}^{\tinytext{N}\tinytext{,}\tinytext{-i}}\smalltext{,}\smalltext{t}}_{\smalltext{\cdot}}} \big[ \varphi^i_2\big(L^N\big(\X^N_{\cdot \land T}\big)\big) \big], \; t \in [0,T],
\end{align*}
admit $\P$-modifications that are right-continuous and $\P\text{--a.s.}$ continuous $(\F,\P^{\alpha^{i,\ell} \otimes_i \hat\balpha^{N,\smallertext{-}i}})$--martingales, which we continue to denote using the same notation. They admit the representations
\begin{gather*}
M^{i,\ell,N}_t = M^{i,\ell,N}_0 + \int_0^t \sum_{j = 1}^N  Z^{i,m,j,\ell,N}_s \cdot \d \big(W_s^{\alpha^{\smalltext{i}\smalltext{,}\smalltext{t}\smalltext{,}\smalltext{\ell}},N}\big)^j, \; t \in [0,T], \; \P\text{\rm--a.s.}, \\
N^{i,\ell,N}_t = N^{i,\ell,N}_0 + \int_0^t \sum_{j = 1}^N  Z^{i,n,j,\ell,N}_s \cdot \d \big(W_s^{\alpha^{\smalltext{i}\smalltext{,}\smalltext{t}\smalltext{,}\smalltext{\ell}},N}\big)^j, \; t \in [0,T], \; \P\text{\rm--a.s.},
\end{gather*}
for $Z^{i,m,j,\ell,N}$ and $Z^{i,n,j,\ell,N}$ in $\L^2_{\mathrm{loc}}(\F_N,\P^{\alpha^{i,\ell} \otimes_i \hat\balpha^{N,\smallertext{-}i}})$, $j \in \{1,\ldots,N\}$. Consequently, taking into account that $\alpha^{i,t,\ell}_s = \hat\alpha^{i,N}_s$, for any $s \in [t+\ell,T]$, It\^o's formula implies 
\begin{align*}
& \E^{\P^{\smalltext{\alpha}^{\tinytext{i}\tinytext{,}\tinytext{t}\tinytext{,}\tinytext{\ell}} \smalltext{\otimes}_{\tinytext{i}} \smalltext{\hat\balpha}^{\tinytext{N}\tinytext{,}\tinytext{-i}}\smalltext{,}\smalltext{N}\smalltext{,}\smalltext{t}}_{\smalltext{\cdot}}} \Bigg[ G^i\big(\varphi^i_1(X^i_{\cdot \land T}),\varphi^i_2\big(L^N\big(\X^N_{\cdot \land T}\big)\big)\big) - \int_t^T \partial^2_{m,n} G^i\big(M^{i,\star,N}_s,N^{i,\star,N}_s\big) \sum_{\ell=1}^N Z^{i,m,\ell,\star,N}_s \cdot Z^{i,n,\ell,\star,N}_s \d s \\
	&\quad- \frac{1}{2} \int_t^T \bigg( \partial^2_{m,m} G^i\big(M^{i,\star,N}_s,N_s^{i,\star,N}\big) \sum_{\ell =1}^N\big\|Z^{i,m,\ell,\star,N}_s\big\|^2 + \partial^2_{n,n} G^i\big(M^{i,\star,N}_s,N^{i,\star,N}_s\big) \sum_{\ell =1}^N \big\|Z^{i,n,\ell,\star,N}_s\big\|^2 \bigg) \d s \Bigg] \\
	&\quad- G^i\bigg( \E^{\P^{\smalltext{\alpha}^{\tinytext{i}\tinytext{,}\tinytext{t}\tinytext{,}\tinytext{\ell}} \smalltext{\otimes}_{\tinytext{i}} \smalltext{\hat\balpha}^{\tinytext{N}\tinytext{,}\tinytext{-i}}\smalltext{,}\smalltext{N}\smalltext{,}\smalltext{t}}_{\smalltext{\cdot}}} \Big[ \varphi_1^i(X^i_{\cdot \land T}) \Big], \E^{\P^{\smalltext{\alpha}^{\tinytext{i}\tinytext{,}\tinytext{t}\tinytext{,}\tinytext{\ell}} \smalltext{\otimes}_{\tinytext{i}} \smalltext{\hat\balpha}^{\tinytext{N}\tinytext{,}\tinytext{-i}}\smalltext{,}\smalltext{N}\smalltext{,}\smalltext{t}}_{\smalltext{\cdot}}} \Big[ \varphi^i_2\big(L^N(\X^N_{\cdot \land T}\big) \Big] \bigg) \\
&= \E^{\P^{\smalltext{\alpha}^{\tinytext{i}\tinytext{,}\tinytext{t}\tinytext{,}\tinytext{\ell}} \smalltext{\otimes}_{\tinytext{i}} \smalltext{\hat\balpha}^{\tinytext{N}\tinytext{,}\tinytext{-i}}\smalltext{,}\smalltext{N}\smalltext{,}\smalltext{t}}_{\smalltext{\cdot}}} \Bigg[ \int_t^{t+\ell} \bigg( \partial^2_{m,n} G^{i,\ell,N}_s \sum_{j =1}^N Z^{i,m,j,\ell,N}_s \cdot Z^{i,n,j,\ell,N}_s - \partial^2_{m,n} G^{i,\star,N}_s \sum_{j =1}^N Z_s^{i,m,j,\star,N} \cdot Z_s^{i,n,j,\star,N} \bigg) \d s \Bigg] \\
&\quad + \frac{1}{2} \E^{\P^{\smalltext{\alpha}^{\tinytext{i}\tinytext{,}\tinytext{t}\tinytext{,}\tinytext{\ell}} \smalltext{\otimes}_{\tinytext{i}} \smalltext{\hat\balpha}^{\tinytext{N}\tinytext{,}\tinytext{-i}}\smalltext{,}\smalltext{N}\smalltext{,}\smalltext{t}}_{\smalltext{\cdot}}} \Bigg[ \int_t^{t+\ell} \bigg( \partial^2_{m,m} G^{i,\ell,N}_s \sum_{j =1}^N \big\|Z_s^{i,m,j,\ell,N}\big\|^2 - \partial^2_{m,m} G^{i,\star,N}_s \sum_{j =1}^N \big\|Z_s^{i,m,j,\star,N}\big\|^2 \bigg) \d s \Bigg] \\
&\quad +\frac{1}{2} \E^{\P^{\smalltext{\alpha}^{\tinytext{i}\tinytext{,}\tinytext{t}\tinytext{,}\tinytext{\ell}} \smalltext{\otimes}_{\tinytext{i}} \smalltext{\hat\balpha}^{\tinytext{N}\tinytext{,}\tinytext{-i}}\smalltext{,}\smalltext{N}\smalltext{,}\smalltext{t}}_{\smalltext{\cdot}}} \Bigg[ \int_t^{t+\ell} \bigg( \partial^2_{n,n} G^{i,\ell,N}_s \sum_{j =1}^N \big\|Z_s^{i,n,j,\ell,N}\big\|^2 - \partial^2_{n,n} G^{i,\star,N}_s \sum_{j =1}^N \big\|Z_s^{i,n,j,\star,N}\big\|^2 \bigg) \d s \Bigg], \; \P\text{--a.s.}, \; t \in [0,T],
\end{align*}
where we have used the notation $\partial^2_{m,n} G^{i,\ell,N}_t \coloneqq \partial^2_{m,n} G^i(M^{i,\ell,N}_t)$ and $\partial^2_{m,n} G^{i,\star,N}_t \coloneqq \partial^2_{m,n} G^i(M^{i,\star,N}_t)$, for $t \in [0,T]$, with analogous notation used for the other derivatives. Consequently, by \Cref{assumpDPP_NplayerGame}.\ref{Glip_NplayerGame} and the estimates in \Cref{align:estimatesBSDEsupSuff}, we deduce the existence of a constant $c_{\partial^\smalltext{2} G} >0$ such that
\begin{align*}
& \Bigg| \E^{\P^{\smalltext{\alpha}^{\tinytext{i}\tinytext{,}\tinytext{t}\tinytext{,}\tinytext{\ell}} \smalltext{\otimes}_{\tinytext{i}} \smalltext{\hat\balpha}^{\tinytext{N}\tinytext{,}\tinytext{-i}}\smalltext{,}\smalltext{N}\smalltext{,}\smalltext{t}}_{\smalltext{\cdot}}} \Bigg[ G^i\big(\varphi^i_1(X^i_{\cdot \land T}),\varphi^i_2\big(L^N\big(\X^N_{\cdot \land T}\big)\big)\big) - \int_t^T \partial^2_{m,n} G^i\big(M^{i,\star,N}_s,N^{i,\star,N}_s\big) \sum_{\ell =1}^N Z^{i,m,\ell,\star,N}_s \cdot Z^{i,n,\ell,\star,N}_s \d s \\
	&\quad- \frac{1}{2} \int_t^T \bigg( \partial^2_{m,m} G^i\big(M^{i,\star,N}_s,N_s^{i,\star,N}\big) \sum_{\ell =1}^N \big\|Z^{i,m,\ell,\star,N}_s\big\|^2 + \partial^2_{n,n} G^i\big(M^{i,\star,N}_s,N^{i,\star,N}_s\big) \sum_{\ell =1}^N \big\|Z^{i,n,\ell,\star,N}_s\big\|^2 \bigg) \d s \Bigg] \\
	&\quad- G^i\bigg( \E^{\P^{\smalltext{\alpha}^{\tinytext{i}\tinytext{,}\tinytext{t}\tinytext{,}\tinytext{\ell}} \smalltext{\otimes}_{\tinytext{i}} \smalltext{\hat\balpha}^{\tinytext{N}\tinytext{,}\tinytext{-i}}\smalltext{,}\smalltext{N}\smalltext{,}\smalltext{t}}_{\smalltext{\cdot}}} \Big[ \varphi_1^i(X^i_{\cdot \land T}) \Big], \E^{\P^{\smalltext{\alpha}^{\tinytext{i}\tinytext{,}\tinytext{t}\tinytext{,}\tinytext{\ell}} \smalltext{\otimes}_{\tinytext{i}} \smalltext{\hat\balpha}^{\tinytext{N}\tinytext{,}\tinytext{-i}}\smalltext{,}\smalltext{N}\smalltext{,}\smalltext{t}}_{\smalltext{\cdot}}} \Big[ \varphi^i_2\big(L^N(\X^N_{\cdot \land T}\big) \Big] \bigg) \Bigg| \\
&\leq c_{\partial^\smalltext{2} G} \E^{\P^{\smalltext{\alpha}^{\tinytext{i}\tinytext{,}\tinytext{t}\tinytext{,}\tinytext{\ell}} \smalltext{\otimes}_{\tinytext{i}} \smalltext{\hat\balpha}^{\tinytext{N}\tinytext{,}\tinytext{-i}}\smalltext{,}\smalltext{N}\smalltext{,}\smalltext{t}}_{\smalltext{\cdot}}} \Bigg[ \int_t^{t+\ell} \sum_{j =1}^N \Big( \big\|Z_s^{i,m,j,\ell,N}\big\|^2 + \big\|Z_s^{i,m,j,\star,N}\big\|^2 + \big\|Z_s^{i,n,j,\ell,N}\big\|^2 + \big\|Z_s^{i,n,j,\star,N}\big\|^2 \Big) \d s \Bigg], \; \P\text{--a.s.}, \; t \in [0,T].
\end{align*}
This, together with the integrability of the processes $Z_s^{i,m,j,\ell,N}$, ${Z}^{i,m,j,\star,N}$, $Z_s^{i,n,j,\ell,N}$ and ${Z}^{i,n,j,\star,N}$, as ensured by the estimates in \Cref{align:estimatesBSDEsupSuff}, implies the existence of some $\ell_\varepsilon >0$ such that the absolute value is smaller than $\varepsilon \ell$. We conclude that for $(\ell,t,\alpha) \in (0,\ell_\varepsilon) \times [0,T] \times \cA_N$, 
\begin{align*}
J^i(t,\hat\alpha^{i,N};\hat\balpha^{N,\smallertext{-}i}) - J^i(t,\alpha^i_{t,\ell};\hat\balpha^{N,\smallertext{-}i}) \geq - \varepsilon \ell, \; \P\text{\rm--a.s.}
\end{align*}
\end{proof}

\section{Auxiliary results}\label{appendix:auxResult}

\begin{proof}[Proof of \Cref{lemma:momBoundAuxiliarySystem}]
We derive several bounds for the auxiliary system introduced in \Cref{align:intermediateSystemNplayerGame}, which plays a key role in the proof of \Cref{theorem:convergenceTheorem}. We first establish estimates for the forward component. Specifically, we notice that for any $p \geq 1$, any $i\in\{1,\dots,N\}$, any $t \in [u,T]$, and for $\P\text{\rm--a.e.} \; \omega\in\Omega$, the following holds:
\begin{align*}
&\E^{\P^{\smalltext{\hat\balpha}^\tinytext{N}\smalltext{,}\smalltext{N}\smalltext{,}\smalltext{u}}_\smalltext{\omega}} \Big[ \big\|\widetilde X^{i}_{\cdot \land t}\big\|^p_\infty \Big] \\
&= \E^{\P^{\smalltext{\hat\balpha}^\tinytext{N}\smalltext{,}\smalltext{N}\smalltext{,}\smalltext{u}}_\smalltext{\omega}} \bigg[ \sup_{s\in[u,t]} \bigg\|X^i_u(\omega) + \int_u^s \sigma_r(\widetilde X^{i}_{\cdot \land r}) b_r\big(\widetilde X^{i}_{\cdot \land r},L^N\big(\widetilde \X^N_{\cdot \land r},\widetilde\balpha^N_r\big),\widetilde\alpha^{i,N}_r\big) \d r + \int_u^t \sigma_r(\widetilde X^{i}_{\cdot \land r}) \d \big(W_r^{\hat\balpha^\smalltext{N},N,u,\omega}\big)^i\bigg\|^p \bigg] \\
&\leq 3^{p-1} \|X^i_u(\omega)\|^p + 3^{p-1} \E^{\P^{\smalltext{\hat\balpha}^\tinytext{N}\smalltext{,}\smalltext{N}\smalltext{,}\smalltext{u}}_\smalltext{\omega}} \bigg[ \sup_{s\in[u,t]} \bigg\| \int_u^s \sigma_r(\widetilde X^{i}_{\cdot \land r}) b_r\big(\widetilde X^{i}_{\cdot \land r},L^N\big(\widetilde \X^N_{\cdot \land r},\widetilde\balpha^N_r\big),\widetilde\alpha^{i,N}_r\big) \d r \bigg\| \bigg] \\
&\quad+ 3^{p-1}\E^{\P^{\smalltext{\hat\balpha}^\tinytext{N}\smalltext{,}\smalltext{N}\smalltext{,}\smalltext{u}}_\smalltext{\omega}} \bigg[ \sup_{s\in[u,t]} \bigg\| \int_u^s \sigma_r(\widetilde X^{i}_{\cdot \land r}) \d \big(W_r^{\hat\balpha^\smalltext{N},N,u,\omega}\big)^i\bigg\|^p \bigg].
\end{align*}

Given that the drift function $b$ is bounded, there exists a constant $c_b>0$ such that, together with \Cref{assumpConvThm}.\ref{lipSigma} and the consistency of the spectral norm with the Euclidean norm, it follows that
\begin{align*}
&\E^{\P^{\smalltext{\hat\balpha}^\tinytext{N}\smalltext{,}\smalltext{N}\smalltext{,}\smalltext{u}}_\smalltext{\omega}} \Big[ \big\|\widetilde X^{i}_{\cdot \land t}\big\|^p_\infty \Big] &\\
&\leq 3^{p-1} \|X^{i}_u(\omega)\|^p + 3^{p-1} c^p_b \ell^p_\sigma \E^{\P^{\smalltext{\hat\balpha}^\tinytext{N}\smalltext{,}\smalltext{N}\smalltext{,}\smalltext{u}}_\smalltext{\omega}} \bigg[\int_u^t \big( 1 + \big\|\widetilde X^{i}_{\cdot \land s}\big\|_\infty \big)^p \d s \bigg] + 3^{p-1} \E^{\P^{\smalltext{\hat\balpha}^\tinytext{N}\smalltext{,}\smalltext{N}\smalltext{,}\smalltext{u}}_\smalltext{\omega}} \Bigg[ \sup_{s \in [u,t]} \bigg\| \int_u^s \sigma_r(\widetilde X^{i}_{\cdot \land r}) \d \big(W_r^{\hat\balpha^\smalltext{N},N,u,\omega}\big)^i \bigg\|^p \Bigg] \\
&\leq 3^{p-1} \|X^{i}_u(\omega)\|^p + 3^{p-1} c^p_b \ell^p_\sigma \E^{\P^{\smalltext{\hat\balpha}^\tinytext{N}\smalltext{,}\smalltext{N}\smalltext{,}\smalltext{u}}_\smalltext{\omega}} \bigg[ \int_u^t \big( 1 + \big\|\widetilde X^{i}_{\cdot \land s}\big\|_\infty \big)^p \d s \bigg] + 3^{p-1} c_{p,{\smallertext{\rm BDG}}} \E^{\P^{\smalltext{\hat\balpha}^\tinytext{N}\smalltext{,}\smalltext{N}\smalltext{,}\smalltext{u}}_\smalltext{\omega}} \bigg[\int_u^t\big \|\sigma_s(\widetilde X^{i}_{\cdot \land s})\big\|^2 \d s \bigg] \\
&\leq 3^{p-1} \|X^{i}_u(\omega)\|^p + 3^{p-1} \big(c^p_b+c_{p,{\smallertext{\rm BDG}}}\big) \ell^p_\sigma \E^{\P^{\smalltext{\hat\balpha}^\tinytext{N}\smalltext{,}\smalltext{N}\smalltext{,}\smalltext{u}}_\smalltext{\omega}} \bigg[ \int_u^t \big( 1 + \big\|\widetilde X^{i}_{\cdot \land s}\big\|_\infty \big)^p \d s \bigg] \\
&\leq 3^{p-1} \|X^{i}_u(\omega)\|^p + 6^{p-1} \big(c^p_b+c_{p,{\smallertext{\rm BDG}}}\big) \ell^p_\sigma \bigg( t + \int_u^t \E^{\P^{\smalltext{\hat\balpha}^\tinytext{N}\smalltext{,}\smalltext{N}\smalltext{,}\smalltext{u}}_\smalltext{\omega}} \Big[ \big\|\widetilde X^{i}_{\cdot \land s}\big\|_\infty^p \Big] \d s \bigg), \; \P\text{\rm--a.e.} \; \omega\in\Omega,
\end{align*}
where the second inequality follows from the Burkholder--Davis--Gundy's inequality with constant $c_{p,{\smallertext{\rm BDG}}}$. Applying Gr\"onwall's inequality yields
\begin{align}\label{align:tildeX_boundP}
\E^{\P^{\smalltext{\hat\balpha}^\tinytext{N}\smalltext{,}\smalltext{N}\smalltext{,}\smalltext{u}}_\smalltext{\omega}} \Big[ \big\|\widetilde X^{i}_{\cdot \land t}\big\|_\infty^p \Big] \leq 3^{p-1} \Big( \|X^{i}_u(\omega)\|^p + 2^{p-1} \big( c^p_b + c_{p,{\smallertext{\rm BDG}}} \big) \ell^p_\sigma t \Big) \mathrm{e}^{6^{\smalltext{p}\smalltext{-}\smalltext{1}} ( c^\smalltext{p}_\smalltext{b} + c_{\smalltext{p}\smalltext{,}{\smallertext{\rm BDG}}} ) \ell^\smalltext{p}_\smalltext{\sigma} t}, \; \P\text{\rm--a.e.} \; \omega\in\Omega.
\end{align}
Similarly, for $\P\text{\rm--a.e.} \; \omega\in\Omega$,
\begin{align}\label{align:tildeSumX_boundP}
\notag \E^{\P^{\smalltext{\hat\balpha}^\tinytext{N}\smalltext{,}\smalltext{N}\smalltext{,}\smalltext{u}}_\smalltext{\omega}} \Bigg[\int_u^T\mathrm{e}^{\beta t}\sum_{\ell =1}^N\big\| \widetilde X^{\ell}_{\cdot\wedge t}\big\|^p_\infty\mathrm{d}t \Bigg] 
&\leq 3^{p-1} \int_u^T \mathrm{e}^{\beta t} \sum_{\ell =1}^N \bigg( \Big( \|X^{\ell}_u(\omega)\|^p + 2^{p-1} \big( c^p_b + c_{p,{\smallertext{\rm BDG}}} \big) \ell^p_\sigma t \Big) \mathrm{e}^{6^{\smalltext{p}\smalltext{-}\smalltext{1}} ( c^\smalltext{p}_\smalltext{b} + c_{\smalltext{p}\smalltext{,}{\smallertext{\rm BDG}}} ) \ell^\smalltext{p}_\sigma t} \bigg) \mathrm{d}t \\
&\leq 3^{p-1}T \mathrm{e}^{\beta T + 6 ( c^\smalltext{p}_\smalltext{b} + c_{\smalltext{p}\smalltext{,}{\smallertext{\rm BDG}}} ) \ell^\smalltext{p}_\smalltext{\sigma} T}\Bigg( \sum_{\ell =1}^N \big\|X^{\ell,N}_u(\omega)\big\|^p  + 2^{p-1} \big( c^p_b + c_{p,{\smallertext{\rm BDG}}} \big) \ell^p_\sigma  TN\Bigg).
\end{align}
For notational simplicity, let
\begin{gather*}
c^1_p\coloneqq 3^{p-1} \mathrm{e}^{6^{\smalltext{p}\smalltext{-}\smalltext{1}} ( c^\smalltext{p}_\smalltext{b} + c_{\smalltext{p}\smalltext{,}{\smallertext{\rm BDG}}} ) \ell^\smalltext{p}_\smalltext{\sigma} T},\; \bar{c}^1_p\coloneqq \mathrm{e}^{6^{\smalltext{p}\smalltext{-}\smalltext{1}} ( c^\smalltext{p}_\smalltext{b} + c_{\smalltext{p}\smalltext{,}{\smallertext{\rm BDG}}} ) \ell^\smalltext{p}_\smalltext{\sigma} T}6^{p-1} \big( c^p_b + c_{p,{\smallertext{\rm BDG}}} \big) \ell^p_\sigma T, \\
c^2_{p}\coloneqq 3^{p-1}T \mathrm{e}^{\beta T + 6 ( c^\smalltext{p}_\smalltext{b} + c_{\smalltext{p}\smalltext{,}{\smallertext{\rm BDG}}} ) \ell^\smalltext{p}_\smalltext{\sigma} T},\; \bar{c}^2_p\coloneqq c^2_p2^{p-1} \big( c^p_b + c_{p,{\smallertext{\rm BDG}}} \big) \ell^p_\sigma  T.
\end{gather*}

Since the dynamics of the forward component $X^{i}$, described in \eqref{align:systemNplayerGame_rcpd}, are analogous to those of $\widetilde X^{i}$, the same estimates apply, and in particular, for $\P\text{\rm--a.e.} \; \omega\in\Omega$, we have
\begin{align*}
\E^{\P^{\smalltext{\hat\balpha}^\tinytext{N}\smalltext{,}\smalltext{N}\smalltext{,}\smalltext{u}}_\smalltext{\omega}} \bigg[\int_u^T\mathrm{e}^{\beta t} \big\| X^{i}_{\cdot\wedge t}\big\|^2_\infty\mathrm{d}t \bigg] \leq c^2_{2} \|X^{i}_u(\omega)\|^2 + \bar{c}^2_{2}, \; \E^{\P^{\smalltext{\hat\balpha}^\tinytext{N}\smalltext{,}\smalltext{N}\smalltext{,}\smalltext{u}}_\smalltext{\omega}} \Bigg[\int_u^T\mathrm{e}^{\beta t}\sum_{\ell =1}^N\big\| X^{\ell}_{\cdot\wedge t}\big\|^2_\infty\mathrm{d}t \Bigg] \leq c^2_2 \sum_{\ell =1}^N \big\|X^{\ell,N}_u(\omega)\big\|^2  + \bar{c}^2_2 N.
\end{align*}

For the bounds on the $Z$-components of the martingale terms $\widetilde M^{i,\star,N}$, for $i \in \{1,\ldots,N\}$, and $N^{\star,N}$, we use the fact that the functions $\varphi_1$ and $\varphi_2$ are assumed to be bounded, as stated in \Cref{assumpConvThm}.\ref{previousAssum} or equivalently in \Cref{assumpDPP_NplayerGame}.\ref{boundedPhisMart_NplayerGame}. Following the computations carried out in \textbf{Step 1} of \Cref{subsubsection:fromNplayerGameToAuxiliary}, it follows that there exist two constants $c_{\varphi_\smalltext{1}}>0$ and $c_{\varphi_\smalltext{2}}>0$ such that, for $\P\text{\rm--a.e.} \; \omega\in\Omega$, we have
\begin{align}\label{align:estZnAuxiliary}
\notag& \E^{\P^{\smalltext{\hat\balpha}^\tinytext{N}\smalltext{,}\smalltext{N}\smalltext{,}\smalltext{u}}_\smalltext{\omega}} \Bigg[ \int_u^T \mathrm{e}^{\beta t} \sum_{\ell =1}^N \big\|\widetilde Z^{i,m,\ell,\star,N}_t\big\|^2 \d t \Bigg] \leq \E^{\P^{\smalltext{\hat\balpha}^\tinytext{N}\smalltext{,}\smalltext{N}\smalltext{,}\smalltext{u}}_\smalltext{\omega}} \Big[ \mathrm{e}^{\beta T} \big|\varphi_1(\widetilde X^{i}_{\cdot \land T}) \big|^2 \Big] \leq \mathrm{e}^{\beta T} c^2_{\varphi_\smalltext{1}}, \\
& \E^{\P^{\smalltext{\hat\balpha}^\tinytext{N}\smalltext{,}\smalltext{N}\smalltext{,}\smalltext{u}}_\smalltext{\omega}} \Bigg[ \int_u^T \mathrm{e}^{\beta t} \sum_{\ell =1}^N \big\|\widetilde Z^{n,\ell,\star,N}_t\big\|^2 \d t \Bigg] \leq \E^{\P^{\smalltext{\hat\balpha}^\tinytext{N}\smalltext{,}\smalltext{N}\smalltext{,}\smalltext{u}}_\smalltext{\omega}} \Big[ \mathrm{e}^{\beta T} \big| \varphi_2\big(L^N\big(\widetilde \X^N_{\cdot \land T}\big)\big) \big|^2 \Big] \leq \mathrm{e}^{\beta T} c^2_{\varphi_\smalltext{2}}.
\end{align}

Thus, for $\P\text{\rm--a.e.} \; \omega\in\Omega$
\begin{align*}
& \E^{\P^{\smalltext{\hat\balpha}^\tinytext{N}\smalltext{,}\smalltext{N}\smalltext{,}\smalltext{u}}_\smalltext{\omega}} \Bigg[ \int_u^T \mathrm{e}^{\beta t} \sum_{(i,\ell) \in \{1,\ldots,N\}^\smalltext{2}} \big\|\widetilde Z^{i,m,\ell,\star,N}_t\big\|^2 \d t \Bigg] \leq \mathrm{e}^{\beta T} c^2_{\varphi_\smalltext{1}} N.
\end{align*}

\medskip
Before proceeding with the computations for the bounds on $\widetilde{Z}^{i,\ell,N}$, with $(i,\ell) \in \{1,\ldots,N\}^2$, we first translate the growth conditions on the functions $f$ and $g$ stated in \Cref{assumpConvThm}.\ref{growth_f_gG}. We begin by noting that the compactness assumption on the set $A$, combined with the growth condition just recalled, ensures that there exists a constant $c_A>0$ such that
\begin{align*}
\Big|f_t\big(\widetilde X^{i}_{\cdot \land t},L^N\big(\widetilde \X^N_{\cdot \land t},\widetilde\balpha^N_t\big),\widetilde\alpha^{i,N}_t\big)\Big| \leq \ell_{f} \Bigg( 1 + \|\widetilde X^{i}_{\cdot \land t}\|^{\bar{p}}_\infty + \frac{1}{N} \sum_{\ell =1}^N \big\|\widetilde X^{\ell}_{\cdot \land s}\big\|^{\bar{p}}_\infty + c_A \Bigg).
\end{align*}
Meanwhile, the boundedness of $\varphi_1$ and $\varphi_2$ implies the existence of a constant $\ell_{g+G,\varphi_\smalltext{1},\varphi_\smalltext{2}} \geq \ell_g$ such that
\begin{align*}
\Big|g\big(\widetilde X^{i}_{\cdot \land T}, L^N\big(\widetilde \X^N_{\cdot \land T}\big)\big) + G\big(\varphi_1(\widetilde X^{i}_{\cdot \land T}),\varphi_2\big(L^N\big(\widetilde \X^N_{\cdot \land T}\big)\big)\big)\Big| \leq \ell_{g+G,\varphi_\smalltext{1},\varphi_\smalltext{2}} \Bigg( 1 + \big\|\widetilde X^{i}_{\cdot \land T}\big\|^{\bar{p}}_\infty + \frac{1}{N} \sum_{\ell \in \{1,\ldots, N\}} \big\|\widetilde X^{\ell}_{\cdot \land T}\big\|^{\bar{p}}_\infty \Bigg) .
\end{align*}

Let us define
\begin{align*}
c^\prime_{\smallertext{\rm{BMO}}_{\smalltext{[}\smalltext{u}\smalltext{,}\smalltext{T}\smalltext{]}}} &\coloneqq \frac{3c^2_{\partial^\smalltext{2}G}}{4} \big\|\widetilde M^{i,\star,N}\big\|^2_{\smallertext{\rm{BMO}}_{\smalltext{[}\smalltext{u}\smalltext{,}\smalltext{T}\smalltext{]}}} + \frac{c^2_{\partial^\smalltext{2}G}}{4} \big\|\widetilde N^{\star,N}\big\|^2_{\smallertext{\rm{BMO}}_{\smalltext{[}\smalltext{u}\smalltext{,}\smalltext{T}\smalltext{]}}}.
\end{align*}
Although the notation has been slightly abused for simplicity, it is clear that all the constant just introduced is uniformly bounded with respect to $N$. This follows from the boundedness of $\varphi^1$ and $\varphi^2$, as well as the boundedness of the drift function $b$. Then, arguing similarly to \Cref{eq:supDeltaY}, we obtain that
\begin{align*}
\notag&\big( 1 - \varepsilon_7 4c^2_{1,\smallertext{\rm BDG}} - \varepsilon_8 c^\prime_{\smallertext{\rm{BMO}}_{\smalltext{[}\smalltext{u}\smalltext{,}\smalltext{T}\smalltext{]}}} \big) \E^{\P^{\smalltext{\hat\balpha}^\tinytext{N}\smalltext{,}\smalltext{N}\smalltext{,}\smalltext{u}}_\smalltext{\omega}} \bigg[ \sup_{t \in [u,T]} \mathrm{e}^{\beta t} \big|\widetilde Y^{i,N}_t\big|^2 \bigg] \\
		\notag&\leq 3 \ell_{g+G,\varphi_\smalltext{1},\varphi_\smalltext{2}} \E^{\P^{\smalltext{\hat\balpha}^\tinytext{N}\smalltext{,}\smalltext{N}\smalltext{,}\smalltext{u}}_\smalltext{\omega}} \Bigg[ \mathrm{e}^{\beta T} \Bigg( 1 + \big\|\widetilde X^{i}_{\cdot \land T}\big\|^{2\bar{p}}_\infty + \frac{1}{N} \sum_{\ell =1}^N \big\|\widetilde X^{\ell}_{\cdot \land T}\big\|^{2\bar{p}}_\infty \Bigg) \Bigg] \\
&\quad+ \big( \ell_f^2(1+c_A)^2+ 2\ell^2_f - \beta \big) \E^{\P^{\smalltext{\hat\balpha}^\tinytext{N}\smalltext{,}\smalltext{N}\smalltext{,}\smalltext{u}}_\smalltext{\omega}} \bigg[ \int_u^T \mathrm{e}^{\beta t} \big|\widetilde Y^{i,N}_t\big|^2 \d t \bigg] + \E^{\P^{\smalltext{\hat\balpha}^\tinytext{N}\smalltext{,}\smalltext{N}\smalltext{,}\smalltext{u}}_\smalltext{\omega}} \Bigg[ \int_u^T \mathrm{e}^{\beta t} \Bigg(1+ \big\|\widetilde X^{i}_{\cdot \land t}\big\|^{2\bar{p}}_\infty + \frac{1}{N} \sum_{\ell =1}^N  \big\|\widetilde X^{\ell}_{\cdot \land t}\big\|^{2\bar{p}}_\infty \Bigg) \d t \Bigg] \\
\notag&\quad+ \E^{\P^{\smalltext{\hat\balpha}^\tinytext{N}\smalltext{,}\smalltext{N}\smalltext{,}\smalltext{u}}_\smalltext{\omega}} \Bigg[ \int_u^T \mathrm{e}^{\beta t} \sum_{\ell=1}^N \bigg( \frac{1}{\varepsilon_7} \big\|\widetilde Z^{i,\ell,N}_t\big\|^2 + \frac{1}{\varepsilon_8} \big\|\widetilde Z^{i,m,\ell,\star,N}_t\big\|^2 + \frac{3}{\varepsilon_8} \big\|\widetilde Z^{n,\ell,\star,N}_t\big\|^2 \bigg) \d t \Bigg], \; \P\text{\rm--a.e.} \; \omega\in\Omega,
\end{align*}
for some $\varepsilon_7>0$ and $\varepsilon_8>0$. Analogously to \Cref{eq:intDeltaZ}, for some $\varepsilon_9>0$, and for any $t \in [u,T]$, we have that
\begin{align*}
\notag&\E^{\P^{\smalltext{\hat\balpha}^\tinytext{N}\smalltext{,}\smalltext{N}\smalltext{,}\smalltext{u}}_\smalltext{\omega}} \Bigg[ \int_u^T \mathrm{e}^{\beta t} \sum_{\ell=1}^N \big\|\widetilde Z^{i,\ell,N}_t\big\|^2 \d t \Bigg] \\
	\notag&\leq 3 \ell_{g+G,\varphi_\smalltext{1},\varphi_\smalltext{2}} \E^{\P^{\smalltext{\hat\balpha}^\tinytext{N}\smalltext{,}\smalltext{N}\smalltext{,}\smalltext{u}}_\smalltext{\omega}} \Bigg[ \mathrm{e}^{\beta T} \Bigg( 1+\big\|\widetilde X^{i}_{\cdot \land T}\big\|^{2\bar{p}}_\infty + \frac{1}{N} \sum_{\ell =1}^N \big\|\widetilde X^{\ell}_{\cdot \land T}\big\|^{2\bar{p}}_\infty \Bigg) \Bigg] + \big( \ell_f^2(1+c_A)^2+ 2\ell^2_f - \beta \big) \E^{\P^{\smalltext{\hat\balpha}^\tinytext{N}\smalltext{,}\smalltext{N}\smalltext{,}\smalltext{u}}_\smalltext{\omega}} \bigg[ \int_u^T \mathrm{e}^{\beta t} \big|\widetilde Y^{i,N}_t\big|^2 \d t \bigg]\\
&\quad + \E^{\P^{\smalltext{\hat\balpha}^\tinytext{N}\smalltext{,}\smalltext{N}\smalltext{,}\smalltext{u}}_\smalltext{\omega}} \Bigg[ \int_u^T \mathrm{e}^{\beta t} \Bigg(1+ \big\|\widetilde X^{i}_{\cdot \land t}\big\|^{2\bar{p}}_\infty + \frac{1}{N} \sum_{\ell =1}^N \big\|\widetilde X^{\ell}_{\cdot \land t}\big\|^{2\bar{p}}_\infty \Bigg) \d t \Bigg] + \frac{1}{\varepsilon_9} \E^{\P^{\smalltext{\hat\balpha}^\tinytext{N}\smalltext{,}\smalltext{N}\smalltext{,}\smalltext{u}}_\smalltext{\omega}} \Bigg[ \int_u^T \mathrm{e}^{\beta t} \sum_{\ell =1}^N \big\|\widetilde Z^{i,m,\ell,\star,N}_t\big\|^2 \d t \Bigg] \\
\notag&\quad+ \frac{3}{\varepsilon_9} \E^{\P^{\smalltext{\hat\balpha}^\tinytext{N}\smalltext{,}\smalltext{N}\smalltext{,}\smalltext{u}}_\smalltext{\omega}} \Bigg[ \int_u^T \mathrm{e}^{\beta t} \sum_{\ell =1}^N \big\|\widetilde Z^{n,\ell,\star,N}_t\big\|^2 \d t \Bigg] + \varepsilon_9 c^\prime_{\smallertext{\rm{BMO}}_{\smalltext{[}\smalltext{u}\smalltext{,}\smalltext{T}\smalltext{]}}} \E^{\P^{\smalltext{\hat\balpha}^\tinytext{N}\smalltext{,}\smalltext{N}\smalltext{,}\smalltext{u}}_\smalltext{\omega}} \bigg[ \sup_{t \in [u,T]} \mathrm{e}^{\beta t} \big|\widetilde Y^{i,N}_t\big|^2 \bigg], \; \P\text{\rm--a.e.} \; \omega\in\Omega.
\end{align*}
Define
\[
c_{\eps_{\smalltext{7}\smalltext{,}\smalltext{8}\smalltext{,}\smalltext{9}}}\coloneqq \frac{ \varepsilon_9 c^\prime_{\smallertext{\rm{BMO}}_{\smalltext{[}\smalltext{u}\smalltext{,}\smalltext{T}\smalltext{]}}} }{1 - \varepsilon_7 4c^2_{1,\smallertext{\rm BDG}} - \varepsilon_8 c^\prime_{\smallertext{\rm{BMO}}_{\smalltext{[}\smalltext{u}\smalltext{,}\smalltext{T}\smalltext{]}}}}.
\]
By combining the previous two inequalities, we deduce
\begin{align*}
\E^{\P^{\smalltext{\hat\balpha}^\tinytext{N}\smalltext{,}\smalltext{N}\smalltext{,}\smalltext{u}}_\smalltext{\omega}} \Bigg[ \int_u^T \mathrm{e}^{\beta t} \sum_{\ell =1}^N \big\|\widetilde Z^{i,\ell,N}_t\big\|^2 \d t \Bigg] &\leq 3 \ell_{g+G,\varphi_\smalltext{1},\varphi_\smalltext{2}} \big( 1 + c_{\eps_{\smalltext{7}\smalltext{,}\smalltext{8}\smalltext{,}\smalltext{9}}} \big) \E^{\P^{\smalltext{\hat\balpha}^\tinytext{N}\smalltext{,}\smalltext{N}\smalltext{,}\smalltext{u}}_\smalltext{\omega}} \Bigg[ \mathrm{e}^{\beta T} \Bigg( 1 + \big\|\widetilde X^{i}_{\cdot \land T}\big\|^{2\bar{p}}_\infty + \frac{1}{N} \sum_{\ell=1}^N \big\|\widetilde X^{\ell}_{\cdot \land T}\big\|^{2\bar{p}}_\infty \Bigg) \Bigg] \\
&\quad+ \big( \ell_f^2(1+c_A)^2+ 2\ell^2_f - \beta \big) \big( 1 + c_{\eps_{\smalltext{7}\smalltext{,}\smalltext{8}\smalltext{,}\smalltext{9}}} \big) \E^{\P^{\smalltext{\hat\balpha}^\tinytext{N}\smalltext{,}\smalltext{N}\smalltext{,}\smalltext{u}}_\smalltext{\omega}} \bigg[ \int_u^T \mathrm{e}^{\beta t} \big|\widetilde Y^{i,N}_t\big|^2 \d t \bigg] \\
&\quad+ \big( 1+  c_{\eps_{\smalltext{7}\smalltext{,}\smalltext{8}\smalltext{,}\smalltext{9}}} \big) \E^{\P^{\smalltext{\hat\balpha}^\tinytext{N}\smalltext{,}\smalltext{N}\smalltext{,}\smalltext{u}}_\smalltext{\omega}} \Bigg[ \int_u^T \mathrm{e}^{\beta t} \Bigg(1+ \big\|\widetilde X^{i}_{\cdot \land t}\big\|^{2\bar{p}}_\infty + \frac{1}{N} \sum_{\ell =1}^N \big\|\widetilde X^{\ell}_{\cdot \land t}\big\|^{2\bar{p}}_\infty \Bigg) \d t \Bigg] \\
\notag&\quad+ \bigg( \frac{1}{\varepsilon_9} + \frac{ c_{\eps_{\smalltext{7}\smalltext{,}\smalltext{8}\smalltext{,}\smalltext{9}}} }{\eps_8}\bigg) \E^{\P^{\smalltext{\hat\balpha}^\tinytext{N}\smalltext{,}\smalltext{N}\smalltext{,}\smalltext{u}}_\smalltext{\omega}} \Bigg[ \int_u^T \mathrm{e}^{\beta t} \sum_{\ell =1}^N \big\|\widetilde Z^{i,m,\ell,\star,N}_t\big\|^2 \d t \Bigg] \\
&\quad+ 3 \bigg( \frac{1}{\varepsilon_9} + \frac{ c_{\eps_{\smalltext{7}\smalltext{,}\smalltext{8}\smalltext{,}\smalltext{9}}} }{\eps_8}\bigg) \E^{\P^{\smalltext{\hat\balpha}^\tinytext{N}\smalltext{,}\smalltext{N}\smalltext{,}\smalltext{u}}_\smalltext{\omega}} \Bigg[ \int_u^T \mathrm{e}^{\beta t} \sum_{\ell =1}^N \big\|\widetilde Z^{n,\ell,\star,N}_t\big\|^2 \d t \Bigg] \\
&\quad+\frac{ c_{\eps_{\smalltext{7}\smalltext{,}\smalltext{8}\smalltext{,}\smalltext{9}}} }{\eps_7}\E^{\P^{\smalltext{\hat\balpha}^\tinytext{N}\smalltext{,}\smalltext{N}\smalltext{,}\smalltext{u}}_\smalltext{\omega}} \Bigg[ \int_u^T \mathrm{e}^{\beta t} \sum_{\ell =1}^N \big\|\widetilde Z^{i,\ell,N}_t\big\|^2 \d t \Bigg], \; \P\text{\rm--a.e.} \; \omega\in\Omega.
\end{align*}
Under the condition that $\beta >  \ell_f^2(1+c_A)^2+ 2\ell^2_f $, it holds that
\begin{align}\label{eq:estim97}
\notag&\bigg( 1 - \frac{c_{\eps_{\smalltext{7}\smalltext{,}\smalltext{8}\smalltext{,}\smalltext{9}}}}{\eps_7}\bigg) \E^{\P^{\smalltext{\hat\balpha}^\tinytext{N}\smalltext{,}\smalltext{N}\smalltext{,}\smalltext{u}}_\smalltext{\omega}} \Bigg[ \int_u^T \mathrm{e}^{\beta t} \sum_{\ell =1}^N \big\|\widetilde Z^{i,\ell,N}_t\big\|^2 \d t \Bigg] \\
\notag&\leq 3 \ell_{g+G,\varphi_\smalltext{1},\varphi_\smalltext{2}} \big( 1 + c_{\eps_{\smalltext{7}\smalltext{,}\smalltext{8}\smalltext{,}\smalltext{9}}}\big) \E^{\P^{\smalltext{\hat\balpha}^\tinytext{N}\smalltext{,}\smalltext{N}\smalltext{,}\smalltext{u}}_\smalltext{\omega}} \Bigg[ \mathrm{e}^{\beta T} \bigg( 1 + \big\|\widetilde X^{i}_{\cdot \land T}\big\|^{2\bar{p}}_\infty + \frac{1}{N} \sum_{\ell =1}^N \big\|\widetilde X^{\ell}_{\cdot \land T}\big\|^{2\bar{p}}_\infty \bigg) \Bigg] \\
\notag&\quad+ \big( 1 +c_{\eps_{\smalltext{7}\smalltext{,}\smalltext{8}\smalltext{,}\smalltext{9}}} \big) \E^{\P^{\smalltext{\hat\balpha}^\tinytext{N}\smalltext{,}\smalltext{N}\smalltext{,}\smalltext{u}}_\smalltext{\omega}} \Bigg[ \int_u^T \mathrm{e}^{\beta t} \Bigg(1+ \big\|\widetilde X^{i}_{\cdot \land t}\big\|^{2\bar{p}}_\infty + \frac{1}{N} \sum_{\ell =1}^N \big\|\widetilde X^{\ell}_{\cdot \land t}\big\|^{2\bar{p}}_\infty  \Bigg) \d t \Bigg] \\
\notag&\quad+ \bigg( \frac{1}{\varepsilon_9} + \frac{c_{\eps_{\smalltext{7}\smalltext{,}\smalltext{8}\smalltext{,}\smalltext{9}}}}{\eps_8} \bigg) \E^{\P^{\smalltext{\hat\balpha}^\tinytext{N}\smalltext{,}\smalltext{N}\smalltext{,}\smalltext{u}}_\smalltext{\omega}} \Bigg[ \int_u^T \mathrm{e}^{\beta t} \sum_{\ell =1}^N \big\|\widetilde Z^{i,m,\ell,\star,N}_t\big\|^2 \d t \Bigg] \\
&\quad+ 3 \bigg( \frac{1}{\varepsilon_9} + \frac{c_{\eps_{\smalltext{7}\smalltext{,}\smalltext{8}\smalltext{,}\smalltext{9}}}}{\eps_8} \bigg) \E^{\P^{\smalltext{\hat\balpha}^\tinytext{N}\smalltext{,}\smalltext{N}\smalltext{,}\smalltext{u}}_\smalltext{\omega}} \Bigg[ \int_u^T \mathrm{e}^{\beta t} \sum_{\ell =1}^N \big\|\widetilde Z^{n,\ell,\star,N}_t\big\|^2 \d t \Bigg], \; \P\text{\rm--a.e.} \; \omega\in\Omega.
\end{align}

Substituting the estimates from \eqref{align:tildeX_boundP}, \eqref{align:tildeSumX_boundP} and \eqref{align:estZnAuxiliary} into \Cref{eq:estim97} leads to
\begin{align*}
& \bigg( 1 - \frac{c_{\eps_{\smalltext{7}\smalltext{,}\smalltext{8}\smalltext{,}\smalltext{9}}}}{\eps_7}\bigg) \E^{\P^{\smalltext{\hat\balpha}^\tinytext{N}\smalltext{,}\smalltext{N}\smalltext{,}\smalltext{u}}_\smalltext{\omega}} \Bigg[ \int_u^T \mathrm{e}^{\beta t} \sum_{\ell =1}^N \big\|\widetilde Z^{i,\ell,N}_t\big\|^2 \d t \Bigg] \\
&\leq 3\mathrm{e}^{\beta T} \ell_{g+G,\varphi_\smalltext{1},\varphi_\smalltext{2}} \big( 1 + c_{\eps_{\smalltext{7}\smalltext{,}\smalltext{8}\smalltext{,}\smalltext{9}}}\big)\Bigg( 1 + 2\bar{c}^1_{2\bar{p}}+c^1_{2\bar{p}} \|X^{i}_u(\omega)\|^{2\bar{p}}+ \frac{c^1_{2\bar{p}}}{N} \sum_{\ell=1}^N \| X^{\ell}_{u}(\omega)\|^{2\bar{p}} \Bigg) \\
&\quad+ \big( 1 +c_{\eps_{\smalltext{7}\smalltext{,}\smalltext{8}\smalltext{,}\smalltext{9}}} \big)\Bigg( 1+ 2\bar{c}^2_{2\bar{p}}+c^2_{2\bar{p}} \|X^{i}_u(\omega)\|^{2\bar{p}}+ \frac{c^2_{2\bar{p}}}{N} \sum_{\ell =1}^N \| X^{\ell}_{u}(\omega)\|^{2\bar{p}} \bigg)+ \mathrm{e}^{\beta T}\bigg( \frac{1}{\varepsilon_9} + \frac{c_{\eps_{\smalltext{7}\smalltext{,}\smalltext{8}\smalltext{,}\smalltext{9}}}}{\eps_8} \Bigg) \big(c^2_{\varphi_\smalltext{1}}+3c^2_{\varphi_\smalltext{2}}\big), \; \P\text{\rm--a.e.} \; \omega\in\Omega.
\end{align*}
If we define
\begin{gather*}
\bar{c}^1_{\eps_{{\smalltext{7}\smalltext{,}\smalltext{8}\smalltext{,}\smalltext{9}}}}\coloneqq\bigg( 1 - \frac{c_{\eps_{\smalltext{7}\smalltext{,}\smalltext{8}\smalltext{,}\smalltext{9}}}}{\eps_7}\bigg)^{-1}\bigg(3\mathrm{e}^{\beta T} \ell_{g+G,\varphi_\smalltext{1},\varphi_\smalltext{2}} \big( 1 + c_{\eps_{\smalltext{7}\smalltext{,}\smalltext{8}\smalltext{,}\smalltext{9}}}\big)\big( 1 + 2\bar{c}^1_{2\bar{p}}\big)+(1+2\bar{c}^2_{2\bar{p}}) \big( 1 +c_{\eps_{\smalltext{7}\smalltext{,}\smalltext{8}\smalltext{,}\smalltext{9}}} \big)+ \mathrm{e}^{\beta T}\bigg( \frac{1}{\varepsilon_9} + \frac{c_{\eps_{\smalltext{7}\smalltext{,}\smalltext{8}\smalltext{,}\smalltext{9}}}}{\eps_8} \bigg) \big(c^2_{\varphi_\smalltext{1}}+3c^2_{\varphi_\smalltext{2}}\big)\bigg),\\
\bar{c}^2_{\eps_{{\smalltext{7}\smalltext{,}\smalltext{8}\smalltext{,}\smalltext{9}}}}\coloneqq \bigg( 1 - \frac{c_{\eps_{\smalltext{7}\smalltext{,}\smalltext{8}\smalltext{,}\smalltext{9}}}}{\eps_7}\bigg)^{-1}\Big(3\mathrm{e}^{\beta T} \ell_{g+G,\varphi_\smalltext{1},\varphi_\smalltext{2}} \big( 1 + c_{\eps_{\smalltext{7}\smalltext{,}\smalltext{8}\smalltext{,}\smalltext{9}}}\big)c^1_{2\bar{p}}+ \big( 1 +c_{\eps_{\smalltext{7}\smalltext{,}\smalltext{8}\smalltext{,}\smalltext{9}}} \big)c^2_{2\bar{p}}\Big),
\end{gather*}
then
\begin{align*}
& \E^{\P^{\smalltext{\hat\balpha}^\tinytext{N}\smalltext{,}\smalltext{N}\smalltext{,}\smalltext{u}}_\smalltext{\omega}} \Bigg[ \int_u^T \mathrm{e}^{\beta t} \sum_{\ell =1}^N \big\|\widetilde Z^{i,\ell,N}_t\big\|^2 \d t \Bigg]\leq \bar{c}^1_{\eps_{{\smalltext{7}\smalltext{,}\smalltext{8}\smalltext{,}\smalltext{9}}}}+\bar{c}^2_{\eps_{{\smalltext{7}\smalltext{,}\smalltext{8}\smalltext{,}\smalltext{9}}}}\Bigg(\|X^i_u(\omega)\|^{2\bar{p}}+ \frac{1}{N} \sum_{\ell=1}^N \| X^{\ell}_{u}(\omega)\|^{2\bar{p}} \Bigg), \; \P\text{\rm--a.e.} \; \omega\in\Omega.
\end{align*}
\end{proof}

\end{appendix}

\end{document}